\documentclass[reqno]{amsart}

\usepackage[mathscr]{euscript}
\usepackage{amssymb}
\usepackage{graphicx}
\usepackage[shortlabels]{enumitem}
\usepackage{mathtools}
\usepackage{booktabs}
\usepackage[dvipsnames]{xcolor}
\usepackage{tikz}
\usepackage{tikz-3dplot}
\usetikzlibrary{positioning,cd,arrows, patterns}
\usepackage{subfigure}
\usepackage{ragged2e}
\usepackage{a4wide}
\usepackage{hyperref}
\hypersetup{colorlinks}
\hypersetup{
    unicode=false,          
    pdftoolbar=true,        
    pdfmenubar=true,        
    pdffitwindow=false,     
    pdfstartview={FitH},    
    pdftitle={My title},    
    pdfauthor={Author},     
    pdfsubject={Subject},   
    pdfcreator={Creator},   
    pdfproducer={Producer}, 
    pdfkeywords={keyword1} {key2} {key3}, 
    pdfnewwindow=true,      
    colorlinks=true,       
    linkcolor=blue,          
    citecolor=red,        
    filecolor=violet,      
    urlcolor=violet       
}
\usepackage{caption} 
\theoremstyle{plain}
\newtheorem{theorem}{Theorem}[section]
\newtheorem{lemma}[theorem]{Lemma}
\newtheorem{proposition}[theorem]{Proposition}
\newtheorem{corollary}[theorem]{Corollary}

\theoremstyle{remark}
\newtheorem{remark}[theorem]{Remark}
\newtheorem{question}[theorem]{Question}

\theoremstyle{definition}
\newtheorem{definition}[theorem]{Definition}

\newtheorem{example}[theorem]{Example}

\DeclareMathOperator{\mmon}{\mu mon}
\DeclareMathOperator{\cl}{Cl}
\DeclareMathOperator{\Sh}{\mathscr{S}\mathsf{h}}
\DeclareMathOperator{\Loc}{\mathscr{L}\mathsf{oc}}

\newcommand\R{\mathbb{R}}
\newcommand\C{\mathbb{C}}

\newcommand\Z{\mathbb{Z}}

\newcommand\bfx{\mathbf{x}}

\newcommand\cA{\mathcal{A}}

\newcommand\cF{\mathcal{F}}
\newcommand\cM{\mathcal{M}}

\newcommand\sfI{\mathsf{I}}
\newcommand\sfT{\mathsf{T}}
\newcommand\sfY{\mathsf{Y}}

\newcommand{\clusterfont}{\mathcal}
\newcommand{\dynkinfont}{\mathsf}
\newcommand{\ngraphfont}{\mathscr}

\newcommand{\quiver}{\clusterfont{Q}}
\newcommand{\qbasis}{\clusterfont{B}}
\newcommand{\qcoxeter}{\mutation_{\quiver}}

\newcommand{\ngraph}{\ngraphfont{G}}
\newcommand{\nbasis}{\ngraphfont{B}}
\newcommand{\ncoxeter}{\mutation_\ngraph}
\newcommand{\coxeterpadding}{\ngraphfont{C}}

\newcommand{\dynA}{\dynkinfont{A}}
\newcommand{\dynB}{\dynkinfont{B}}
\newcommand{\dynC}{\dynkinfont{C}}
\newcommand{\dynD}{\dynkinfont{D}}
\newcommand{\dynE}{\dynkinfont{E}}
\newcommand{\dynF}{\dynkinfont{F}}
\newcommand{\dynG}{\dynkinfont{G}}
\newcommand{\dynX}{\dynkinfont{X}}
\newcommand{\dynY}{\dynkinfont{Y}}

\newcommand{\dynADE}{\dynkinfont{ADE}}
\newcommand{\dynBCFG}{\dynkinfont{BCFG}}

\newcommand{\facet}{\mathscr{F}}
\newcommand{\seed}{\Sigma}
\newcommand{\initialseed}{\Sigma_{t_0}}
\newcommand{\flags}{\mathcal{F}_\legendrian}
\newcommand{\mutation}{\mu}

\newcommand{\annulus}{\mathbb{A}}
\newcommand{\disk}{\mathbb{D}}
\newcommand{\sphere}{\mathbb{S}}

\newcommand{\legendrian}{\lambda}
\newcommand{\Legendrian}{\Lambda}

\newcommand{\boundary}{\partial}
\newcommand{\cycle}{\gamma}

\newcommand{\field}{\mathbb{F}}

\newcommand{\Roots}{\Phi}
\newcommand{\SRoots}{\Pi}

\newcommand{\Move}[1]{{\rm{(#1)}}}

\newcommand{\wavefront}{\Gamma}

\newcommand{\boundellipse}[3]
{(#1) ellipse (#2 and #3)}

\numberwithin{equation}{section}

\tikzset{ynode/.style = {circle, fill = yellow, inner sep = 2pt, opacity = 0.5}}
\tikzset{gnode/.style = {circle, fill=green, inner sep = 2pt, opacity = 0.5}}
\tikzstyle{Dnode}=[draw, circle, inner sep = 0.07cm]
\tikzstyle{double line} = [
double distance = 1.5pt, 
double=\pgfkeysvalueof{/tikz/commutative diagrams/background color}
]

\newlength{\starsize}
\newlength{\starspread}
\tikzset{starsize/.code={\setlength{\starsize}{#1}},
	starspread/.code={\setlength{\starspread}{#1}}}

\tikzset{starsize=1mm,
	starspread=3mm}

\pgfdeclarepatternformonly[\starspread,\starsize]
{star180}
{\pgfpointorigin}
{\pgfqpoint{\starspread}{\starspread}}
{\pgfqpoint{\starspread}{\starspread}}
{
	\pgftransformshift{\pgfqpoint{\starsize}{\starsize}}
	\pgfpathmoveto{\pgfqpointpolar{198}{\starsize}}
	\pgfpathlineto{\pgfqpointpolar{342}{\starsize}}
	\pgfpathlineto{\pgfqpointpolar{486}{\starsize}}
	\pgfpathlineto{\pgfqpointpolar{270}{\starsize}}
	\pgfpathlineto{\pgfqpointpolar{414}{\starsize}}
	\pgfpathclose%
	\pgfusepath{fill}
}

\pgfdeclarepatternformonly[\starspread,\starsize]
{star120}
{\pgfpointorigin}
{\pgfqpoint{\starspread}{\starspread}}
{\pgfqpoint{\starspread}{\starspread}}
{
	\pgftransformshift{\pgfqpoint{\starsize}{\starsize}}
	\pgfpathmoveto{\pgfqpointpolar{138}{\starsize}}
	\pgfpathlineto{\pgfqpointpolar{282}{\starsize}}
	\pgfpathlineto{\pgfqpointpolar{426}{\starsize}}
	\pgfpathlineto{\pgfqpointpolar{210}{\starsize}}
	\pgfpathlineto{\pgfqpointpolar{354}{\starsize}}
	\pgfpathclose%
	\pgfusepath{fill}
}

\pgfdeclarepatternformonly[\starspread,\starsize]
{star240}
{\pgfpointorigin}
{\pgfqpoint{\starspread}{\starspread}}
{\pgfqpoint{\starspread}{\starspread}}
{
	\pgftransformshift{\pgfqpoint{\starsize}{\starsize}}
	\pgfpathmoveto{\pgfqpointpolar{258}{\starsize}}
	\pgfpathlineto{\pgfqpointpolar{402}{\starsize}}
	\pgfpathlineto{\pgfqpointpolar{546}{\starsize}}
	\pgfpathlineto{\pgfqpointpolar{330}{\starsize}}
	\pgfpathlineto{\pgfqpointpolar{474}{\starsize}}
	\pgfpathclose%
	\pgfusepath{fill}
}

\makeatletter 
\tikzset{curlybrace/.style={rounded corners=2pt,line cap=round}}%
\pgfkeys{
/curlybrace/.cd,%
tip angle/.code     =  \def\cb@angle{#1},
/curlybrace/.unknown/.code ={\let\searchname=\pgfkeyscurrentname
                              \pgfkeysalso{\searchname/.try=#1,
                              /tikz/\searchname/.retry=#1}}}  
\def\curlybrace{\pgfutil@ifnextchar[{\curly@brace}{\curly@brace[]}}%

\def\curly@brace[#1]#2#3#4{%
\pgfkeys{/curlybrace/.cd,
tip angle = 0.75}%
\pgfqkeys{/curlybrace}{#1}%
\ifnum 1>#4 \def\cbrd{0.05} \else \def\cbrd{0.075} \fi
\draw[/curlybrace/.cd,curlybrace,#1]  (#2:#4-\cbrd) -- (#2:#4) arc (#2:{(#2+#3)/2-\cb@angle}:#4) --({(#2+#3)/2}:#4+\cbrd) coordinate (curlybracetipn);
\draw[/curlybrace/.cd,curlybrace,#1] ({(#2+#3)/2}:#4+\cbrd) -- ({(#2+#3)/2+\cb@angle}:#4) arc ({(#2+#3)/2+\cb@angle} :#3:#4) --(#3:#4-\cbrd);
}
\makeatother

\title{Lagrangian fillings for Legendrian links of finite type}

\author{Byung Hee An}
\email{anbyhee@knu.ac.kr}
\address{Department of Mathematics Education, Kyungpook National University, Republic of Korea}

\author{Youngjin Bae}
\email{yjbae@inu.ac.kr}
\address{Department of Mathematics, Incheon National University, Republic of Korea}

\author{Eunjeong Lee}
\email{eunjeong.lee@ibs.ac.kr}
\address{Center for Geometry and Physics, Institute for Basic Science, Republic of Korea}
\keywords{Legendrian link, Lagrangian filling, Cluster algebra}
\subjclass[2010]{Primary: 53D10, 13F60. Secondary: 57R17.}

\begin{document}

\begin{abstract}
We prove that there are at least seeds many exact embedded Lagrangian
fillings for Legendrian links of type $\dynADE$. We also provide seeds many Lagrangian fillings with certain symmetries for type $\dynBCFG$. Our main tools are $N$-graphs
and the combinatorics of seed patterns of finite type.
\end{abstract}

\maketitle

\tableofcontents

\section{Introduction}
\subsection{Backgrounds}
Legendrian knots are central object in the study of contact 3-dimensional
contact manifolds. Classification of Legendrian knots are important as its
own right, and also play a prominent role in constructing 4-dimensional
Weinstein manifold.

Classical Legendrian knot invariants are Thurston--Bennequin number and rotation
number~\cite{Gei2008} which distinguish the pair of Legendrian knots with the
same knot type. There are non-classical invariants including the Legendrian
contact algebra via the method of Floer theory~\cite{EGH2000, Che2002}, and
the space of constructible sheaves using microlocal
analysis~\cite{GKS2012,STZ2017}. These non-classical invariants distinguish
the Chekanov pair, a pair of Legendrian knots of type $m5_2$ having the same
classical invariants.

Recently, the study of exact Lagrangian fillings for Legendrian links has
been extremely plentiful. In the context of Legendrian contact algebra, an exact
Lagrangian filling gives an augmentation through the functorial view 
point~\cite{EHK2016}. There are several level of equivalence between 
augmentations
and the constructible sheaves for Legendrian links from counting to
categorical equivalence~\cite{NRSSZ2015}. By using these idea of
augmentations and constructible sheaves, people construct infinitely many fillings for certain Legendrian links~\cite{CG2020, GSW2020b,
CZ2020}. Here is the summarized list of methods of constructing Lagrangian
fillings for Legendrian links:
\begin{enumerate}
\item Decomposable Lagrangian fillings via pinching sequences \cite{EHK2016}.
\item Alternating Legendrians and its conjugate Lagrangian fillings
\cite{STWZ2019}. 
\item Legendrian weaves via $N$-graphs \cite{TZ2018,
CZ2020}. 
\item Donaldson--Thomas transformation on augmentation varieties
\cite{SW2019, GSW2020a, GSW2020b}.
\end{enumerate}

The cluster structure introduced by \cite{FZ1_2002} plays a crucial role in
the above constructions and applications. More precisely, the space of
augmentations, or equally the moduli of constructible sheaves adapted to
Legendrian links, admits a structure of cluster algebra~\cite{STWZ2019}. Note
that a seed of cluster algebra consists of a quiver whose vertices are
decorated with cluster variables. An involutory operation at each vertex,
called \emph{mutation}, generates all seeds of the cluster pattern.
The main point is to identify the
mutation in the cluster pattern and an operation in the space of Lagrangian
fillings. This geometric operation is deeply related to the Lagrangian surgery \cite{Pol1991} and
the wall-crossing phenomenon \cite{Aur2007}.

Indeed, a Legendrian torus link of type $(2,n)$ admits Catalan number many
Lagrangian fillings up to exact Lagrangian isotopy \cite{Pan2017, STWZ2019,
TZ2018}. Interestingly enough, the Catalan number is the number of seeds in a
cluster pattern of Dynkin type $\dynA_{n-1}$. There are also Legendrian links
corresponding to Dynkin type $\dynD_n, \dynE_6, \dynE_7,$ and $\dynE_8$
\cite{GSW2020b}. A conjecture in \cite[Conjecture~5.1]{Cas2020} says that the number of
distinct exact embedded Lagrangian fillings (up to exact Lagrangian isotopy) for
Legendrian links of type $\dynADE$ is exactly the same as the number of seeds
of the corresponding cluster algebras.

Furthermore, it is also conjectured in \cite[Conjecture~5.4]{Cas2020} that for Legendrian links of type $\dynA_{2n-1}, \dynD_{n+1}, \dynE_6$ and $\dynD_4$, Lagrangian fillings having certain $\Z/2\Z$ or $\Z/3\Z$-symmetry form the cluster patterns of type $\dynB_n, \dynC_n, \dynF_4$ and $\dynG_2$, which are Dynkin diagrams obtained by \emph{folding} introduced in \cite{FZ_Ysystem03}.

\subsection{The results}
Our main result is that there are seeds many Lagrangian fillings for
Legendrian links of finite type. We deal with $N$-graphs in \cite{CZ2020} to
construct the Lagrangian fillings. An $N$-graph $\ngraph$ on $\disk^2$ gives
a Legendrian surface $\Legendrian(\ngraph)$ in $J^1\disk^2$ while the
boundary $\boundary \ngraph$ on $\sphere^1$ induces a Legendrian link
$\legendrian(\boundary \ngraph)$. Then projection of $\Legendrian(\ngraph)$
along the Reeb direction becomes a Lagrangian filling of
$\legendrian(\boundary \ngraph)$.

As mentioned above, we interpret an $N$-graph as a seed in the corresponding
cluster pattern. A one-cycle in the Legendrian surface
$\Legendrian(\ngraph)$ corresponds to a vertex of the quiver, and a signed
intersection between one-cycles gives an arrow between corresponding
vertices. From constructible sheaves adapted to $\Legendrian(\ngraph)$, one
can assign a monodromy to each one-cycle which becomes the cluster
variable at each vertex.

There is an operation so called a \emph{Legendrian mutation}
$\mutation_\cycle$ on an $N$-graph $\ngraph$ along one-cycle $[\cycle]\in
H_1(\Legendrian(\ngraph))$ which is the counterpart of the mutation on the
cluster pattern, see Proposition~\ref{proposition:equivariance of mutations}. The
delicate and challenging part is that we do not know whether Legendrian
mutations are always possible or not. Simply put, this is because the
mutation in cluster side is algebraic, whereas the Legendrian mutation is
rather geometric.

The main idea of our construction is to consider the following bichromatic
({\color{blue} blue} and {\color{red} red}) graph $\ngraph(a,b,c)$, i.e.
$N$-graph with $N=3$, bounding a Legendrian link $\legendrian(a,b,c)$, which is the closure of the braid $\beta(a,b,c)$ as follows:
\begin{align*}
\legendrian(a,b,c)&=\cl(\beta(a,b,c))\subset J^1\sphere^1,&
\beta(a,b,c)&=\sigma_2\sigma_1^{a+1}\sigma_2\sigma_1^{b+1}\sigma_2\sigma_1^{c+1}.
\end{align*}
Then $\beta(a,b,c)$ corresponds to the rainbow closure of the braid $\beta_0(a,b,c)=\sigma_1\sigma_2^a\sigma_1^{b-1}\sigma_2^c$.
\begin{figure}[ht]
\begin{align*}
&\begin{tikzpicture}[yshift=0.75cm,baseline=-.5ex,xscale=0.8]
\draw[thick] (-3,-1)--(-1.5,-1) (-0.5,-1) -- (1.5, -1) (2.5,-1) -- (3,-1);
\draw[thick, rounded corners] (-3,-1.5) -- (-2.5,-1.5) -- (-2,-2) -- (0,-2) (-0.5,-1.5) -- (0, -1.5) (1, -1.5) -- (1.5,-1.5) (2.5, -1.5) -- (3,-1.5);
\draw[thick, rounded corners] (-3,-2) -- (-2.5,-2) -- (-2,-1.5) -- (-1.5,-1.5) (1,-2) -- (3,-2);
\draw[thick] (-1.5,-0.9) rectangle node {$a$} (-0.5, -1.6) (0,-2.1) rectangle node {$b-1$} (1,-1.4) (1.5,-1.6) rectangle node {$c$} (2.5,-0.9);
\draw[thick] (-3,-1) to[out=180,in=0] (-3.5,-0.75) to[out=0,in=180] (-3,-0.5) -- (3,-0.5) to[out=0,in=180] (3.5,-0.75) to[out=180,in=0] (3,-1);
\draw[thick] (-3,-1.5) to[out=180,in=0] (-4,-0.75) to[out=0,in=180] (-3,0) -- (3,0) to[out=0,in=180] (4,-0.75) to[out=180,in=0] (3,-1.5);
\draw[thick] (-3,-2) to[out=180,in=0] (-4.5,-0.75) to[out=0,in=180] (-3,0.5) -- (3,0.5) to[out=0,in=180] (4.5,-0.75) to[out=180,in=0] (3,-2);
\draw[dashed] (-3,-2.2) rectangle (3,-0.8) (0,-2.2) node[below] {$\beta_0(a,b,c)$};
\end{tikzpicture}&
\begin{tikzpicture}[yshift=0.25cm,baseline=-.5ex,xscale=0.8]
\draw[thick] (-1.5,-0.5) -- (-1,-0.5) (-1.5,0) -- (-1,0) (0.5,-0.5) -- (0,-0.5) (0.5,0) -- (0,0) (-1,-0.6) rectangle node {$r$} (0,0.1);
\end{tikzpicture}&=
\begin{tikzpicture}[yshift=0.75cm,baseline=-.5ex,xscale=0.8]
\draw[thick, rounded corners] (-1.5,-0.5) -- (-1,-0.5) -- (-0.5, -1) -- (-0.3, -1) (-1.5,-1) -- (-1,-1) -- (-0.5, -0.5) -- (-0.3, -0.5) (1.5,-0.5) -- (1,-0.5) -- (0.5, -1) -- (0.3, -1) (1.5,-1) -- (1,-1) -- (0.5, -0.5) -- (0.3, -0.5);
\draw[line cap=round, dotted] (-0.3,-0.5) -- (0.3,-0.5) (-0.3,-1) -- node[below] {$\underbrace{\hphantom{\hspace{2cm}}}_r$} (0.3,-1);
\end{tikzpicture}
\end{align*}
\caption{The rainbow closure of the braid $\beta_0(a,b,c)$ in $\R^3$}
\end{figure}
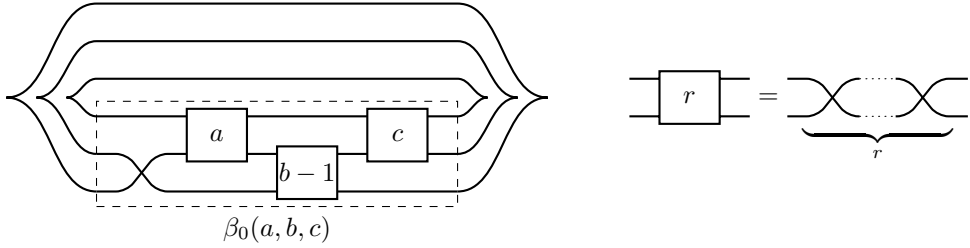

One-cycles $\nbasis(a,b,c)$ of the Legendrian surface
$\Legendrian(\ngraph(a,b,c))$ are given by the {\color{orange!90!yellow} yellow}- and {\color{green!60!black} green}-shaded
edges as depicted in Figure~\ref{figure:intro_tripod}.
See \S\ref{sec:1-cycles in Legendrian weaves} for the detail.
\begin{figure}[ht]
\[
(\ngraph(a,b,c),\nbasis(a,b,c))=\begin{tikzpicture}[baseline=-.5ex,scale=0.6]
\draw[thick] (0,0) circle (3cm);
\draw[green, line cap=round, line width=5, opacity=0.5] (60:1) -- (50:1.5) (70:1.75) -- (50:2) (180:1) -- (170:1.5) (190:1.75) -- (170:2) (300:1) -- (290:1.5) (310:1.75) -- (290:2);
\draw[yellow, line cap=round, line width=5, opacity=0.5] (0,0) -- (60:1) (0,0) -- (180:1) (0,0) -- (300:1) (50:1.5) -- (70:1.75) (170:1.5) -- (190:1.75) (290:1.5) -- (310:1.75);
\draw[red, thick] (0,0) -- (0:3) (0,0) -- (120:3) (0,0) -- (240:3);
\draw[blue, thick, fill] (0,0) -- (60:1) circle (2pt) -- (100:3) (60:1) -- (50:1.5) circle (2pt) -- (20:3) (50:1.5) -- (70:1.75) circle (2pt) -- (80:3) (70:1.75) -- (50:2) circle (2pt) -- (40:3);
\draw[blue, thick, dashed] (50:2) -- (60:3);
\draw[blue, thick, fill] (0,0) -- (180:1) circle (2pt) -- (220:3) (180:1) -- (170:1.5) circle (2pt) -- (140:3) (170:1.5) -- (190:1.75) circle (2pt) -- (200:3) (190:1.75) -- (170:2) circle (2pt) -- (160:3);
\draw[blue, thick, dashed] (170:2) -- (180:3);
\draw[blue, thick, fill] (0,0) -- (300:1) circle (2pt) -- (340:3) (300:1) -- (290:1.5) circle (2pt) -- (260:3) (290:1.5) -- (310:1.75) circle (2pt) -- (320:3) (310:1.75) -- (290:2) circle (2pt) -- (280:3);
\draw[blue, thick, dashed] (290:2) -- (300:3);
\draw[thick, fill=white] (0,0) circle (2pt);
\curlybrace[]{10}{110}{3.2};
\draw (60:3.5) node[rotate=-30] {\small ${ a+1}$};
\curlybrace[]{130}{230}{3.2};
\draw (180:3.5) node[rotate=90] {\small $b+1$};
\curlybrace[]{250}{350}{3.2};
\draw (300:3.5) node[rotate=30] {\small $c+1$};
\end{tikzpicture}
\]
\caption{The tripod $N$-graph $\ngraph(a,b,c)$ and the good tuple $\nbasis(a,b,c)$ of cycles}
\label{figure:intro_tripod}
\end{figure}

There are several good properties of $\ngraph(a,b,c)$ as follows:
\begin{enumerate}
\item The geometric- and algebraic intersection numbers of the one-cycles in
$\nbasis(a,b,c)$ coincide. 
\item The corresponding quiver $\quiver(a,b,c)$ is
bipartite, see \S\ref{sec:N-graphs and seeds} for the details. 
\item It covers Legendrian links of type $\dynADE$. More precisely,
the underlying graphs of $\quiver(1,b,c)$ for $b+c-1=n$, $\quiver(n-2,2,2)$, and $\quiver(2,3,n-3)$ are the same as Dynkin diagrams of type $\dynA_{n}$,
$\dynD_{n}$, and $\dynE_{n}$, respectively.
\end{enumerate}

Let us consider the finite type $N$-graph $\ngraph(a,b,c)$, that is,
$\frac{1}{a}+\frac{1}{b}+\frac{1}{c} > 1$. Denote the corresponding rank
$n$-root system by $\Roots=\Roots(a,b,c)$, where $n=a+b+c-2$. Let us consider
an \emph{exchange graph} $E(\Roots)$ of the corresponding cluster
pattern whose vertices are the seeds and whose edges connect the vertices
are given by a single mutation. Note from \cite{CFZ02_polytopal} that the
exchange graph $E(\Roots)$ can be realized as vertices and edges of a
polytope $P(\Roots)\subset \R^n$ called a \emph{generalized associahedron}.

The combinatorics of the exchange graph $E(\Roots)$ is the key ingredient in
investigating the Legendrian mutability. 
All facets of the polytope $P(\Roots)$ can be recovered
from a sequence of mutations obtained by a Coxeter element 
together with a subset of facets of $P(\Roots)$ corresponding to $P(\Roots([n] \setminus \{i\}))$, see
\cite{FZ_Ysystem03} or Proposition~\ref{prop_FZ_finite_type_Coxeter_element}. 
We call this specific sequence of mutations a
\emph{Coxeter mutation}. In order to interpret a Coxeter mutation in terms of
$N$-graphs, let us consider a partition $\nbasis_+, \nbasis_-$ of the
one-cycles $\nbasis$, consisting of {\color{orange!90!yellow} yellow}- and {\color{green!60!black} green}-shaded edges,
respectively. Then the $N$-graph realization of the Coxeter mutation is called the \emph{Legendrian Coxeter mutation} and given
by the sequence of Legendrian mutations:
\[
\ncoxeter=\prod_{\cycle \in \nbasis_-} \mutation_{\cycle}\cdot\prod_{\cycle\in \nbasis_+} \mutation_{\cycle}.
\]

Then the resulting $N$-graph $\ncoxeter(\ngraph(a,b,c),\nbasis(a,b,c))$ becomes the
$N$-graph shown in Figure~\ref{figure:intro_Legendrian Coxeter mutation} up to a sequence of Move~\Move{II} in Figure~\ref{fig:move1-6}.
\begin{figure}[ht]
\[
\begin{tikzpicture}[baseline=-.5ex,scale=0.5]
\draw[thick] (0,0) circle (5cm);
\draw[dashed]  (0,0) circle (3cm);
\fill[opacity=0.1, even odd rule] (0,0) circle (3) (0,0) circle (5);
\foreach \i in {1,2,3} {
\begin{scope}[rotate=\i*120]
\draw[green, line cap=round, line width=5, opacity=0.5] (60:1) -- (50:1.5) (70:1.75) -- (50:2);
\draw[yellow, line cap=round, line width=5, opacity=0.5] (0,0) -- (60:1) (50:1.5) -- (70:1.75);
\draw[blue, thick, rounded corners] (0,0) -- (0:3.4) to[out=-75,in=80] (-40:4);
\draw[red, thick, fill] (0,0) -- (60:1) circle (2pt) (60:1) -- (50:1.5) circle (2pt) -- (70:1.75) circle (2pt) -- (50:2) circle (2pt);
\draw[red, thick, dashed, rounded corners] (50:2) -- (60:2.8) -- (60:3.3) to[out=0,in=220] (40:4);
\draw[red, thick, rounded corners] (50:2) -- (40:2.8) -- (40:3.3) to[out=-20,in=200] (20:4) (70:1.75) -- (80:2.8) -- (80:3.3) to[out=20,in=240] (60:4) (60:1) -- (100:2.8) -- (100:3.3) to[out=40,in=260] (80:4);
\draw[red, thick, rounded corners] (50:1.5) -- (20:3) -- (20:3.5) to[out=-70,in=50] (-40:4) (20:4) to[out=-50,in=120] (0:4.5) -- (0:5);
\draw[red, thick] (20:4) to[out=100,in=-40] (40:4) (60:4) to[out=140,in=0] (80:4);
\draw[blue, thick] (20:5) -- (20:4) to[out=140,in=-80] (40:4) (60:5) -- (60:4) to[out=180,in=-40] (80:4) -- (80:5);
\draw[thick, dotted] (40:4) arc (40:60:4);
\draw[blue, thick, rounded corners] (20:4) to[out=-70,in=100] (-20:4.5) -- (-20:5);
\draw[blue, thick, dashed] (40:4) -- (40:5);
\draw[fill=white, thick] (20:4) circle (2pt) (40:4) circle (2pt) (60:4) circle (2pt) (80:4) circle (2pt) (-40:4) circle (2pt);
\end{scope}
\draw[fill=white, thick] (0,0) circle (2pt);
}
\end{tikzpicture}
\]
\caption{Legendrian Coxeter mutation on $(\ngraph(a,b,c),\nbasis(a,b,c))$}
\label{figure:intro_Legendrian Coxeter mutation}
\end{figure}

Removing the gray-shaded annulus region, the only difference between
$(\ngraph(a,b,c),\nbasis(a,b,c))$ and $\ncoxeter(\ngraph(a,b,c),\nbasis(a,b,c))$ is the reversing of the color. Note that the intersection pattern between one-cycles and the
Legendrian mutability are preserved under the action of the Legendrian Coxeter mutation
$\ncoxeter$. Moreover, the operation $\qcoxeter$ also acts on
the face poset of the generalized associahedron of the root system $\Roots$.
By the induction argument on the rank of root system, we conclude that there
in no (geometric) obstruction to realize each seed via the $N$-graph,
especially for finite type case. This guarantees that there are at least
seeds many Lagrangian fillings for $\legendrian(a,b,c)$.

For the infinite type, i.e. $\frac{1}{a}+\frac{1}{b}+\frac{1}{c} \leq 1$, the operation $\qcoxeter$ is of infinite order and so is $\ncoxeter$, hence Legendrian weaves 
\[\Legendrian(\ncoxeter^r(\ngraph(a,b,c),\nbasis(a,b,c)))\]
produce infinitely many distinct Lagrangian fillings.
Indeed, the quiver $\quiver(a,b,c)$ is also bipartite and the one can perform 
the Legendrian Coxeter mutation $\mutation_{\ngraph}$ on the $N$-graph $\ngraph(a,b,c)$ by stacking
the gray-shaded annulus like as before. Therefore, there is no obstruction to
realize seeds obtained by mutations $\ncoxeter^r$ via the $N$-graphs.
Since the order of the Legendrian Coxeter mutation is infinite (see Lemma~\ref{lemma:order of coxeter mutation}), we obtain infinitely many $N$-graphs
and hence infinitely many exact embedded Lagrangian fillings for the Legendrian
link $\legendrian(a,b,c)$ with $\frac{1}{a}+\frac{1}{b}+\frac{1}{c} \leq 1$.

\begin{theorem}[Theorem~\ref{theorem:infinite fillings}]
For each $a,b,c\ge 1$, the Legendrian knot or link $\legendrian(a,b,c)$ has distinct
infinitely many Lagrangian fillings if
\[
\frac1a+\frac1b+\frac1c\le 1,
\]
or equivalently, the tripod $\quiver(a,b,c)$ is of infinite type. 
\end{theorem}

\begin{theorem}[Theorem~\ref{theorem:ADE type}]
There are at least seeds many distinct exact embedded Lagrangian fillings for
Legendrian links of type $\dynADE$. 
\end{theorem}

There are several way of constructing exact embedded Lagrangian fillings as mentioned above.
Especially in $\dynD_4$ case, there are 34 distinct Lagrangian fillings constructed by the method of the alternating Legendrians in \cite{BFFH2018,STWZ2019},
while the above $N$-graphs give seeds many 50 Lagrangian fillings.

The remaining finite type Dynkin diagrams, which are non-simply
laced, are of type $\dynBCFG$, obtain by the folding procedure from type $\dynADE$, see
\S\ref{sec:folding}. By keep tracking the folding process, seeds and
mutations in $\dynB_n$, $\dynC_n$, $\dynF_4$, and $\dynG_2$ cluster patterns
can be regarded as certain subsets of seeds and sequences of mutations in
$\dynA_{2n-1}$, $\dynD_{n-1}$, $\dynE_6$, and $\dynD_4$, respectively. Those
specified seeds of type $\dynADE$ admit $N$-graphs with certain symmetries 
given by an action of a finite group $G$, and we call such seeds and $N$-graphs
\emph{$G$-admissible}. If a seed (or an $N$-graph) is again $G$-admissible 
after performing a sequence of mutations indexed by vertices in the same 
$G$-orbit, then we call it \emph{globally foldable} with respect to $G$.  

The following four $N$-graphs are examples of type $\dynBCFG$. 
Indeed, they are $\ngraph(1,3,3)$, $\ngraph(3,2,2)$, $\ngraph(2,3,3)$, and $\ngraph(2,2,2)$, respectively.
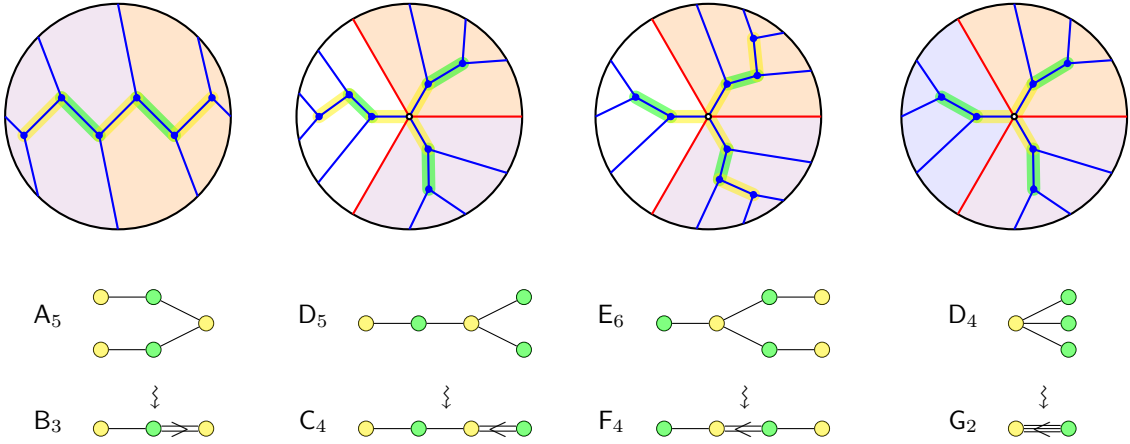
\[
\begin{tikzcd}[column sep=small, row sep=small]
\begin{tikzpicture}[baseline=-.5ex,scale=0.5]
\draw[orange, opacity=0.2, fill] (-90:3) arc(-90:90:3) (90:3) -- (0.5,0.5) -- (-0.5,-0.5) -- (-90:3);
\draw[violet, opacity=0.1, fill] (90:3) arc(90:270:3) (270:3) -- (-0.5,-0.5) -- (0.5,0.5) -- (90:3);
\draw[thick] (0,0) circle (3);
\draw[green,line cap=round, line width=5, opacity=0.5] (-1.5,0.5) -- (-0.5, -0.5) 
(0.5, 0.5) -- (1.5, -0.5);
\draw[yellow,line cap=round, line width=5, opacity=0.5] (-2.5,-0.5) -- (-1.5, 0.5) (-0.5, -0.5) -- (0.5, 0.5) (1.5, -0.5) -- (2.5, 0.5);
\draw[blue, thick, fill] (0:3) -- (2.5,0.5) circle (2pt) -- (45:3) (2.5,0.5) -- (1.5,-0.5) circle (2pt) -- (-45:3) (1.5,-0.5) -- (0.5,0.5) circle (2pt) -- (90:3) (0.5,0.5) -- (-0.5, -0.5) circle (2pt) -- (-90:3) (-0.5, -0.5) -- (-1.5, 0.5) circle (2pt) -- (135:3) (-1.5, 0.5) -- (-2.5, -0.5) circle (2pt) -- (-135:3);
\draw[blue, thick] (-2.5,-0.5) -- (-180:3);
\end{tikzpicture}
&
\begin{tikzpicture}[baseline=-.5ex,scale=0.5]
\draw[orange, opacity=0.2, fill] (0:3) arc(0:120:3) (120:3) -- (0,0) -- (0:3);
\draw[violet, opacity=0.1, fill] (0:3) -- (0,0) -- (-120:3) arc(-120:0:3) (0:3);
\draw[thick] (0,0) circle (3cm);
\draw[green, line cap=round, line width=5, opacity=0.5] (60:1) -- (45:2)  (180:1) -- (160:1.7) (300:1) -- (285:2);
\draw[yellow, line cap=round, line width=5, opacity=0.5] (0,0) -- (60:1) (0,0) -- (180:1) (0,0) -- (300:1)  (160:1.7) -- (180:2.4) ;
\draw[red, thick] (0,0) -- (0:3) (0,0) -- (120:3) (0,0) -- (240:3);
\draw[blue, thick, fill] 
(0,0) -- (60:1) circle (2pt) -- (90:3) 
(60:1) -- (45:2) circle (2pt) -- (30:3) 
(45:2) -- (60:3);
\draw[blue, thick, fill] 
(0,0) -- (180:1) circle (2pt) -- (216:3) 
(180:1) -- (160:1.7) circle (2pt) -- (144:3) 
(160:1.7) -- (180:2.4) circle (2pt) -- (192:3) 
(180:2.4) -- (168:3);
\draw[blue, thick, fill] 
(0,0) -- (300:1) circle (2pt) -- (330:3) 
(300:1) -- (285:2) circle (2pt) -- (270:3) 
(285:2) -- (300:3);
\draw[thick, fill=white] (0,0) circle (2pt);
\end{tikzpicture}
&
\begin{tikzpicture}[baseline=-.5ex,scale=0.5]
\draw[orange, opacity=0.2, fill] (0:3) arc(0:120:3) (120:3) -- (0,0) -- (0:3);
\draw[violet, opacity=0.1, fill] (0:3) -- (0,0) -- (-120:3) arc(-120:0:3) (0:3);
\draw[thick] (0,0) circle (3cm);
\draw[green, line cap=round, line width=5, opacity=0.5] (60:1) -- (40:1.7)  (180:1) -- (165:2) (280:1.7) -- (300:1);
\draw[yellow, line cap=round, line width=5, opacity=0.5] (0,0) -- (60:1) (0,0) -- (180:1) (0,0) -- (300:1) (40:1.7)--(60:2.4) (280:1.7)--(300:2.4);
\draw[red, thick] (0,0) -- (0:3) (0,0) -- (120:3) (0,0) -- (240:3);
\draw[blue, thick, fill] 
(0,0) -- (60:1) circle (2pt) -- (96:3) 
(60:1) -- (40:1.7) circle (2pt) -- (24:3) 
(40:1.7) -- (60:2.4) circle (2pt) -- (72:3)
(60:2.4) -- (48:3);
\draw[blue, thick, fill] 
(0,0) -- (180:1) circle (2pt) -- (210:3) 
(180:1) -- (165:2) circle (2pt) -- (150:3) 
(165:2) -- (180:3);
\draw[blue, thick, fill] 
(0,0) -- (300:1) circle (2pt) -- (336:3) 
(300:1) -- (280:1.7) circle (2pt) -- (264:3) 
(280:1.7) -- (300:2.4) circle (2pt) -- (312:3)
(300:2.4) -- (288:3);
\draw[thick, fill=white] (0,0) circle (2pt);
\end{tikzpicture}
&
\begin{tikzpicture}[baseline=-.5ex,scale=0.5]
\draw[orange, opacity=0.2, fill] (0:3) arc(0:120:3) (120:3) -- (0,0) -- (0:3);
\draw[violet, opacity=0.1, fill] (0:3) -- (0,0) -- (-120:3) arc(-120:0:3) (0:3);
\draw[blue, opacity=0.1, fill] (120:3) arc(120:240:3) (240:3) -- (0,0) -- (120:3);
\draw[thick] (0,0) circle (3cm);
\draw[green, line cap=round, line width=5, opacity=0.5] (60:1) -- (45:2)  (180:1) -- (165:2) (300:1) -- (285:2);
\draw[yellow, line cap=round, line width=5, opacity=0.5] (0,0) -- (60:1) (0,0) -- (180:1) (0,0) -- (300:1);
\draw[red, thick] (0,0) -- (0:3) (0,0) -- (120:3) (0,0) -- (240:3);
\draw[blue, thick, fill] 
(0,0) -- (60:1) circle (2pt) -- (90:3) 
(60:1) -- (45:2) circle (2pt) -- (30:3) 
(45:2) -- (60:3);
\draw[blue, thick, fill] 
(0,0) -- (180:1) circle (2pt) -- (210:3) 
(180:1) -- (165:2) circle (2pt) -- (150:3) 
(165:2) -- (180:3);
\draw[blue, thick, fill] 
(0,0) -- (300:1) circle (2pt) -- (330:3) 
(300:1) -- (285:2) circle (2pt) -- (270:3) 
(285:2) -- (300:3);
\draw[thick, fill=white] (0,0) circle (2pt);
\end{tikzpicture}\\
\begin{tikzpicture}[baseline=-.5ex,scale=0.7]
    
\node[Dnode] (a1) at (-2,0.5) {};
\node[Dnode] (a2) at (-1,0.5) {};
\node[Dnode] (a3) at (0,0) {};
\node[Dnode] (a4) at (-1,-0.5) {};
\node[Dnode] (a5) at (-2,-0.5) {};

\node[ynode] at (a1) {};
\node[gnode] at (a2) {};
\node[ynode] at (a3) {};
\node[gnode] at (a4) {};
\node[ynode] at (a5) {};

\draw (a1)--(a2)--(a3)--(a4)--(a5);

\node at (-3,0) {$\dynA_{5}$};
\node at (-1,-1.5) {\rotatebox[origin=c]{-90}{$\rightsquigarrow$}};

\node at (-3,-2) {$\dynB_{3}$};

\coordinate[Dnode] (1) at (-2,-2) {};
\coordinate[Dnode] (2) at (-1,-2) {};
\coordinate[Dnode] (3) at (0,-2) {};

\node[ynode] at (1) {};
\node[ynode] at (3) {};
\node[gnode] at (2) {};

\draw (1)--(2);
\draw[double line] (2)-- ++ (3) node[midway,yshift=-0.1cm ] {$>$};

\end{tikzpicture}
&
\begin{tikzpicture}[baseline=-.5ex,scale=0.7]

\coordinate[Dnode] (d1) at (-1,0) {};
\coordinate[Dnode] (d2) at (0,0) {};
\coordinate[Dnode] (d3) at (1,0) {};
\coordinate[Dnode] (d4) at (2,0.5) {};
\coordinate[Dnode] (d5) at (2,-0.5) {};

\node[ynode] at (d1) {};
\node[ynode] at (d3) {};
\node[gnode] at (d2) {};
\node[gnode] at (d4) {};
\node[gnode] at (d5) {};

\draw  (d1)--(d2)--(d3)--(d4)
(d3)--(d5);

\node at (-2,0) {$\dynD_{5}$};
\node at (0.5,-1.5) {\rotatebox[origin=c]{-90}{$\rightsquigarrow$}};

\node at (-2,-2) {$\dynC_{4}$};
\coordinate[Dnode] (1) at (-1,-2) {};
\coordinate[Dnode] (2) at (0,-2) {};
\coordinate[Dnode] (3) at (1,-2) {};
\coordinate[Dnode] (4) at (2,-2) {};

\node[ynode] at (1) {};
\node[gnode] at (2) {};
\node[ynode] at (3) {};
\node[gnode] at (4) {};

\draw (1)--(2)--(3) ;
\draw[double line] (3)-- ++ (4) node[midway,yshift=-0.1cm ] {$<$};

\end{tikzpicture}
&
\begin{tikzpicture}[baseline=-.5ex,scale=0.7]

\coordinate[Dnode] (e1) at (-2,0) {};
\coordinate[Dnode] (e2) at (-1,0) {};
\coordinate[Dnode] (e3) at (0,0.5) {};
\coordinate[Dnode] (e4) at (1,0.5) {};
\coordinate[Dnode] (e5) at  (0,-0.5) {};
\coordinate[Dnode] (e6) at  (1,-0.5) {};

\foreach \y in {e2,e4,e6} {
	\node[ynode] at (\y) {};
}
\foreach \g in {e1,e3,e5}{
	\node[gnode] at (\g) {};
}
\draw (e1)--(e2)--(e3)--(e4)
(e2)--(e5)--(e6);

\node at (-3,0) {$\dynE_{6}$};
\node at (-0.5,-1.5) {\rotatebox[origin=c]{-90}{$\rightsquigarrow$}};

\node at (-3,-2) {$\dynF_{4}$};
\coordinate[Dnode] (1) at (-2,-2) {};
\coordinate[Dnode] (2) at (-1,-2) {};
\coordinate[Dnode] (3) at (0,-2) {};
\coordinate[Dnode] (4) at (1,-2) {};

\foreach \y in {2,4} {
	\node[ynode] at (\y) {};
}
\foreach \g in {1,3}{
	\node[gnode] at (\g) {};
}

\draw (1)--(2)
(3)--(4);
\draw[double line] (2)-- ++ (3) node[midway,yshift=-0.1cm ] {$<$};
\end{tikzpicture}
&
\begin{tikzpicture}[baseline=-.5ex,scale=0.7]
\tikzstyle{triple line} = [
double distance = 2pt, 
double=\pgfkeysvalueof{/tikz/commutative diagrams/background color}
]
\coordinate[Dnode] (d1) at (-2,0) {};
\coordinate[Dnode] (d2) at  (-1,0)  {};
\coordinate[Dnode] (d3) at (-1,0.5) {};
\coordinate[Dnode] (d4) at (-1,-0.5) {};

\foreach \y in {d1} {
	\node[ynode] at (\y) {};
}
\foreach \g in {d2,d3,d4}{
	\node[gnode] at (\g) {};
}

\draw (d1)--(d2)
(d1)--(d3)
(d1)--(d4);

\node at (-3,0) {$\dynD_{4}$};
\node at (-1.5,-1.5) {\rotatebox[origin=c]{-90}{$\rightsquigarrow$}};

\node at (-3,-2) {$\dynG_{2}$};
\coordinate[Dnode] (1) at (-2,-2) {};
\coordinate[Dnode] (2) at (-1,-2) {};

\node[ynode] at (1) {};
\node[gnode] at (2) {};

\draw[triple line] (1)-- ++ (2) node[midway,yshift=-0.1cm ] {$<$};
\draw (1)--(2);

\end{tikzpicture}
\end{tikzcd}
\]

The colored regions represent how the group $G$ acts on the $N$-graphs and
the induced Lagrangian fillings. The first three $N$-graphs are globally
foldable with respect to $\Z/2\Z$ by folding {\color{orange!70!red} orange}- and
{\color{violet} violet}-colored regions in an orientation preserving way.
Similarly, the $\Z/3\Z$-rotational symmetry of the last one implies that it is
globally foldable with respect to $\Z/3\Z$ by folding three colored regions.

\begin{theorem}[Theorem~\ref{theorem:BCFG type}]
The following holds:
\begin{enumerate}
\item The Legendrian link $\lambda(\dynA_{2n-1})$ has $\binom{2n}{n}$ $\Z/2\Z$-admissible $N$-graphs which admits the cluster pattern of type $\dynB_n$.
\item The Legendrian link $\lambda(\dynD_{n+1})$ has $\binom{2n}{n}$ $\Z/2\Z$-admissible $N$-graphs which admits the cluster pattern of type $\dynC_n$.
\item The Legendrian link $\lambda(\dynE_{6})$ has $105$ $\Z/2\Z$-admissible $N$-graphs which admits the cluster pattern of type $\dynF_4$.
\item The Legendrian link $\lambda(\dynD_{4})$ has $8$ $\Z/3\Z$-admissible $N$-graphs which admits the cluster pattern of type $\dynG_2$.
\end{enumerate}
\end{theorem}

\subsection*{Acknowledgement}
B. An was supported by Kyungpook National University Research Fund, 2020.
Y. Bae was supported by the National Research Foundation of Korea (NRF) grant funded by the Korea government (MSIT) (No. 2020R1A2C1A0100320).
E. Lee was supported by IBS-R003-D1.

\section{Legendrians and \texorpdfstring{$N$}{N}-graphs}

We recall from \cite{CZ2020} the notion of $N$-graphs and their combinatorial moves which encode the Legendrian isotopy data of corresponding Legendrian surfaces. As an application, we review how $N$-graphs can be use to find and to distinguish Lagrangian fillings for Legendrian links.

\subsection{Geometric setup}\label{sec:geometric setup}
Let us start with the standard contact structure on $\R^3$ whose contact structure is given by $\xi_{\rm st}^3=\ker(dz-ydx)$. Consider the symplectization $(\R^4, d(e^s(dz-ydx)))$ and its contactization $(\R^5,\alpha_{\rm st}^5=e^s(dz-ydx)-dt)$.

Now consider a contact 3-dimensional space 
\[
(\R^3_{\ell},\xi_{\rm st}^3):=\{(x,y,z,s,t)\in\R^5 \mid (s,t)=(\ell,1)\}
\]
for each symplectization level $s=\ell$. Take Legendrians $\legendrian_1\subset (\R^3_1,\xi_{\rm st}^3)$ and $\legendrian_2\subset (\R^3_2,\xi_{\rm st}^3)$
and consider a Legendrian surface 
\[
\Legendrian \subset (\R^5,\xi_{\rm st}^5=\ker\alpha_{\rm st}^5) \cap \{1 \leq s\leq2\}
\] whose boundary is $\legendrian_1\cup \legendrian_2$, i.e., 
\begin{align*}
\Legendrian \cap (\R_1^3,\xi_{\rm st}^3)&= \legendrian_1,&
\Legendrian \cap (\R_2^3,\xi_{\rm st}^3)&= \legendrian_2.
\end{align*}

Let $\pi:\R^5\to \R^4$ be the projection along the contactization coordinate $t$,
then $\pi(\Legendrian)$ becomes an exact Lagrangian (possibly immersed) cobordism from $\pi(\legendrian_2)$ to $\pi(\legendrian_1)$.
Note that $\partial_t$ is the Reeb vector field of $(\R^5,\alpha_{\rm st}^5)$ and the above immersed points on $\pi(\Legendrian)$ correspond to Reeb chords in $\Legendrian$.

Relating the construction of the current article, any Legendrian link in $(\R^3,\xi_{\rm st}^3)$ can be seen as a satellite link of the standard Legendrian unknot $\legendrian_{\rm unknot}\subset (\R^3,\xi_{\rm st}^3)$. Note that a neighborhood of $\legendrian_{\rm unknot}$ is contactomorphic to $(J^1\sphere^1, \ker(dz-p_\theta d\theta))$. Denote the corresponding contact embedding by $\iota:J^1\sphere^1\hookrightarrow\R^3$. 

Let us denote the corresponding satellite links of $\legendrian_1,\legendrian_2$ by 
\begin{align*}
\legendrian_{\beta_i}=\cl(\beta_i)\subset (J^1(\sphere^1\times \{i\}),\ker(dz-p_\theta d\theta)), \quad i=1,2.
\end{align*}
Here $\beta_i$ is a positive braid word for the satellite link of $\legendrian_i$, and $\cl(\beta)$ denotes the closure of a braid word $\beta$. Extending the contact embedding $\iota:J^1\sphere^1\hookrightarrow\R^3$, by abuse of notation, we have $\iota:J^1(\sphere^1 \times [1,2]) \hookrightarrow\R^5\cap \{1 \leq s\leq2\}$. Denote the corresponding Legendrian surface by 
\[
\Legendrian\subset(J^1(\sphere^1 \times [1,2]),\ker(dz-p_\theta d\theta -p_\sigma d\sigma)),
\]
where $\sigma$ is the coordinate for the interval $[1,2]$ which corresponds to the $e^s$-coordinate.
By a strict contactomorphism, we can regard
\[
\Legendrian\subset (J^1\sphere^1 \times \R_s\times \R_t, \ker(e^s(dz- p_\theta d\theta)-dt)).
\]
Then its Lagrangian projection $\pi\circ \iota(\Legendrian)$ gives an exact Lagrangian cobordism from $\legendrian_2$ to $\legendrian_1$, where $\pi$ is the projection along the $t$-coordinate.
Especially when $\legendrian_1=\emptyset$, the boundary $J^1(\sphere^1\times \{1\})$ of $J^1(\sphere^1 \times [1,2])$ can be compactified by $J^1\disk^2$.
Under the Lagrangian projection, this corresponds to a exact symplectic filling $(\disk^4,\omega_{\rm st})$ of $(\sphere^3,\xi_{\rm st})=(\R^3_{1},\xi_{\rm st}^3)\cup\{\infty\}$.
We end this section by stating the relation between Legendrian- and Lagrangian fillings.

\begin{lemma}\label{lem:legendrian and lagrangian}
As in the above setup, let $\legendrian_\beta \subset J^1\sphere^1$ be a Legendrian link, and let $\iota(\legendrian_\beta)$ be an induced Legendrian link in $(\sphere^3,\xi_{\rm st})$. Let $\Legendrian,\Legendrian'\subset J^1\disk^2$ be two Legendrian surfaces, without Reeb chords, bounding $\legendrian_\beta$. If the corresponding exact Lagrangian fillings $\pi\circ\iota(\Legendrian), \pi\circ\iota(\Legendrian')\subset (\disk^4,\omega_{\rm st})$ of $\iota(\legendrian_\beta)$ are exact Lagrangian isotopic relative to the boundary, then $\Legendrian, \Legendrian'$ are Legendrian isotopic relative to the boundary.
\end{lemma}

\subsection{\texorpdfstring{$N$}{N}-graphs and Legendrian weaves}

\begin{definition}\cite[Definition~2.2]{CZ2020}
An  $N$-graph $\ngraph$ on a smooth surface $S$ is an $(N-1)$-tuple of graphs $(\ngraph_1,\dots, \ngraph_{N-1})$ satisfying the following conditions:
\begin{enumerate}
\item Each graph $\ngraph_i$ is embedded, trivalent, possibly empty and non necessarily connected.
\item Any consecutive pair of graphs $(\ngraph_i,\ngraph_{i+1})$, $1\leq i \leq N-2$, intersects only at hexagonal points depicted as in Figure~\ref{fig:hexagonal_point}.
\item Any pair of graphs $(\ngraph_i, \ngraph_j)$ with $1\leq i,j\leq N-1$ and $|i-j|>1$  intersects transversely at edges.
\end{enumerate}
\end{definition}

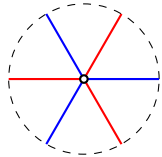
\begin{figure}[ht]
\begin{tikzpicture}
\begin{scope}
\draw[dashed] (0,0) circle (1cm);
\draw[red, thick] (60:1)--(0,0) (180:1)--(0,0) (-60:1)--(0,0);
\draw[blue, thick] (0:1)--(0,0) (120:1)--(0,0) (240:1)--(0,0);
\draw[thick,black,fill=white] (0,0) circle (0.05);
\end{scope}
\end{tikzpicture}
\caption{A hexagonal point}
\label{fig:hexagonal_point}
\end{figure}

\begin{remark}
For the result of the current article, we mainly consider the case $N=3$ and $S=\disk^2$.
In other words, we focus bicolored graphs with monochromatic trivalent vertices and bichromatic hexagonal points as in Figure~\ref{fig:hexagonal_point}.
\end{remark}

For any $N$-graph $\ngraph$ on a surface $S$, we associate a Legendrian surface $\Legendrian(\ngraph)\subset J^1S $. Basically, we construct the Legendrian surface by weaving the wavefronts in $S \times \R$ constructed from a local chart of $S$. 

Let $\ngraph\subset S$ be an $N$-graph.
A finite cover $\{U_i\}_{i\in I}$ is called {\em $\ngraph$-compatible} if
\begin{enumerate}
\item each $U_i$ is diffeomorphic to the open disk $\mathring{\disk}^2$,
\item $U_i \cap \ngraph$ is connected, and
\item $U_i \cap \ngraph$ contains at most one vertex.
\end{enumerate}

For each $U_i$, we associate a wavefront $\wavefront(U_i)\subset U_i\times \R \subset S\times \R$.
Note that there are only four types of local charts for any $N$-graph $\ngraph$ as follows:
\begin{enumerate}[Type 1.]
\item A chart without any graph component whose corresponding wavefront becomes
\[
\bigcup_{i=1,\dots,N}\mathring{\disk}^2\times\{i\}\subset \mathring{\disk}^2\times \R.
\]

\item A chart with single edge. The corresponding wavefront is the union of the $\dynA_1^2$-germ along the two sheets $\mathring\disk^2\times \{i\}$ and $\mathring\disk^2\times\{i+1\}$, and trivial disks $\disk^2\times\{i\}$, $i\in \{1,\dots,N\}\setminus\{i,i+1\}$.
The local model of $\dynA_1^2$ comes from the origin of the singular surface
\[
\wavefront(\dynA_1^2)=\{(x,y,z)\in \R^3 \mid x^2-z^2=0\}
\]
See Figure~\ref{fig:A_1^2 germ}.

\item A chart with a monochromatic trivalent vertex whose wavefront is the union of the $\dynD_4^-$-germ, see \cite[\S2.4]{Arn1990}, and trivial disks $\disk^2\times\{i\}$, $i\in \{1,\dots,N\}\setminus\{i,i+1\}$.
The local model for Legendrian singularity of type $\dynD_4^-$ is given by the image at the origin of 
\begin{align*}
\delta_4^-:\R^2\to \R^3:(x,y)\mapsto \left( x^2-y^2, 2xy, \frac{2}{3}(x^3-3xy^2) \right).
\end{align*}
See Figure~\ref{fig:D_4^- germ}.

\item A chart with a bichromatic hexagonal point. The induced wavefront is the union of the $\dynA_1^3$-germ along the three sheets $\mathring\disk^2\times \{*\}$, $*=i,i+1,i+2$, and the trivial disks $\disk^2\times\{i\}$, $i\in \{1,\dots,N\}\setminus\{i,i+1,i+2\}$. The local model of $\dynA_1^3$ is given by the origin of the singular surface
\[
\{(x,y,z)\in \R^3 \mid (x^2-z^2)(y-z)=0\}.
\]
See Figure~\ref{fig:A_1^3 germ}.
\end{enumerate}

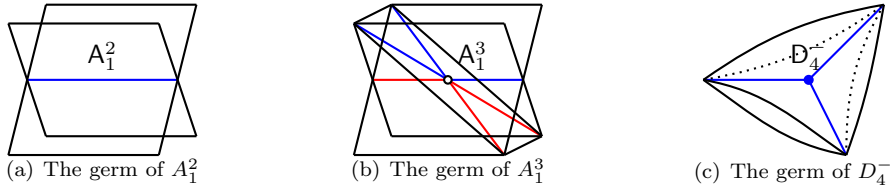
\begin{figure}[ht]
\subfigure[The germ of $A_1^2$\label{fig:A_1^2 germ}]{\makebox[0.3\textwidth]{$
\begin{tikzpicture}[baseline=-.5ex,scale=1]
\begin{scope}
\draw[blue, thick] (-1,0)--(1,0) node[black,above, midway] {$\dynA_1^2$};
\draw[thick] (-3/4,1)--(5/4,1);
\draw[thick] (5/4,1)--(3/4,-1);
\draw[thick] (3/4,-1)--(-5/4,-1);
\draw[thick] (-5/4,-1)--(-3/4,1);
\draw[thick] (-5/4,3/4)--(3/4,3/4);
\draw[thick] (3/4,3/4)--(5/4,-3/4);
\draw[thick] (5/4,-3/4)--(-3/4,-3/4);
\draw[thick] (-3/4,-3/4)--(-5/4,3/4);
\end{scope}
\end{tikzpicture}
$}}
\subfigure[The germ of $A_1^3$\label{fig:A_1^3 germ}]{\makebox[0.3\textwidth]{$
\begin{tikzpicture}[baseline=-.5ex,scale=1]
\begin{scope}
\draw[blue, thick] (0,0)--(1,0);
\draw[blue, thick] (-5/4,3/4)--(0,0);
\draw[blue, thick] (-3/4,1)--(0,0);

\draw[red, thick] (-1,0)--(0,0);
\draw[red, thick] (0,0)--(5/4,-3/4);
\draw[red, thick] (0,0)--(3/4,-1);

\draw[thick] (-3/4,1)--(5/4,1);
\draw[thick] (5/4,1)--(3/4,-1);
\draw[thick] (3/4,-1)--(-5/4,-1);
\draw[thick] (-5/4,-1)--(-3/4,1);
\draw[thick,black,fill=white] (0,0) circle (0.05) node[above right] {$\dynA_1^3$};
\draw[thick] (-5/4,3/4)--(3/4,3/4);
\draw[thick] (3/4,3/4)--(5/4,-3/4);
\draw[thick] (5/4,-3/4)--(-3/4,-3/4);
\draw[thick] (-3/4,-3/4)--(-5/4,3/4);

\draw[thick] (-5/4,3/4)--(-3/4,1);
\draw[thick] (-3/4,1)--(5/4,-3/4);
\draw[thick] (5/4,-3/4)--(3/4,-1);
\draw[thick] (3/4,-1)--(-5/4,3/4);
\end{scope}
\end{tikzpicture}
$}}
\subfigure[The germ of $D_4^-$\label{fig:D_4^- germ}]{\makebox[0.3\textwidth]{$
\begin{tikzpicture}[baseline=-.5ex,scale=1]
\begin{scope}
\draw[blue, thick, fill] (-1.4,0)--(0,0) circle (1.5pt);
\draw[blue, thick] (0,0)--(1,1);
\draw[blue, thick] (0,0)--(1/2,-1);
\node[above] at (0,0) {$\dynD_4^-$};

\draw[thick] (-1.4,0) to[out=35,in=190] (1,1);
\draw[dotted, thick] (-1.4,0) to[out=10,in=215] (1,1);

\draw[thick] (-1.4,0) to[out=-35,in=170] (1/2,-1);
\draw[thick] (-1.4,0) to[out=-10,in=145] (1/2,-1);

\draw[dotted,thick] (1/2,-1) to[out=90,in=240] (1,1);
\draw[thick] (1/2,-1) to[out=70,in=260] (1,1);
\end{scope}
\end{tikzpicture}
$}}

\caption{Three-types of wavefronts of Legendrian singularities.}
\label{fig:legendrian_singularities}
\end{figure}

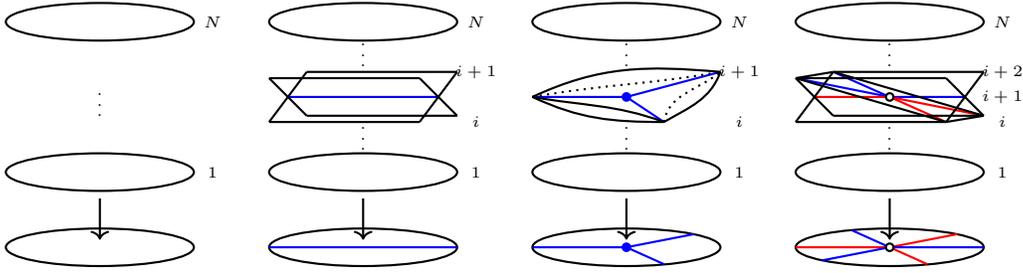
\begin{figure}[ht]
\begin{tikzpicture}

\begin{scope}[xshift=-3.5cm]
\draw[thick] \boundellipse{0,1}{1.25}{0.25};

\draw[thick] \boundellipse{0,-1}{1.25}{0.25};
\draw[thick] \boundellipse{0,-2}{1.25}{0.25};
\draw[thick, ->] (0,-1.35)--(0,-1.9);

\node at (1.5,1){\tiny $N$};
\node at (0,0){\tiny $\vdots$};
\node at (1.5,-1){\tiny $1$};

\end{scope}

\begin{scope}
\draw[thick] \boundellipse{0,1}{1.25}{0.25};

\draw[blue, thick] (-1,0)--(1,0);
\draw[thick] (-3/4,1/3)--(5/4,1/3);
\draw[thick] (5/4,1/3)--(3/4,-1/3);
\draw[thick] (3/4,-1/3)--(-5/4,-1/3);
\draw[thick] (-5/4,-1/3)--(-3/4,1/3);
\draw[thick] (-5/4,3/12)--(3/4,3/12);
\draw[thick] (3/4,3/12)--(5/4,-3/12);
\draw[thick] (5/4,-3/12)--(-3/4,-3/12);
\draw[thick] (-3/4,-3/12)--(-5/4,3/12);

\draw[thick] \boundellipse{0,-1}{1.25}{0.25};
\draw[thick] \boundellipse{0,-2}{1.25}{0.25};
\draw[blue, thick] (-5/4,-2)--(5/4,-2);
\draw[thick, ->] (0,-1.35)--(0,-1.9);

\node at (1.5,1){\tiny $N$};
\node at (0,2/3){\tiny $\vdots$};
\node at (1.5,1/3){\tiny $i+1$};
\node at (1.5,-1/3){\tiny $i$};
\node at (0,-0.45){\tiny $\vdots$};
\node at (1.5,-1){\tiny $1$};

\end{scope}

\begin{scope}[xshift=3.5cm]
\draw[thick] \boundellipse{0,1}{1.25}{0.25};

\draw[blue, thick] (-1.25,0)--(0,0);
\draw[blue, thick] (0,0)--(1.25,1/3);
\draw[blue, thick] (0,0)--(1/2,-1/3);
\draw[thick,blue,fill=blue] (0,0) circle (0.05);

\draw[thick] (-1.25,0) to[out=25,in=175] (1.25,1/3);
\draw[dotted, thick] (-1.25,0) to[out=10,in=190] (1.25,1/3);

\draw[thick] (-1.25,0) to[out=-25,in=180] (1/2,-1/3);
\draw[thick] (-1.25,0) to[out=-10,in=160] (1/2,-1/3);

\draw[dotted,thick] (1/2,-1/3) to[out=80,in=220] (1.25,1/3);
\draw[thick] (1/2,-1/3) to[out=30,in=250] (1.25,1/3);

\draw[thick] \boundellipse{0,-1}{1.25}{0.25};
\draw[thick] \boundellipse{0,-2}{1.25}{0.25};
\draw[blue, thick] (-5/4,-2)--(0,-2);
\draw[blue, thick] (0,-2)--(0.90,-1.825);
\draw[blue, thick] (0,-2)--(1/2,-2.225);

\draw[thick,blue,fill=blue] (0,-2) circle (0.05);
\draw[thick, ->] (0,-1.35)--(0,-1.9);

\node at (1.5,1){\tiny $N$};
\node at (0,2/3){\tiny $\vdots$};
\node at (1.5,1/3){\tiny $i+1$};
\node at (1.5,-1/3){\tiny $i$};
\node at (0,-0.45){\tiny $\vdots$};
\node at (1.5,-1){\tiny $1$};

\end{scope}

\begin{scope}[xshift=7cm]
\draw[thick] \boundellipse{0,1}{1.25}{0.25};

\draw[blue, thick] (0,0)--(1,0);
\draw[blue, thick] (-5/4,3/12)--(0,0);
\draw[blue, thick] (-3/4,1/3)--(0,0);

\draw[red, thick] (-1,0)--(0,0);
\draw[red, thick] (0,0)--(5/4,-3/12);
\draw[red, thick] (0,0)--(3/4,-1/3);

\draw[thick] (-3/4,1/3)--(5/4,1/3);
\draw[thick] (5/4,1/3)--(3/4,-1/3);
\draw[thick] (3/4,-1/3)--(-5/4,-1/3);
\draw[thick] (-5/4,-1/3)--(-3/4,1/3);
\draw[thick,black,fill=white] (0,0) circle (0.05);
\draw[thick] (-5/4,3/12)--(3/4,3/12);
\draw[thick] (3/4,3/12)--(5/4,-3/12);
\draw[thick] (5/4,-3/12)--(-3/4,-3/12);
\draw[thick] (-3/4,-3/12)--(-5/4,3/12);

\draw[thick] (-5/4,3/12)--(-3/4,1/3);
\draw[thick] (-3/4,1/3)--(5/4,-3/12);
\draw[thick] (5/4,-3/12)--(3/4,-1/3);
\draw[thick] (3/4,-1/3)--(-5/4,3/12);

\draw[thick] \boundellipse{0,-1}{1.25}{0.25};
\draw[thick] \boundellipse{0,-2}{1.25}{0.25};
\draw[red, thick] (-5/4,-2)--(0,-2);
\draw[red, thick] (0,-2)--(0.90,-1.825);
\draw[red, thick] (0,-2)--(1/2,-2.225);

\draw[blue, thick] (5/4,-2)--(0,-2);
\draw[blue, thick] (0,-2)--(-0.90,-2.175);
\draw[blue, thick] (0,-2)--(-1/2,-1.775);
\draw[thick,black,fill=white] (0,-2) circle (0.05);

\draw[thick, ->] (0,-1.35)--(0,-1.9);

\node at (1.5,1){\tiny $N$};
\node at (0,2/3){\tiny $\vdots$};
\node at (1.5,1/3){\tiny $i+2$};
\node at (1.5,0){\tiny $i+1$};
\node at (1.5,-1/3){\tiny $i$};
\node at (0,-0.45){\tiny $\vdots$};
\node at (1.5,-1){\tiny $1$};

\end{scope}

\end{tikzpicture}
\caption{Four-types of local charts for $N$-graphs.}
\label{fig:local_chart_3-graphs}
\end{figure}

\begin{definition}\cite[Definition~2.7]{CZ2020}
Let $\ngraph$ be an $N$-graph on a surface $S$. 
The {\em Legendrian weave} $\Legendrian(\ngraph)\subset J^1 S$ is an embedded Legendrian surface whose wavefront $\wavefront(\ngraph)\subset S\times \R$ is constructed by weaving the wavefronts $\{\wavefront(U_i)\}_{i\in I}$ from a $\ngraph$-compatible cover $\{U_i\}_{i\in I}$ with respect to the gluing data given by $\ngraph$.
\end{definition}

\begin{remark}
Note that $\Legendrian(\ngraph)$ is well-defined up to the choice of cover and up to planar isotopies.
Let $\{\varphi_t\}_{t\in[0,1]}$ be a compactly supported isotopy of $S$.
Then this induces a Legendrian isotopy of Legendrian surface $\Legendrian(\varphi_t(\ngraph))\subset J^1 S$ relative to the boundary.
\end{remark}

\subsection{Legendrian isotopies and moves on \texorpdfstring{$N$}{N}-graphs}

The idea of $N$-graph is useful in the study of Legendrian surface, because the Legendrian isotopy of the Legendrian weave $\Legendrian(\ngraph)$ can be encoded in combinatorial moves of $N$-graphs.

\begin{theorem}\cite[Theorem~1.1]{CZ2020}\label{thm:N-graph moves and legendrian isotopy}
Let $\ngraph$ be a local $N$-graph. The combinatorial moves in Figure~\ref{fig:move1-6} and Figure~\ref{fig:move7-12} are Legendrian isotopies for $\Legendrian(\ngraph)$.
\end{theorem}

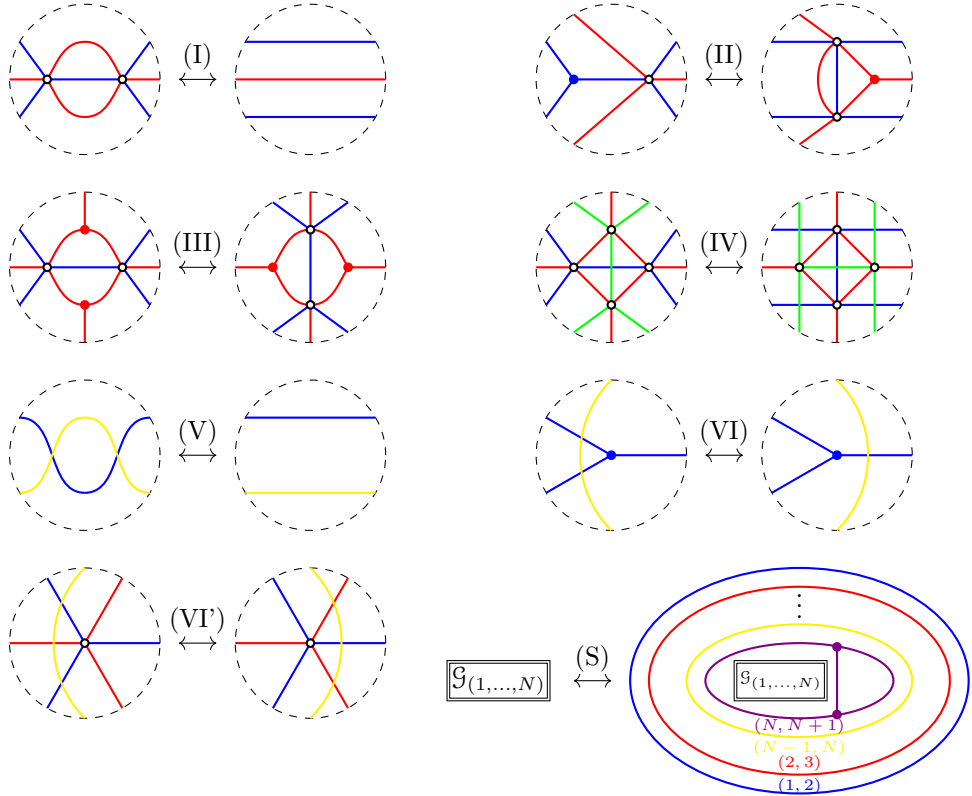
\begin{figure}[ht]
\begin{tikzpicture}
\begin{scope}
\draw [dashed] (0,0) circle [radius=1];
\draw [dashed] (3,0) circle [radius=1];
\draw [<->] (1.25,0) -- (1.75,0) node[midway, above] {\Move{I}};

\draw [blue, thick] ({-sqrt(3)/2},1/2)--(-1/2,0);
\draw [blue, thick] ({-sqrt(3)/2},-1/2)--(-1/2,0);
\draw [blue, thick] ({sqrt(3)/2},1/2)--(1/2,0);
\draw [blue, thick] ({sqrt(3)/2},-1/2)--(1/2,0);
\draw [blue, thick] (-1/2,0)--(1/2,0);

\draw [red, thick] (-1,0)--(-1/2,0) to[out=60,in=180] (0,1/2) to[out=0,in=120] (1/2,0)--(1,0);
\draw [red, thick] (-1/2,0) to[out=-60,in=180] (0, -1/2) to[out=0, in=-120] (1/2,0); 

\draw[thick,black,fill=white] (-1/2,0) circle (0.05);
\draw[thick,black,fill=white] (1/2,0) circle (0.05);

\draw [blue, thick] ({3-sqrt(3)/2},1/2)--({3+sqrt(3)/2},1/2);
\draw [blue, thick] ({3-sqrt(3)/2},-1/2)--({3+sqrt(3)/2},-1/2);
\draw [red, thick] (2,0)--(4,0);

\end{scope}

\begin{scope}[xshift=7cm]
\draw [dashed] (0,0) circle [radius=1];
\draw [dashed] (3,0) circle [radius=1];
\draw [<->] (1.25,0) -- (1.75,0) node[midway, above] {\Move{II}};

\draw [blue, thick] ({-sqrt(3)/2},1/2)--(-1/2,0);
\draw [blue, thick] ({-sqrt(3)/2},-1/2)--(-1/2,0);
\draw [blue, thick] ({sqrt(3)/2},1/2)--(1/2,0);
\draw [blue, thick] ({sqrt(3)/2},-1/2)--(1/2,0);
\draw [blue, thick] (-1/2,0)--(1/2,0);

\draw [red, thick] (-1/2,{sqrt(3)/2}) -- (1/2,0)--(1,0);
\draw [red, thick] (-1/2,{-sqrt(3)/2}) -- (1/2,0);

\draw[thick,blue,fill=blue] (-1/2,0) circle (0.05);
\draw[thick,black,fill=white] (1/2,0) circle (0.05);

\draw [blue, thick] ({3-sqrt(3)/2},1/2)--({3+sqrt(3)/2},1/2);
\draw [blue, thick] ({3-sqrt(3)/2},-1/2)--({3+sqrt(3)/2},-1/2);
\draw [blue, thick] (3,1/2)--(3,-1/2);

\draw [red, thick] (5/2,{sqrt(3)/2})--(3,1/2) to[out=-150,in=150] (3,-1/2)--(5/2,{-sqrt(3)/2});
\draw [red, thick] (3,1/2)--(7/2,0) -- (4,0);
\draw [red, thick] (3,-1/2)--(7/2,0);

\draw[thick,black,fill=white] (3,1/2) circle (0.05);
\draw[thick,black,fill=white] (3,-1/2) circle (0.05);
\draw[thick,red,fill=red] (7/2,0) circle (0.05);

\end{scope}

\begin{scope}[yshift=-2.5cm]
\draw [dashed] (0,0) circle [radius=1];
\draw [<->] (1.25,0) -- (1.75,0) node[midway, above] {\Move{III}};

\draw [blue, thick] ({-sqrt(3)/2},1/2)--(-1/2,0);
\draw [blue, thick] ({-sqrt(3)/2},-1/2)--(-1/2,0);
\draw [blue, thick] ({sqrt(3)/2},1/2)--(1/2,0);
\draw [blue, thick] ({sqrt(3)/2},-1/2)--(1/2,0);
\draw [blue, thick] (-1/2,0)--(1/2,0);

\draw [red, thick] (-1,0)--(-1/2,0) to[out=60,in=180] (0,1/2) to[out=0,in=120] (1/2,0)--(1,0);
\draw [red, thick] (-1/2,0) to[out=-60,in=180] (0, -1/2) to[out=0, in=-120] (1/2,0); 
\draw [red, thick] (0,1) to (0,1/2);
\draw [red, thick] (0,-1) to (0,-1/2);

\draw[thick,black,fill=white] (-1/2,0) circle (0.05);
\draw[thick,black,fill=white] (1/2,0) circle (0.05);

\draw[thick,red,fill=red] (0,1/2) circle (0.05);
\draw[thick,red,fill=red] (0,-1/2) circle (0.05);

\end{scope}

\begin{scope}[yshift=-2.5cm, xshift=3cm]

\draw [dashed] (0,0) circle [radius=1];

\draw [blue, thick] (-1/2,{sqrt(3)/2}) to (0,1/2) to (0,-1/2) to (-1/2,-{sqrt(3)/2});
\draw [blue, thick] (1/2,{sqrt(3)/2})--(0,1/2);
\draw [blue, thick] (1/2,-{sqrt(3)/2})--(0,-1/2);

\draw [red, thick] (-1,0)--(-1/2,0) to[out=60,in=180] (0,1/2) to[out=0,in=120] (1/2,0)--(1,0);
\draw [red, thick] (-1/2,0) to[out=-60,in=180] (0, -1/2) to[out=0, in=-120] (1/2,0); 
\draw [red, thick] (0,1) to (0,1/2);
\draw [red, thick] (0,-1) to (0,-1/2);

\draw[thick,red,fill=red] (-1/2,0) circle (0.05);
\draw[thick,red,fill=red] (1/2,0) circle (0.05);

\draw[thick,black,fill=white] (0,1/2) circle (0.05);
\draw[thick,black,fill=white] (0,-1/2) circle (0.05);
\end{scope}

\begin{scope}[xshift=7cm, yshift=-2.5cm]
\draw [dashed] (0,0) circle [radius=1];
\draw [<->] (1.25,0) -- (1.75,0) node[midway, above] {\Move{IV}};

\draw [blue, thick] ({-sqrt(3)/2},1/2)--(-1/2,0);
\draw [blue, thick] ({-sqrt(3)/2},-1/2)--(-1/2,0);
\draw [blue, thick] ({sqrt(3)/2},1/2)--(1/2,0);
\draw [blue, thick] ({sqrt(3)/2},-1/2)--(1/2,0);
\draw [blue, thick] (-1/2,0)--(1/2,0);

\draw [red, thick] (-1,0)--(-1/2,0) to (0,1/2) to (1/2,0)--(1,0);
\draw [red, thick] (-1/2,0) to (0, -1/2) to (1/2,0); 
\draw [red, thick] (0,1) to (0,1/2);
\draw [red, thick] (0,-1) to (0,-1/2);

\draw [green, thick] (-1/2,{sqrt(3)/2}) to (0,1/2) to (0,-1/2) to (-1/2,-{sqrt(3)/2});
\draw [green, thick] (1/2,{sqrt(3)/2})--(0,1/2);
\draw [green, thick] (1/2,-{sqrt(3)/2})--(0,-1/2);

\draw[thick,black,fill=white] (-1/2,0) circle (0.05);
\draw[thick,black,fill=white] (1/2,0) circle (0.05);

\draw[thick,black,fill=white] (0,1/2) circle (0.05);
\draw[thick,black,fill=white] (0,-1/2) circle (0.05);

\end{scope}

\begin{scope}[xshift=10cm, yshift=-2.5cm]
\draw [dashed] (0,0) circle [radius=1];
\draw [blue, thick] ({-sqrt(3)/2},1/2)--({+sqrt(3)/2},1/2);
\draw [blue, thick] ({-sqrt(3)/2},-1/2)--({+sqrt(3)/2},-1/2);
\draw [blue, thick] (0,1/2)--(0,-1/2);

\draw [red, thick] (-1,0)--(-1/2,0) to (0,1/2) to (1/2,0)--(1,0);
\draw [red, thick] (-1/2,0) to (0, -1/2) to (1/2,0); 
\draw [red, thick] (0,1) to (0,1/2);
\draw [red, thick] (0,-1) to (0,-1/2);

\draw [green, thick] (-1/2,{sqrt(3)/2}) to (-1/2,{-sqrt(3)/2});
\draw [green, thick] (1/2,{sqrt(3)/2}) to (1/2,{-sqrt(3)/2});
\draw [green, thick] (-1/2,0) to (1/2,0);

\draw[thick,black,fill=white] (-1/2,0) circle (0.05);
\draw[thick,black,fill=white] (1/2,0) circle (0.05);

\draw[thick,black,fill=white] (0,1/2) circle (0.05);
\draw[thick,black,fill=white] (0,-1/2) circle (0.05);

\end{scope}

\begin{scope}[xshift=0cm, yshift=-5cm]
\draw [dashed] (0,0) circle [radius=1];
\draw [<->] (1.25,0) -- (1.75,0) node[midway, above] {\Move{V}};

\draw [blue, thick] ({-sqrt(3)/2},1/2)to[out=0,in=180](0,-1/2);
\draw [blue, thick] ({sqrt(3)/2},1/2)to[out=180,in=0](0,-1/2);

\draw [yellow, thick] ({-sqrt(3)/2},-1/2)to[out=0,in=180](0,1/2);
\draw [yellow, thick] ({sqrt(3)/2},-1/2)to[out=180,in=0](0,1/2);

\end{scope}

\begin{scope}[xshift=3cm, yshift=-5cm]
\draw [dashed] (0,0) circle [radius=1];

\draw [blue, thick] ({-sqrt(3)/2},1/2) to ({sqrt(3)/2},1/2);

\draw [yellow, thick] ({-sqrt(3)/2},-1/2) to ({sqrt(3)/2},-1/2);
\end{scope}

\begin{scope}[xshift=7cm, yshift=-5cm]
\draw [dashed] (0,0) circle [radius=1];
\draw [<->] (1.25,0) -- (1.75,0) node[midway, above] {\Move{VI}};
\draw [blue, thick] ({-sqrt(3)/2},1/2) to (0,0) to(1,0);
\draw [blue, thick] ({-sqrt(3)/2},-1/2) to (0,0);

\draw[thick,blue,fill=blue] (0,0) circle (0.05);

\draw[yellow, thick] (0,1) to[out=-135,in=135] (0,-1);
\end{scope}

\begin{scope}[xshift=10cm, yshift=-5cm]
\draw [dashed] (0,0) circle [radius=1];
\draw [blue, thick] ({-sqrt(3)/2},1/2) to (0,0) to(1,0);
\draw [blue, thick] ({-sqrt(3)/2},-1/2) to (0,0);

\draw[thick,blue,fill=blue] (0,0) circle (0.05);

\draw[yellow, thick] (0,1) to[out=-45,in=45] (0,-1);
\end{scope}

\begin{scope}[xshift=0cm, yshift=-7.5cm]
\draw [dashed] (0,0) circle [radius=1];
\draw [<->] (1.25,0) -- (1.75,0) node[midway, above] {\Move{VI'}};
\draw [blue, thick] ({-1/2},{sqrt(3)/2}) to (0,0) to(1,0);
\draw [blue, thick] ({-1/2},{-sqrt(3)/2}) to (0,0);

\draw [red, thick] (-1,0) to (0,0) to(1/2,{sqrt(3)/2});
\draw [red, thick] (0,0) to(1/2,{-sqrt(3)/2});

\draw[thick,black,fill=white] (0,0) circle (0.05);

\draw[yellow, thick] (0,1) to[out=-135,in=135] (0,-1);
\end{scope}

\begin{scope}[xshift=3cm, yshift=-7.5cm]
\draw [dashed] (0,0) circle [radius=1];
\draw [blue, thick] ({-1/2},{sqrt(3)/2}) to (0,0) to(1,0);
\draw [blue, thick] ({-1/2},{-sqrt(3)/2}) to (0,0);

\draw [red, thick] (-1,0) to (0,0) to(1/2,{sqrt(3)/2});
\draw [red, thick] (0,0) to(1/2,{-sqrt(3)/2});

\draw[thick,black,fill=white] (0,0) circle (0.05);

\draw[yellow, thick] (0,1) to[out=-45,in=45] (0,-1);
\end{scope}

\begin{scope}[xshift=5.5cm, yshift=-8cm]
\draw [double] (2/3,1/4) -- (-2/3,1/4) -- (-2/3,-1/4) -- (2/3,-1/4)-- cycle;
\node at (0,0) {$\ngraph_{(1,\dots,N)}$};
\draw [<->] (1,0) -- (1.5,0) node[midway, above] {\Move{S}};
\end{scope}

\begin{scope}[xshift=9.25cm, yshift=-8cm]
\draw [double] (0.6,1/4) -- (-0.6,1/4) -- (-0.6,-1/4) -- (0.6,-1/4)-- cycle;
\node at (0,0) {\tiny$\ngraph_{(1,\dots,N)}$};
\draw[thick, violet] \boundellipse{0.25,0}{1.25}{0.5};
\draw[thick, violet] (0.75,0.45) to (0.75,-0.45) ;
\draw[thick,violet,fill=violet] (0.75,0.45) circle (0.05);
\draw[thick,violet,fill=violet] (0.75,-0.45) circle (0.05);

\draw[thick, yellow] \boundellipse{0.25,0}{1.5}{0.75};
\draw[thick, red] \boundellipse{0.25,0}{2}{1.25};
\draw[thick, blue] \boundellipse{0.25,0}{2.25}{1.5};

\node at (0.25,1.1) {$\vdots$};
\node at (0.25,-0.6) {\tiny\color{violet}$(N,N+1)$};
\node at (0.25,-0.9) {\tiny\color{yellow}$(N-1,N)$};
\node at (0.25,-1.1) {\tiny\color{red}$(2,3)$};
\node at (0.25,-1.4) {\tiny\color{blue}$(1,2)$};

\end{scope}

\end{tikzpicture}
\caption{Combinatorial moves for Legendrian isotopies of surface $\Legendrian(\ngraph)$.
Here the pairs ({\color{blue} blue}, {\color{red} red}) and ({\color{red} red}, {\color{green} green}) are consecutive. Other pairs are not.}
\label{fig:move1-6}
\end{figure}

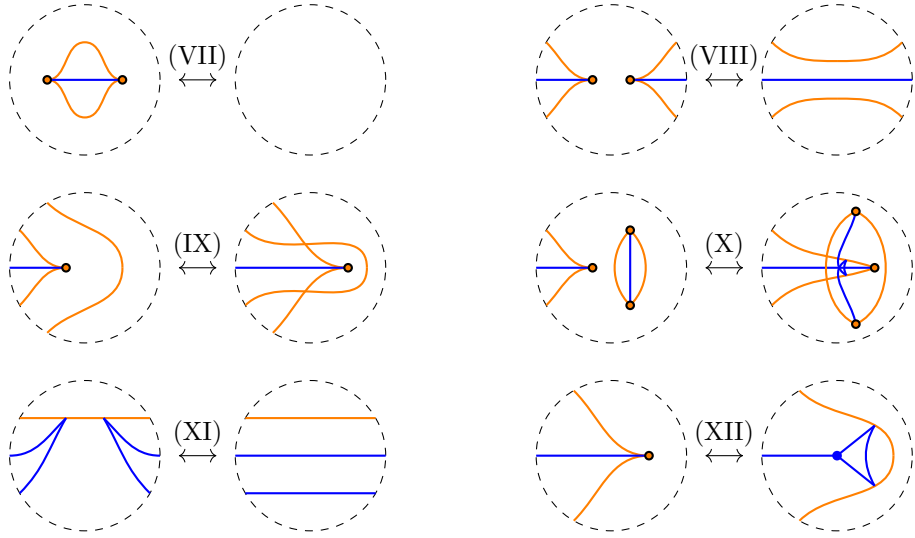
\begin{figure}[ht]
\begin{tikzpicture}
\begin{scope}
\draw [dashed] (0,0) circle [radius=1];
\draw [dashed] (3,0) circle [radius=1];
\draw [<->] (1.25,0) -- (1.75,0) node[midway, above] {\Move{VII}};
\draw [orange, thick] (-1/2,0) to[out=0,in=180] (0,1/2) to[out=0,in=180] (1/2,0);
\draw [orange, thick] (-1/2,0) to[out=0,in=180] (0,-1/2) to[out=0,in=180] (1/2,0);
\draw [blue, thick] (-1/2,0) to (1/2,0);
\draw[thick,black,fill=orange] (-1/2,0) circle (0.05);
\draw[thick,black,fill=orange] (1/2,0) circle (0.05);

\end{scope}

\begin{scope}[xshift=7cm, yshift=0cm]
\draw [dashed] (0,0) circle [radius=1];
\draw [dashed] (3,0) circle [radius=1];
\draw [<->] (1.25,0) -- (1.75,0) node[midway, above] {\Move{VIII}};
\draw [orange, thick] ({-sqrt(3)/2},1/2) to[out=-45,in=180] (-1/4,0);
\draw [orange, thick] ({-sqrt(3)/2},-1/2) to[out=45,in=180] (-1/4,0);
\draw [orange, thick] ({sqrt(3)/2},1/2) to[out=-135,in=0] (1/4,0);
\draw [orange, thick] ({sqrt(3)/2},-1/2) to[out=135,in=0] (1/4,0);
\draw [blue,thick] (-1,0)--(-1/4,0);
\draw [blue,thick] (1,0)--(1/4,0);
\draw[thick,black,fill=orange] (-1/4,0) circle (0.05);
\draw[thick,black,fill=orange] (1/4,0) circle (0.05);
\end{scope}

\begin{scope}[xshift=10cm, yshift=0cm]
\draw [orange, thick] ({-sqrt(3)/2},1/2) to[out=-45,in=180] (0,1/4);
\draw [orange, thick] ({sqrt(3)/2},1/2) to[out=-135,in=0] (0,1/4);
\draw [orange, thick] ({-sqrt(3)/2},-1/2) to[out=45,in=180] (0,-1/4);
\draw [orange, thick] ({sqrt(3)/2},-1/2) to[out=135,in=0] (0,-1/4);
\draw [blue, thick] (-1,0)--(1,0);

\end{scope}

\begin{scope}[xshift=0cm, yshift=-2.5cm]
\draw [dashed] (0,0) circle [radius=1];
\draw [dashed] (3,0) circle [radius=1];
\draw [<->] (1.25,0) -- (1.75,0) node[midway, above] {\Move{IX}};
\draw [orange, thick] ({-sqrt(3)/2},1/2) to[out=-45,in=180] (-1/4,0);
\draw [orange, thick] ({-sqrt(3)/2},-1/2) to[out=45,in=180] (-1/4,0);
\draw [blue,thick] (-1,0)--(-1/4,0);
\draw[thick,black,fill=orange] (-1/4,0) circle (0.05);
\draw [thick, orange] (-1/2,{sqrt(3)/2}) to[out=-45,in=90] (1/2,0);
\draw [thick, orange] (-1/2,{-sqrt(3)/2}) to[out=45,in=-90] (1/2,0);
\end{scope}

\begin{scope}[xshift=3cm, yshift=-2.5cm]
\draw [thick, orange] (-1/2,{sqrt(3)/2}) to[out=-45,in=180] (1/2,0);
\draw [thick, orange] (-1/2,{-sqrt(3)/2}) to[out=45,in=180] (1/2,0);
\draw [blue,thick] (-1,0)--(1/2,0);
\draw[thick,black,fill=orange] (1/2,0) circle (0.05);
\draw [orange, thick] ({-sqrt(3)/2},1/2) to[out=-45,in=90] (3/4,0);
\draw [orange, thick] ({-sqrt(3)/2},-1/2) to[out=45,in=-90] (3/4,0);

\end{scope}

\begin{scope}[xshift=7cm, yshift=-2.5cm]
\draw [dashed] (0,0) circle [radius=1];
\draw [dashed] (3,0) circle [radius=1];
\draw [<->] (1.25,0) -- (1.75,0) node[midway, above] {\Move{X}};
\draw [orange, thick] ({-sqrt(3)/2},1/2) to[out=-45,in=180] (-1/4,0);
\draw [orange, thick] ({-sqrt(3)/2},-1/2) to[out=45,in=180] (-1/4,0);
\draw [blue,thick] (-1,0)--(-1/4,0);
\draw [thick,black,fill=orange] (-1/4,0) circle (0.05);

\draw [orange, thick] (1/4,1/2) to[out=-135,in=135] (1/4,-1/2);
\draw [orange, thick] (1/4,1/2) to[out=-45,in=45] (1/4,-1/2);
\draw [blue, thick] (1/4,1/2) to (1/4,-1/2);
\draw [thick,black,fill=orange] (1/4,1/2) circle (0.05);
\draw [thick,black,fill=orange] (1/4,-1/2) circle (0.05);

\end{scope}

\begin{scope}[xshift=10cm, yshift=-2.5cm]
\draw [orange, thick] ({-sqrt(3)/2},1/2) to[out=-45,in=160] (1/2,0);
\draw [orange, thick] ({-sqrt(3)/2},-1/2) to[out=45,in=200] (1/2,0);
\draw [blue,thick] (-1,0)--(1/2,0);
\draw [thick,black,fill=orange] (1/2,0) circle (0.05);

\draw [orange, thick] (1/4,0.75) to[out=-155,in=155] (1/4,-0.75);
\draw [orange, thick] (1/4,0.75) to[out=-25,in=25] (1/4,-0.75);

\draw [blue, thick] (1/4,0.75) to[out=-100,in=150] (1/8,-0.1);
\draw [blue, thick] (1/4,-0.75) to[out=100,in=-150] (1/8,0.1);
\draw [blue, thick] (1/8,0.1) to[out=-120,in=120] (1/8,-0.1);
\draw [thick,black,fill=orange] (1/4,0.75) circle (0.05);
\draw [thick,black,fill=orange] (1/4,-0.75) circle (0.05);

\end{scope}

\begin{scope}[xshift=0cm, yshift=-5cm]
\draw [dashed] (0,0) circle [radius=1];
\draw [dashed] (3,0) circle [radius=1];
\draw [<->] (1.25,0) -- (1.75,0) node[midway, above] {\Move{XI}};

\draw [orange, thick] ({-sqrt(3)/2},1/2) to ({sqrt(3)/2},1/2);
\draw [blue, thick] (-1,0) to[out=0,in=-135] (-1/4,1/2);
\draw [blue, thick] ({-sqrt(3)/2},-1/2) to[out=45,in=-120] (-1/4,1/2);

\draw [blue, thick] (1,0) to[out=180,in=-45] (1/4,1/2);
\draw [blue, thick] ({sqrt(3)/2},-1/2) to[out=135,in=-60] (1/4,1/2);

\end{scope}

\begin{scope}[xshift=3cm, yshift=-5cm]

\draw [orange, thick] ({-sqrt(3)/2},1/2) to ({sqrt(3)/2},1/2);
\draw [blue, thick] (-1,0) to (1,0);
\draw [blue, thick] ({-sqrt(3)/2},-1/2) to ({sqrt(3)/2},-1/2);

\end{scope}

\begin{scope}[xshift=7cm, yshift=-5cm]
\draw [dashed] (0,0) circle [radius=1];
\draw [dashed] (3,0) circle [radius=1];
\draw [<->] (1.25,0) -- (1.75,0) node[midway, above] {\Move{XII}};
\draw [thick, orange] (-1/2,{sqrt(3)/2}) to[out=-45,in=180] (1/2,0);
\draw [thick, orange] (-1/2,{-sqrt(3)/2}) to[out=45,in=180] (1/2,0);
\draw [blue,thick] (-1,0)--(1/2,0);
\draw[thick,black,fill=orange] (1/2,0) circle (0.05);

\end{scope}

\begin{scope}[xshift=10cm, yshift=-5cm]

\draw [thick, orange] (-1/2,{sqrt(3)/2}) to[out=-45,in=90] (3/4,0);
\draw [thick, orange] (-1/2,{-sqrt(3)/2}) to[out=45,in=-90] (3/4,0);
\draw [blue,thick] (-1,0)--(0,0);
\draw [blue,thick] (0,0) to (1/2,0.4);
\draw [blue,thick] (0,0) to (1/2,-0.4);
\draw [blue,thick] (1/2,0.4) to[out=-120,in=120] (1/2,-0.4);
\draw[thick,blue,fill=blue] (0,0) circle (0.05);
\end{scope}

\end{tikzpicture}
\caption{Combinatorial moves for Legendrian isotopies of surface $\Legendrian(\ngraph)$: These are moves involving $A_3$-swallowtail singularities, the {\color{orange} orange} vertex. Here the {\color{orange} orange} lines are locus of cusp singularities.}
\label{fig:move7-12}
\end{figure}

\subsubsection{$N$-graphs on $\disk^2$}

Let $\legendrian_\beta \subset J^1\sphere^1$ be a Legendrian link obtained from a Legendrian line in $\R^3$ by satelliting the Legendrian unknot. Here $\legendrian_\beta$ is the closure of a positive $N$-strand braid $\beta$.
The braid word $\beta$ consist of alphabets $\sigma_1,\dots, \sigma_{N-1}$, and these give an $(N-1)$-tuple of sets of points in $\sphere^1$ which can be regarded as a boundary data of $N$-graphs on $\disk^2$.
By the setup in \S\ref{sec:geometric setup}, $\pi\circ\iota(\Legendrian(\ngraph))$ induces an exact (possibly immersed) Lagrangian filling in $(\R^4,\omega_{\rm st})$ of $\iota(\legendrian_\beta)$.
Let us denote the equivalence class of a $N$-graph $\ngraph\subset \disk^2$ up to the moves $\Move{I},\dots,\Move{XII}$ by $[\ngraph]$.

\begin{remark}\label{remark:stabilization}
For $N$-graphs $\ngraph$ on $\disk^2$ the stabilization, becomes Move $\Move{S}$ in Figure~\ref{fig:stab. on disk}:
\begin{figure}[ht]
\begin{tikzpicture}
\begin{scope}
\draw [thick] (0,0) circle [radius=1];
\node at (0,0) {$\ngraph_{(1,\dots,N)}$};
\draw [<->] (1.5,0) -- (2,0) node[midway, above] {\Move{S}};
\end{scope}

\begin{scope}[xshift=4.5cm]
\draw [thick] (0,0) circle [radius=1.5];
\draw [double] (0,1.5) to (0,-1.5) arc [radius=1.5, start angle=-90, end angle= 90] -- cycle;
\node at (0.75,0) {$\ngraph_{(1,\dots,N)}$};
\draw [thick, blue] ({1.5*cos(165)},{1.5*sin(165)}) to[out=-45,in=90] (-1.25,0); 
\draw [thick, blue] ({1.5*cos(195)},{1.5*sin(195)}) node[left] {\tiny$(1,2)$} to[out=45,in=-90] (-1.25,0); 
\draw [thick, red] ({1.5*cos(150)},{1.5*sin(150)}) to[out=-45,in=90] (-0.95,0); 
\draw [thick, red] ({1.5*cos(210)},{1.5*sin(210)}) node[left] {\tiny$(2,3)$} to[out=45,in=-90] (-0.95,0); 
\draw [thick, yellow] ({1.5*cos(135)},{1.5*sin(135)}) to[out=-45,in=90] (-0.5,0); 
\draw [thick, yellow] ({1.5*cos(225)},{1.5*sin(225)}) node[left] {\tiny$(N-1,N)$} to[out=45,in=-90] (-0.5,0); 
\draw [thick, violet] ({1.5*cos(100)},{1.5*sin(100)}) to ({1.5*cos(260)},{1.5*sin(260)});
\draw [thick, violet] ({1.5*cos(115)},{1.5*sin(115)}) node[left] {\tiny$(N,N+1)$} to ({0.6*cos(115)},{0.6*sin(115)}); 
\node at (-0.7,0) {$\cdots$};
\draw[thick,violet,fill=violet] ({0.6*cos(115)},{0.6*sin(115)}) circle (0.05);
\end{scope}

\end{tikzpicture}
\caption{A stabilization of an $N$-graph on $\disk^2$.}
\label{fig:stab. on disk}
\end{figure}
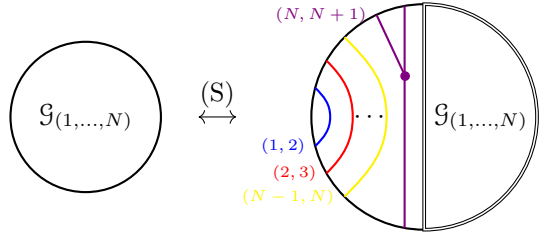
\end{remark}

\begin{definition}
An $N$-graph $\ngraph\subset \disk^2$ is called {\em free} if the induced Legendrian weave $\Legendrian(\ngraph)\subset J^1\disk^2$ can be woven without interior Reeb chord. 
\end{definition}

\begin{example}\cite[Example 7.3]{CZ2020}\label{ex:free N-graph}
Let $\ngraph\subset \disk^2$ be a $2$-graph such that $\disk^2\setminus \ngraph$ is simply connected relative to the boundary $\boundary\disk^2\cap(\disk^2\setminus \ngraph)$. Then $\ngraph$ is free if and only if $\ngraph$ has no faces contained in $\mathring\disk^2$.
Note that each of such faces admits at least one Reeb chord, see Figure~\ref{fig:N-graphs with Reeb chords}.
\end{example}

\begin{figure}[ht]
\begin{tikzcd}
\begin{tikzpicture}
\draw [dashed] (0,0) circle [radius=1.5];
\draw [thick,blue] (-1.5,0)-- (-1,0) to[out=90,in=90] (1,0)--(1.5,0) (-1,0) to[out=-90,in=-90] (1,0);
\draw [thick, blue, fill] (-1,0) circle (1.5pt)  (1,0) circle (1.5pt);
\end{tikzpicture}
&
\begin{tikzpicture}
\draw [dashed] (0,0) circle [radius=1.5];
\draw [thick,blue] (90:1.5) -- (90:1) (210:1.5) -- (210:1) (-30:1.5) -- (-30:1)
(90:1) -- (210:1) -- (-30:1)-- (90:1);
\draw [thick, blue, fill] (90:1) circle (1.5pt) (210:1) circle (1.5pt) (-30:1) circle (1.5pt);
\end{tikzpicture}
\end{tikzcd}
\caption{$N$-graphs with Reeb chords}
\label{fig:N-graphs with Reeb chords}
\end{figure}
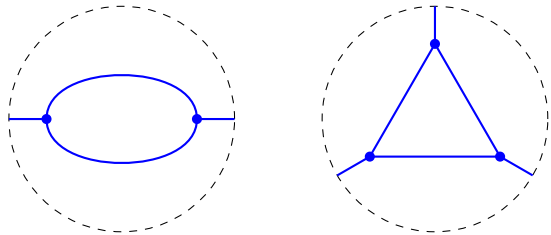

To investigate the Reeb chords of $\Legendrian(\ngraph)$ in $J^1\disk^2$,
let us consider the wavefront $\wavefront(\ngraph)$ in $\disk^2\times \R$.
Label the sheets of the wavefront 
\begin{equation}\label{equation:wavefront decomposition}
\wavefront(\ngraph)=\bigcup_{i=1}^{N}\wavefront_i
\end{equation}
by the $z$-coordinate from the bottom to top. Let $f_i:\disk^2\to \R$ be a function whose graph becomes $\wavefront_i$, and let $h_{ij}:\disk^2\to \R$ be a difference function given by $f_j-f_i$ for any $i,j\in [N]$ with $i<j$.
By the construction $h_{i\, i+1}^{-1}(0)$ gives $\ngraph_i\subset \ngraph$. 
The critical points of $h_{ij}$ on $\mathring{\disk}^2\setminus \ngraph$ are the possible candidates for the Reeb chords. In other words, to guarantee that $\ngraph$ is free, it suffices to show that $h_{ij}$ has no critical point on $\mathring{\disk}^2\setminus \ngraph$.

\begin{figure}[ht]
\begin{tikzcd}
\begin{tikzpicture}[baseline=-.5ex,scale=0.6]
\draw[thick] (0,0) circle (3cm);
\draw[red, thick] (0,0) -- (0:3) (0,0) -- (120:3) (0,0) -- (240:3);
\draw[blue, thick, fill] 
(0,0) -- (60:1) circle (2pt) -- (90:3) 
(60:1) -- (45:2) circle (2pt) -- (30:3) 
(45:2) -- (60:3);
\draw[blue, thick, fill] 
(0,0) -- (180:1) circle (2pt) -- (210:3) 
(180:1) -- (165:2) circle (2pt) -- (150:3) 
(165:2) -- (180:3);
\draw[blue, thick, fill] 
(0,0) -- (300:1) circle (2pt) -- (330:3) 
(300:1) -- (285:2) circle (2pt) -- (270:3) 
(285:2) -- (300:3);
\draw[thick, fill=white] (0,0) circle (2pt);
\end{tikzpicture}
&
\begin{tikzpicture}[baseline=-.5ex,scale=0.6]
\draw[thick] (0,0) circle (3cm);
\draw[blue, thick, fill] 
(0,0) circle (2pt) -- (60:1) circle (2pt) -- (90:3) 
(60:1) -- (45:2) circle (2pt) -- (30:3) 
(45:2) -- (60:3);
\draw[blue, thick, fill] 
(0,0) -- (180:1) circle (2pt) -- (210:3) 
(180:1) -- (165:2) circle (2pt) -- (150:3) 
(165:2) -- (180:3);
\draw[blue, thick, fill] 
(0,0) -- (300:1) circle (2pt) -- (330:3) 
(300:1) -- (285:2) circle (2pt) -- (270:3) 
(285:2) -- (300:3);
\node at (0:1.5) {$F_{1;1}$};
\node at (120:1.5) {$F_{1;4}$};
\node at (240:1.5) {$F_{1;7}$};
\node at (41:2.55) {$F_{1;2}$};
\node at (167:2.55) {$F_{1;5}$};
\node at (286:2.7) {$F_{1;8}$};
\node at (67:2.2) {$F_{1;3}$};
\node at (192:2.2) {$F_{1;6}$};
\node at (307:2.2) {$F_{1;9}$};
\end{tikzpicture}
&
\begin{tikzpicture}[baseline=-.5ex,scale=0.6]
\draw[thick] (0,0) circle (3cm);
\draw[red, thick, fill] (0,0) circle (2pt) -- (0:3) (0,0) -- (120:3) (0,0) -- (240:3);
\node at (60:1.5) {$F_{2;1}$};
\node at (180:1.5) {$F_{2;2}$};
\node at (-60:1.5) {$F_{2;3}$};
\end{tikzpicture}
\end{tikzcd}
\caption{$\ngraph(2,2,2)$, $\ngraph_1$, and $\ngraph_2$}
\end{figure}

Now apply this idea to the $3$-graph $\ngraph(a,b,c)$ in the introduction.
In order to construct $h_{1\, 2}$ and $h_{2\,3}$, consider the following graph complements $\mathring\disk^2\setminus \ngraph_i$ for $i=1,2$.
Let us denote the closure of  connected components of $\mathring\disk^2\setminus \ngraph_i$ by $\{F_{i;k}\}_{k\in K_i}$. Each $F_{i;k}$ is a polygon and exactly one edge comes from the boundary $\boundary \disk^2$.

\begin{figure}[ht]
\begin{tikzpicture}
\draw[thick] (-1.5,0) -- (1.5,0) ;
\draw[thick,blue] (1.5,0) -- (1,1) -- (0.5,1.5) (-1.5,0) -- (-1,1) -- (-0.5,1.5);
\draw[thick,blue,dotted] (0.5,1.5) -- (-0.5,1.5);
\foreach \x in {0,1,2,3,4}
{
\draw[opacity=0.5] 
(1+0.1*\x,1-0.2*\x)--(1+0.1*\x,0)
(-1-0.1*\x,1-0.2*\x)--(-1-0.1*\x,0)
(0.9-0.1*\x,1.1+0.1*\x)--(0.9-0.1*\x,0)
(-0.9+0.1*\x,1.1+0.1*\x)--(-0.9+0.1*\x,0)
(0.4-0.1*\x,1.5) --(0.4-0.1*\x,0)
(-0.4+0.1*\x,1.5) --(-0.4+0.1*\x,0)
;
}
\end{tikzpicture}
\caption{The vertical gray lines present gradient flow lines of $h_{1\,2}|_{F_{1;k}}$. Here the bottom (black) line comes from $\boundary\disk^2$.}
\label{fig:gradient}
\end{figure}
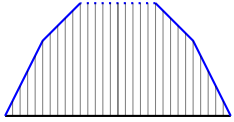

We then consider functions $h_{i\, i+1}$, $i=1,2$ satisfying the followings:
\begin{itemize}
\item $h_{i\, i+1}$ is smooth and nonnegative.
\item $h_{i\, i+1}(x)=0$ if and only if $x\in \ngraph_i$.
\item $h_{i\, i+1}$ has no critical point on $\mathring\disk^2\setminus \ngraph_i$.
\item For any $k\in K_i$, positive gradient flow lines of $h_{i\, i+1}|_{F_{i;k}}$ head for the edge from $\boundary\disk^2$, see Figure~\ref{fig:gradient}.
\end{itemize}
By the construction of $h_{i\, i+1}$ and the definition of Reeb chord, there is no Reeb chord connecting $\wavefront_i$ and $\wavefront_{i+1}$ for $i=1,2$. Now consider the the gradient flow lines of $h_{1\,2}+h_{2\,3}$ to see the Reeb chords from $\wavefront_1$ to $\wavefront_3$.
Without loss of generality, we may assume that $\|\nabla h_{1\,2}\|<\|\nabla h_{2\,3}\|$ except a small neighborhood of $\ngraph_2$. Then by the configuration of $\ngraph(a,b,c)=\ngraph_1\cup \ngraph_2$ the gradient flow lines of $h_{1\,2}+h_{2\,3}$ never vanish except the hexagonal point. In conclusion, we can construct the wavefront $\wavefront(\ngraph(a,b,c))$ without interior Reeb chords.

\begin{lemma}
The $3$-graph $\ngraph(a,b,c)$ is free.
\end{lemma}

\subsubsection{$N$-graphs on $\annulus$}
Let $\annulus$ be the oriented annulus with two boundary components $\boundary_+\annulus$ and $\boundary_-\annulus$, and let $\ngraph$ be a $N$-graph on $\annulus$.
We say that $\ngraph$ is of type $(\legendrian_+, \legendrian_-)$ if $\ngraph$ on $\boundary_+\annulus$ and $\boundary_-\annulus$ are given by Legendrian links $\legendrian_+$ and $\legendrian_-$, respectively.
We may regard the $N$-graph $\ngraph$ of type $(\legendrian_+, \legendrian_-)$ as a cobordism between $\legendrian_+$ and $\legendrian_-$.

Suppose that two annular $N$-graphs $\ngraph_1$ and $\ngraph_2$ are of type $(\legendrian_1,\legendrian_2)$ and $(\legendrian_2, \legendrian_3)$.
Then two $N$-graphs can be merged or piled in a natural way to obtain the annular $N$-graph, denoted by $\ngraph_1 \cdot \ngraph_2$ of type $(\legendrian_1,\legendrian_3)$.
Let $\ngraph\subset\disk^2$ be an $N$-graph with $\boundary\ngraph=\legendrian_1$, then the {\em padding} operation $\ngraph_1\ngraph$ is defined by gluing along the boundary $\legendrian_1$.
Note that if there is a rotational symmetry on $\legendrian_1$, then the operation $\ngraph_1\ngraph$ is well-defined only up to that symmetry.

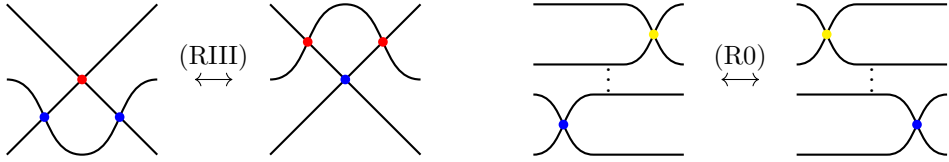
\begin{figure}[ht]
\begin{tikzpicture}
\begin{scope}
\draw[thick] (-1,-1) to (1,1);
\draw[thick] (-1,1) to (1,-1);
\draw[thick] (-1,0) to[out=0, in=180] (0,-1) to[out=0,in=180] (1,0);

\draw[thick,blue,fill=blue] (-0.5,-0.5) circle (0.05);
\draw[thick,blue,fill=blue] (0.5,-0.5) circle (0.05);
\draw[thick,red,fill=red] (0,0) circle (0.05);

\draw [<->] (1.5,0) -- (2,0) node[midway, above] {\Move{RIII}};
\end{scope}

\begin{scope}[xshift=3.5cm]
\draw[thick] (-1,-1) to (1,1);
\draw[thick] (-1,1) to (1,-1);
\draw[thick] (-1,0) to[out=0, in=180] (0,1) to[out=0,in=180] (1,0);

\draw[thick,red,fill=red] (-0.5,0.5) circle (0.05);
\draw[thick,red,fill=red] (0.5,0.5) circle (0.05);
\draw[thick,blue,fill=blue] (0,0) circle (0.05);

\end{scope}

\begin{scope}[xshift=7cm]
\draw[thick] (1,1) to[out=180,in=0] (0.2,0.2) to (-1,0.2);
\draw[thick] (1,0.2) to[out=180,in=0] (0.2,1) to (-1,1);

\node at (0,0.1) {$\vdots$};

\draw[thick] (1,-0.2) to (-0.2,-0.2) to[out=180,in=0] (-1,-1);
\draw[thick] (1,-1) to (-0.2,-1) to[out=180,in=0] (-1,-0.2);

\draw[thick,yellow,fill=yellow] (0.6,0.6) circle (0.05);
\draw[thick,blue,fill=blue] (-0.6,-0.6) circle (0.05);

\draw [<->] (1.5,0) -- (2,0) node[midway, above] {\Move{R0}};
\end{scope}

\begin{scope}[xshift=10.5cm]
\draw[thick] (-1,1) to[out=0,in=180] (-0.2,0.2) to (1,0.2);
\draw[thick] (-1,0.2) to[out=0,in=180] (-0.2,1) to (1,1);

\node at (0,0.1) {$\vdots$};

\draw[thick] (-1,-0.2) to (0.2,-0.2) to[out=0,in=180] (1,-1);
\draw[thick] (-1,-1) to (0.2,-1) to[out=0,in=180] (1,-0.2);

\draw[thick,yellow,fill=yellow] (-0.6,0.6) circle (0.05);
\draw[thick,blue,fill=blue] (0.6,-0.6) circle (0.05);
\end{scope}

\end{tikzpicture}
\caption{Reidemeister moves in $J^1\sphere^1$ or $\R^3$ avoiding cusp singularities.}
\end{figure}

Let us illustrate \emph{elementary annulus $N$-graph}s coming from the Legendrian isotopies in $J^1\sphere^1$. 
The following two Legendrian Reidemeister moves \Move{RIII} and \Move{R0} can be interpreted as $N$-graphs $\ngraph_{\Move{RIII}}$ and $\ngraph_{\Move{R0}}$ on the annulus $\annulus$, respectively, as depicted in Figure~\ref{fig:elementary annulus N-graph}.
The Move \Move{I} and \Move{V} of $N$-graphs in Figure~\ref{fig:move1-6} imply that the inverses $\ngraph_{\Move{RIII}}^{-1}$ and $\ngraph_{\Move{R0}}^{-1}$ can be obtained by reversing the role of the inner- and outer boundaries.

Suppose that there are certain rotational symmetry on $N$-graphs.
Let us consider a \emph{rotational annulus $N$-graph} which is trivial as an $N$-graph but rotated respecting the symmetry. A typical example comes from Legendrian torus link $\legendrian(n,m)$ of maximal Thurston-Bennequin number. The right one in Figure~\ref{fig:elementary annulus N-graph} is a rotational annulus $N$-graph for $\legendrian(3,3)$.
This type of annular $N$-graphs play a crucial role in producing a sequence of distinct exact Lagrangian fillings of positive braid Legendrian links, see \cite{Kal2006, CG2020, GSW2020b}.

\begin{figure}[ht]
\begin{tikzcd}[row sep=0.25cm]
\begin{tikzpicture}[baseline=-.5ex,scale=1]
\draw [thick] (0,0) circle [radius=0.7];
\draw [thick] (0,0) circle [radius=1.5];
\draw [thick, dotted] (260:1) arc (260:280:1);
\draw [thick, blue] (60:0.7) to[out=60, in=-30] (90:1.1) (120:0.7) to[out=120, in=-150] (90:1.1) -- (90:1.5);
\draw [thick, red] (90:0.7) -- (90:1.1) to[out=150,in=-60] (120:1.5) (90:1.1) to[out=30,in=-120] (60:1.5);
\draw [thick, black]
(30:0.7) -- (30:1.5) 
(150:0.7) -- (150:1.5)
(0:0.7) -- (0:1.5) 
(180:0.7) -- (180:1.5);
\node at (0,0) {$\ngraph_{\Move{RIII}}$};
\end{tikzpicture}
\cdot
\begin{tikzpicture}[baseline=-.5ex,scale=1]
\draw [thick] (0,0) circle [radius=0.7];
\draw [thick, black]
(30:0.5) -- (30:0.7) 
(150:0.5) -- (150:0.7)
(0:0.5) -- (0:0.7) 
(180:0.5) -- (180:0.7);
\draw [thick, red] (90:0.5) -- (90:0.7);
\draw [thick, blue] (60:0.5) -- (60:0.7) (120:0.5) -- (120:0.7) ;
\draw [thick, dotted] (250:0.6) arc (250:290:0.6);
\draw [double] (0,0) circle [radius=0.5];
\node at (0,0) {$\ngraph$};
\end{tikzpicture}
=
\begin{tikzpicture}[baseline=-.5ex,scale=1]
\draw [thick] (0,0) circle [radius=1.5];
\draw [thick, dotted] (260:1) arc (260:280:1);
\draw [thick, blue] (60:0.7) to[out=60, in=-30] (90:1.1) (120:0.7) to[out=120, in=-150] (90:1.1) -- (90:1.5);
\draw [thick, red] (90:0.7) -- (90:1.1) to[out=150,in=-60] (120:1.5) (90:1.1) to[out=30,in=-120] (60:1.5);
\draw [thick, black]
(30:0.7) -- (30:1.5) 
(150:0.7) -- (150:1.5)
(0:0.7) -- (0:1.5) 
(180:0.7) -- (180:1.5);
\draw [thick, black]
(30:0.5) -- (30:0.7) 
(150:0.5) -- (150:0.7)
(0:0.5) -- (0:0.7) 
(180:0.5) -- (180:0.7);
\draw [thick, red] (90:0.5) -- (90:0.7);
\draw [thick, blue] (60:0.5) -- (60:0.7) (120:0.5) -- (120:0.7) ;
\draw [double] (0,0) circle [radius=0.5];
\node at (0,0) {$\ngraph$};
\end{tikzpicture}
\\
\begin{tikzpicture}[baseline=-.5ex,scale=1]
\draw [thick] (0,0) circle [radius=0.7];
\draw [thick] (0,0) circle [radius=1.5];
\draw [thick, dotted] (260:1) arc (260:280:1);
\draw [thick, blue] (120:0.7) to[out=120, in=-120] (60:1.5);
\draw [thick, yellow] (60:0.7) to[out=60,in=-60] (120:1.5);
\draw [thick, black]
(30:0.7) -- (30:1.5) 
(150:0.7) -- (150:1.5)
(0:0.7) -- (0:1.5) 
(180:0.7) -- (180:1.5);
\node at (0,0) {$\ngraph_{\Move{R0}}$};
\end{tikzpicture}
\cdot
\begin{tikzpicture}[baseline=-.5ex,scale=1]
\draw [thick] (0,0) circle [radius=0.7];
\draw [thick, black]
(30:0.5) -- (30:0.7) 
(150:0.5) -- (150:0.7)
(0:0.5) -- (0:0.7) 
(180:0.5) -- (180:0.7);
\draw [thick, yellow] (60:0.5) -- (60:0.7);
\draw [thick, blue] (120:0.5) -- (120:0.7) ;
\draw [thick, dotted] (250:0.6) arc (250:290:0.6);
\draw [double] (0,0) circle [radius=0.5];
\node at (0,0) {$\ngraph$};
\end{tikzpicture}
=
\begin{tikzpicture}[baseline=-.5ex,scale=1]
\draw [thick] (0,0) circle [radius=1.5];
\draw [thick, dotted] (260:1) arc (260:280:1);
\draw [thick, blue] (120:0.7) to[out=120, in=-120] (60:1.5);
\draw [thick, yellow] (60:0.7) to[out=60,in=-60] (120:1.5);
\draw [thick, black]
(30:0.7) -- (30:1.5) 
(150:0.7) -- (150:1.5)
(0:0.7) -- (0:1.5) 
(180:0.7) -- (180:1.5);
\draw [thick, black]
(30:0.5) -- (30:0.7) 
(150:0.5) -- (150:0.7)
(0:0.5) -- (0:0.7) 
(180:0.5) -- (180:0.7);
\draw [thick, yellow] (60:0.5) -- (60:0.7);
\draw [thick, blue] (120:0.5) -- (120:0.7) ;
\draw [double] (0,0) circle [radius=0.5];
\node at (0,0) {$\ngraph$};
\end{tikzpicture}
\\
\begin{tikzpicture}
\begin{scope}

\draw [thick, red] (0.5*1/2,{1/2*sqrt(3)/2}) to[out=60,in=180] (1.5,0);
\draw [rotate around={60:(0,0)},thick, red] (0.5*1/2,{1/2*sqrt(3)/2}) to[out=60,in=180] (1.5,0);
\draw [rotate around={120:(0,0)},thick, red] (0.5*1/2,{1/2*sqrt(3)/2}) to[out=60,in=180] (1.5,0);
\draw [rotate around={180:(0,0)},thick, red] (0.5*1/2,{1/2*sqrt(3)/2}) to[out=60,in=180] (1.5,0);
\draw [rotate around={240:(0,0)},thick, red] (0.5*1/2,{1/2*sqrt(3)/2}) to[out=60,in=180] (1.5,0);
\draw [rotate around={300:(0,0)},thick, red] (0.5*1/2,{1/2*sqrt(3)/2}) to[out=60,in=180] (1.5,0);

\draw [thick, blue] ({1.5*sqrt(3)/2},{1.5*1/2}) to[out=-120, in=90] (0,0.5);
\draw [rotate around={60:(0,0)}, thick, blue] ({1.5*sqrt(3)/2},{1.5*1/2}) to[out=-120, in=90] (0,0.5);
\draw [rotate around={120:(0,0)}, thick, blue] ({1.5*sqrt(3)/2},{1.5*1/2}) to[out=-120, in=90] (0,0.5);
\draw [rotate around={180:(0,0)}, thick, blue] ({1.5*sqrt(3)/2},{1.5*1/2}) to[out=-120, in=90] (0,0.5);
\draw [rotate around={240:(0,0)}, thick, blue] ({1.5*sqrt(3)/2},{1.5*1/2}) to[out=-120, in=90] (0,0.5);
\draw [rotate around={300:(0,0)}, thick, blue] ({1.5*sqrt(3)/2},{1.5*1/2}) to[out=-120, in=90] (0,0.5);

\draw [thick] (0,0) circle [radius=1.5];
\draw [thick] (0,0) circle [radius=0.5];

\end{scope}
\end{tikzpicture}
\end{tikzcd}
\caption{Elementary annulus operations on $N$-graphs on $\disk^2$, and a rotational annulus $N$-graph.}
\label{fig:elementary annulus N-graph}
\end{figure}
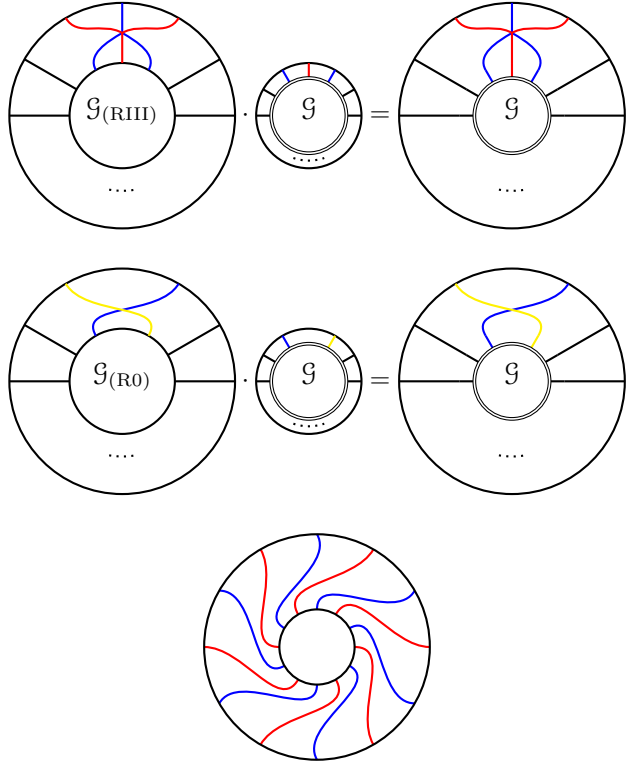

\section{Cluster algebras}
Cluster algebras, introduced by Fomin and Zelevinsky~\cite{FZ1_2002}, are
commutative algebras with specific generators, called \emph{cluster
	variables}, defined recursively.
In this section, we recall basic notions in the theory of cluster algebras.
For more details, we refer the reader to~\cite{FZ1_2002, FZ2_2003}.

Throughout this section, we fix $m, n \in \Z_{>0}$ such that $n \leq m$, and
we let $\field$ be the rational function field with $m$ independent
variables over $\C$.

\subsection{Basics on cluster algebras}

\begin{definition}[{cf. \cite{FZ1_2002, FZ2_2003}}]
	A \emph{seed} $\seed = (\bfx, \qbasis)$ is a pair of 
	\begin{itemize}
		\item a tuple $\mathbf x = (x_1,\dots,x_m)$ of algebraically
		independent generators of $\field$, that is, 
		$\field = \C(x_1,\dots,x_m)$;
		\item an $n \times m$ integer matrix $\qbasis = (b_{i,j})_{i,j}$ such
		that the \emph{principal part} $\qbasis^{\textrm{pr}} \coloneqq
		(b_{i,j})_{1\leq i,j\leq n}$ is skew-symmetrizable, that is, there
		exist positive integers $d_1,\dots,d_n$ such that
		$$\textrm{diag}(d_1,\dots,d_n) \cdot \qbasis^{\textrm{pr}}$$ is a
		skew-symmetric matrix.
	\end{itemize}
	We call elements $x_1,\dots,x_m$ \emph{cluster variables} and call
	$\qbasis$ \emph{exchange matrix}. Moreover, we call $x_1,\dots,x_n$
	\emph{unfrozen} (or, \emph{mutable}) variables and $x_{n+1},\dots,x_m$
	\emph{frozen} variables.
\end{definition}

To define cluster algebras, we introduce mutations on seeds, exchange matrices, and quivers as follows. 
\begin{enumerate}
\item (Mutation on seeds)	For a seed $\seed = (\bfx, \qbasis)$ and an integer $1 \leq k \leq n$, the \emph{mutation} $\mutation_k(\seed) = (\bfx', \qbasis')$ is defined as follows:
\begin{equation*}
\begin{split}
x_i' &= \begin{cases}
x_i &\text{ if } i \neq k,\\
\displaystyle 
x_k^{-1}\left( \prod_{b_{k,j} > 0} x_j^{b_{k,j}} + \prod_{b_{k,j} < 0}x_j^{-b_{k,j}}
\right) & \text{ otherwise}.
\end{cases}\\[1em]
b_{i,j}' &= \begin{cases}
-b_{i,j} & \text{ if } i = k \text{ or } j = k, \\
\displaystyle b_{i,j} + \frac{|b_{i,k}| b_{k,j} + b_{i,k} | b_{k,j}|} {2} & \text{ otherwise}.
\end{cases}
\end{split}
\end{equation*}
\item (Mutation on exchange matrices)
We define $\mu_k(\qbasis) = (b_{i,j}')$, and say that \emph{$\qbasis' =(b_{i,j}')$ is the mutation of $\qbasis$ at $k$}.
\item (Mutation on quivers)
We call a finite directed multigraph $\quiver$ a \emph{quiver} if it does not have
oriented cycles of length at most $2$. The adjacency matrix $\qbasis(\quiver)$ of a
quiver is always skew-symmetric. Moreover, $\mu_k(\qbasis(\quiver))$ is again 
the adjacency matrix of a quiver $\quiver'$. We define $\mu_k(\quiver)$ to be 
the quiver satisfying 
\[
 \qbasis(\mu_k(\quiver)) = \mu_k(\qbasis(\quiver)),
\]
and say that \emph{$\mu_k(\quiver)$ is the mutation of $\quiver$ at $k$}.
\end{enumerate}

We say a quiver $\quiver'$ is \emph{mutation equivalent} to another quiver $\quiver$ if 
there exists a sequence  of mutations $\mutation_{j_1},\dots,\mutation_{j_{\ell}}$ 
which connects $\quiver'$ and $\quiver$, that is,
\[
\quiver' = (\mutation_{j_{\ell}} \cdots \mutation_{j_1})(\quiver).
\]
Also, we say a quiver $\quiver$ is \emph{acyclic} if there is no directed cycle.
\begin{remark}\label{rmk_acyclic_quivers}
	It is proved in~\cite[Corollary~4]{CalderoKeller06} that 
	if two acyclic quivers are mutation equivalent, then
	there exists a sequence of mutations from one to other
	such that intermediate quivers are all acyclic.
	Indeed, two mutation
	equivalent acyclic quivers have the same underlying 
	(undirected) graph.
\end{remark}

An immediate check shows that $\mutation_k(\seed)$ is again a seed, and a mutation is an involution, that is, its square is the identity.
Also, note that the mutation on seeds does not change frozen variables $x_{n+1},\dots,x_m$.
Let $\mathbb{T}_n$ denote the $n$-regular tree whose edges are labeled by $1,\dots,n$. Except for $n = 1$, there are infinitely many vertices on the tree $\mathbb{T}_n$. For example, we present regular trees $\mathbb{T}_2$ and $\mathbb{T}_3$ in Figure~\ref{figure_regular_trees_2_and_3}.
\begin{figure}
	\begin{tabular}{cc}
		\begin{tikzpicture}
		\tikzset{every node/.style={scale=0.8}}
		\tikzset{cnode/.style = {circle, fill,inner sep=0pt, minimum size= 1.5mm}}
		\node[cnode] (1) {};
		\node[cnode, right of=1 ] (2) {};
		\node[cnode, right of=2 ] (3) {};
		\node[cnode, right of=3 ] (4) {};	
		\node[cnode, right of=4 ] (5) {};	
		\node[cnode, right of=5 ] (6) {};	
		
		\node[left of=1] {$\cdots$};
		\node[right of=6] {$\cdots$};
		
		\draw (1)--(2) node[above, midway] {$1$};
		\draw (2)--(3) node[above, midway] {$2$};
		\draw (3)--(4) node[above, midway] {$1$};
		\draw (4)--(5) node[above, midway] {$2$};
		\draw (5)--(6) node[above, midway] {$1$};
		\end{tikzpicture} &
		\begin{tikzpicture}
			\tikzset{every node/.style={scale=0.8}}
		\tikzset{cnode/.style = {circle, fill,inner sep=0pt, minimum size= 1.5mm}}
		\node[cnode] (1) {};
		\node[cnode, below right of =1] (2) {};
		\node[cnode, below of =2] (3) {};
		\node[cnode, above right of=2] (4){};
		\node[cnode, above of =4] (5) {};
		\node[cnode, below right of = 4] (6) {};
		\node[cnode, below of= 6] (7) {};
		\node[cnode, above right of = 6] (8) {};
		\node[cnode, above of = 8] (9) {};
		\node[cnode, below right of = 8] (10) {};
		\node[cnode, below of = 10] (11) {};
		\node[cnode, above right of = 10] (12) {};
		
		\node[left of = 1] {$\cdots$};
		\node[right of = 12] {$\cdots$};
		
		\draw (1)--(2) node[above, midway, sloped] {$1$};
		\draw (4)--(6) node[above, midway, sloped] {$1$};
		\draw (8)--(10) node[above, midway, sloped] {$1$};

		\foreach \x [evaluate ={ \x as \y using int(\x +2)} ] in {2, 6, 10}{
			\draw (\x)--(\y)  node[below, midway, sloped] {$2$}; 
		}
		\foreach \x [evaluate = {\x as \y using int(\x +1)}] in {2, 4, 6, 8, 10}{
			\draw (\x)--(\y) node[above, midway, sloped] {$3$};
		}
		\end{tikzpicture}\\[2ex]
		$\mathbb{T}_2$ & $\mathbb{T}_3$
	\end{tabular}
	\caption{The $n$-regular trees for $n=2$ and $n = 3$.}
	\label{figure_regular_trees_2_and_3}	
\end{figure}
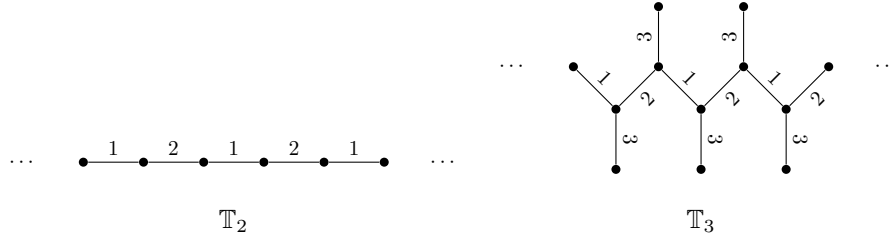
A \emph{cluster pattern} (or \emph{seed pattern}) is an assignment
\[
\mathbb{T}_n \to \{\text{seeds in } \field\}, \quad t \mapsto \seed_t = (\bfx_t, \qbasis_t)
\]
such that if $\begin{tikzcd} t \arrow[r,dash, "k"] & t' \end{tikzcd}$ in $\mathbb{T}_n$, then $\mutation_k(\seed_t) = \seed_{t'}$.
Let $\{ \seed_t = (\bfx_t, \qbasis_t)\}_{t \in \mathbb{T}_n}$ be a cluster pattern with $\bfx_t = (x_{1;t},\dots,x_{m;t})$. Since the mutation does not change frozen variables, we may let $x_{n+1} = x_{n+1;t},\dots,x_m = x_{m;t}$.
\begin{definition}[{cf. \cite{FZ2_2003}}]
	Let $\{ \seed_t = (\bfx_t, \qbasis_t)\}_{t \in \mathbb{T}_n}$ be a cluster pattern with $\bfx_t = (x_{1;t},\dots,x_{m;t})$.
	The \emph{cluster algebra} $\cA(\{\seed_t\}_{t \in \mathbb{T}_n})$ is defined to be the $\C[x_{n+1},\dots,x_m]$-subalgebra 
	of $\field$ generated by all the cluster variables 
	 $\bigcup_{t \in \mathbb{T}_n} \{x_{1;t},\dots,x_{n;t}\}$.
\end{definition}
If we fix a vertex $t_0 \in \mathbb{T}_n$, then a cluster pattern $\{ \seed_t \}_{t \in \mathbb{T}_n}$ is 
constructed from the seed $\seed_{t_0}$. In this case, we call $\seed_{t_0}$ an \emph{initial seed}. 
Because of this reason, we simply denote by $\cA(\seed_{t_0})$ the cluster algebra given by the cluster 
pattern constructed from the initial seed $\seed_{t_0}$.
\begin{example}\label{example_A2_example}
	Let $n = m = 2$. Suppose that an initial seed is given by 
	\[
		\seed_{t_0} = \left(
			(x_1,x_2), \begin{pmatrix}
				0 & 1  \\ -1 & 0
			\end{pmatrix}
		\right).
	\]
	We present a part of the cluster pattern obtained by the initial seed $\seed_{t_0}$.
	\begin{center}
	\begin{tikzcd}
		\left( (x_2,x_1), 
		\begin{pmatrix}
			0 & -1 \\ 1 & 0
		\end{pmatrix}
		\right)
		\arrow[r,equal,"\text{up to}", "\text{relabelling}"']
		& \seed_{t_0} 
			= \left(
			(x_1,x_2), 
			\begin{pmatrix}
				0 & 1 \\ -1 & 0
			\end{pmatrix}
		\right) \arrow[d,<->, "\mutation_1"]
		\\
		\left(
			(\frac{1+x_1}{x_2}, x_1), \begin{pmatrix}
				0 & 1 \\ -1 & 0
			\end{pmatrix}
		\right) \arrow[u,<->, "\mutation_1"]
		& 
		\left(
			\left(\frac{1+x_2}{x_1}, x_2\right), \begin{pmatrix}
				0 & -1 \\ 1 & 0
			\end{pmatrix}
		\right) \arrow[d, <->,"\mutation_2"]\\
		\left(
			\left(\frac{1+x_1}{x_2}, \frac{1+x_1+x_2}{x_1x_2}\right),
			\begin{pmatrix}
				0 & -1 \\ 1 & 0
			\end{pmatrix}
		\right) \arrow[u,<->, "\mutation_2"]
		&
		\left(
			\left(\frac{1+x_2}{x_1}, \frac{1+x_1+x_2}{x_1x_2}\right),
			\begin{pmatrix}
				0 & 1  \\ -1 & 0
			\end{pmatrix}
		\right) \arrow[l,<->, "\mutation_1"]
    \end{tikzcd}
\end{center}
\end{example}

\begin{remark}\label{rmk_x_cluster_mutation}
There is another mutation operation called the \emph{cluster
$\mathcal{X}$-mutation}.
Let $\{ \seed_t = (\bfx_t, \qbasis_t)\}_{t \in \mathbb{T}_n}$ be a cluster
pattern with $\bfx_t = (x_{1;t},\dots,x_{m;t})$. For $t \in \mathbb{T}_n$ and
$i \in [n]$, we set $\mathbf{y}_t = (y_{1;t},\dots,y_{n;t})$ by
\[
    y_{i;t} = \prod_{j \in [m]} x_{j;t}^{b^{(t)}_{i,j}}
\]
where $\qbasis_t = (b^{(t)}_{i,j})$.
Then the assignment $t \mapsto (\mathbf{y}_t, \qbasis_t)$ is called a cluster
\emph{$Y$-pattern} and for $\begin{tikzcd} t \arrow[r,dash, "k"] & t'
\end{tikzcd}$ in $\mathbb{T}_n$, we have
\[
y_{i;t'} = \begin{cases}
    \displaystyle y_{i;t} y_{k;t}^{\max\{b_{i,k}^{(t)},0\}}(1+y_{k;t})^{-b_{i,k}^{(t)}} & \text{ if }i \neq k, \\
   y_{k;t}^{-1} &\text{ otherwise;}
\end{cases}
\]
see~\cite[Proposition~3.9]{FZ4_2007}. For $\begin{tikzcd} t \arrow[r,dash,
"k"] & t' \end{tikzcd}$ in $\mathbb{T}_n$, the operation sends $(\mathbf y_t,
\qbasis_t)$ to $(\mathbf y_{t'}, \qbasis_{t'})$ is called the cluster
$\mathcal{X}$-mutation (or, $\mathcal{X}$-cluster mutation).
For exchange matrices and quivers, the cluster $\mathcal{X}$-mutation is defined
the same as before. 
\end{remark}

\subsection{Cluster algebras of finite type}
The number of cluster variables in Example~\ref{example_A2_example} is finite
even though the number of vertices in the graph $\mathbb{T}_2$ is infinite.
We call such cluster algebras \emph{of finite type}. More precisely, we
recall the following definition.
\begin{definition}[{\cite{FZ2_2003}}]
	A cluster algebra is said to be \emph{of finite type} if it has finitely
	many cluster variables.
\end{definition}
It has been realized that classifying finite type cluster algebras is related
to studying exchange matrices. The \emph{Cartan counterpart}
$C(\qbasis^{\textrm{pr}}_{t_0}) = (c_{i,j})$ of the principal part
$\qbasis^{\textrm{pr}}_{t_0}$ of an exchange matrix is defined by
\[
c_{i,j} = \begin{cases}
2 & \text{ if } i = j, \\
-|b_{i,j}| & \text{ otherwise}.
\end{cases}
\]
Since $\qbasis^{\text{pr}}_{t_0}$ is skew-symmetrizable, its Cartan
counterpart $C(\qbasis^{\textrm{pr}}_{t_0})$ is symmetrizable. 
Note that two mutation equivalent acyclic quivers produce the same 
Cartan counterpart (cf. Remark~\ref{rmk_acyclic_quivers}).
The following
theorem presents a classification of cluster algebras of finite type.
\begin{theorem}[{\cite{FZ2_2003}}] \label{thm_FZ_finite_type}
	Let $\{ \seed_t = (\bfx_t, \qbasis_t)\}_{t \in \mathbb{T}_n}$ be a
	cluster pattern with an initial seed $\seed_{t_0} = (\bfx_{t_0},
	\qbasis_{t_0})$. Let $\mathcal{A}(\qbasis_{t_0})$ be the corresponding
	cluster algebra. Then we have the following.
	\begin{enumerate}
		\item The cluster algebra $\mathcal{A}(\qbasis_{t_0})$ is of finite
		type if and only if $C(\qbasis^{\textrm{pr}}_{t_0})$ is a Cartan
		matrix of finite type.
		\item If the cluster algebra $\mathcal{A}(\qbasis_{t_0})$ is of
		finite type, then there is a bijective correspondence between the set
		of positive roots for $C(\qbasis^{\textrm{pr}}_{t_0})$ and the set of
		noninitial cluster variables. More precisely, for the set $\SRoots=\{\alpha_1,\dots,\alpha_n\}$
		of simple roots, a positive root $\sum_{i=1}^n d_i \alpha_i$ is associated to a cluster variable of the
		form
		\begin{align*}
		&\frac{f(\bfx_{t_0})}{x_{1;t_0}^{d_1} \cdots x_{n;t_0}^{d_n}},&
		f(\bfx_{t_0})&\in\C[x_{1;t_0},\dots, x_{m;t_0}].
		\end{align*}
		Here, $x_{1;t_0},\dots,x_{m;t_0}$ are cluster variables in the
		initial seed $\bfx_{t_0}$.
		Accordingly, there is a bijective correspondence between the set of
		cluster variables and the set $\Roots_{\geq -1}$ of \emph{almost
		positive roots} $\Roots_{\geq -1} \coloneqq \Roots^+ \cup (-
		\SRoots)$, where $\Roots$ it the root system whose Cartan matrix is~$C(\qbasis^{\textrm{pr}}_{t_0})$.
	\end{enumerate}
\end{theorem}

We provide a list of finite type root systems and their Dynkin diagram in
Table~\ref{Dynkin}.
In what follows, we fix an ordering on the simple roots as in Table~\ref{Dynkin}; 
our conventions agree with that in the standard textbook of Humphreys~\cite{Humphreys}.
\begin{table}[t]
\begin{center}
\begin{tabular}{c|l  }
\toprule
$\Roots$ & Dynkin diagram \\
\midrule
\raisebox{1em}{$\dynA_n$ $(n \geq 1)$} &
\begin{tikzpicture}[scale=.5]
	\tikzstyle{state}=[draw, circle, inner sep = 0.07cm]
	\tikzset{every node/.style={scale=0.7}}
	\node[state, label=below:{$1$}] (1) {};
	\node[state, label=below:{$2$}] (2) [right = of 1] {};
	\node[state, label=below:{$3$}] (3) [right = of 2] {};
	\node[state, label=below:{$n-1$}] (4) [right =of 3] {};
	\node[state, label=below:{$n$}] (5) [right =of 4] {};			

	\draw (1)--(2)--(3)
		(4)--(5);
	\draw[dotted] (3)--(4);
\end{tikzpicture}  \\[0.5em]
			
\raisebox{1em}{$\dynB_n$ $(n \geq 2)$} &
\begin{tikzpicture}[scale=.5]
	\tikzstyle{state}=[draw, circle, inner sep = 0.07cm]
	\tikzstyle{double line} = [
		double distance = 1.5pt, 
		double=\pgfkeysvalueof{/tikz/commutative diagrams/background color}
	]
	\tikzset{every node/.style={scale=0.7}}

	\node[state, label=below:{$1$}] (1) {};
	\node[state, label=below:{$2$}] (2) [right = of 1] {};
	\node[state, label=below:{$n-2$}] (3) [right = of 2] {};
	\node[state, label=below:{$n-1$}] (4) [right =of 3] {};
	\node[state, label=below:{$n$}] (5) [right =of 4] {};

	\draw (1)--(2)
		(3)--(4);
	\draw [dotted] (2)--(3);
	\draw[double line] (4)-- node{\scalebox{1.3}{ $>$}} (5);
\end{tikzpicture}  \\[0.5em]			
\raisebox{1.5em}{$\dynC_n$ $(n \geq 3)$} & 
\begin{tikzpicture}[scale=.5]
	\tikzstyle{state}=[draw, circle, inner sep = 0.07cm]
	\tikzstyle{double line} = [
		double distance = 1.5pt, 
		double=\pgfkeysvalueof{/tikz/commutative diagrams/background color}
	]
	\tikzset{every node/.style={scale=0.7}}

	\node[state, label=below:{$1$}] (1) {};
	\node[state, label=below:{$2$}] (2) [right = of 1] {};
	\node[state, label=below:{$n-2$}] (3) [right = of 2] {};
	\node[state, label=below:{$n-1$}] (4) [right =of 3] {};
	\node[state, label=below:{$n$}] (5) [right =of 4] {};

	\draw (1)--(2)
		(3)--(4);
	\draw [dotted] (2)--(3);
	\draw[double line] (4)-- node{\scalebox{1.3}{ $<$}} (5);
\end{tikzpicture}  \\[0.5em]	
			
\raisebox{1.5em}{$\dynD_n$ $(n \geq 4)$} & 
\begin{tikzpicture}[scale=.5]
	\tikzstyle{state}=[draw, circle, inner sep = 0.07cm]
	\tikzstyle{double line} = [
		double distance = 1.5pt, 
		double=\pgfkeysvalueof{/tikz/commutative diagrams/background color}
	]
	\tikzset{every node/.style={scale=0.7}}
			
	\node[state, label=below:{$1$}] (1) {};
	\node[state, label=below:{$2$}] (2) [right = of 1] {};
	\node[state, label=below:{$n-3$}] (3) [right = of 2] {};
	\node[state, label=below:{$n-2$}] (4) [right =of 3] {};			
			
	\node[state, label=right:{$n-1$}] (5) [above right= 0.3cm and 1cm of 4] {};
	\node[state, label=right:{$n$}] (6) [below right= 0.3cm and 1cm of 4] {};

	\draw(1)--(2)
		(3)--(4)--(5)
		(4)--(6);
	\draw[dotted] (2)--(3);
\end{tikzpicture}
\\[0.5em]
\raisebox{1.5em}{$\dynE_6$} & 
	
\begin{tikzpicture}[scale=.5]
	\tikzstyle{state}=[draw, circle, inner sep = 0.07cm]
	\tikzstyle{double line} = [
		double distance = 1.5pt, 
		double=\pgfkeysvalueof{/tikz/commutative diagrams/background color}
	]
	\tikzset{every node/.style={scale=0.7}}

	\node[state, label=below:{$1$}] (1) {};
	\node[state, label=below:{$3$}] (3) [right=of 1] {};
	\node[state, label=below:{$4$}] (4) [right=of 3] {};
	\node[state, label=right:{$2$}] (2) [above=of 4] {};
	\node[state, label=below:{$5$}] (5) [right=of 4] {};
	\node[state, label=below:{$6$}] (6) [right=of 5]{};

	\draw(1)--(3)--(4)--(5)--(6)
		(2)--(4);
\end{tikzpicture}			
\\
\raisebox{1.5em}{$\dynE_7$} & 
\begin{tikzpicture}[scale=.5]
	\tikzstyle{state}=[draw, circle, inner sep = 0.07cm]
	\tikzstyle{double line} = [
		double distance = 1.5pt, 
		double=\pgfkeysvalueof{/tikz/commutative diagrams/background color}
	]
	\tikzset{every node/.style={scale=0.7}}

	\node[state, label=below:{$1$}] (1) {};
	\node[state, label=below:{$3$}] (3) [right=of 1] {};
	\node[state, label=below:{$4$}] (4) [right=of 3] {};
	\node[state, label=right:{$2$}] (2) [above=of 4] {};
	\node[state, label=below:{$5$}] (5) [right=of 4] {};
	\node[state, label=below:{$6$}] (6) [right=of 5]{};
	\node[state, label=below:{$7$}] (7) [right=of 6]{};

	\draw(1)--(3)--(4)--(5)--(6)--(7)
		(2)--(4);
\end{tikzpicture}	
\\
\raisebox{1.5em}{$\dynE_8$} & 
\begin{tikzpicture}[scale=.5]
	\tikzstyle{state}=[draw, circle, inner sep = 0.07cm]
	\tikzstyle{double line} = [
		double distance = 1.5pt, 
		double=\pgfkeysvalueof{/tikz/commutative diagrams/background color}
	]
	\tikzset{every node/.style={scale=0.7}}

	\node[state, label=below:{$1$}] (1) {};
	\node[state, label=below:{$3$}] (3) [right=of 1] {};
	\node[state, label=below:{$4$}] (4) [right=of 3] {};
	\node[state, label=right:{$2$}] (2) [above=of 4] {};
	\node[state, label=below:{$5$}] (5) [right=of 4] {};
	\node[state, label=below:{$6$}] (6) [right=of 5]{};
	\node[state, label=below:{$7$}] (7) [right=of 6]{};
	\node[state, label=below:{$8$}] (8) [right=of 7]{};

	\draw(1)--(3)--(4)--(5)--(6)--(7)--(8)
		(2)--(4);
\end{tikzpicture}
\\
\raisebox{1em}{$\dynF_4$} 
&
\begin{tikzpicture}[scale = .5]
	\tikzstyle{state}=[draw, circle, inner sep = 0.07cm]
	\tikzstyle{point} = [
		draw=black,
		cross out,
		inner sep=0pt,
		minimum width=4pt,
		minimum height=4pt]		
	\tikzstyle{double line} = [
		double distance = 1.5pt, 
		double=\pgfkeysvalueof{/tikz/commutative diagrams/background color}
	]
	\tikzset{every node/.style={scale=0.7}}
			
	\node[state, label=below:{1}] (1) {};
	\node[state, label=below:{2}] (2) [right = of 1] {};
	\node[state, label=below:{3}] (3) [right = of 2] {};
	\node[state, label=below:{4}] (4) [right =of 3] {};
			
	\draw (1)--(2)
		(3)--(4);
	\draw[double line] (2)-- node{\scalebox{1.3}{$>$}} (3);
\end{tikzpicture}  \\
\raisebox{1em}{$\dynG_2$} 
&
\begin{tikzpicture}[scale =.5]
	\tikzstyle{state}=[draw, circle, inner sep = 0.07cm]
	\tikzstyle{point} = [
		draw=black,
		cross out,
		inner sep=0pt,
		minimum width=4pt,
		minimum height=4pt]
	\tikzstyle{triple line} = [
		double distance = 2pt, 
		double=\pgfkeysvalueof{/tikz/commutative diagrams/background color}
	]
	\tikzset{every node/.style={scale=0.7}}
			
	\node[state, label=below:{1}] (1) {};
	\node[state, label=below:{2}] (2) [right = of 1] {};
			
	\draw[triple line] (1)-- node{\scalebox{1.3}{$<$}} (2);
	\draw (1)--(2);
			
\end{tikzpicture}\\
\bottomrule
\end{tabular}
\end{center}
\caption{Dynkin diagrams of finite type}\label{Dynkin}
\end{table}
In Table~\ref{table_seeds_and_cluster_variables}, we provide enumeration on
the number of cluster variables and clusters in each cluster algebra of
finite (irreducible) type (cf.~\cite[Figure~5.17]{FWZ_chapter45}). 
\begin{definition}\label{def_quiver_of_type_X}
For a quiver $\quiver$, we say that \emph{$\quiver$ is of type $\dynX$} if 
it is mutation equivalent to a quiver~$\quiver'$ whose underlying 
unoriented graph on mutable vertices is the Dynkin diagram of type~$\dynX$. 
Equivalently, $\quiver$ is of type $\dynX$ if it is mutation equivalent to 
a quiver $\quiver'$ whose  Cartan counterpart 
$C(\qbasis^{\text{pr}}(\quiver'))$ of the principal part of the adjacency matrix 
is the Cartan matrix of type $\dynX$. 
\end{definition}
\begin{table}[b]
	\begin{tabular}{c|ccccccccc}
		\toprule
		$\Roots$ & $\dynA_n$ & $\dynB_n$ & $\dynC_n$ & $\dynD_n$ & $\dynE_6$ & $\dynE_7$ & $\dynE_8$ & $\dynF_4$ & $\dynG_2$ \\
        \midrule
        $\#$seeds &  $\displaystyle \frac{1}{n+2}{\binom{2n+2}{n+1}}$ & $\displaystyle  \binom{2n}{n}$
        & $\displaystyle  \binom{2n}{n}$ & $\displaystyle  \frac{3n-2}{n} \binom{2n-2}{n-1}$ & $833$ & $4160$ & $25080$ & $105$ & $8$ \\[1.5em]
        $\#$clvar & $\displaystyle  \frac{n(n+3)}{2}$ & $n(n+1)$ & $n(n+1)$ & $n^2$ & $42$ & $70$ & $128$ & $28$ & $8$ \\
		\bottomrule 
	\end{tabular}
\caption{Enumeration of seeds and cluster variables}\label{table_seeds_and_cluster_variables}
\end{table}
\begin{example}
	Continuing Example~\ref{example_A2_example}, the Cartan
	counterpart of the principal part $\qbasis_{t_0}^{\text{pr}}$ is given by
	\[
		C(\qbasis_{t_0}^{\text{pr}}) = \begin{pmatrix}
			2 & -1  \\ -1 & 2
		\end{pmatrix},
	\]
	which is the Cartan matrix of Lie type $\dynA_2$. 
	Accordingly, by Theorem~\ref{thm_FZ_finite_type}, the cluster algebra
	$\cA(\seed_{t_0})$ is of finite type. Indeed, there are five cluster
	variables and we present the bijective correspondence between them and
	the set of almost positive roots as described in
	Theorem~\ref{thm_FZ_finite_type}(2).
	\[
		\begin{tikzcd}[row sep = 0.2cm]
			x_1 
			& x_2 
			& \displaystyle \frac{1+x_2}{x_1} 
			& \displaystyle \frac{1+x_1}{x_2} 
			& \displaystyle \frac{1+x_1+x_2}{x_1x_2} 
			\\
			-\alpha_1 & -\alpha_2 & \alpha_1 & \alpha_2 & \alpha_1 + \alpha_2
		\end{tikzcd}
	\]
	Here, $\alpha_1$ and $\alpha_2$ are simple roots of the Lie algebra of type $\dynA_2$. 
\end{example}

\subsection{Folding}\label{sec:folding}
Under certain conditions, one can \textit{fold} seed patterns to produce new
ones. This procedure is used to study cluster algebras of type
$\dynBCFG$ from those of simply-laced type~$\dynADE$. As before, we fix $m,
n \in \Z_{>0}$ such that $n \leq m$.

Let $\quiver$ be a labeled quiver having $m$ vertices labeled $1,\dots,m$. 
Let $G$ be a finite group acting on the set $[m]$. 
The notation $i \sim i'$ will mean that $i$ and $i'$ lie in the same
$G$-orbit. To study folding of cluster algebras, we prepare some
terminologies.
\begin{definition}[{cf.~\cite[\S4.4]{FWZ_chapter45}}]\label{definition:admissible quiver}
    Let $\quiver$ be a labeled quiver having $m$ vertices and $G$ a finite
    group acting on the set $[m]$.
    \begin{enumerate}
        \item The quiver $\quiver$ (or the corresponding $m \times n$
        exchanged matrix $\qbasis=\qbasis(\quiver)$) is \emph{$G$-admissible}
        if
        \begin{enumerate}
            \item for any $i \sim i'$, index $i$ is mutable if and only if so is $i'$;
            \item for any indices $i$ and $j$, and any $g \in G$, we have $b_{i,j} = b_{g(i),g(j)}$;
            \item for mutable indices $i \sim i'$, we have $b_{i,i'} = 0$;
            \item for any $i \sim i'$, and any mutable $j$, we have $b_{i,j} b_{i',j} \geq 0$.
        \end{enumerate}
        \item For a $G$-admissible quiver $\quiver$, we call a $G$-orbit \emph{mutable} 
        (respectively, \emph{frozen}) if it consists of mutable (respectively, frozen) vertices. 
    \end{enumerate}        
\end{definition}
For a $G$-admissible quiver $\quiver$, we define the matrix $\qbasis^G =
\qbasis(\quiver)^G = (b_{I,J}^G)$ whose rows (respectively, columns) are
labeled by the $G$-orbits (respectively, mutable $G$-orbits) by
\[
    b_{I,J}^G = \sum_{i \in I} b_{i,j}
\]
where $j$ is an arbitrary index in $J$. We then say $\qbasis^G$ is obtained
from $\qbasis$ (or from the quiver $\quiver$) by \textit{folding} with
respect to the given $G$-action.
\begin{example}\label{example_D4_to_G2}
    Let $\quiver$ be a quiver of type $\dynD_4$ given as follows.
\[
    \begin{tikzpicture}[node distance=0.7cm]
        \tikzstyle{state}=[draw, circle, inner sep = 0.07cm]
        \tikzset{every node/.style={scale=0.7}}    
        \tikzstyle{double line} = [
            double distance = 1.5pt, 
            double=\pgfkeysvalueof{/tikz/commutative diagrams/background color}
        ]
        \tikzstyle{triple line} = [
            double distance = 2pt, 
            double=\pgfkeysvalueof{/tikz/commutative diagrams/background color}
        ]
        \node[state, label=left:{$2$}] (1) {};
        \node[state, label =right:{$1$}] (3) [above right = 0.4cm and 0.7cm of 1] {};
        \node[state, label=right:{$3$}] (2) [right = 0.7cm of 1] {};
        \node[state, label= right:{$4$}] (4) [below right = 0.4cm and 0.7cm of 1] {};
        
        \draw (3)--(1)--(4)
        (1)--(2);
    
        \node[label={below:\normalsize{$\rightsquigarrow$}}] [above right = 0.1cm and 1.5cm of 2] {};

        \node[ynode] at (1) {};
        \node[gnode] at (2) {};
        \node[gnode] at (3) {};
        \node[gnode] at (4) {};

    \draw[->] (1)--(2);
    \draw[->] (1)--(3);
    \draw[->] (1)--(4);
    \end{tikzpicture}
    \qquad
    \text{\raisebox{1.5em}{$\qbasis(\quiver) = \begin{pmatrix}
        0 & -1 & 0 & 0 \\
        1 & 0 & 1 & 1 \\
        0 & -1 & 0 & 0 \\
        0 & -1 & 0 & 0
    \end{pmatrix}$}}
\]
The finite group $G = \Z / 3 \Z$ acts on $[4]$ by sending $1 \mapsto 3
\mapsto 4 \mapsto 1$ and $2 \mapsto 2$. 
Here, we decorate vertices of the quiver $\quiver$ with green and yellow colors
for presenting sources and sinks, respectively. 
One may check that the quiver
$\quiver$ is $G$-admissible. By setting $I_1 = \{2\}$ and $I_2 = \{ 1,3,4\}$,
we obtain
\[
    \begin{split}
    b_{I_1,I_2}^G &= \sum_{i \in I_1} b_{i,1} = b_{2,1} = 1, \\
    b_{I_2,I_1}^G &= \sum_{i \in I_2} b_{i,2} = b_{1,2} + b_{3,2} + b_{4,2} = -3.
    \end{split}
\]
Accordingly, we obtain the matrix $\qbasis^G = \begin{pmatrix} 0 & 1 \\ -3 &
0 \end{pmatrix}$ whose Cartan counterpart is the Cartan matrix of type
$\dynG_2$.
\end{example}

For a $G$-admissible quiver $\quiver$ and a mutable $G$-orbit $I$, we
consider a composition of mutations given by
\[
    \mutation_I = \prod_{i \in I} \mutation_i
\]
which is well-defined because of the definition of admissible quivers.
If $\mutation_I(\quiver)$ is again $G$-admissible, then we have that
\begin{equation*}
    (\mutation_I(\qbasis))^G = \mutation_I(\qbasis^G).
\end{equation*}
We notice that the quiver $\mutation_I(\quiver)$ is \textit{not}
$G$-admissible in general. Therefore, we present the following definition.
\begin{definition}
    Let $G$ be a group acting on the vertex set of a quiver $\quiver$.
    We say that $\quiver$ is \emph{globally foldable} with respect to $G$ if
    $\quiver$ is $G$-admissible and moreover for any sequence of mutable
    $G$-orbits $I_1,\dots,I_\ell$, the quiver $(\mutation_{I_\ell}  \dots
     \mutation_{I_1})(\quiver)$ is $G$-admissible.
\end{definition}
For a globally foldable quiver, we can fold all the seeds in the
corresponding seed pattern.
Let $m^G$ denote the number of orbits of the action of $G$ on $[m]$. Let
$\field^G$ be the field of rational functions in $m^G$ independent variables.
Let $\psi \colon \field \to \field^G$ be a surjective 
homomorphism.
A seed $\seed = (\mathbf{x}, \qbasis(\quiver))$ is called \emph{$(G, \psi)$-admissible} if 
\begin{itemize}
    \item $\quiver$ is a $G$-admissible quiver;
    \item for any $i \sim i'$, we have $\psi(x_i) = \psi(x_{i'})$.
\end{itemize}
In this situation, we define a new ``folded'' seed $\seed^G = (\bfx^G,
\qbasis^G)$ in $\field^G$ whose exchange matrix is given as before
and cluster variables $\bfx^G = (x_I)$ are indexed by the $G$-orbits and
given by $x_I = \psi(x_i)$.

\begin{proposition}[{cf.~\cite[Corollary~4.4.11]{FWZ_chapter45}}]\label{proposition:folded cluster pattern}
    Let $\quiver$ be a quiver which is globally foldable with respect to a
    group $G$ acting on the set of its vertices. Let $\seed = (\mathbf{x},
    \qbasis(\quiver))$ be a seed in the field $\field$ of rational functions
    freely generated by a cluster $\mathbf{x} = (x_1,\dots,x_m)$.
    Define $\psi \colon \field  \to \field^G$ so that
    $\seed$ is a $(G, \psi)$-admissible seed. Then, for any mutable
    $G$-orbits $I_1,\dots,I_\ell$, the seed $(\mutation_{I_\ell}  \dots 
    \mutation_{I_1})(\seed)$ is $(G, \psi)$-admissible, and moreover the
    folded seeds $((\mutation_{I_\ell}  \dots 
    \mutation_{I_1})(\seed))^G$ form a seed pattern in $\field^G$ 
    with the initial seed $\seed^G=(\bfx^G, (\qbasis(\quiver))^G)$.
\end{proposition}

\begin{example}\label{example_folding_ADE}
The quiver in Example~\ref{example_D4_to_G2} is globally foldable, and
moreover the corresponding seed pattern is of type $\dynG_2$. In fact, 
seed patterns of type~$\dynBCFG$  are obtained by 
folding quivers of type~$\dynADE$  in general.
In Figure~\ref{fig_folding}, we present the corresponding quivers of
type~$\dynADE$. We decorate vertices of quivers with yellow and green 
colors for presenting source and sink, respectively. 
As one may see, we have to put arrows on the Dynkin diagram alternatingly.
For each case, the finite group action that makes each quiver globally foldable is given as follows. 
\begin{enumerate}
	\item $\dynA_{2n-1} \rightsquigarrow \dynB_n$: The finite group $G = \Z/2\Z$ acts on the set  $[2n-1]$ of vertices of the quiver of type $\dynA_{2n-1}$ by 
	\[
	i \mapsto 2n-i \quad \text{ for } i \in [2n-1].
	\]
	There are $n$ orbits: $I_i = \{ i, 2n-i\}$ for $i \in [n]$. 
	\item $\dynD_{n+1} \rightsquigarrow \dynC_n$: The finite group $G = \Z/2\Z$ acts on the set $[n+1]$ of vertices of the quiver of type $\dynD_{n+1}$ by 
	\[
	\begin{split}
	i &\mapsto i \quad \text{ for } i \in [n-1], \\
	n  & \mapsto n+1, 	\quad n+1 \mapsto n.
	\end{split}
	\]
	There are $n$ orbits: $I_i = \{i\}$ for $i \in [n-1]$, and $I_n = \{n,n+1\}$. 
	\item $\dynE_6 \rightsquigarrow \dynF_4$: The finite group $G = \Z/2\Z$ acts on the set $[6]$ of vertices of the quiver of type $\dynE_6$ by 
	\[
	\begin{split}
	i &\mapsto i \quad \text{ for } i = 2, 4, \\
	1 & \mapsto 6, \quad 6 \mapsto 1, \\
	3 &\mapsto 5, \quad 5 \mapsto 3.
	\end{split}
	\]
	There are $4$ orbits: $I_1 = \{1,6\}$, $I_2=\{3,5\}$, $I_3 = \{4\}$, and $I_4 = \{2\}$.
	\item $\dynD_4 \rightsquigarrow \dynG_2$: The finite group $G = \Z/3\Z$ acts on the set $[4]$ of vertices of the quiver of type $\dynD_4$ by 
	\[
	\begin{split}
	2 &\mapsto 2, \\
	1 & \mapsto 3, \quad 3 \mapsto 4,\quad 4 \mapsto 1.
	\end{split}
	\]
	There are $2$ orbits: $I_1 = \{2\}$ and $I_2 = \{1,3,4\}$.
\end{enumerate}
The alternating colorings on quivers of type $\dynADE$ provide 
that on quivers of type $\dynBCFG$ as displayed in the right column of Figure~\ref{fig_folding}.
\end{example}
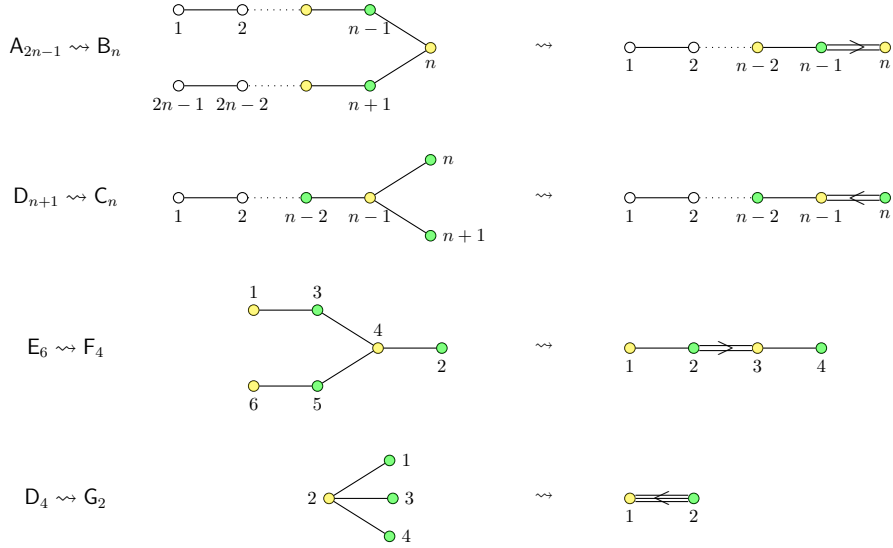
\begin{figure}
    \begin{tikzpicture}[node distance=0.7cm]
        \tikzstyle{state}=[draw, circle, inner sep = 0.07cm]
        \tikzset{every node/.style={scale=0.7}}    
        \tikzstyle{double line} = [
            double distance = 1.5pt, 
            double=\pgfkeysvalueof{/tikz/commutative diagrams/background color}
        ]
        \tikzstyle{triple line} = [
            double distance = 2pt, 
            double=\pgfkeysvalueof{/tikz/commutative diagrams/background color}
        ]
        \begin{scope}[xshift=-1.5cm, yshift=-0.5cm]
            \node[color=white] {\textcolor{black}{{\Large $\dynA_{2n-1} \rightsquigarrow \dynB_n$}}};
        \end{scope}
        \begin{scope}[xshift=-1.5cm, yshift=-2.5cm]
            \node[color=white] {\textcolor{black}{{\Large $\dynD_{n+1} \rightsquigarrow \dynC_n$}}};
        \end{scope}
        \begin{scope}[xshift=-1.5cm, yshift=-4.5cm]
            \node[color=white] {\textcolor{black}{{\Large $\dynE_6 \rightsquigarrow \dynF_4$}}};
        \end{scope}
        \begin{scope}[xshift=-1.5cm, yshift=-6.5cm]
            \node[color=white] {\textcolor{black}{{\Large $\dynD_{4} \rightsquigarrow \dynG_2$}}};
        \end{scope}
\begin{scope}

	\node[state, label=below:{$1$}] (1) {};
	\node[state, label=below:{$2$}] (2) [right = of 1] {};
	\node[state, label=below:{}] (3) [right = of 2] {};
	\node[state, label=below:{$n-1$}] (4) [right =of 3] {};
    \node[state, label=below:{$n$}] (5) [below right= 0.4cm and 0.7cm of 4] {};
    \node[state, label=below:{$n+1$}] (6) [below left = 0.4cm and 0.7cm of 5] {};
    \node[state] (7) [left = of 6] {};
    \node[state, label = below:{$2n-2$}] (8) [left = of 7] {};
    \node[state, label=below:{$2n-1$}] (9) [left = of 8] {};

    \foreach \y in {3, 5, 7} {
        \node[ynode] at (\y) {};
    }

    \foreach \g in {4,6} {
        \node[gnode] at (\g) {};
    }

	\draw (1)--(2)
        (3)--(4)--(5)--(6)--(7)
        (8)--(9);
    \draw[dotted] (2)--(3)
    (7)--(8);

\end{scope}
\begin{scope}[xshift = 6cm, yshift = -0.5cm]

	\node[state, label=below:{$1$}] (1) {};
	\node[state, label=below:{$2$}] (2) [right = of 1] {};
	\node[state, label=below:{$n-2$}] (3) [right = of 2] {};
	\node[state, label=below:{$n-1$}] (4) [right =of 3] {};
	\node[state, label=below:{$n$}] (5) [right =of 4] {};

    \foreach \y in {3, 5} {
        \node[ynode] at (\y) {};
    }

    \foreach \g in {4} {
        \node[gnode] at (\g) {};
    }

	\draw (1)--(2)
		(3)--(4);
	\draw [dotted] (2)--(3);
	\draw[double line] (4)-- node{\scalebox{1.3}{ $>$}} (5);
	\node[label={below:\normalsize{$\rightsquigarrow$}}] [above left = 0.1cm and 1cm of 1] {};
\end{scope}

\begin{scope}[yshift= -2.5cm]

	\node[state, label=below:{$1$}] (1) {};
	\node[state, label=below:{$2$}] (2) [right = of 1] {};
	\node[state, label=below:{$n-2$}] (3) [right = of 2] {};
	\node[state, label=below:{$n-1$}] (4) [right =of 3] {};			
			
	\node[state, label=right:{$n$}] (5) [above right= 0.4cm and .7cm of 4] {};
	\node[state, label=right:{$n+1$}] (6) [below right= 0.4cm and 0.7cm of 4] {};
    \foreach \y in {4} {
        \node[ynode] at (\y) {};
    }

    \foreach \g in {3,5,6} {
        \node[gnode] at (\g) {};
    }
	\draw(1)--(2)
		(3)--(4)--(5)
		(4)--(6);
	\draw[dotted] (2)--(3);

\end{scope}
\begin{scope}[xshift = 6cm, yshift=-2.5cm]
    \node[state, label=below:{$1$}] (1) {};
	\node[state, label=below:{$2$}] (2) [right = of 1] {};
	\node[state, label=below:{$n-2$}] (3) [right = of 2] {};
	\node[state, label=below:{$n-1$}] (4) [right =of 3] {};
	\node[state, label=below:{$n$}] (5) [right =of 4] {};

    \foreach \y in {4} {
        \node[ynode] at (\y) {};
    }

    \foreach \g in {3,5} {
        \node[gnode] at (\g) {};
    }

	\draw (1)--(2)
		(3)--(4);
	\draw [dotted] (2)--(3);
    \draw[double line] (4)-- node{\scalebox{1.3}{ $<$}} (5);
    	\node[label={below:\normalsize{$\rightsquigarrow$}}] [above left = 0.1cm and 1cm of 1] {};
\end{scope}

\begin{scope}[xshift=3.5cm, yshift=-4.5cm]

    \node[state, label=below:{$2$}] (2) {};
    \node[state, label=above:{$4$}] (4) [left = of 2] {};
    \node[state, label=above:{$3$}] (3) [above left = 0.4cm and 0.7cm of 4] {};
    \node[state, label=above:{$1$}] (1) [left = of 3] {};
    \node[state, label=below:{$5$}] (5) [below left = 0.4cm and 0.7cm of 4] {};
    \node[state, label=below:{$6$}] (6) [left = of 5] {};
    \foreach \y in {1,4,6} {
        \node[ynode] at (\y) {};
    }

    \foreach \g in {2,3,5} {
        \node[gnode] at (\g) {};
    }	
	\draw(1)--(3)--(4)--(5)--(6)
		(2)--(4);

\end{scope}
\begin{scope}[xshift = 6cm, yshift = -4.5cm]
    \node[state, label=below:{1}] (1) {};
	\node[state, label=below:{2}] (2) [right = of 1] {};
	\node[state, label=below:{3}] (3) [right = of 2] {};
	\node[state, label=below:{4}] (4) [right =of 3] {};
    \foreach \y in {1,3} {
        \node[ynode] at (\y) {};
    }

    \foreach \g in {2,4} {
        \node[gnode] at (\g) {};
    }		
	\draw (1)--(2)
		(3)--(4);
	\draw[double line] (2)-- node{\scalebox{1.3}{$>$}} (3);
	\node[label={below:\normalsize{$\rightsquigarrow$}}] [above left = 0.1cm and 1cm of 1] {};	
\end{scope}

\begin{scope}[xshift=2cm, yshift=-6.5cm]

    \node[state, label=left:{$2$}] (1) {};
    \node[state, label =right:{$1$}] (3) [above right = 0.4cm and 0.7cm of 1] {};
    \node[state, label=right:{$3$}] (2) [right = 0.7cm of 1] {};
    \node[state, label= right:{$4$}] (4) [below right = 0.4cm and 0.7cm of 1] {};

    \foreach \y in {1} {
        \node[ynode] at (\y) {};
    }

    \foreach \g in {2,3,4} {
        \node[gnode] at (\g) {};
    }
    
    \draw (3)--(1)--(4)
    (1)--(2);

\end{scope}
\begin{scope}[xshift = 6cm, yshift = -6.5cm]
	\node[state, label=below:{1}] (1) {};
	\node[state, label=below:{2}] (2) [right = of 1] {};
			
	\draw[triple line] (1)-- node{\scalebox{1.3}{$<$}} (2);
    \draw (1)--(2);
    \foreach \y in {1} {
        \node[ynode] at (\y) {};
    }

    \foreach \g in {2} {
        \node[gnode] at (\g) {};
    }
	\node[label={below:\normalsize{$\rightsquigarrow$}}] [above left = 0.1cm and 1cm of 1] {};	
\end{scope}
    \end{tikzpicture}
    \caption{Foldings in Dynkin diagrams of finite type (for seed patterns)}\label{fig_folding}
\end{figure}

\subsection{Combinatorics of exchange graphs}
\label{sec_comb_of_exchange_graphs}
The \emph{exchange graph} of a cluster pattern is the $n$-regular (finite or
infinite) connected graph whose vertices are the seeds of the cluster pattern
and whose edges connect the seeds related by a single mutation. For example,
the exchange graph in Example~\ref{example_A2_example} is a cycle graph
with~$5$~vertices. In this section, we recall the combinatorics of exchange
graphs which will be used later. For more details, we refer the reader
to~\cite{FZ2_2003, FZ_Ysystem03, FZ4_2007}.

As we already have seen in Theorem~\ref{thm_FZ_finite_type}, cluster algebras
of finite type are classified by Cartan matrices of finite type. Moreover,
for a cluster algebra of finite type, the exchange graph depends only on the
exchange matrix (see~\cite{FZ2_2003}). Because of this reason, we denote by
$E(\Roots)$ the exchange graph of a cluster pattern corresponding to the root
system $\Roots$.

To study the combinatorics of exchange graphs of cluster algebras, we prepare
some terminologies.
We call a graph over $[m]$ \emph{bipartite} if there is a function $\varepsilon \colon [m] \to \{+,-\}$, called a \emph{coloring}, such that for all $i$ and $j$ in $[m]$, 
\[
b_{i,j} \neq 0  \implies \begin{cases}
\varepsilon(i) = +, \\
\varepsilon(j) = -.
\end{cases}
\]
Here, $b_{i,j}$ is the adjacency matrix of the graph. 
For example, every tree is bipartite, but cycle graphs with an odd number of
vertices are not bipartite.
Dynkin diagrams of finite type are bipartite since they are trees.

Let $\Roots$ be a rank $n$ root system of finite type with the set of simple
roots $\SRoots = \{\alpha_i \mid i \in [n]\}$ and the set of positive roots
$\Roots^+$. 
For every subset $J \subset [n]$, let $\Roots(J)$ denote the root subsystem
of $\Roots$ spanned by the set of simple roots $\{ \alpha_i \mid i \in J \}$.
Note that $\Phi(J)$ may not be irreducible even if $\Phi$ is.
Let $W$ be the Weyl group of $\Roots$ which is
generated by the simple reflections $s_i = s_{\alpha_i}$.
Since the Dynkin diagram of $\Roots$ is a bipartite
graph, let $I_+$ and $I_-$ be two parts of the set~$[n]$; they are determined
uniquely up to renaming. 
Recall that a \emph{Coxeter element} is the product of all simple
reflections. 
The order $h$ of a Coxeter element in $W$ is called the \emph{Coxeter number}
of $\Roots$. We present the known formula of Coxeter numbers $h$ in
Table~\ref{table_Coxeter_number} (see~\cite[Appendix]{Bourbaki02}).
\begin{table}[b]
	\begin{tabular}{c|ccccccccc}
		\toprule
		$\Roots$ & $\dynA_n$ & $\dynB_n$ & $\dynC_n$ & $\dynD_n$ & $\dynE_6$ & $\dynE_7$ & $\dynE_8$ & $\dynF_4$ & $\dynG_2$ \\
		\midrule
		$h$ & $n+1$ & $2n$ & $2n$ & $2n-2$ & $12$ & $18$ & $30$ & $12$ & $6$ 		\\
		\bottomrule 
	\end{tabular}
\caption{Coxeter numbers}\label{table_Coxeter_number}
\end{table}

Let $\Delta(\Roots)$ be a simplicial complex whose ground set is
$\Roots_{\geq -1}$ and maximal simplices are called \emph{clusters}. The dual
graph of $\Delta(\Roots)$ is known to be the exchange graph $E(\Roots)$. 
Recall from~\cite{CFZ02_polytopal, FZ_Ysystem03} that there is 
a polytopal realization of the simplicial complex~$\Delta(\Roots)$, that is, 
there is a simple convex polytope $P(\Roots) \subset \R^n$ such that 
the dual complex of~$P(\Roots)$ agrees with $\Delta(\Roots)$.
The polytope $P(\Roots)$ is called the \emph{generalized associahedron}. We
denote by $F_{\beta}$ the facet of the polytope $P(\Roots)$ corresponding to
a root $\beta \in \Roots_{\geq -1}$. Here, a facet of a polytope of dimension
$n$ is a face of dimension $n-1$.

Consider the composition
$\qcoxeter = \mutation_- \mutation_+$ of a sequence of mutations where
\[
\mutation_{\varepsilon} = \prod_{i \in I_{\varepsilon}} \mutation_i \qquad \text{ for } \varepsilon \in \{ +, -\}.
\]
We call $\qcoxeter$ a \emph{Coxeter mutation}.
 It is known from~\cite[Proposition~3.2]{FZ_Ysystem03} that both $\mutation_+$ and
 $\mutation_-$ act on the face poset of the polytope $P(\Roots)$. Moreover, we
 have the following properties.
\begin{proposition}[{cf.~\cite[Propositions~2.5, 3.2, and~3.7]{FZ_Ysystem03}}]
    \label{prop_FZ_finite_type_Coxeter_element}
    The following holds.
\begin{enumerate} 
	\item Both $\mutation_+$ and $\mutation_-$ act on the face poset of the
	polytope $P(\Roots)$.
    \item 	Suppose that $h = 2e$ is even.  The map 
    $(r,i) \mapsto \qcoxeter^r(F_{-\alpha_i})$ induces a bijection 
\[
\{0,1,\dots,e\} \times I \to \facet(P(\Roots))
\]
where $\facet(P(\Roots))$ is the set of facets, which are codimension
one faces, of the polytope $P(\Roots)$.
\item The face poset of a facet  $\mutation_{\quiver}^r(F_{-\alpha_i})$  in $P(\Roots)$
is the same as that of the generalized
associahedron $P(\Roots ([n] \setminus \{i\}))$ of dimension $n-1$.
\item The facets corresponding to negative simple roots
$-\alpha_1,\dots,-\alpha_n$ intersect at a vertex.
\end{enumerate}
\end{proposition}

As a direct consequence of Proposition~\ref{prop_FZ_finite_type_Coxeter_element}, we 
have the following lemma which will be used later. 
\begin{lemma}\label{lemma:normal form}
	Let $\seed$ be a seed in a cluster pattern of finite type with even Coxeter number $h = 2e$.
	Suppose that $\seed \in \mutation_{\quiver}^r(F_{-\alpha_i})$ for  
	$r \in \{0,1,\dots,e\}$ and $\alpha_i \in \SRoots$. 
Then, there exists a sequence $j_1,\dots,j_{\ell} \in [n] \setminus \{i\}$ 
which gives a sequence $\mutation_{j_1},\dots,\mutation_{j_\ell}$ of mutations 
from $\qcoxeter^r(\initialseed)$ to $\seed$ inside a facet~$\mutation_{\quiver}^r(F_{-\alpha_i})$, that is, 
\begin{align*}
\qcoxeter^r(\initialseed), \mutation_{j_1}(\qcoxeter^r(\initialseed)),
(\mutation_{j_2} \mutation_{j_1})(\qcoxeter^r(\initialseed)), \dots,
(\mutation_{j_\ell} \cdots \mutation_{j_1})(\qcoxeter^r(\initialseed)) \in \mutation_{\quiver}^r(F_{-\alpha_i})
\end{align*}
and 
\[
\seed = (\mutation_{j_\ell} \cdots \mutation_{j_1})(\qcoxeter^r(\initialseed)).
\]
\end{lemma}
\begin{proof}
	Since $\initialseed\in F_{-\alpha_i}$, we have $\qcoxeter^r(\initialseed)\in \mutation_{\quiver}^r(F_{-\alpha_i})$. 
	Accordingly, both seeds $\qcoxeter^r(\initialseed)$ and $\seed$ are contained in the same facet $\mutation_{\quiver}^r(F_{-\alpha_i})$, so there exists a sequence $\mutation_{j_1},\dots,\mutation_{j_\ell}$ of mutations from $\qcoxeter^r(\initialseed)$ to $\seed$ inside $\mutation_{\quiver}^r(F_{-\alpha_i})$ as desired.
\end{proof}

\begin{example}
	Consider the root system $\Roots$ of type $\dynA_3$. In this case, the
	Coxeter number is $4$, which is even (cf.
	Table~\ref{table_Coxeter_number}). In Table~\ref{table_A3_tau_action}, we
	present how $\qcoxeter$ acts on the set of facets. Here, we use the convention that
	$I_+ = \{1,3\}$ and $I_- = \{2\}$.
	\begin{table}[b]
		\begin{tabular}{c|ccc}
			\toprule
			$r$ & $\qcoxeter^r(F_{-\alpha_1})$ 
			& $ \qcoxeter^r(F_{-\alpha_2})$ 
			& $ \qcoxeter^r(F_{-\alpha_3})$\\ 
			\midrule
			$0$ & $F_{-\alpha_1}$ & $F_{-\alpha_2}$ & $F_{-\alpha_3}$ \\
			$1$ & $F_{\alpha_1 + \alpha_2}$ & $F_{\alpha_2}$ & $F_{\alpha_2 + \alpha_3}$ \\
			$2$ & $F_{\alpha_3}$ & $F_{\alpha_1 + \alpha_2 + \alpha_3}$ & $F_{\alpha_1}$\\
			\bottomrule
		\end{tabular}
		\caption{Computation $\qcoxeter^r(F_{-\alpha_i})$ for type $\dynA_3$}\label{table_A3_tau_action}
	\end{table}
	The corresponding generalized associahedron is presented in
	Figure~\ref{fig_asso_A3}. We label each facet the corresponding almost
	positive root. The back-side facets are associated with the set of
	negative simple roots. As one may see that the face posets of $\qcoxeter^r(F_{-\alpha_i})$ 
	are the same as that of the
	generalized associahedron $P(\Roots ([n]\setminus \{i\}))$. Indeed, the facets
	$\qcoxeter^r(F_{-\alpha_1})$ and
	$\qcoxeter^r(F_{-\alpha_3})$ are pentagons, and the facets 
	$\qcoxeter^r(F_{-\alpha_2})$ are squares.
For $\seed = F_{-\alpha_1} \cap F_{-\alpha_2} \cap F_{-\alpha_3}$, 
we decorate the vertices $\{ \qcoxeter^r(\seed) \mid k = 0,1,2 \}$ with green. 
\end{example}

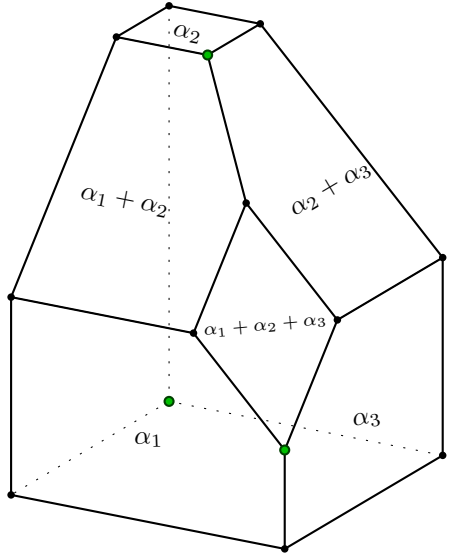
\begin{figure}
\tdplotsetmaincoords{110}{-30}
\begin{tikzpicture}%
[tdplot_main_coords,
scale= 2,
scale=0.700000,
back/.style={loosely dotted, thin},
edge/.style={color=black, thick},
facet/.style={fill=blue!95!black,fill opacity=0.100000},
vertex/.style={inner sep=1pt,circle,fill=black,thick,anchor=base},
gvertex/.style={inner sep=1.2pt,circle,draw=green!25!black,fill=green!75!black,thick,anchor=base}]
\coordinate (1) at (-0.50000, -1.50000, 2.00000);
\coordinate (2) at (1.50000, 1.50000, -2.00000);
\coordinate (3) at (0.50000, 0.50000, 1.00000);
\coordinate (4) at (0.50000, 1.50000, 0.00000);
\coordinate (5) at (1.50000, 1.50000, -1.00000);
\coordinate (6) at (-0.50000, -0.50000, 2.00000);
\coordinate (7) at (1.50000, 0.50000, 0.00000);
\coordinate (8) at (1.50000, -1.50000, 0.00000);
\coordinate (9) at (1.50000, -1.50000, -2.00000);
\coordinate (10) at (-1.50000, -1.50000, -2.00000);
\coordinate (11) at (-1.50000, -1.50000, 2.00000);
\coordinate (12) at (-1.50000, 1.50000, -2.00000);
\coordinate (13) at (-1.50000, 1.50000, 0.00000);
\coordinate (14) at (-1.50000, -0.50000, 2.00000);


\draw[edge,back] (9) -- (10);
\draw[edge,back] (10) -- (11);
\draw[edge,back] (10) -- (12);
\node[vertex] at (10)     {};

\draw (7)--(8) node[midway, sloped, above, yshift=1.3cm] {$\alpha_2 + \alpha_3$};
\draw (13)--(4) node[midway, sloped, above, yshift=1.3cm] {$\alpha_1 + \alpha_2$};
\draw[color=white] (13)--(2) node[midway, sloped, below, rotate = 45] {\textcolor{black}{$\alpha_1$}};
\draw[color=white] (2)--(8) node[midway, sloped, below, rotate = - 50] {\textcolor{black}{$\alpha_3$}};

\draw (14)--(6) node[midway, sloped, above, xshift=0.3cm] {$\alpha_2$};
\draw[color=white] (4)--(7) node[midway, sloped, above, yshift=-0.2cm] {\textcolor{black}{\scriptsize $\alpha_1 + \alpha_2 + \alpha_3$}};

\draw[edge] (1) -- (6);
\draw[edge] (1) -- (8);
\draw[edge] (1) -- (11);
\draw[edge] (2) -- (5);
\draw[edge] (2) -- (9);
\draw[edge] (2) -- (12);
\draw[edge] (3) -- (4);
\draw[edge] (3) -- (6);
\draw[edge] (3) -- (7);
\draw[edge] (4) -- (5);
\draw[edge] (4) -- (13);
\draw[edge] (5) -- (7);
\draw[edge] (6) -- (14);
\draw[edge] (7) -- (8);
\draw[edge] (8) -- (9);
\draw[edge] (11) -- (14);
\draw[edge] (12) -- (13);
\draw[edge] (13) -- (14);
\node[vertex] at (1)     {};
\node[vertex] at (2)     {};
\node[vertex] at (3)     {};
\node[vertex] at (4)     {};
\node[vertex] at (5)     {};
\node[vertex] at (6)     {};
\node[vertex] at (7)     {};
\node[vertex] at (8)     {};
\node[vertex] at (9)     {};
\node[vertex] at (11)     {};
\node[vertex] at (12)     {};
\node[vertex] at (13)     {};
\node[vertex] at (14)     {};

\foreach \g in {10, 6, 5} {
\node[gvertex] at (\g) {};
}
\end{tikzpicture}
\caption{The type $\dynA_3$ generalized associahedron}\label{fig_asso_A3}
\end{figure}

\begin{example}
	We consider the generalized associahedron of type $\dynD_4$ and present
	four facets corresponding to the negative simple roots in
	Figure~\ref{fig_asso_D4}. The facet corresponding to $-\alpha_2$ is
    combinatorially equivalent to $P(\Roots(\{1\})) \times P(\Roots(\{3\})) 
    \times P(\Roots(\{4\}))$, which is a $3$-cube presented in the boundary. The
	intersection of these four facets is a vertex sits in the bottom colored
	in green. The Coxeter mutation $\qcoxeter$ acts on the face poset
	of the permutohedron, especially, four green vertices are in the same
	orbit.
\end{example}

\begin{remark}\label{remark:folding and Coxeter mutation}
As we have seen in Example~\ref{example_folding_ADE}, bipartite coloring 
on quivers of type~$\dynADE$ induce that on quivers of type~$\dynBCFG$.
Accordingly, if a seed pattern of simply-laced type $\dynX$ 
gives a seed pattern of type $\dynY$ via the folding procedure, then
the Coxeter mutation of type~$\dynY$ is the same as
that of type~$\dynX$. 
More precisely, for a globally foldable seed $\seed$ with respect to $G$ 
defining a cluster algebra
of type $\dynX$ and its Coxeter mutation $\qcoxeter^{\dynX}$, we have
\[
 \qcoxeter^{\dynY}(\seed^G) = (\qcoxeter^{\dynX}(\seed))^G.
\]
Here,  $\qcoxeter^{\dynY}$ is the Coxeter mutation on 
the seed pattern determined by $\seed^G$.

Moreover, Coxeter numbers of $\dynX$
and $\dynY$ are the same. Indeed, 
\[
\begin{split}
& h(\dynA_{2n-1}) = h (\dynB_n) = 2n, \\
& h(\dynD_{n+1}) = h(\dynC_n) = 2n, \\
& h(\dynE_6) = h(\dynF_4) = 12, \\
& h(\dynD_4) = h(\dynG_2) = 6.
\end{split}
\]
\end{remark}
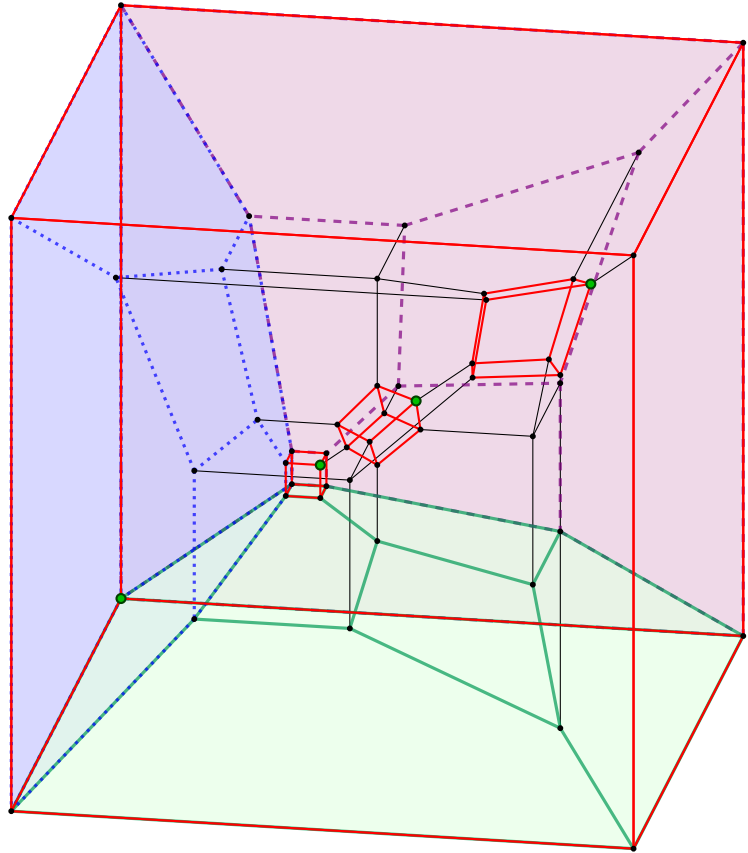
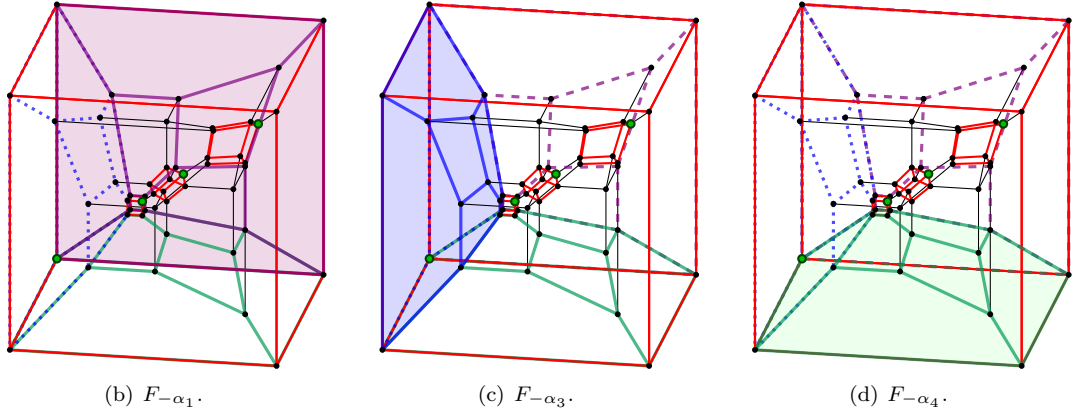
\begin{figure}
	\subfigure[The generalized associahedron of type $\dynD_4$.]{
        \centering
\tdplotsetmaincoords{110}{260}
\begin{tikzpicture}%
    [tdplot_main_coords,scale = 7,
    back/.style={loosely dotted, thin},
    edge/.style={color=black},
    cube/.style={color=red, thick},
    facet/.style={fill=blue!95!black,fill opacity=0.100000},
    vertex/.style={inner sep=1.2pt,circle,draw=green!25!black,fill=green!75!black,thick,anchor=base},
    vertex_normal/.style = {inner sep=0.5pt,circle,draw=black,fill=black,thick,anchor=base},
    edge0/.style = {color=blue!50!red, very thick,  dashed, opacity=0.7},
    edge8/.style= {color=ForestGreen, very thick, opacity=0.7},
    edge12/.style= {blue,dotted, very thick, opacity=0.7},
    face1/.style = {fill=blue!20!white, fill opacity = 0.5}]

    \coordinate (0) at (0.6, 0.6, 0.6);
    \coordinate (1) at (0.42857142857142855, -0.42857142857142855, -0.42857142857142855);
    \coordinate (2) at (-0.13333333333333333, 0.2, -0.13333333333333333);
    \coordinate (3) at (0.375, -0.25, 0.25);
    \coordinate (4) at (0.3, 0.0, 0.0);
    \coordinate (5) at (0.09090909090909091, -0.09090909090909091, 0.09090909090909091);
    \coordinate (6) at (0.2222222222222222, -0.2222222222222222, 0.2222222222222222);
    \coordinate (7) at (0.42857142857142855, -0.2857142857142857, 0.42857142857142855);
    \coordinate (8) at (0.08333333333333333, 0.0, 0.0);
    \coordinate (9) at (0.0, 0.07692307692307693, 0.0);
    \coordinate (10) at (-0.07142857142857142, 0.07142857142857142, -0.07142857142857142);
    \coordinate (11) at (0.0, 0.0, -0.07692307692307693);
    \coordinate (12) at (-0.07692307692307693, 0.0, 0.0);
    \coordinate (13) at (0.0, -0.08333333333333333, 0.0);
    \coordinate (14) at (0.0, 0.0, 0.08333333333333333);
    \coordinate (15) at (-0.13333333333333333, 0.13333333333333333, -0.13333333333333333);
    \coordinate (16) at (0.0, 0.0, 0.3);
    \coordinate (17) at (0.25, -0.25, 0.375);
    \coordinate (18) at (0.2857142857142857, -0.42857142857142855, 0.42857142857142855);
    \coordinate (19) at (0.0, -0.3, 0.0);
    \coordinate (20) at (0.5, -0.5, 0.5);
    \coordinate (21) at (0.42857142857142855, -0.42857142857142855, 0.2857142857142857);
    \coordinate (22) at (0.25, -0.375, 0.25);
    \coordinate (23) at (0.0, 0.3, 0.3);
    \coordinate (24) at (0.3, 0.3, 0.0);
    \coordinate (25) at (0.42857142857142855, 0.42857142857142855, 0.42857142857142855);
    \coordinate (26) at (0.0, 0.23076923076923078, 0.0);
    \coordinate (27) at (0.3, 0.3, -0.3);
    \coordinate (28) at (-0.13333333333333333, 0.2, -0.2);
    \coordinate (29) at (0.3, 0.0, -0.3);
    \coordinate (30) at (0.0, 0.0, -0.23076923076923078);
    \coordinate (31) at (-0.13333333333333333, 0.13333333333333333, -0.2);
    \coordinate (32) at (0.0, -0.3, -0.3);
    \coordinate (33) at (0.6, 0.6, -0.6);
    \coordinate (34) at (0.6, -0.6, -0.6);
    \coordinate (35) at (0.6, -0.6, 0.6);
    \coordinate (36) at (-0.6, -0.6, 0.6);
    \coordinate (37) at (-0.6, 0.6, 0.6);
    \coordinate (38) at (-0.2, 0.2, -0.13333333333333333);
    \coordinate (39) at (-0.23076923076923078, 0.0, 0.0);
    \coordinate (40) at (-0.2, 0.13333333333333333, -0.13333333333333333);
    \coordinate (41) at (-0.3, 0.0, 0.3);
    \coordinate (42) at (-0.42857142857142855, -0.42857142857142855, 0.42857142857142855);
    \coordinate (43) at (-0.3, -0.3, 0.0);
    \coordinate (44) at (-0.3, 0.3, 0.3);
    \coordinate (45) at (-0.2, 0.2, -0.2);
    \coordinate (46) at (-0.2, 0.13333333333333333, -0.2);
    \coordinate (47) at (-0.3, -0.3, -0.3);
    \coordinate (48) at (-0.6, 0.6, -0.6);
    \coordinate (49) at (-0.6, -0.6, -0.6);

    \fill[red!60!blue, opacity=0.15] (37)--(48)--(49)--(36)--cycle {};
\fill[face1] (37)--(48)--(45)--(38)--(44)--cycle {};
\fill[face1] (37)--(48)--(33)--(0)--cycle{};
\fill[face1] (37)--(44)--(23)--(25)--(0)--cycle {};
\fill[face1] (44)--(23)--(26)--(2)--(38)--cycle {};
\fill[face1] (2)--(28)--(45)--(38)--cycle {};
\fill[face1] (45)--(28)--(27)--(33)--(48)--cycle {};
\fill[face1] (2)--(26)--(24)--(27)--(28)--cycle {};
\fill[face1] (25)--(23)--(26)--(24)--cycle;
\fill[face1] (0)--(25)--(24)--(27)--(33)--cycle;

\fill[green!10!white, opacity=0.7] (48)--(45)--(46)--(47)--(49)--(34)--(33)--cycle{};

    \draw[edge] (3)--(4);
    \draw[edge] (5)--(6);
    \draw[edge] (4)--(8);
    \draw[edge] (10)--(15);
    \draw[edge] (14)--(16);
    \draw[edge] (16)--(17);
    \draw[edge] (13)--(19);
    \draw[edge] (1)--(21);
    \draw[edge] (19)--(22);
    \draw[edge] (16)--(23);
    \draw[edge] (4)--(24);
    \draw[edge] (7)--(25);
    \draw[edge] (9)--(26);
    \draw[edge] (4)--(29);
    \draw[edge] (11)--(30);
    \draw[edge] (19)--(32);
    \draw[edge] (20)--(35);
    \draw[edge] (12)--(39);
    \draw[edge] (16)--(41);
    \draw[edge] (18)--(42);
    \draw[edge] (19)--(43);

    \draw[cube] (0)--(33);
    \draw[cube] (33)--(34);
    \draw[cube] (34)--(35);
    \draw[cube] (0)--(35);
    \draw[cube] (35)--(36);
    \draw[cube] (0)--(37);
    \draw[cube] (36)--(37);
    \draw[cube] (37)--(48);
    \draw[cube] (33)--(48);
    \draw[cube] (48)--(49);
    \draw[cube] (34)--(49);
    \draw[cube] (36)--(49);

    \draw[edge0] (36)--(37);
    \draw[edge0] (39)--(40);
    \draw[edge0] (38)--(40);
    \draw[edge0] (39)--(41);
    \draw[edge0] (41)--(42);
    \draw[edge0] (36)--(42);
    \draw[edge0] (39)--(43);
    \draw[edge0] (42)--(43);
    \draw[edge0] (37)--(44);
    \draw[edge0] (38)--(44);
    \draw[edge0] (41)--(44);
    \draw[edge0] (38)--(45);
    \draw[edge0] (45)--(46);
    \draw[edge0] (40)--(46);
    \draw[edge0] (43)--(47);
    \draw[edge0] (46)--(47);
    \draw[edge0] (37)--(48);
    \draw[edge0] (45)--(48);
    \draw[edge0] (48)--(49);
    \draw[edge0] (47)--(49);
    \draw[edge0] (36)--(49);

    \draw[edge8] (27)--(28);
    \draw[edge8] (1)--(29);
    \draw[edge8] (27)--(29);
    \draw[edge8] (29)--(30);
    \draw[edge8] (28)--(31);
    \draw[edge8] (30)--(31);
    \draw[edge8] (30)--(32);
    \draw[edge8] (1)--(32);
    \draw[edge8] (27)--(33);
    \draw[edge8] (33)--(34);
    \draw[edge8] (1)--(34);
    \draw[edge8] (28)--(45);
    \draw[edge8] (45)--(46);
    \draw[edge8] (31)--(46);
    \draw[edge8] (32)--(47);
    \draw[edge8] (46)--(47);
    \draw[edge8] (45)--(48);
    \draw[edge8] (33)--(48);
    \draw[edge8] (48)--(49);
    \draw[edge8] (47)--(49);
    \draw[edge8] (34)--(49);

    \draw[edge12] (24)--(25);
    \draw[edge12] (23)--(25);
    \draw[edge12] (0)--(25);
    \draw[edge12] (2)--(26);
    \draw[edge12] (23)--(26);
    \draw[edge12] (24)--(26);
    \draw[edge12] (24)--(27);
    \draw[edge12] (2)--(28);
    \draw[edge12] (27)--(28);
    \draw[edge12] (0)--(33);
    \draw[edge12] (27)--(33);
    \draw[edge12] (0)--(37);
    \draw[edge12] (2)--(38);
    \draw[edge12] (37)--(44);
    \draw[edge12] (38)--(44);
    \draw[edge12] (23)--(44);
    \draw[edge12] (38)--(45);
    \draw[edge12] (28)--(45);
    \draw[edge12] (37)--(48);
    \draw[edge12] (45)--(48);
    \draw[edge12] (33)--(48);

        \draw[cube] (0)--(33);
        \draw[cube] (33)--(34);
        \draw[cube] (34)--(35);
        \draw[cube] (0)--(35);
        \draw[cube] (35)--(36);
        \draw[cube] (0)--(37);
        \draw[cube] (36)--(37);
        \draw[cube] (37)--(48);
        \draw[cube] (33)--(48);
        \draw[cube] (48)--(49);
        \draw[cube] (34)--(49);
        \draw[cube] (36)--(49);

        \draw[cube] (2)--(15);
        \draw[cube] (2)--(28);
        \draw[cube] (15)--(31);
        \draw[cube] (28)--(31);
        \draw[cube] (2)--(38);
        \draw[cube] (38)--(40);
        \draw[cube] (15)--(40);
        \draw[cube] (38)--(45);
        \draw[cube] (28)--(45);
        \draw[cube] (45)--(46);
        \draw[cube] (31)--(46);
        \draw[cube] (40)--(46);

        \draw[cube] (3)--(6);
        \draw[cube] (3)--(7);
        \draw[cube] (6)--(17);
        \draw[cube] (7)--(17);
        \draw[cube] (17)--(18);
        \draw[cube] (18)--(20);
        \draw[cube] (7)--(20);
        \draw[cube] (3)--(21);
        \draw[cube] (20)--(21);
        \draw[cube] (18)--(22);
        \draw[cube] (21)--(22);
        \draw[cube] (6)--(22);

        \draw[cube] (5)--(8);
        \draw[cube] (8)--(9);
        \draw[cube] (9)--(10);
        \draw[cube] (8)--(11);
        \draw[cube] (10)--(11);
        \draw[cube] (10)--(12);
        \draw[cube] (11)--(13);
        \draw[cube] (12)--(13);
        \draw[cube] (5)--(13);
        \draw[cube] (9)--(14);
        \draw[cube] (12)--(14);
        \draw[cube] (5)--(14);

        \foreach \x in {0,...,49}{
            \node[vertex_normal] at (\x) {};
    
        }
    
        \foreach \x in {48, 20, 15, 5}{
            \node[vertex] at (\x) {};
        }
	\end{tikzpicture}
	}
	\vspace{1em}
\subfigure[$F_{-\alpha_1}$.]{
    \centering
    \tdplotsetmaincoords{110}{260}
\begin{tikzpicture}%
	[tdplot_main_coords, scale = 3,
    back/.style={loosely dotted, thin},
    edge/.style={color=black},
    cube/.style={color=red, thick},
    facet/.style={fill=blue!95!black,fill opacity=0.100000},
    vertex/.style={inner sep=1pt,circle,draw=green!25!black,fill=green!75!black,thick,anchor=base},
    vertex_normal/.style = {inner sep=0.5pt,circle,draw=black,fill=black,thick,anchor=base},
    edge0/.style = {color=blue!50!red, very thick, opacity=0.7},
    edge8/.style= {color=ForestGreen, very thick, opacity=0.7},
    edge12/.style= {blue,dotted, very thick, opacity=0.7},
    face1/.style = {fill=blue!20!white, fill opacity = 0.5}]

    \coordinate (0) at (0.6, 0.6, 0.6);
    \coordinate (1) at (0.42857142857142855, -0.42857142857142855, -0.42857142857142855);
    \coordinate (2) at (-0.13333333333333333, 0.2, -0.13333333333333333);
    \coordinate (3) at (0.375, -0.25, 0.25);
    \coordinate (4) at (0.3, 0.0, 0.0);
    \coordinate (5) at (0.09090909090909091, -0.09090909090909091, 0.09090909090909091);
    \coordinate (6) at (0.2222222222222222, -0.2222222222222222, 0.2222222222222222);
    \coordinate (7) at (0.42857142857142855, -0.2857142857142857, 0.42857142857142855);
    \coordinate (8) at (0.08333333333333333, 0.0, 0.0);
    \coordinate (9) at (0.0, 0.07692307692307693, 0.0);
    \coordinate (10) at (-0.07142857142857142, 0.07142857142857142, -0.07142857142857142);
    \coordinate (11) at (0.0, 0.0, -0.07692307692307693);
    \coordinate (12) at (-0.07692307692307693, 0.0, 0.0);
    \coordinate (13) at (0.0, -0.08333333333333333, 0.0);
    \coordinate (14) at (0.0, 0.0, 0.08333333333333333);
    \coordinate (15) at (-0.13333333333333333, 0.13333333333333333, -0.13333333333333333);
    \coordinate (16) at (0.0, 0.0, 0.3);
    \coordinate (17) at (0.25, -0.25, 0.375);
    \coordinate (18) at (0.2857142857142857, -0.42857142857142855, 0.42857142857142855);
    \coordinate (19) at (0.0, -0.3, 0.0);
    \coordinate (20) at (0.5, -0.5, 0.5);
    \coordinate (21) at (0.42857142857142855, -0.42857142857142855, 0.2857142857142857);
    \coordinate (22) at (0.25, -0.375, 0.25);
    \coordinate (23) at (0.0, 0.3, 0.3);
    \coordinate (24) at (0.3, 0.3, 0.0);
    \coordinate (25) at (0.42857142857142855, 0.42857142857142855, 0.42857142857142855);
    \coordinate (26) at (0.0, 0.23076923076923078, 0.0);
    \coordinate (27) at (0.3, 0.3, -0.3);
    \coordinate (28) at (-0.13333333333333333, 0.2, -0.2);
    \coordinate (29) at (0.3, 0.0, -0.3);
    \coordinate (30) at (0.0, 0.0, -0.23076923076923078);
    \coordinate (31) at (-0.13333333333333333, 0.13333333333333333, -0.2);
    \coordinate (32) at (0.0, -0.3, -0.3);
    \coordinate (33) at (0.6, 0.6, -0.6);
    \coordinate (34) at (0.6, -0.6, -0.6);
    \coordinate (35) at (0.6, -0.6, 0.6);
    \coordinate (36) at (-0.6, -0.6, 0.6);
    \coordinate (37) at (-0.6, 0.6, 0.6);
    \coordinate (38) at (-0.2, 0.2, -0.13333333333333333);
    \coordinate (39) at (-0.23076923076923078, 0.0, 0.0);
    \coordinate (40) at (-0.2, 0.13333333333333333, -0.13333333333333333);
    \coordinate (41) at (-0.3, 0.0, 0.3);
    \coordinate (42) at (-0.42857142857142855, -0.42857142857142855, 0.42857142857142855);
    \coordinate (43) at (-0.3, -0.3, 0.0);
    \coordinate (44) at (-0.3, 0.3, 0.3);
    \coordinate (45) at (-0.2, 0.2, -0.2);
    \coordinate (46) at (-0.2, 0.13333333333333333, -0.2);
    \coordinate (47) at (-0.3, -0.3, -0.3);
    \coordinate (48) at (-0.6, 0.6, -0.6);
    \coordinate (49) at (-0.6, -0.6, -0.6);

    \fill[red!60!blue, opacity=0.15] (37)--(48)--(49)--(36)--cycle {};


    \draw[edge] (3)--(4);
    \draw[edge] (5)--(6);
    \draw[edge] (4)--(8);
    \draw[edge] (10)--(15);
    \draw[edge] (14)--(16);
    \draw[edge] (16)--(17);
    \draw[edge] (13)--(19);
    \draw[edge] (1)--(21);
    \draw[edge] (19)--(22);
    \draw[edge] (16)--(23);
    \draw[edge] (4)--(24);
    \draw[edge] (7)--(25);
    \draw[edge] (9)--(26);
    \draw[edge] (4)--(29);
    \draw[edge] (11)--(30);
    \draw[edge] (19)--(32);
    \draw[edge] (20)--(35);
    \draw[edge] (12)--(39);
    \draw[edge] (16)--(41);
    \draw[edge] (18)--(42);
    \draw[edge] (19)--(43);

    \draw[cube] (0)--(33);
    \draw[cube] (33)--(34);
    \draw[cube] (34)--(35);
    \draw[cube] (0)--(35);
    \draw[cube] (35)--(36);
    \draw[cube] (0)--(37);
    \draw[cube] (36)--(37);
    \draw[cube] (37)--(48);
    \draw[cube] (33)--(48);
    \draw[cube] (48)--(49);
    \draw[cube] (34)--(49);
    \draw[cube] (36)--(49);

    \draw[edge8] (27)--(28);
    \draw[edge8] (1)--(29);
    \draw[edge8] (27)--(29);
    \draw[edge8] (29)--(30);
    \draw[edge8] (28)--(31);
    \draw[edge8] (30)--(31);
    \draw[edge8] (30)--(32);
    \draw[edge8] (1)--(32);
    \draw[edge8] (27)--(33);
    \draw[edge8] (33)--(34);
    \draw[edge8] (1)--(34);
    \draw[edge8] (28)--(45);
    \draw[edge8] (45)--(46);
    \draw[edge8] (31)--(46);
    \draw[edge8] (32)--(47);
    \draw[edge8] (46)--(47);
    \draw[edge8] (45)--(48);
    \draw[edge8] (33)--(48);
    \draw[edge8] (48)--(49);
    \draw[edge8] (47)--(49);
    \draw[edge8] (34)--(49);

    \draw[edge12] (24)--(25);
    \draw[edge12] (23)--(25);
    \draw[edge12] (0)--(25);
    \draw[edge12] (2)--(26);
    \draw[edge12] (23)--(26);
    \draw[edge12] (24)--(26);
    \draw[edge12] (24)--(27);
    \draw[edge12] (2)--(28);
    \draw[edge12] (27)--(28);
    \draw[edge12] (0)--(33);
    \draw[edge12] (27)--(33);
    \draw[edge12] (0)--(37);
    \draw[edge12] (2)--(38);
    \draw[edge12] (37)--(44);
    \draw[edge12] (38)--(44);
    \draw[edge12] (23)--(44);
    \draw[edge12] (38)--(45);
    \draw[edge12] (28)--(45);
    \draw[edge12] (37)--(48);
    \draw[edge12] (45)--(48);
    \draw[edge12] (33)--(48);

        \draw[cube] (0)--(33);
        \draw[cube] (33)--(34);
        \draw[cube] (34)--(35);
        \draw[cube] (0)--(35);
        \draw[cube] (35)--(36);
        \draw[cube] (0)--(37);
        \draw[cube] (36)--(37);
        \draw[cube] (37)--(48);
        \draw[cube] (33)--(48);
        \draw[cube] (48)--(49);
        \draw[cube] (34)--(49);
        \draw[cube] (36)--(49);

        \draw[cube] (2)--(15);
        \draw[cube] (2)--(28);
        \draw[cube] (15)--(31);
        \draw[cube] (28)--(31);
        \draw[cube] (2)--(38);
        \draw[cube] (38)--(40);
        \draw[cube] (15)--(40);
        \draw[cube] (38)--(45);
        \draw[cube] (28)--(45);
        \draw[cube] (45)--(46);
        \draw[cube] (31)--(46);
        \draw[cube] (40)--(46);

        \draw[cube] (3)--(6);
        \draw[cube] (3)--(7);
        \draw[cube] (6)--(17);
        \draw[cube] (7)--(17);
        \draw[cube] (17)--(18);
        \draw[cube] (18)--(20);
        \draw[cube] (7)--(20);
        \draw[cube] (3)--(21);
        \draw[cube] (20)--(21);
        \draw[cube] (18)--(22);
        \draw[cube] (21)--(22);
        \draw[cube] (6)--(22);

        \draw[cube] (5)--(8);
        \draw[cube] (8)--(9);
        \draw[cube] (9)--(10);
        \draw[cube] (8)--(11);
        \draw[cube] (10)--(11);
        \draw[cube] (10)--(12);
        \draw[cube] (11)--(13);
        \draw[cube] (12)--(13);
        \draw[cube] (5)--(13);
        \draw[cube] (9)--(14);
        \draw[cube] (12)--(14);
        \draw[cube] (5)--(14);

        \draw[edge0] (36)--(37);
        \draw[edge0] (39)--(40);
        \draw[edge0] (38)--(40);
        \draw[edge0] (39)--(41);
        \draw[edge0] (41)--(42);
        \draw[edge0] (36)--(42);
        \draw[edge0] (39)--(43);
        \draw[edge0] (42)--(43);
        \draw[edge0] (37)--(44);
        \draw[edge0] (38)--(44);
        \draw[edge0] (41)--(44);
        \draw[edge0] (38)--(45);
        \draw[edge0] (45)--(46);
        \draw[edge0] (40)--(46);
        \draw[edge0] (43)--(47);
        \draw[edge0] (46)--(47);
        \draw[edge0] (37)--(48);
        \draw[edge0] (45)--(48);
        \draw[edge0] (48)--(49);
        \draw[edge0] (47)--(49);
        \draw[edge0] (36)--(49);
    
        \foreach \x in {0,...,49}{
            \node[vertex_normal] at (\x) {};
    
        }
    
        \foreach \x in {48, 20, 15, 5}{
            \node[vertex] at (\x) {};
        }
\end{tikzpicture}
}
\subfigure[$F_{-\alpha_3}$.]{
    \centering
    \tdplotsetmaincoords{110}{260}
\begin{tikzpicture}%
    [tdplot_main_coords,scale = 3,
    back/.style={loosely dotted, thin},
    edge/.style={color=black},
    cube/.style={color=red, thick},
    facet/.style={fill=blue!95!black,fill opacity=0.100000},
    vertex/.style={inner sep=1pt,circle,draw=green!25!black,fill=green!75!black,thick,anchor=base},
    vertex_normal/.style = {inner sep=0.5pt,circle,draw=black,fill=black,thick,anchor=base},
    edge0/.style = {color=blue!50!red, very thick,  dashed, opacity=0.7},
    edge8/.style= {color=ForestGreen, very thick, opacity=0.7},
    edge12/.style= {blue,dotted, very thick, opacity=0.7},
    edge12_/.style={blue, very thick, opacity = 0.7},
    face1/.style = {fill=blue!20!white, fill opacity = 0.5}]

    \coordinate (0) at (0.6, 0.6, 0.6);
    \coordinate (1) at (0.42857142857142855, -0.42857142857142855, -0.42857142857142855);
    \coordinate (2) at (-0.13333333333333333, 0.2, -0.13333333333333333);
    \coordinate (3) at (0.375, -0.25, 0.25);
    \coordinate (4) at (0.3, 0.0, 0.0);
    \coordinate (5) at (0.09090909090909091, -0.09090909090909091, 0.09090909090909091);
    \coordinate (6) at (0.2222222222222222, -0.2222222222222222, 0.2222222222222222);
    \coordinate (7) at (0.42857142857142855, -0.2857142857142857, 0.42857142857142855);
    \coordinate (8) at (0.08333333333333333, 0.0, 0.0);
    \coordinate (9) at (0.0, 0.07692307692307693, 0.0);
    \coordinate (10) at (-0.07142857142857142, 0.07142857142857142, -0.07142857142857142);
    \coordinate (11) at (0.0, 0.0, -0.07692307692307693);
    \coordinate (12) at (-0.07692307692307693, 0.0, 0.0);
    \coordinate (13) at (0.0, -0.08333333333333333, 0.0);
    \coordinate (14) at (0.0, 0.0, 0.08333333333333333);
    \coordinate (15) at (-0.13333333333333333, 0.13333333333333333, -0.13333333333333333);
    \coordinate (16) at (0.0, 0.0, 0.3);
    \coordinate (17) at (0.25, -0.25, 0.375);
    \coordinate (18) at (0.2857142857142857, -0.42857142857142855, 0.42857142857142855);
    \coordinate (19) at (0.0, -0.3, 0.0);
    \coordinate (20) at (0.5, -0.5, 0.5);
    \coordinate (21) at (0.42857142857142855, -0.42857142857142855, 0.2857142857142857);
    \coordinate (22) at (0.25, -0.375, 0.25);
    \coordinate (23) at (0.0, 0.3, 0.3);
    \coordinate (24) at (0.3, 0.3, 0.0);
    \coordinate (25) at (0.42857142857142855, 0.42857142857142855, 0.42857142857142855);
    \coordinate (26) at (0.0, 0.23076923076923078, 0.0);
    \coordinate (27) at (0.3, 0.3, -0.3);
    \coordinate (28) at (-0.13333333333333333, 0.2, -0.2);
    \coordinate (29) at (0.3, 0.0, -0.3);
    \coordinate (30) at (0.0, 0.0, -0.23076923076923078);
    \coordinate (31) at (-0.13333333333333333, 0.13333333333333333, -0.2);
    \coordinate (32) at (0.0, -0.3, -0.3);
    \coordinate (33) at (0.6, 0.6, -0.6);
    \coordinate (34) at (0.6, -0.6, -0.6);
    \coordinate (35) at (0.6, -0.6, 0.6);
    \coordinate (36) at (-0.6, -0.6, 0.6);
    \coordinate (37) at (-0.6, 0.6, 0.6);
    \coordinate (38) at (-0.2, 0.2, -0.13333333333333333);
    \coordinate (39) at (-0.23076923076923078, 0.0, 0.0);
    \coordinate (40) at (-0.2, 0.13333333333333333, -0.13333333333333333);
    \coordinate (41) at (-0.3, 0.0, 0.3);
    \coordinate (42) at (-0.42857142857142855, -0.42857142857142855, 0.42857142857142855);
    \coordinate (43) at (-0.3, -0.3, 0.0);
    \coordinate (44) at (-0.3, 0.3, 0.3);
    \coordinate (45) at (-0.2, 0.2, -0.2);
    \coordinate (46) at (-0.2, 0.13333333333333333, -0.2);
    \coordinate (47) at (-0.3, -0.3, -0.3);
    \coordinate (48) at (-0.6, 0.6, -0.6);
    \coordinate (49) at (-0.6, -0.6, -0.6);

\fill[face1] (37)--(48)--(45)--(38)--(44)--cycle {};
\fill[face1] (37)--(48)--(33)--(0)--cycle{};
\fill[face1] (37)--(44)--(23)--(25)--(0)--cycle {};
\fill[face1] (44)--(23)--(26)--(2)--(38)--cycle {};
\fill[face1] (2)--(28)--(45)--(38)--cycle {};
\fill[face1] (45)--(28)--(27)--(33)--(48)--cycle {};
\fill[face1] (2)--(26)--(24)--(27)--(28)--cycle {};
\fill[face1] (25)--(23)--(26)--(24)--cycle;
\fill[face1] (0)--(25)--(24)--(27)--(33)--cycle;


    \draw[edge] (3)--(4);
    \draw[edge] (5)--(6);
    \draw[edge] (4)--(8);
    \draw[edge] (10)--(15);
    \draw[edge] (14)--(16);
    \draw[edge] (16)--(17);
    \draw[edge] (13)--(19);
    \draw[edge] (1)--(21);
    \draw[edge] (19)--(22);
    \draw[edge] (16)--(23);
    \draw[edge] (4)--(24);
    \draw[edge] (7)--(25);
    \draw[edge] (9)--(26);
    \draw[edge] (4)--(29);
    \draw[edge] (11)--(30);
    \draw[edge] (19)--(32);
    \draw[edge] (20)--(35);
    \draw[edge] (12)--(39);
    \draw[edge] (16)--(41);
    \draw[edge] (18)--(42);
    \draw[edge] (19)--(43);

    \draw[cube] (0)--(33);
    \draw[cube] (33)--(34);
    \draw[cube] (34)--(35);
    \draw[cube] (0)--(35);
    \draw[cube] (35)--(36);
    \draw[cube] (0)--(37);
    \draw[cube] (36)--(37);
    \draw[cube] (37)--(48);
    \draw[cube] (33)--(48);
    \draw[cube] (48)--(49);
    \draw[cube] (34)--(49);
    \draw[cube] (36)--(49);

    \draw[edge0] (36)--(37);
    \draw[edge0] (39)--(40);
    \draw[edge0] (38)--(40);
    \draw[edge0] (39)--(41);
    \draw[edge0] (41)--(42);
    \draw[edge0] (36)--(42);
    \draw[edge0] (39)--(43);
    \draw[edge0] (42)--(43);
    \draw[edge0] (37)--(44);
    \draw[edge0] (38)--(44);
    \draw[edge0] (41)--(44);
    \draw[edge0] (38)--(45);
    \draw[edge0] (45)--(46);
    \draw[edge0] (40)--(46);
    \draw[edge0] (43)--(47);
    \draw[edge0] (46)--(47);
    \draw[edge0] (37)--(48);
    \draw[edge0] (45)--(48);
    \draw[edge0] (48)--(49);
    \draw[edge0] (47)--(49);
    \draw[edge0] (36)--(49);

    \draw[edge8] (27)--(28);
    \draw[edge8] (1)--(29);
    \draw[edge8] (27)--(29);
    \draw[edge8] (29)--(30);
    \draw[edge8] (28)--(31);
    \draw[edge8] (30)--(31);
    \draw[edge8] (30)--(32);
    \draw[edge8] (1)--(32);
    \draw[edge8] (27)--(33);
    \draw[edge8] (33)--(34);
    \draw[edge8] (1)--(34);
    \draw[edge8] (28)--(45);
    \draw[edge8] (45)--(46);
    \draw[edge8] (31)--(46);
    \draw[edge8] (32)--(47);
    \draw[edge8] (46)--(47);
    \draw[edge8] (45)--(48);
    \draw[edge8] (33)--(48);
    \draw[edge8] (48)--(49);
    \draw[edge8] (47)--(49);
    \draw[edge8] (34)--(49);

        \draw[cube] (0)--(33);
        \draw[cube] (33)--(34);
        \draw[cube] (34)--(35);
        \draw[cube] (0)--(35);
        \draw[cube] (35)--(36);
        \draw[cube] (0)--(37);
        \draw[cube] (36)--(37);
        \draw[cube] (37)--(48);
        \draw[cube] (33)--(48);
        \draw[cube] (48)--(49);
        \draw[cube] (34)--(49);
        \draw[cube] (36)--(49);

        \draw[cube] (2)--(15);
        \draw[cube] (2)--(28);
        \draw[cube] (15)--(31);
        \draw[cube] (28)--(31);
        \draw[cube] (2)--(38);
        \draw[cube] (38)--(40);
        \draw[cube] (15)--(40);
        \draw[cube] (38)--(45);
        \draw[cube] (28)--(45);
        \draw[cube] (45)--(46);
        \draw[cube] (31)--(46);
        \draw[cube] (40)--(46);

        \draw[cube] (3)--(6);
        \draw[cube] (3)--(7);
        \draw[cube] (6)--(17);
        \draw[cube] (7)--(17);
        \draw[cube] (17)--(18);
        \draw[cube] (18)--(20);
        \draw[cube] (7)--(20);
        \draw[cube] (3)--(21);
        \draw[cube] (20)--(21);
        \draw[cube] (18)--(22);
        \draw[cube] (21)--(22);
        \draw[cube] (6)--(22);

        \draw[cube] (5)--(8);
        \draw[cube] (8)--(9);
        \draw[cube] (9)--(10);
        \draw[cube] (8)--(11);
        \draw[cube] (10)--(11);
        \draw[cube] (10)--(12);
        \draw[cube] (11)--(13);
        \draw[cube] (12)--(13);
        \draw[cube] (5)--(13);
        \draw[cube] (9)--(14);
        \draw[cube] (12)--(14);
        \draw[cube] (5)--(14);

    \draw[edge12_] (24)--(25);
    \draw[edge12_] (23)--(25);
    \draw[edge12_] (0)--(25);
    \draw[edge12_] (2)--(26);
    \draw[edge12_] (23)--(26);
    \draw[edge12_] (24)--(26);
    \draw[edge12_] (24)--(27);
    \draw[edge12_] (2)--(28);
    \draw[edge12_] (27)--(28);
    \draw[edge12_] (0)--(33);
    \draw[edge12_] (27)--(33);
    \draw[edge12_] (0)--(37);
    \draw[edge12_] (2)--(38);
    \draw[edge12_] (37)--(44);
    \draw[edge12_] (38)--(44);
    \draw[edge12_] (23)--(44);
    \draw[edge12_] (38)--(45);
    \draw[edge12_] (28)--(45);
    \draw[edge12] (37)--(48);
    \draw[edge12] (45)--(48);
    \draw[edge12] (33)--(48);
    
        \foreach \x in {0,...,49}{
            \node[vertex_normal] at (\x) {};
    
        }
    
        \foreach \x in {48, 20, 15, 5}{
            \node[vertex] at (\x) {};
        }
\end{tikzpicture}
}
\subfigure[$F_{-\alpha_4}$.]{
    \centering
    \tdplotsetmaincoords{110}{260}
\begin{tikzpicture}%
    [tdplot_main_coords,scale = 3,
    back/.style={loosely dotted, thin},
    edge/.style={color=black},
    cube/.style={color=red, thick},
    facet/.style={fill=blue!95!black,fill opacity=0.100000},
    vertex/.style={inner sep=1pt,circle,draw=green!25!black,fill=green!75!black,thick,anchor=base},
    vertex_normal/.style = {inner sep=0.5pt,circle,draw=black,fill=black,thick,anchor=base},
    edge0/.style = {color=blue!50!red, very thick,  dashed, opacity=0.7},
    edge8/.style= {color=ForestGreen, very thick, opacity=0.7},
    edge12/.style= {blue,dotted, very thick, opacity=0.7},
    face1/.style = {fill=blue!20!white, fill opacity = 0.5}]

    \coordinate (0) at (0.6, 0.6, 0.6);
    \coordinate (1) at (0.42857142857142855, -0.42857142857142855, -0.42857142857142855);
    \coordinate (2) at (-0.13333333333333333, 0.2, -0.13333333333333333);
    \coordinate (3) at (0.375, -0.25, 0.25);
    \coordinate (4) at (0.3, 0.0, 0.0);
    \coordinate (5) at (0.09090909090909091, -0.09090909090909091, 0.09090909090909091);
    \coordinate (6) at (0.2222222222222222, -0.2222222222222222, 0.2222222222222222);
    \coordinate (7) at (0.42857142857142855, -0.2857142857142857, 0.42857142857142855);
    \coordinate (8) at (0.08333333333333333, 0.0, 0.0);
    \coordinate (9) at (0.0, 0.07692307692307693, 0.0);
    \coordinate (10) at (-0.07142857142857142, 0.07142857142857142, -0.07142857142857142);
    \coordinate (11) at (0.0, 0.0, -0.07692307692307693);
    \coordinate (12) at (-0.07692307692307693, 0.0, 0.0);
    \coordinate (13) at (0.0, -0.08333333333333333, 0.0);
    \coordinate (14) at (0.0, 0.0, 0.08333333333333333);
    \coordinate (15) at (-0.13333333333333333, 0.13333333333333333, -0.13333333333333333);
    \coordinate (16) at (0.0, 0.0, 0.3);
    \coordinate (17) at (0.25, -0.25, 0.375);
    \coordinate (18) at (0.2857142857142857, -0.42857142857142855, 0.42857142857142855);
    \coordinate (19) at (0.0, -0.3, 0.0);
    \coordinate (20) at (0.5, -0.5, 0.5);
    \coordinate (21) at (0.42857142857142855, -0.42857142857142855, 0.2857142857142857);
    \coordinate (22) at (0.25, -0.375, 0.25);
    \coordinate (23) at (0.0, 0.3, 0.3);
    \coordinate (24) at (0.3, 0.3, 0.0);
    \coordinate (25) at (0.42857142857142855, 0.42857142857142855, 0.42857142857142855);
    \coordinate (26) at (0.0, 0.23076923076923078, 0.0);
    \coordinate (27) at (0.3, 0.3, -0.3);
    \coordinate (28) at (-0.13333333333333333, 0.2, -0.2);
    \coordinate (29) at (0.3, 0.0, -0.3);
    \coordinate (30) at (0.0, 0.0, -0.23076923076923078);
    \coordinate (31) at (-0.13333333333333333, 0.13333333333333333, -0.2);
    \coordinate (32) at (0.0, -0.3, -0.3);
    \coordinate (33) at (0.6, 0.6, -0.6);
    \coordinate (34) at (0.6, -0.6, -0.6);
    \coordinate (35) at (0.6, -0.6, 0.6);
    \coordinate (36) at (-0.6, -0.6, 0.6);
    \coordinate (37) at (-0.6, 0.6, 0.6);
    \coordinate (38) at (-0.2, 0.2, -0.13333333333333333);
    \coordinate (39) at (-0.23076923076923078, 0.0, 0.0);
    \coordinate (40) at (-0.2, 0.13333333333333333, -0.13333333333333333);
    \coordinate (41) at (-0.3, 0.0, 0.3);
    \coordinate (42) at (-0.42857142857142855, -0.42857142857142855, 0.42857142857142855);
    \coordinate (43) at (-0.3, -0.3, 0.0);
    \coordinate (44) at (-0.3, 0.3, 0.3);
    \coordinate (45) at (-0.2, 0.2, -0.2);
    \coordinate (46) at (-0.2, 0.13333333333333333, -0.2);
    \coordinate (47) at (-0.3, -0.3, -0.3);
    \coordinate (48) at (-0.6, 0.6, -0.6);
    \coordinate (49) at (-0.6, -0.6, -0.6);


\fill[green!10!white, opacity=0.7] (48)--(45)--(46)--(47)--(49)--(34)--(33)--cycle{};

    \draw[edge] (3)--(4);
    \draw[edge] (5)--(6);
    \draw[edge] (4)--(8);
    \draw[edge] (10)--(15);
    \draw[edge] (14)--(16);
    \draw[edge] (16)--(17);
    \draw[edge] (13)--(19);
    \draw[edge] (1)--(21);
    \draw[edge] (19)--(22);
    \draw[edge] (16)--(23);
    \draw[edge] (4)--(24);
    \draw[edge] (7)--(25);
    \draw[edge] (9)--(26);
    \draw[edge] (4)--(29);
    \draw[edge] (11)--(30);
    \draw[edge] (19)--(32);
    \draw[edge] (20)--(35);
    \draw[edge] (12)--(39);
    \draw[edge] (16)--(41);
    \draw[edge] (18)--(42);
    \draw[edge] (19)--(43);

    \draw[cube] (0)--(33);
    \draw[cube] (33)--(34);
    \draw[cube] (34)--(35);
    \draw[cube] (0)--(35);
    \draw[cube] (35)--(36);
    \draw[cube] (0)--(37);
    \draw[cube] (36)--(37);
    \draw[cube] (37)--(48);
    \draw[cube] (33)--(48);
    \draw[cube] (48)--(49);
    \draw[cube] (34)--(49);
    \draw[cube] (36)--(49);

    \draw[edge0] (36)--(37);
    \draw[edge0] (39)--(40);
    \draw[edge0] (38)--(40);
    \draw[edge0] (39)--(41);
    \draw[edge0] (41)--(42);
    \draw[edge0] (36)--(42);
    \draw[edge0] (39)--(43);
    \draw[edge0] (42)--(43);
    \draw[edge0] (37)--(44);
    \draw[edge0] (38)--(44);
    \draw[edge0] (41)--(44);
    \draw[edge0] (38)--(45);
    \draw[edge0] (45)--(46);
    \draw[edge0] (40)--(46);
    \draw[edge0] (43)--(47);
    \draw[edge0] (46)--(47);
    \draw[edge0] (37)--(48);
    \draw[edge0] (45)--(48);
    \draw[edge0] (48)--(49);
    \draw[edge0] (47)--(49);
    \draw[edge0] (36)--(49);

    \draw[edge12] (24)--(25);
    \draw[edge12] (23)--(25);
    \draw[edge12] (0)--(25);
    \draw[edge12] (2)--(26);
    \draw[edge12] (23)--(26);
    \draw[edge12] (24)--(26);
    \draw[edge12] (24)--(27);
    \draw[edge12] (2)--(28);
    \draw[edge12] (27)--(28);
    \draw[edge12] (0)--(33);
    \draw[edge12] (27)--(33);
    \draw[edge12] (0)--(37);
    \draw[edge12] (2)--(38);
    \draw[edge12] (37)--(44);
    \draw[edge12] (38)--(44);
    \draw[edge12] (23)--(44);
    \draw[edge12] (38)--(45);
    \draw[edge12] (28)--(45);
    \draw[edge12] (37)--(48);
    \draw[edge12] (45)--(48);
    \draw[edge12] (33)--(48);

        \draw[cube] (0)--(33);
        \draw[cube] (33)--(34);
        \draw[cube] (34)--(35);
        \draw[cube] (0)--(35);
        \draw[cube] (35)--(36);
        \draw[cube] (0)--(37);
        \draw[cube] (36)--(37);
        \draw[cube] (37)--(48);
        \draw[cube] (33)--(48);
        \draw[cube] (48)--(49);
        \draw[cube] (34)--(49);
        \draw[cube] (36)--(49);

        \draw[cube] (2)--(15);
        \draw[cube] (2)--(28);
        \draw[cube] (15)--(31);
        \draw[cube] (28)--(31);
        \draw[cube] (2)--(38);
        \draw[cube] (38)--(40);
        \draw[cube] (15)--(40);
        \draw[cube] (38)--(45);
        \draw[cube] (28)--(45);
        \draw[cube] (45)--(46);
        \draw[cube] (31)--(46);
        \draw[cube] (40)--(46);

        \draw[cube] (3)--(6);
        \draw[cube] (3)--(7);
        \draw[cube] (6)--(17);
        \draw[cube] (7)--(17);
        \draw[cube] (17)--(18);
        \draw[cube] (18)--(20);
        \draw[cube] (7)--(20);
        \draw[cube] (3)--(21);
        \draw[cube] (20)--(21);
        \draw[cube] (18)--(22);
        \draw[cube] (21)--(22);
        \draw[cube] (6)--(22);

        \draw[cube] (5)--(8);
        \draw[cube] (8)--(9);
        \draw[cube] (9)--(10);
        \draw[cube] (8)--(11);
        \draw[cube] (10)--(11);
        \draw[cube] (10)--(12);
        \draw[cube] (11)--(13);
        \draw[cube] (12)--(13);
        \draw[cube] (5)--(13);
        \draw[cube] (9)--(14);
        \draw[cube] (12)--(14);
        \draw[cube] (5)--(14);
    
        \draw[edge8] (27)--(28);
    \draw[edge8] (1)--(29);
    \draw[edge8] (27)--(29);
    \draw[edge8] (29)--(30);
    \draw[edge8] (28)--(31);
    \draw[edge8] (30)--(31);
    \draw[edge8] (30)--(32);
    \draw[edge8] (1)--(32);
    \draw[edge8] (27)--(33);
    \draw[edge8] (33)--(34);
    \draw[edge8] (1)--(34);
    \draw[edge8] (28)--(45);
    \draw[edge8] (45)--(46);
    \draw[edge8] (31)--(46);
    \draw[edge8] (32)--(47);
    \draw[edge8] (46)--(47);
    \draw[edge8] (45)--(48);
    \draw[edge8] (33)--(48);
    \draw[ForestGreen, thick, opacity = 0.7, dashed] (48)--(49);
    \draw[edge8] (47)--(49);
    \draw[edge8] (34)--(49);

        \foreach \x in {0,...,49}{
            \node[vertex_normal] at (\x) {};
    
        }
    
        \foreach \x in {48, 20, 15, 5}{
            \node[vertex] at (\x) {};
        }
	\end{tikzpicture}  
}
\caption{The generalized associahedron of type $\dynD_4$ and facets corresponding to some negative simple roots $-\alpha_1$, $-\alpha_3$, and $-\alpha_4$.}\label{fig_asso_D4}
\end{figure}

In the remaining part of this section, we recall~\cite{FZ4_2007} which
considers the combinatorics on mutations in a more general setting. Let
$\quiver$ be a bipartite quiver and $I_+$ and $I_-$ be the bipartite
decomposition of the vertex set of $\quiver$. Consider the composition
$\qcoxeter = \mutation_- \mutation_+$ of a sequence of mutations where
\[
    \mutation_{\varepsilon} = \prod_{i \in I_{\varepsilon}} \mutation_i \qquad \text{ for } \varepsilon \in \{ +, -\}.
\]
We call $\qcoxeter$ a \emph{Coxeter mutation} as before.
We enclose this section by recalling the following result which will be used later.
\begin{lemma}[{\cite[Theorem~8.8]{FZ4_2007}}]\label{lemma:order of coxeter mutation}
Let $\seed_{t_0} = (\bfx_{t_0}, \qbasis_{t_0})$ be an initial seed. Suppose
that the exchange matrix $\qbasis_{t_0}$ is the adjacency matrix of a
bipartite quiver $\quiver$. Then the set $\{ \qcoxeter^r (\seed_{t_0}) \}_{r \in \Z_{\geq 0}}$ 
of seeds is finite if and only if 
the Cartan counterpart $C(\qbasis_{t_0}^{\text{pr}})$ is a Cartan matrix of
finite type.

Moreover, for a quiver $\quiver$ of finite type, the order the $\qcoxeter$-action is given by $(h+2)/2$ if $h$ is even, or $h+2$ otherwise.
\end{lemma}

\section{\texorpdfstring{$N$}{N}-graphs and seeds}
Let us recall from \cite{CZ2020} how to construct a seed from an $N$-graph~$\ngraph$.
Each one-cycle in $\Legendrian(\ngraph)$ corresponds to a vertex of the quiver,
and a monodromy along that cycle gives a coordinate function at that vertex.
The quiver is obtained from the intersection data among one-cycles.
Moreover, there is an operation in $N$-graph, called \emph{Legendrian mutation}, which is a counterpart of the mutation in the cluster structure.
The Legendrian mutation is crucial in constructing and distinguishing $N$-graphs.
In turn, these will give seeds many Lagrangian fillings of Legendrian links.

\subsection{One-cycles in Legendrian weaves}\label{sec:1-cycles in Legendrian weaves}
Let $\ngraph\subset \disk^2$ be a free $N$-graph and $\Legendrian(\ngraph)$ be the induced Legendrian weave.
We express one-cycles of $\Legendrian(\ngraph)$ in terms of subgraphs of $\ngraph$.

\begin{definition}
A subgraph $\sfT\subset \ngraph$ is said to be \emph{admissible} if it satisfies the following conditions:
\begin{itemize}
\item every vertex of $\sfT$ is at most trivalent,
\item each univalent vertex in $\sfT$ is a trivalent vertex in $\ngraph$,
\item each bivalent vertex in $\sfT$ corresponding to a hexagonal point in $\ngraph$ connects two opposite edges in $\ngraph$, and 
\item each trivalent vertex in $\sfT$ corresponding to a hexagonal point in $\ngraph$ and connects three edges in the same color.
\end{itemize}

An admissible graph $\sfT\subset \ngraph$ is \emph{good} if it is a (connected) tree and only univalent vertices of $\sfT$ are trivalent vertices in $\ngraph$.
\end{definition}

For each admissible subgraph $\sfT$, we can define an oriented immersed loop $\ell(\sfT)\subset\disk^2$ is defined by paths whose local pictures look as depicted in Figure~\ref{fig:T cycle}.
Each arc $\ell_j(\sfT)\subset\ell(\sfT)$ cut by $\ngraph$ is labelled as $s_j\in\{1,\dots, N\}$, which lifts to the $s_j$-th sheet $\wavefront_{s_j}$ via $\pi_{\disk^2}:\wavefront(\ngraph)\to\disk^2$.
By concatenating the lifts, we have an oriented embedded loop $\cycle(\sfT)$ in $\Legendrian(\ngraph)$ and
a one-cycle $[\cycle]\in H_1(\Legendrian(\ngraph),\Z)$ is called a \emph{$\sfT$-cycle} if $[\cycle] = [\cycle(\sfT)]$.

\begin{figure}[ht]
\subfigure[Near a trivalent vertex of $\ngraph$
\label{figure:loop near vertex}]{\makebox[.4\textwidth]{
\begin{tikzpicture}[baseline=-.5ex]
\begin{scope}[xshift=-5cm]
\draw [dashed] (0,0) circle [radius=2];
\clip (0,0) circle [radius=2];
\draw [yellow, line cap=round, line width=5, opacity=0.5] (0,0) to (2,0);

\draw [blue, thick] (135:2) -- (0:0)--(0:2) (-135:2)--(0,0);
\draw[->] (30:2) to[out=180,in=90] node[pos=0.5, above, sloped] {$i+1$} node[pos=1, left] {$i$} (-1,0);
\draw[] (-1,0) to[out=-90,in=180] node[pos=0.5,below, sloped] {$i+1$}  (-30:2);
\draw[thick,blue,fill=blue] (0,0) circle (0.05);
\end{scope}
\begin{scope}
\draw [dashed] (0,0) circle [radius=2];
\clip (0,0) circle [radius=2];
\draw [yellow, line cap=round, line width=5, opacity=0.5] (0,0) to (2,0) (-135:2)--(0,0);

\draw [blue, thick] (135:2) -- (0:0)--(0:2) (-135:2)--(0,0);
\draw[->] (30:2) to[out=180,in=90] node[pos=0.15, below] {$i+1$}  node[pos=1, above, sloped, rotate = 180] {$i$} (-1,0);
\draw[] (-1,0) to[out=-90,in=180]  node[near end, above] {$i+1$}  (-30:2);
\draw[->] (240:2) to[out=45,in=-45] node[pos=0.2, above,sloped] {$i+1$} node[pos=1,above, sloped]  {$i$} (45:0.5);
\draw[] (45:0.5) to[out=135,in=45] node[pos=0.85, below,sloped] {$i+1$} (210:2);

\draw[thick,blue,fill=blue] (0,0) circle (0.05);
\end{scope}
\begin{scope}[xshift=5cm]
\draw [dashed] (0,0) circle [radius=2];
\clip (0,0) circle [radius=2];
\draw [yellow, line cap=round, line width=5, opacity=0.5] (0,0) to (2,0) (-135:2)--(0,0) (135:2) -- (0:0);

\draw [blue, thick] (135:2) -- (0:0)--(0:2) (-135:2)--(0,0);

\draw[->] (30:2)  to[out=180,in=90] node[pos=0.2, below,sloped] {$i+1$}  node[pos=1, above, sloped, rotate = 180] {$i$} (-1,0);
\draw[] (-1,0) to[out=-90,in=180] node[pos=0.8, above,sloped] {$i+1$}  (-30:2);
\draw[->] (250:2) to[out=45,in=-45] node[pos=0.2, above,sloped] {$i+1$} node[pos=0.9,above, sloped]  {$i$} (45:0.5);
\draw[] (45:0.5) to[out=135,in=45] node[pos=0.85, below,sloped] {$i+1$} (210:2);
\begin{scope}
\draw[->] (150:2) to[out=-45,in=-135] node[pos=0.15, above,sloped] {$i+1$}  node[sloped, pos=1, below] {$i$} (-45:0.25) ;
\draw[] (-45:0.25) to[out=45,in=-45] node[pos=0.85, below,sloped] {$i+1$} (120:2);
\end{scope}

\draw[thick,blue,fill=blue] (0,0) circle (0.05);
\end{scope}
\end{tikzpicture}}}

\subfigure[Near a hexagonal point of $\ngraph$
\label{figure:loop near hexagon II}]{\makebox[.4\textwidth]{
\begin{tikzpicture}
\begin{scope}[xshift=-5cm]
\draw [dashed] (0,0) circle [radius=2];
\clip (0,0) circle [radius=2];
\draw [yellow, line cap=round, line width=5, opacity=0.5] (-2,0) to (2,0);

\draw[red, thick] (180:2) -- (0,0) -- (60:2) (0,0)-- (-60:2);
\draw[blue, thick] (0:2) -- (0,0) -- (120:2) (0,0)-- (-120:2);
\draw[thick,black,fill=white] (0,0) circle (0.05);
\draw[->] (15:2) --  node[above, midway, pos=1]{$i+2$} 
node[above, pos=0.3] {$i+1$}
(0,{2*sin(15)});
\draw[->] ({180+15}:2) -- node[below, pos=0.3]{$i+2$} (0,{-2*sin(15)}) ;
\draw[] (0,{2*sin(15)}) -- node[above, pos=0.7] {$i+2$} ({180-15}:2);
\draw[] (0,{-2*sin(15)}) -- 
node[below, midway, pos=0]{$i+2$} 
node[below, pos=0.7] {$i+1$}
(-15:2);

\end{scope}

\begin{scope}
\draw [dashed] (0,0) circle [radius=2];
\clip (0,0) circle [radius=2];
\begin{scope}
\draw [yellow, line cap=round, line width=5, opacity=0.5] 
(0:0) -- (-30:2);
\draw [blue, thick](0:0)--(270:2);
\draw [red, thick](0,0)--(330:2);
\draw[->] (340:2) to[out=160,in=-60] (30:0.75);
\draw(270:0.75) to[out=0,in=140] (320:2);
\end{scope}
\begin{scope}[rotate=120]
\draw [yellow, line cap=round, line width=5, opacity=0.5] 
(0:0) -- (-30:2);
\draw [blue, thick](0:0)--(270:2);
\draw [red, thick](0,0)--(330:2);
\draw[->] (340:2) to[out=160,in=-60] (30:0.75);
\draw(270:0.75) to[out=0,in=140] (320:2);
\end{scope}
\begin{scope}[rotate=240]
\draw [yellow, line cap=round, line width=5, opacity=0.5] 
(0:0) -- (-30:2);
\draw [blue, thick](0:0)--(270:2);
\draw [red, thick](0,0)--(330:2);
\draw[->] (340:2) to[out=160,in=-60] (30:0.75);
\draw(270:0.75) to[out=0,in=140] (320:2);
\end{scope}

\node[rotate=0] at (300:1.4) {$i+2$};
\node[rotate=-60] at (0:1.2) {$i+2$};
\node[rotate=-60] at (60:1.2) {$i+2$};
\node[rotate=60] at (120:1.2) {$i+2$};
\node[rotate=60] at (180:1.2) {$i+2$};
\node[rotate=0] at (240:1.4) {$i+2$};

\draw[thick,black, fill=white] (0:0) circle (1.5pt);
\end{scope}
\begin{scope}[xshift=5cm]
\draw [dashed] (0,0) circle [radius=2];
\clip (0,0) circle [radius=2];
\begin{scope}
\draw [yellow, line cap=round, line width=5, opacity=0.5] 
(0:0) -- (-30:2);
\draw [red, thick](0:0)--(270:2);
\draw [blue, thick](0,0)--(330:2);
\draw (320:2) to[out=160,in=-60] (30:0.75);
\draw[->] (340:2) to[out=140,in=0] (270:0.75);
\end{scope}
\begin{scope}[rotate=120]
\draw [yellow, line cap=round, line width=5, opacity=0.5] 
(0:0) -- (-30:2);
\draw [red, thick](0:0)--(270:2);
\draw [blue, thick](0,0)--(330:2);
\draw (320:2) to[out=160,in=-60] (30:0.75);
\draw[->] (340:2) to[out=140,in=0] (270:0.75);
\end{scope}
\begin{scope}[rotate=240]
\draw [yellow, line cap=round, line width=5, opacity=0.5] 
(0:0) -- (-30:2);
\draw [red, thick](0:0)--(270:2);
\draw [blue, thick](0,0)--(330:2);
\draw (320:2) to[out=160,in=-60] (30:0.75);
\draw[->] (340:2) to[out=140,in=0] (270:0.75);
\end{scope}

\node[rotate=-60] at (0:1) {$i$};
\node[rotate=-60] at (60:1) {$i$};
\node[rotate=60] at (120:1) {$i$};
\node[rotate=60] at (180:1) {$i$};
\node[rotate=0] at (250:1.1) {$i$};
\node[rotate=0] at (290:1.1) {$i$};

\node[rotate=-30] at (350:1.6) {$i+1$};
\node[rotate=-30] at (310:1.6) {$i+1$};
\node[rotate=30] at (230:1.6) {$i+1$};
\node[rotate=30] at (190:1.6) {$i+1$};
\node[rotate=0] at (65:1.6) {$i+1$};
\node[rotate=0] at (115:1.6) {$i+1$};

\draw[thick,black, fill=white] (0:0) circle (1.5pt);
\end{scope}
\end{tikzpicture}}}
\caption{Local configurations on cycles and corresponding arcs of $\ngraph\subset\disk^2$
}
\label{fig:T cycle}
\end{figure}

\begin{example}[(Long) $\sfI$-cycles]
For an edge $e$ of $\ngraph$ connecting two trivalent vertices, let $\sfI(e)$ be the subgraph of $\ngraph$ consisting of a single edge $e$.
Then $\sfI(e)$ is a good subgraph of $\ngraph$ and the cycle $[\cycle(\sfI(e))]$ depicted in  Figure~\ref{figure:I-cycle} is called an \emph{$\sfI$-cycle}.

In general, a linear chain of edges $(e_1,e_2,\dots, e_n)$ satisfying
\begin{itemize}
\item $e_i$ connects a trivalent vertex and a hexagonal point for $i=1,n$;
\item $e_i$ and $e_{i+1}$ meet at a hexagonal point in the opposite way, see Figure~\ref{figure:long I-cycle}, for $i=2,\dots, n-1$
\end{itemize}
forms a good subgraph $\sfI(e_1,\dots, e_n)$, and the cycle $[\cycle(\sfI(e_1,\dots, e_n))]$ is called a \emph{long $\sfI$-cycle}. See Figure~\ref{figure:long I-cycle}.
\end{example}

\begin{example}[$\sfY$-cycles]
Let $e_1,e_2,e_3$ be monochromatic edges joining a hexagonal point $h$ and trivalent vertices $v_i$ for $i=1,2,3$.
Then the subgraph $\sfY(e_1,e_2,e_3)$ consisting of three edges $e_1, e_2$ and $e_3$ is a good subgraph of $\ngraph$ and it defines a cycle $[\cycle(\sfY(e_1,e_2,e_3))]$ called an \emph{upper} or \emph{lower} \emph{$\sfY$-cycle} according to the relative position of sheets that edges represent.
See Figures~\ref{figure:Y-cycle_1} and \ref{figure:Y-cycle_2}.
\end{example}

\begin{figure}[ht]
\subfigure[An $\sfI$-cycle $\cycle(\sfI(e))$\label{figure:I-cycle}]{\makebox[.4\textwidth]{
\begin{tikzpicture}
\draw [dashed] (0,0) circle [radius=1.5];

\draw [yellow, line cap=round, line width=5, opacity=0.5] (-1/2,0) to (1/2,0);
\draw [blue, thick] ({-3*sqrt(3)/4},3/4)--(-1/2,0);
\draw [blue, thick] ({-3*sqrt(3)/4},-3/4)--(-1/2,0);
\draw [blue, thick] ({3*sqrt(3)/4},3/4)--(1/2,0);
\draw [blue, thick] ({3*sqrt(3)/4},-3/4)--(1/2,0);
\draw [blue, thick] (-1/2,0)--(1/2,0) node[above, midway] {$e$};

\draw[thick,blue,fill=blue] (-1/2,0) circle (0.05);
\draw[thick,blue,fill=blue] (1/2,0) circle (0.05);

\draw[->] (1,0) node[right] {$i$} to[out=90,in=0] (0,0.5) node[above] {$i+1$} to[out=180,in=90] (-1,0) node[left] {$i$} to[out=-90,in=180] (0,-0.5)node[below] {$i+1$} to[out=0,in=-90] (1,0);

\end{tikzpicture}
}}
\subfigure[A long $\sfI$-cycle $\cycle(\sfI(e_1,e_2))$\label{figure:long I-cycle}]{\makebox[.4\textwidth]{
\begin{tikzpicture}

\draw[dashed] \boundellipse{0,0}{3}{1.5};

\draw [yellow, line cap=round, line width=5, opacity=0.5] (-1.5,0) to (1.5,0);

\draw[red, thick] (0,0)--(1.35,1.35);
\draw[red, thick] (0,0)--(1.35,-1.35);
\draw[red, thick] (0,0)--(-1.5,0) node[above, midway] {$e_1$};
\draw[red, thick] (-1.5,0)--(-1.5-0.9,0.9);
\draw[red, thick] (-1.5,0)--(-1.5-0.9,-0.9);

\draw[blue, thick] (0,0)--(-1.35,1.35);
\draw[blue, thick] (0,0)--(-1.35,-1.35);
\draw[blue, thick] (0,0)--(1.5,0) node[above, midway] {$e_2$};
\draw[blue, thick] (1.5,0)--(1.5+0.9,0.9);
\draw[blue, thick] (0,0)--(1.5,0)--(1.5+0.9,-0.9);

\draw[thick,red,fill=red] (-1.5,0) circle (0.05);
\draw[thick,blue,fill=blue] (1.5,0) circle (0.05);
\draw[thick,black,fill=white] (0,0) circle (0.05);

\draw[->] (2,0) node[right]{$i$} to[out=90,in=0] (0,0.5) node[above]{$i+2$} to[out=180,in=90] (-2,0) node[left]{$i+1$} to[out=-90,in=180] (0,-0.5) node[below]{$i+2$} to[out=0,in=-90] (2,0);
\node[above] at (1.5,0.5) {$i+1$};
\node[below] at (1.5,-0.5) {$i+1$};
\node[above] at (-1.5,0.5) {$i+2$};
\node[below] at (-1.5,-0.5) {$i+2$};

\end{tikzpicture}}}

\subfigure[An upper $\sfY$-cycle $\cycle(\sfY(e_1,e_2,e_3))$\label{figure:Y-cycle_1}]{\makebox[.4\textwidth]{
\begin{tikzpicture}
\draw [dashed] (0,0) circle [radius=2];

\begin{scope}
\draw [yellow, line cap=round, line width=5, opacity=0.5] 
(0:0) -- (-30:1);
\draw [blue, thick](0:0)--(270:2);
\draw [red, thick](0,0)--(330:1) (310:2)--(330:1)--(350:2);
\draw[->] (270:0.5) to[out=0,in=-120] (330:1.5);
\draw(330:1.5) to[out=60,in=-60] (30:0.5);
\node[rotate=60] at (330:1.75) {$\scriptstyle i+1$};
\node[rotate=-30] at (290:1) {$\scriptstyle i+2$};
\node[rotate=-30] at (0:1) {$\scriptstyle i+2$};
\end{scope}
\begin{scope}[rotate=120]
\draw [yellow, line cap=round, line width=5, opacity=0.5] 
(0:0) -- (-30:1);
\draw [blue, thick](0:0)--(270:2);
\draw [red, thick](0,0)--(330:1) (310:2)--(330:1)--(350:2);
\draw[->] (270:0.5) to[out=0,in=-120] (330:1.5);
\draw(330:1.5) to[out=60,in=-60] (30:0.5);
\end{scope}
\begin{scope}[rotate=240]
\draw [yellow, line cap=round, line width=5, opacity=0.5] 
(0:0) -- (-30:1);
\draw [blue, thick](0:0)--(270:2);
\draw [red, thick](0,0)--(330:1) (310:2)--(330:1)--(350:2);
\draw[->] (270:0.5) to[out=0,in=-120] (330:1.5);
\draw(330:1.5) to[out=60,in=-60] (30:0.5);
\end{scope}
\draw[thick, red, fill] 
(90:1) circle (1.5pt)
(210:1) circle (1.5pt)
(330:1) circle (1.5pt);

\node at (90:1.75) {$\scriptstyle i+1$};
\node[rotate=-90] at (60:1) {$\scriptstyle i+2$};
\node[rotate=90] at (120:1) {$\scriptstyle i+2$};
\node[rotate=-60] at (210:1.75) {$\scriptstyle i+1$};
\node[rotate=30] at (180:1) {$\scriptstyle i+2$};
\node[rotate=30] at (250:1) {$\scriptstyle i+2$};

\draw[thick,black, fill=white] (0:0) circle (1.5pt);
\end{tikzpicture}}}
\subfigure[A lower $\sfY$-cycle $\cycle(\sfY(e_1,e_2,e_3))$\label{figure:Y-cycle_2}]{\makebox[.4\textwidth]{
\begin{tikzpicture}
\draw [dashed] (0,0) circle [radius=2];

\begin{scope}
\draw [yellow, line cap=round, line width=5, opacity=0.5] 
(0:0) -- (-30:1);
\draw [red, thick](0:0)--(270:2);
\draw [blue, thick](0,0)--(330:1) (310:2)--(330:1)--(350:2);
\draw (270:0.3) to[out=0,in=60] (330:1.5);
\draw[->] (30:0.3) to[out=-60,in=-120] (330:1.5)  ;
\end{scope}
\begin{scope}[rotate=120]
\draw [yellow, line cap=round, line width=5, opacity=0.5] 
(0:0) -- (-30:1);
\draw [red, thick](0:0)--(270:2);
\draw [blue, thick](0,0)--(330:1) (310:2)--(330:1)--(350:2);
\draw (270:0.3) to[out=0,in=60] (330:1.5);
\draw[->] (30:0.3) to[out=-60,in=-120] (330:1.5)  ;
\end{scope}
\begin{scope}[rotate=240]
\draw [yellow, line cap=round, line width=5, opacity=0.5] 
(0:0) -- (-30:1);
\draw [red, thick](0:0)--(270:2);
\draw [blue, thick](0,0)--(330:1) (310:2)--(330:1)--(350:2);
\draw (270:0.3) to[out=0,in=60] (330:1.5);
\draw[->] (30:0.3) to[out=-60,in=-120] (330:1.5)  ;
\end{scope}
\draw[thick, blue, fill] 
(90:1) circle (1.5pt)
(210:1) circle (1.5pt)
(330:1) circle (1.5pt);

\node[rotate=60] at (330:1.75) {\small$i$};
\node at (90:1.75) {\small$i$};
\node[rotate=-60] at (210:1.75) {\small$i$};

\node[right] at (90:1) {\scriptsize$i+1$};
\node[left] at (90:1) {\scriptsize$i+1$};
\node[above right]  at (330:1) {\scriptsize$i+1$};
\node[below left]  at (330:1) {\scriptsize$i+1$};
\node[above left]  at (210:1) {\scriptsize$i+1$};
\node[below right]  at (210:1) {\scriptsize$i+1$};
\node[above]  at (30:0.3) {\scriptsize$i$};
\node[right]  at (30:0.3) {\scriptsize$i$};
\node[above]  at (150:0.3) {\scriptsize$i$};
\node[left]  at (150:0.3) {\scriptsize$i$};
\node[below left]  at (270:0.2) {\scriptsize$i$};
\node[below right]  at (270:0.2) {\scriptsize$i$};

\draw[thick,black, fill=white] (0:0) circle (1.5pt);
\end{tikzpicture}}}
\caption{(Long) $\sfI$- and $\sfY$-cycles}
\label{fig:I and Y cycle}
\end{figure}

One of the benifit of cycles from admissible subgraphs is that one can keep track how cycles are changed under the $N$-graph moves described in Figure~\ref{fig:move1-6}, especially under Move~\Move{I} and Move~\Move{II}.
Note that Move~\Move{III} can be decomposed into a sequence of Move~\Move{I} and Move~\Move{II}.
Some of such changes are given in Figure~\ref{fig:cycles under moves}.
Then it is easy to check that any $\sfT$-cycle coming from a good subgraph $\sfT$ can be transformed to an $\sfI$-cycle. 

\begin{figure}[ht]
\[
\begin{tikzcd}[row sep=0pc]
\begin{tikzpicture}[baseline=-.5ex]
\draw [dashed] (0,0) circle [radius=1];
\clip (0,0) circle (1);
\draw [yellow, line cap=round, line width=5, opacity=0.5] (-1,0) to (1,0);
\draw [blue, thick] ({-sqrt(3)/2},1/2)--(-1/2,0);
\draw [blue, thick] ({-sqrt(3)/2},-1/2)--(-1/2,0);
\draw [blue, thick] ({sqrt(3)/2},1/2)--(1/2,0);
\draw [blue, thick] ({sqrt(3)/2},-1/2)--(1/2,0);
\draw [blue, thick] (-1/2,0)--(1/2,0);
\draw [red, thick] (-1,0)--(-1/2,0) to[out=60,in=180] (0,1/2) to[out=0,in=120] (1/2,0)--(1,0);
\draw [red, thick] (-1/2,0) to[out=-60,in=180] (0, -1/2) to[out=0, in=-120] (1/2,0); 
\draw[thick,black,fill=white] (-1/2,0) circle (0.05);
\draw[thick,black,fill=white] (1/2,0) circle (0.05);
\end{tikzpicture}
\arrow[leftrightarrow,r,"\Move{I}"]&
\begin{tikzpicture}[baseline=-.5ex]
\draw [dashed] (3,0) circle [radius=1];
\clip (3,0) circle (1);
\draw [yellow, line cap=round, line width=5, opacity=0.5] (2,0) to (4,0);
\draw [blue, thick] ({3-sqrt(3)/2},1/2)--({3+sqrt(3)/2},1/2);
\draw [blue, thick] ({3-sqrt(3)/2},-1/2)--({3+sqrt(3)/2},-1/2);
\draw [red, thick] (2,0)--(4,0);
\end{tikzpicture}
&
\begin{tikzpicture}[baseline=-.5ex]
\draw [dashed] (0,0) circle [radius=1];
\clip (0,0) circle (1);
\draw [yellow, line cap=round, line width=5, opacity=0.5] ({-sqrt(3)/2},1/2)--(-1/2,0) -- (1/2,0) --({sqrt(3)/2},1/2);
\begin{scope}[yscale=-1]
\draw [yellow, line cap=round, line width=5, opacity=0.5] ({-sqrt(3)/2},1/2)--(-1/2,0) -- (1/2,0) --({sqrt(3)/2},1/2);
\end{scope}
\draw [blue, thick] ({-sqrt(3)/2},1/2)--(-1/2,0);
\draw [blue, thick] ({-sqrt(3)/2},-1/2)--(-1/2,0);
\draw [blue, thick] ({sqrt(3)/2},1/2)--(1/2,0);
\draw [blue, thick] ({sqrt(3)/2},-1/2)--(1/2,0);
\draw [blue, thick] (-1/2,0)--(1/2,0);
\draw [red, thick] (-1,0)--(-1/2,0) to[out=60,in=180] (0,1/2) to[out=0,in=120] (1/2,0)--(1,0);
\draw [red, thick] (-1/2,0) to[out=-60,in=180] (0, -1/2) to[out=0, in=-120] (1/2,0); 
\draw[thick,black,fill=white] (-1/2,0) circle (0.05);
\draw[thick,black,fill=white] (1/2,0) circle (0.05);
\end{tikzpicture}
\arrow[leftrightarrow,r,"\Move{I}"]&
\begin{tikzpicture}[baseline=-.5ex]
\draw [dashed] (3,0) circle [radius=1];
\clip (3,0) circle (1);
\draw [yellow, line cap=round, line width=5, opacity=0.5] (2,1/2) to (4,1/2);
\draw [yellow, line cap=round, line width=5, opacity=0.5] (2,-1/2) to (4,-1/2);
\draw [blue, thick] ({3-sqrt(3)/2},1/2)--({3+sqrt(3)/2},1/2);
\draw [blue, thick] ({3-sqrt(3)/2},-1/2)--({3+sqrt(3)/2},-1/2);
\draw [red, thick] (2,0)--(4,0);
\end{tikzpicture}
\\
\begin{tikzpicture}[baseline=-.5ex]
\draw [dashed] (0,0) circle [radius=1];
\clip (0,0) circle (1);
\draw [yellow, line cap=round, line width=5, opacity=0.5] ({-sqrt(3)/2},1/2)--(-1/2,0);
\draw [yellow, line cap=round, line width=5, opacity=0.5] (-1/2,0) to[out=-60,in=180] (0, -1/2) to[out=0, in=-120] (1/2,0); 
\draw [yellow, line cap=round, line width=5, opacity=0.5] ({sqrt(3)/2},1/2)--(1/2,0);
\draw [blue, thick] ({-sqrt(3)/2},1/2)--(-1/2,0);
\draw [blue, thick] ({-sqrt(3)/2},-1/2)--(-1/2,0);
\draw [blue, thick] ({sqrt(3)/2},1/2)--(1/2,0);
\draw [blue, thick] ({sqrt(3)/2},-1/2)--(1/2,0);
\draw [blue, thick] (-1/2,0)--(1/2,0);
\draw [red, thick] (-1,0)--(-1/2,0) to[out=60,in=180] (0,1/2) to[out=0,in=120] (1/2,0)--(1,0);
\draw [red, thick] (-1/2,0) to[out=-60,in=180] (0, -1/2) to[out=0, in=-120] (1/2,0); 
\draw[thick,black,fill=white] (-1/2,0) circle (0.05);
\draw[thick,black,fill=white] (1/2,0) circle (0.05);
\end{tikzpicture}
\arrow[leftrightarrow,r,"\Move{I}"]&
\begin{tikzpicture}[baseline=-.5ex]
\draw [dashed] (3,0) circle [radius=1];
\clip (3,0) circle (1);
\draw [yellow, line cap=round, line width=5, opacity=0.5] ({3-sqrt(3)/2},1/2)--({3+sqrt(3)/2},1/2);
\draw [blue, thick] ({3-sqrt(3)/2},1/2)--({3+sqrt(3)/2},1/2);
\draw [blue, thick] ({3-sqrt(3)/2},-1/2)--({3+sqrt(3)/2},-1/2);
\draw [red, thick] (2,0)--(4,0);
\end{tikzpicture}
&
\begin{tikzpicture}[baseline=-.5ex]
\draw [dashed] (0,0) circle [radius=1];
\clip (0,0) circle (1);
\draw [yellow, line cap=round, line width=5, opacity=0.5] (-1/2,0) to (1,0);
\draw [blue, thick] ({-sqrt(3)/2},1/2)--(-1/2,0);
\draw [blue, thick] ({-sqrt(3)/2},-1/2)--(-1/2,0);
\draw [blue, thick] ({sqrt(3)/2},1/2)--(1/2,0);
\draw [blue, thick] ({sqrt(3)/2},-1/2)--(1/2,0);
\draw [blue, thick] (-1/2,0)--(1/2,0);
\draw [red, thick] (-1/2,{sqrt(3)/2}) -- (1/2,0)--(1,0);
\draw [red, thick] (-1/2,{-sqrt(3)/2}) -- (1/2,0);
\draw[thick,blue,fill=blue] (-1/2,0) circle (0.05);
\draw[thick,black,fill=white] (1/2,0) circle (0.05);
\end{tikzpicture}
\arrow[leftrightarrow,r,"\Move{II}"]&
\begin{tikzpicture}[baseline=-.5ex]
\draw [dashed] (3,0) circle [radius=1];
\clip (3,0) circle (1);
\draw [yellow, line cap=round, line width=5, opacity=0.5] (3.5,0) to (4,0);
\draw [blue, thick] ({3-sqrt(3)/2},1/2)--({3+sqrt(3)/2},1/2);
\draw [blue, thick] ({3-sqrt(3)/2},-1/2)--({3+sqrt(3)/2},-1/2);
\draw [blue, thick] (3,1/2)--(3,-1/2);
\draw [red, thick] (5/2,{sqrt(3)/2})--(3,1/2) to[out=-150,in=150] (3,-1/2)--(5/2,{-sqrt(3)/2});
\draw [red, thick] (3,1/2)--(7/2,0) -- (4,0);
\draw [red, thick] (3,-1/2)--(7/2,0);

\draw[thick,black,fill=white] (3,1/2) circle (0.05);
\draw[thick,black,fill=white] (3,-1/2) circle (0.05);
\draw[thick,red,fill=red] (7/2,0) circle (0.05);
\end{tikzpicture}
\\
\begin{tikzpicture}[baseline=-.5ex]
\draw [dashed] (0,0) circle [radius=1];
\clip (0,0) circle (1);
\draw [yellow, line cap=round, line width=5, opacity=0.5] (-1/2,{sqrt(3)/2}) -- (1/2,0)--(1,0);
\draw [yellow, line cap=round, line width=5, opacity=0.5] (-1/2,{-sqrt(3)/2}) -- (1/2,0);
\draw [blue, thick] ({-sqrt(3)/2},1/2)--(-1/2,0);
\draw [blue, thick] ({-sqrt(3)/2},-1/2)--(-1/2,0);
\draw [blue, thick] ({sqrt(3)/2},1/2)--(1/2,0);
\draw [blue, thick] ({sqrt(3)/2},-1/2)--(1/2,0);
\draw [blue, thick] (-1/2,0)--(1/2,0);

\draw [red, thick] (-1/2,{sqrt(3)/2}) -- (1/2,0)--(1,0);
\draw [red, thick] (-1/2,{-sqrt(3)/2}) -- (1/2,0);

\draw[thick,blue,fill=blue] (-1/2,0) circle (0.05);
\draw[thick,black,fill=white] (1/2,0) circle (0.05);
\end{tikzpicture}
\arrow[leftrightarrow,r,"\Move{II}"]&
\begin{tikzpicture}[baseline=-.5ex]
\draw [dashed] (3,0) circle [radius=1];
\clip (3,0) circle (1);
\draw [yellow, line cap=round, line width=5, opacity=0.5] (5/2,{sqrt(3)/2})--(3,1/2)--(3,1/2)--(3,-1/2) -- (3,-1/2)--(5/2,{-sqrt(3)/2});
\draw [yellow, line cap=round, line width=5, opacity=0.5] (3.5,0)--(4,0);
\draw [blue, thick] ({3-sqrt(3)/2},1/2)--({3+sqrt(3)/2},1/2);
\draw [blue, thick] ({3-sqrt(3)/2},-1/2)--({3+sqrt(3)/2},-1/2);
\draw [blue, thick] (3,1/2)--(3,-1/2);

\draw [red, thick] (5/2,{sqrt(3)/2})--(3,1/2) to[out=-150,in=150] (3,-1/2)--(5/2,{-sqrt(3)/2});
\draw [red, thick] (3,1/2)--(7/2,0) -- (4,0);
\draw [red, thick] (3,-1/2)--(7/2,0);

\draw[thick,black,fill=white] (3,1/2) circle (0.05);
\draw[thick,black,fill=white] (3,-1/2) circle (0.05);
\draw[thick,red,fill=red] (7/2,0) circle (0.05);
\end{tikzpicture}
&
\begin{tikzpicture}[baseline=-.5ex]
\draw [dashed] (0,0) circle [radius=1];
\clip (0,0) circle (1);
\draw [yellow, line cap=round, line width=5, opacity=0.5] (-1/2,{sqrt(3)/2}) -- (1/2,0)--({sqrt(3)/2},-1/2);
\draw [blue, thick] ({-sqrt(3)/2},1/2)--(-1/2,0);
\draw [blue, thick] ({-sqrt(3)/2},-1/2)--(-1/2,0);
\draw [blue, thick] ({sqrt(3)/2},1/2)--(1/2,0);
\draw [blue, thick] ({sqrt(3)/2},-1/2)--(1/2,0);
\draw [blue, thick] (-1/2,0)--(1/2,0);

\draw [red, thick] (-1/2,{sqrt(3)/2}) -- (1/2,0)--(1,0);
\draw [red, thick] (-1/2,{-sqrt(3)/2}) -- (1/2,0);

\draw[thick,blue,fill=blue] (-1/2,0) circle (0.05);
\draw[thick,black,fill=white] (1/2,0) circle (0.05);
\end{tikzpicture}
\arrow[leftrightarrow,r,"\Move{II}"]&
\begin{tikzpicture}[baseline=-.5ex]
\draw [dashed] (3,0) circle [radius=1];
\clip (3,0) circle (1);
\draw [yellow, line cap=round, line width=5, opacity=0.5] (5/2,{sqrt(3)/2})--(3,1/2) to[out=-150,in=150] (3,-1/2)--({3+sqrt(3)/2},-1/2);
\draw [yellow, line cap=round, line width=5, opacity=0.5] (3.5,0)--(3,1/2);
\draw [blue, thick] ({3-sqrt(3)/2},1/2)--({3+sqrt(3)/2},1/2);
\draw [blue, thick] ({3-sqrt(3)/2},-1/2)--({3+sqrt(3)/2},-1/2);
\draw [blue, thick] (3,1/2)--(3,-1/2);

\draw [red, thick] (5/2,{sqrt(3)/2})--(3,1/2) to[out=-150,in=150] (3,-1/2)--(5/2,{-sqrt(3)/2});
\draw [red, thick] (3,1/2)--(7/2,0) -- (4,0);
\draw [red, thick] (3,-1/2)--(7/2,0);

\draw[thick,black,fill=white] (3,1/2) circle (0.05);
\draw[thick,black,fill=white] (3,-1/2) circle (0.05);
\draw[thick,red,fill=red] (7/2,0) circle (0.05);
\end{tikzpicture}
\\
\begin{tikzpicture}[baseline=-.5ex]
\draw [dashed] (0,0) circle [radius=1];
\clip (0,0) circle (1);
\draw [yellow, line cap=round, line width=5, opacity=0.5] ({-sqrt(3)/2},1/2)--(-1/2,0);
\draw [blue, thick] ({-sqrt(3)/2},1/2)--(-1/2,0);
\draw [blue, thick] ({-sqrt(3)/2},-1/2)--(-1/2,0);
\draw [blue, thick] ({sqrt(3)/2},1/2)--(1/2,0);
\draw [blue, thick] ({sqrt(3)/2},-1/2)--(1/2,0);
\draw [blue, thick] (-1/2,0)--(1/2,0);

\draw [red, thick] (-1/2,{sqrt(3)/2}) -- (1/2,0)--(1,0);
\draw [red, thick] (-1/2,{-sqrt(3)/2}) -- (1/2,0);

\draw[thick,blue,fill=blue] (-1/2,0) circle (0.05);
\draw[thick,black,fill=white] (1/2,0) circle (0.05);
\end{tikzpicture}
\arrow[leftrightarrow,r,"\Move{II}"]&
\begin{tikzpicture}[baseline=-.5ex]
\draw [dashed] (3,0) circle [radius=1];
\clip (3,0) circle (1);
\draw [yellow, line cap=round, line width=5, opacity=0.5] ({3-sqrt(3)/2},1/2)--(3,1/2)--(3.5,0);
\draw [blue, thick] ({3-sqrt(3)/2},1/2)--({3+sqrt(3)/2},1/2);
\draw [blue, thick] ({3-sqrt(3)/2},-1/2)--({3+sqrt(3)/2},-1/2);
\draw [blue, thick] (3,1/2)--(3,-1/2);

\draw [red, thick] (5/2,{sqrt(3)/2})--(3,1/2) to[out=-150,in=150] (3,-1/2)--(5/2,{-sqrt(3)/2});
\draw [red, thick] (3,1/2)--(7/2,0) -- (4,0);
\draw [red, thick] (3,-1/2)--(7/2,0);

\draw[thick,black,fill=white] (3,1/2) circle (0.05);
\draw[thick,black,fill=white] (3,-1/2) circle (0.05);
\draw[thick,red,fill=red] (7/2,0) circle (0.05);
\end{tikzpicture}
&
\begin{tikzpicture}[baseline=-.5ex]
\draw [dashed] (0,0) circle [radius=1];
\clip (0,0) circle (1);
\draw [yellow, line cap=round, line width=5, opacity=0.5] ({sqrt(3)/2},1/2) -- (1/2,0)--({sqrt(3)/2},-1/2);
\draw [yellow, line cap=round, line width=5, opacity=0.5] (-1/2,0) -- (1/2,0);
\draw [blue, thick] ({-sqrt(3)/2},1/2)--(-1/2,0);
\draw [blue, thick] ({-sqrt(3)/2},-1/2)--(-1/2,0);
\draw [blue, thick] ({sqrt(3)/2},1/2)--(1/2,0);
\draw [blue, thick] ({sqrt(3)/2},-1/2)--(1/2,0);
\draw [blue, thick] (-1/2,0)--(1/2,0);

\draw [red, thick] (-1/2,{sqrt(3)/2}) -- (1/2,0)--(1,0);
\draw [red, thick] (-1/2,{-sqrt(3)/2}) -- (1/2,0);

\draw[thick,blue,fill=blue] (-1/2,0) circle (0.05);
\draw[thick,black,fill=white] (1/2,0) circle (0.05);
\end{tikzpicture}
\arrow[leftrightarrow,r,"\Move{II}"]&
\begin{tikzpicture}[baseline=-.5ex]
\draw [dashed] (3,0) circle [radius=1];
\clip (3,0) circle (1);
\draw [yellow, line cap=round, line width=5, opacity=0.5] ({3+sqrt(3)/2},1/2)--(3,1/2) to[out=-150,in=150] (3,-1/2)--({3+sqrt(3)/2},-1/2);
\draw [blue, thick] ({3-sqrt(3)/2},1/2)--({3+sqrt(3)/2},1/2);
\draw [blue, thick] ({3-sqrt(3)/2},-1/2)--({3+sqrt(3)/2},-1/2);
\draw [blue, thick] (3,1/2)--(3,-1/2);

\draw [red, thick] (5/2,{sqrt(3)/2})--(3,1/2) to[out=-150,in=150] (3,-1/2)--(5/2,{-sqrt(3)/2});
\draw [red, thick] (3,1/2)--(7/2,0) -- (4,0);
\draw [red, thick] (3,-1/2)--(7/2,0);

\draw[thick,black,fill=white] (3,1/2) circle (0.05);
\draw[thick,black,fill=white] (3,-1/2) circle (0.05);
\draw[thick,red,fill=red] (7/2,0) circle (0.05);
\end{tikzpicture}
\end{tikzcd}
\]
\caption{Cycles under Move~\Move{I} and \Move{II}.}
\label{fig:cycles under moves}
\end{figure}
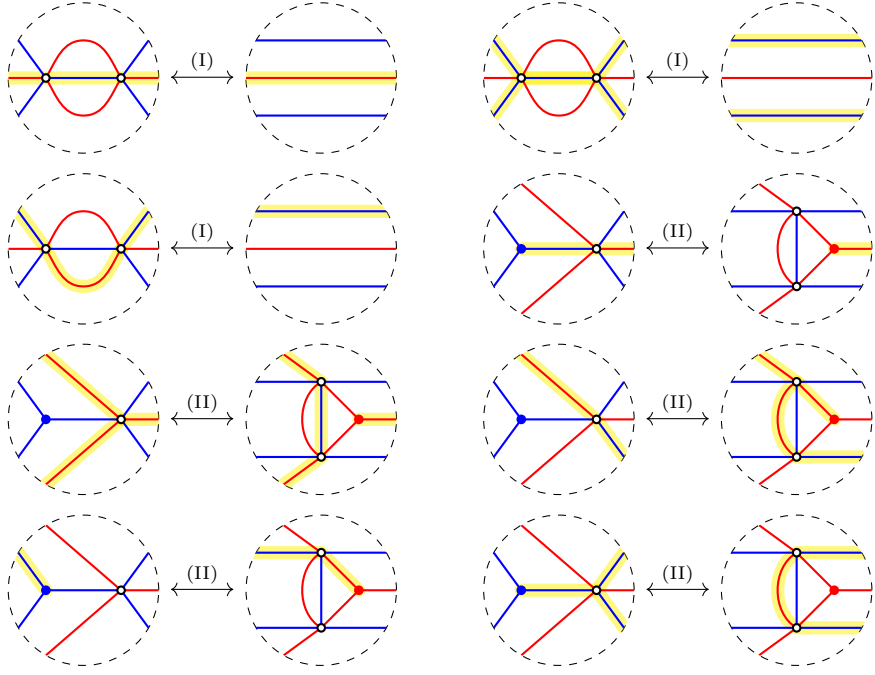

\begin{remark}
It is important to note that not every cycle can be represented by a subgraph. For example, the cycle on the left of the following picture can not be expressed by a subtree but it can be after Move~\Move{I}.
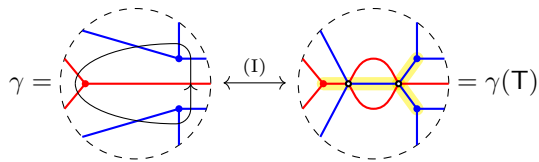
\[
\begin{tikzcd}
\cycle=\begin{tikzpicture}[baseline=-.5ex,scale=0.667]
\draw [dashed] (0,0) circle [radius=1.5];

\draw [blue, thick] ({-1.5*sqrt(2)/2},{1.5*sqrt(2)/2})--({sqrt(3)/2},1/2);
\draw [blue, thick] ({-1.5*sqrt(2)/2},{-1.5*sqrt(2)/2})--({sqrt(3)/2},-1/2);

\draw [blue, thick] ({sqrt(3)/2},1/2)--({sqrt(2)},1/2);
\draw [blue, thick] ({sqrt(3)/2},-1/2)--({sqrt(2)},-1/2);
\draw [blue, thick] ({sqrt(3)/2},1/2)--({sqrt(3)/2},{sqrt(6)/2});
\draw [blue, thick] ({sqrt(3)/2},-1/2)--({sqrt(3)/2},{-sqrt(6)/2});

\draw [red,thick] ({-1.5*2*sqrt(2)/3},{1.5*1/3}) to (-1,0);
\draw [red,thick] ({-1.5*2*sqrt(2)/3},{-1.5*1/3}) to (-1,0);
\draw [red, thick] (-1,0)--(1.5,0);

\draw[thick,red,fill=red] (-1,0) circle (0.05);
\draw[thick,blue,fill=blue] ({sqrt(3)/2},1/2) circle (0.05);
\draw[thick,blue,fill=blue] ({sqrt(3)/2},-1/2) circle (0.05);

\draw[->] (-1.2,0) arc (180:270:2 and 0.8) arc (-90:0:0.3) -- ++(0,0.5);
\draw (-1.2,0) arc (180:90:2 and 0.8) arc (90:0:0.3) -- ++(0,-0.5);
\end{tikzpicture}
\arrow[leftrightarrow, r, "\Move{I}"]&
\begin{tikzpicture}[baseline=-.5ex,scale=0.667]
\begin{scope}
\draw [dashed] (0,0) circle [radius=1.5];
\draw [yellow, line cap=round, line width=5, opacity=0.5] (-1,0) to (0.5,0);
\draw [yellow, line cap=round, line width=5, opacity=0.5] ({sqrt(3)/2},1/2)--(1/2,0);
\draw [yellow, line cap=round, line width=5, opacity=0.5] ({sqrt(3)/2},-1/2)--(1/2,0);

\draw [blue, thick] ({-1.5*sqrt(2)/2},{1.5*sqrt(2)/2})--(-1/2,0);
\draw [blue, thick] ({-1.5*sqrt(2)/2},{-1.5*sqrt(2)/2})--(-1/2,0);
\draw [blue, thick] ({sqrt(3)/2},1/2)--(1/2,0);
\draw [blue, thick] ({sqrt(3)/2},-1/2)--(1/2,0);
\draw [blue, thick] ({sqrt(3)/2},1/2)--({sqrt(2)},1/2);
\draw [blue, thick] ({sqrt(3)/2},-1/2)--({sqrt(2)},-1/2);
\draw [blue, thick] ({sqrt(3)/2},1/2)--({sqrt(3)/2},{sqrt(6)/2});
\draw [blue, thick] ({sqrt(3)/2},-1/2)--({sqrt(3)/2},{-sqrt(6)/2});
\draw [blue, thick] (-1/2,0)--(1/2,0);

\draw [red,thick] ({-1.5*2*sqrt(2)/3},{1.5*1/3}) to (-1,0);
\draw [red,thick] ({-1.5*2*sqrt(2)/3},{-1.5*1/3}) to (-1,0);
\draw [red, thick] (-1,0)--(-1/2,0) to[out=60,in=180] (0,1/2) to[out=0,in=120] (1/2,0)--(1.5,0);
\draw [red, thick] (-1/2,0) to[out=-60,in=180] (0, -1/2) to[out=0, in=-120] (1/2,0); 

\draw[thick,black,fill=white] (-1/2,0) circle (0.05);
\draw[thick,black,fill=white] (1/2,0) circle (0.05);
\draw[thick,red,fill=red] (-1,0) circle (0.05);
\draw[thick,blue,fill=blue] ({sqrt(3)/2},1/2) circle (0.05);
\draw[thick,blue,fill=blue] ({sqrt(3)/2},-1/2) circle (0.05);
\end{scope}
\end{tikzpicture}=\cycle(\sfT)
\end{tikzcd}
\]

On the other hand, there might be a one-cycle having two different subgraph presentations as follows:
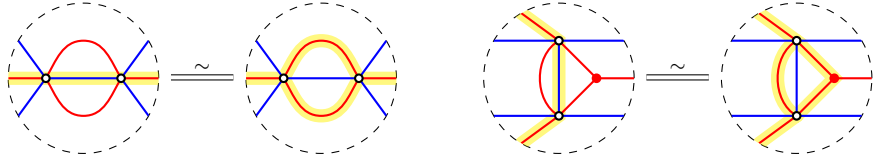
\[
\begin{tikzcd}
\begin{tikzpicture}[baseline=-.5ex]
\draw [dashed] (0,0) circle [radius=1];
\clip (0,0) circle (1);
\draw [yellow, line cap=round, line width=5, opacity=0.5] (-1,0) to (1,0);
\draw [blue, thick] ({-sqrt(3)/2},1/2)--(-1/2,0);
\draw [blue, thick] ({-sqrt(3)/2},-1/2)--(-1/2,0);
\draw [blue, thick] ({sqrt(3)/2},1/2)--(1/2,0);
\draw [blue, thick] ({sqrt(3)/2},-1/2)--(1/2,0);
\draw [blue, thick] (-1/2,0)--(1/2,0);
\draw [red, thick] (-1,0)--(-1/2,0) to[out=60,in=180] (0,1/2) to[out=0,in=120] (1/2,0)--(1,0);
\draw [red, thick] (-1/2,0) to[out=-60,in=180] (0, -1/2) to[out=0, in=-120] (1/2,0); 
\draw[thick,black,fill=white] (-1/2,0) circle (0.05);
\draw[thick,black,fill=white] (1/2,0) circle (0.05);
\end{tikzpicture}
\arrow[equal, r, "\sim"]&
\begin{tikzpicture}[baseline=-.5ex]
\draw [dashed] (3,0) circle [radius=1];
\clip (3,0) circle (1);
\draw [yellow, line cap=round, line width=5, opacity=0.5] 
(2,0) -- (2.5,0) to[out=60,in=180] (3,1/2) to[out=0,in=120] (3.5,0)-- (4,0)
(2.5,0) to[out=-60,in=-180] (3,-1/2) to[out=0,in=-120] (3.5,0);
\begin{scope}[xshift=3cm]
\draw [blue, thick] ({-sqrt(3)/2},1/2)--(-1/2,0);
\draw [blue, thick] ({-sqrt(3)/2},-1/2)--(-1/2,0);
\draw [blue, thick] ({sqrt(3)/2},1/2)--(1/2,0);
\draw [blue, thick] ({sqrt(3)/2},-1/2)--(1/2,0);
\draw [blue, thick] (-1/2,0)--(1/2,0);

\draw [red, thick] (-1,0)--(-1/2,0) to[out=60,in=180] (0,1/2) to[out=0,in=120] (1/2,0)--(1,0);
\draw [red, thick] (-1/2,0) to[out=-60,in=180] (0, -1/2) to[out=0, in=-120] (1/2,0); 

\draw[thick,black,fill=white] (-1/2,0) circle (0.05);
\draw[thick,black,fill=white] (1/2,0) circle (0.05);
\end{scope}
\end{tikzpicture}&
\begin{tikzpicture}[baseline=-.5ex]
\draw [dashed] (0,0) circle [radius=1];
\clip (0,0) circle (1);
\draw [yellow, line cap=round, line width=5, opacity=0.5] (-1/2,{sqrt(3)/2}) -- (0,1/2)--(0,-1/2) -- (-1/2,{-sqrt(3)/2});
\draw [blue, thick] ({-sqrt(3)/2},1/2)--({sqrt(3)/2},1/2);
\draw [blue, thick] ({-sqrt(3)/2},-1/2)--({sqrt(3)/2},-1/2);
\draw [blue, thick] (0,1/2)--(0,-1/2);
\draw [red, thick] (-0.5,{sqrt(3)/2})--(0,1/2) to[out=-150,in=150] (0,-1/2)--(-0.5,{-sqrt(3)/2});
\draw [red, thick] (0,1/2)--(0.5,0) -- (1,0);
\draw [red, thick] (0,-1/2)--(0.5,0);
\draw[thick,black,fill=white] (0,1/2) circle (0.05);
\draw[thick,black,fill=white] (0,-1/2) circle (0.05);
\draw[thick,red,fill=red] (0.5,0) circle (0.05);
\end{tikzpicture}
\arrow[equal, r, "\sim"]&
\begin{tikzpicture}[baseline=-.5ex]
\draw [dashed] (3,0) circle [radius=1];
\clip (3,0) circle (1);
\draw [yellow, line cap=round, line width=5, opacity=0.5] (5/2,{sqrt(3)/2})--(3,1/2) to[out=-150,in=150] (3,-0.5) (3,1/2) -- (3.5,0) -- (3,-1/2) (3,-0.5) -- (5/2,{-sqrt(3)/2});
\draw [blue, thick] ({3-sqrt(3)/2},1/2)--({3+sqrt(3)/2},1/2);
\draw [blue, thick] ({3-sqrt(3)/2},-1/2)--({3+sqrt(3)/2},-1/2);
\draw [blue, thick] (3,1/2)--(3,-1/2);
\draw [red, thick] (5/2,{sqrt(3)/2})--(3,1/2) to[out=-150,in=150] (3,-1/2)--(5/2,{-sqrt(3)/2});
\draw [red, thick] (3,1/2)--(7/2,0) -- (4,0);
\draw [red, thick] (3,-1/2)--(7/2,0);
\draw[thick,black,fill=white] (3,1/2) circle (0.05);
\draw[thick,black,fill=white] (3,-1/2) circle (0.05);
\draw[thick,red,fill=red] (7/2,0) circle (0.05);
\end{tikzpicture}
\end{tikzcd}
\]
Therefore, there is a bit subtle issue for picking up nice cycles in a consistent way.
\end{remark}

\begin{definition}\label{def:good cycle}
Let $\ngraph\subset \disk^2$ be an $N$-graph, and $\Legendrian(\ngraph)$ be an induced Legendrian surface in $J^1\disk^2$.
A cycle $[\cycle]\in H_1(\Legendrian(\ngraph))$ is \emph{good} if $[\cycle]=[\cycle(\sfT)]$ for some good $\sfT\subset\ngraph$.

A tuple of linearly independent good cycles $\nbasis=\{[\gamma_i]\}_{i\in I}$ in $H_1(\Legendrian(\ngraph))$ is \emph{good} if for any pair of cycles $[\gamma_i]$ and $[\gamma_j]$, $\ngraph$ can be transformed into $\ngraph'$ via $N$-graph moves so that two cycles $[\cycle_i]$ and $[\cycle_j]$ become $\sfI$-cycles in $H_1(\Legendrian(\ngraph'))$.
\end{definition}

\begin{remark}
Suppose $[\cycle_i]$ and $[\cycle_j]$ be a pair of cycles in a good tuple $\nbasis$ of $H_1(\Legendrian(\ngraph))$.
If $\ngraph$ is free, then by Example~\ref{ex:free N-graph}, there is no bigon in $\ngraph$ up to $N$-graph moves. So, under this assumption, two corresponding $\sfI$-cycles in Definition~\ref{def:good cycle} intersect at most one trivalent vertex.
\end{remark}

\begin{definition}\label{def:equiv on N-graph and N-basis}
Let $(\ngraph, \nbasis)$ and $(\ngraph', \nbasis')$ be pairs of an $N$-graph and good tuples of one-cycles.
We say that $(\ngraph, \nbasis)$ and $(\ngraph', \nbasis')$ are \emph{equivalent}
if there is a sequence of moves between $\ngraph$ and $\ngraph'$ inducing moves as depicted in Figure~\ref{fig:cycles under moves} between representatives of cycles in $\nbasis$ and $\nbasis'$.
We denote the equivalent class of $(\ngraph, \nbasis)$ by $[\ngraph, \nbasis]$.
\end{definition}
\begin{remark}
For two equivalent pairs $(\ngraph, \nbasis)$ and $(\ngraph',\nbasis')$, all moves between $N$-graphs $\ngraph$ and $\ngraph'$ can be realized by isotopies betwen Legendrian weaves $\Legendrian(\ngraph)$ and $\Legendrian(\ngraph')$ by Theorem~1.1 in \cite{CZ2020}, and then the induced isomorphism $H_1(\Legendrian(\ngraph))\cong H_1(\Legendrian(\ngraph'))$ identifies $\nbasis$ with $\nbasis'$.
\end{remark}

\subsection{\texorpdfstring{$N$}{N}-graphs and flag moduli space}
We recall from \cite{CZ2020} a central algebraic invariant $\mathcal{M}(\ngraph)$ of the Legendrian weave $\Legendrian(\ngraph)$.
The main idea is to consider moduli spaces of constructible sheaves associated to $\Legendrian(\ngraph)$.
To introduce a legible model for such constructible sheaves, let us consider a full flag, i.e. a nested sequence of subspaces in $\C^N$;
\[
\cF^\bullet \in \{(\cF^i)_{i=0}^N \mid \dim \cF^i=i,\  \cF^j\subset \cF^{j+1}, 1\leq j\leq N-1,\  \cF^N=\C^N \}.
\]
\begin{definition}\cite{CZ2020}\label{def:flag moduli space}
Let $\ngraph\subset \disk^2$ be an $N$-graph. Let $\{F_i\}_{i\in I}$ be a set of closures of connected components of $\disk^2\setminus \ngraph$, call each closure a \emph{face}. 
The \emph{framed flag moduli space} $\widetilde \cM(\ngraph)$ is a collection of \emph{flags} $\cF_{\Legendrian(\ngraph)}=\{\cF^\bullet(F_i)\}_{i\in I}$ in $\C^N$ satisfying the following: 

Let $F_1,F_2$ be a pair of faces sharing an edge in $\ngraph_i$. Then the corresponding flags $\cF^\bullet(F_1),\cF^\bullet(F_2)$ satisfy
\begin{align}\label{equation:flag conditions}
\begin{cases}
\cF^j(F_1)=\cF^j(F_2), \qquad 0\leq j \leq N, \quad j\neq i;\\
\cF^i(F_1)\neq \cF^i(F_2).
\end{cases}
\end{align}

Let us consider the general linear group $\operatorname{GL}_N$ action on $\cM(\ngraph)$ by acting on all flags at once. The \emph{flag moduli space} of the $N$-graph $\ngraph$ is defined by the quotient space (a stack, in general)
\[
\cM(\ngraph):=\widetilde{\cM}(\ngraph)/\operatorname{GL}_N.
\] 
\end{definition}

Let $\Sh(\disk^2 \times \R)$ be the category of \emph{constructible sheaves} on $\disk^2\times \R$. Under the identification $J^1\disk^2\cong T^{\infty,-}(\disk^2\times \R)$, an $N$-graph $\ngraph\subset \disk^2$ gives a Legendrian 
\[
\Legendrian(\ngraph)\subset J^1 \disk^2 
\cong T^{\infty,-}(\disk^2\times \R) 
\subset T^\infty(\disk^2\times \R).
\]
This can be used to define a Legendrian isotopy invariant $\Sh_{\Legendrian(\ngraph)}^1(\disk^2 \times \R)_{0}$ of $\Sh(\disk^2 \times \R)$ consisting of constructible sheaves 
\begin{itemize}
\item whose singular support at infinity lies in $\Legendrian(\ngraph)
\subset T^\infty(\disk^2\times \R)$,
\item whose microlocal rank is one, and
\item which are zero near $\disk^2\times \{-\infty\}$.
\end{itemize} 
See \cite{GKS2012,STZ2017} for the detail.

\begin{theorem}[{\cite[Theorem~5.3]{CZ2020}}]
The flag moduli space $\cM(\ngraph)$ is isomorphic to $\Sh_{\Legendrian(\ngraph)}^1(\disk^2\times\R)_0$. Hence $\cM(\ngraph)$ is a Legendrian isotopy invariant of $\Legendrian(\ngraph)$.
\end{theorem}

\begin{remark}
Indeed, the actual theorem is about a connected surface, not only for $\disk^2$.
\end{remark}

Let $\legendrian=\legendrian_\beta$ be a Legendrian in $J^1\sphere^1$, which gives us an $(N-1)$-tuple of points $X=(X_1,\dots, X_{N-1})$  in $\sphere^1$ which given by the alphabet $\sigma_1,\dots,\sigma_{N-1}$ of the braid word $\beta$.
Let $\{f_j\}_{j\in J}$ be the set of closures of connected components of $\sphere^1\setminus X$.
The flags $\flags=\{\cF^\bullet(f_j)\}_{j\in J}$ in $\C^N$ satisfying exactly the same conditions in \eqref{equation:flag conditions} will be called simply by \emph{flags on $\legendrian$}.
It is well known that the moduli space $\cM(X)$ of such flags $\flags$ up to $\operatorname{GL}_N$ is isomorphic to $\Sh_\lambda^1(\sphere^1\times \R)_0$ which is a Legendrian isotopy invariant, see \cite[Theorem 1.1]{STZ2017}.

\begin{definition}\label{def:good N-graph}
Let $\ngraph\subset \disk^2$ be an $N$-graph, and let $\flags$ be flags 
adapted to $\legendrian\subset J^1\boundary\disk^2$ given by $\boundary\ngraph$.
An $N$-graph $\ngraph$ is \emph{good}, if the flags $\flags$ uniquely determine flags $\cF_{\Legendrian(\ngraph)}$ in Definition~\ref{def:flag moduli space}.
\end{definition}
Note that $\ngraph(a,b,c)$ in the introduction is good in an obvious way. 
If an $N$-graph $\ngraph\subset \disk^2$ is good and $[\ngraph]=[\ngraph']$, then $\ngraph'$ is also good.

\subsection{\texorpdfstring{$N$}{N}-graphs and seeds}\label{sec:N-graphs and seeds}

Let $\ngraph\subset \disk^2$ be an $N$-graph, and $\nbasis=\{[\gamma_i]\}_{i\in [n]}\subset H_1(\Legendrian(\ngraph))$ be a good tuple of cycles.
For two cycles $[\cycle_i]$ and $[\cycle_j]$, let $i([\cycle_i], [\cycle_j])$ be the algebraic intersection number in $H_1(\Legendrian(\ngraph))$ which can be computed explicitly as follows:
without loss of generality, we may assume that both $[\cycle_i]$ and $[\cycle_j]$ are $\sfI$-cycles represented by $\cycle(\sfI(e))$ and $\cycle(\sfI(e'))$ for some edges $e$ and $e'$ in $\ngraph$, respectively.
Suppose that $e$ and $e'$ intersect at the vertex in $\ngraph$.
Then two representatives of $\cycle_i$ and $\cycle_j$ look locally as depicted in Figure~\ref{fig:I-cycle with orientation and intersections} and their intersection is defined to be $\pm1$ by using the clockwise rotation convention.

\begin{figure}[ht]
\subfigure[Positively intersecting $\sfI$-cycles]{
$
\def\arraycolsep{1pc}
\begin{array}{cccc}
\begin{tikzpicture}[baseline=-.5ex]
\draw [dashed] (0,0) circle [radius=1];
\clip (0,0) circle (1);
\draw [blue, thick] (0,0)--({cos(30)},{sin(30)}) node[above, midway, rotate=30] {\color{black}$e'$};
\draw [blue, thick] (0,0)--({cos(-90)},{sin(-90)}) node[left, midway] {\color{black}$e$};
\draw [blue, thick] (0,0)--({cos(150)},{sin(150)});
\draw[thick,blue,fill=blue] (0,0) circle (0.05);
\end{tikzpicture}
&
\begin{tikzpicture}[baseline=-.5ex]
\draw [dashed] (0,0) circle [radius=1];
\clip (0,0) circle (1);
\draw [green, line cap=round, line width=5, opacity=0.5] (0,0) to (0,-1);
\draw [yellow, line cap=round, line width=5, opacity=0.5] (0,0) to ({cos(30)},{sin(30)});
\draw [blue, thick] (0,0)--({cos(30)},{sin(30)});
\draw [blue, thick] (0,0)--({cos(-90)},{sin(-90)});
\draw [blue, thick] (0,0)--({cos(150)},{sin(150)});
\draw[thick,blue,fill=blue] (0,0) circle (0.05);
\draw [green!50!black,->] (0.3,-1) -- node[midway,right] {\color{black}$\cycle_i$} (0.3, 0) arc (0:180:0.3 and 0.5) -- (-0.3,-0.9);
\begin{scope}[rotate=120]
\draw [orange!90!yellow,->] (0.2,-1) -- node[midway,above] {\color{black}$\cycle_j$} (0.2, 0) arc (0:180:0.2 and 0.5) -- (-0.2,-0.9);
\end{scope}
\draw [fill] (-12:0.3) circle (1pt);
\end{tikzpicture}
&
\begin{tikzpicture}[baseline=-.5ex]
\begin{scope}
\draw [green!50!black,->] (90:-0.5) -- (90:0.7) node[left] {\color{black}$\cycle_i$};
\draw [orange!90!yellow,->] (30:-0.5) -- (30:0.7) node[below right] {\color{black}$\cycle_j$};
\draw [fill] (0,0) circle (1pt);
\draw [->] (90:0.3) arc (90:30:0.3);
\draw (0,0) node[below right] {$(+)$};
\end{scope}
\end{tikzpicture}
&
\begin{tikzpicture}[baseline=-.5ex]
\tikzstyle{state}=[draw, circle, inner sep = 0.07cm]
\tikzset{every node/.style={scale=0.7}}    
\node[state, label=above:{$1$}] (1) at (3,0) {};
\node[state, label=above:{$2$}] (2) [right = of 1] {};
\node[ynode] at (2) {};
\node[gnode] at (1) {};
\draw[->] (1)--(2);
\end{tikzpicture}
\end{array}
$
}
\subfigure[Negatively intersecting $\sfI$-cycles]{
$
\def\arraycolsep{1pc}
\begin{array}{cccc}
\begin{tikzpicture}[baseline=-.5ex]
\draw [dashed] (0,0) circle [radius=1];
\clip (0,0) circle (1);
\draw [blue, thick] (0,0)--({cos(30)},{sin(30)});
\draw [blue, thick] (0,0)--({cos(-90)},{sin(-90)}) node[left, midway] {\color{black}$e$};
\draw [blue, thick] (0,0)--({cos(150)},{sin(150)}) node[above, midway, rotate=-30] {\color{black}$e'$};
\draw[thick,blue,fill=blue] (0,0) circle (0.05);
\end{tikzpicture}
&
\begin{tikzpicture}[baseline=-.5ex]
\draw [dashed] (0,0) circle [radius=1];
\clip (0,0) circle (1);
\draw [green, line cap=round, line width=5, opacity=0.5] (0,0) to (0,-1);
\draw [yellow, line cap=round, line width=5, opacity=0.5] (0,0) to ({cos(30)},{sin(30)});
\draw [blue, thick] (0,0)--({cos(30)},{sin(30)});
\draw [blue, thick] (0,0)--({cos(-90)},{sin(-90)});
\draw [blue, thick] (0,0)--({cos(150)},{sin(150)});
\draw[thick,blue,fill=blue] (0,0) circle (0.05);
\draw [green!50!black,->] (0.3,-1) -- node[midway,right] {\color{black}$\cycle_i$} (0.3, 0) arc (0:180:0.3 and 0.5) -- (-0.3,-0.9);
\begin{scope}[rotate=-120]
\draw [orange!90!yellow,->] (0.2,-1) -- node[midway,below] {\color{black}$\cycle_j$} (0.2, 0) arc (0:180:0.2 and 0.5) -- (-0.2,-0.9);
\end{scope}
\draw [fill] (-168:0.3) circle (1pt);
\end{tikzpicture}
&
\begin{tikzpicture}[baseline=-.5ex]
\begin{scope}
\draw [green!50!black,->] (-90:-0.5) -- (-90:0.7) node[left] {\color{black}$\cycle_i$};
\draw [orange!90!yellow,->] (-30:-0.5) -- (-30:0.7) node[below right] {\color{black}$\cycle_j$};
\draw [fill] (0,0) circle (1pt);
\draw [->] (-90:0.3) arc (-90:-30:0.3);
\draw (0,0) node[above right] {$(-)$};
\end{scope}
\end{tikzpicture}
&
\begin{tikzpicture}[baseline=-.5ex]
\tikzstyle{state}=[draw, circle, inner sep = 0.07cm]
\tikzset{every node/.style={scale=0.7}}    
\node[state, label=above:{$1$}] (1) at (3,0) {};
\node[state, label=above:{$2$}] (2) [right = of 1] {};
\node[ynode] at (2) {};
\node[gnode] at (1) {};
\draw[->] (2)--(1);
\end{tikzpicture}
\end{array}
$
}
\caption{$\sfI$-cycles with intersections.}
\label{fig:I-cycle with orientation and intersections}
\end{figure}
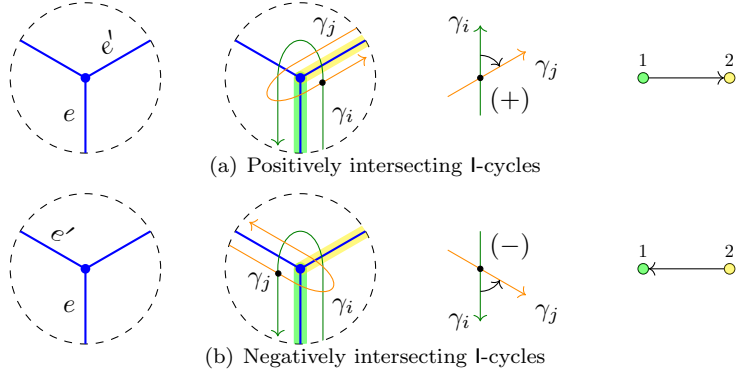

\begin{definition}
For each a pair $(\ngraph, \nbasis)$ of an $N$-graph and a good tuple of cycles, we define a quiver $\quiver=\quiver(\ngraph,\nbasis)$ as follows:
\begin{enumerate}
\item the set of vertices is $[n]$ where $\nbasis=\{[\cycle_i]\mid i\in[n]\}\subset H_1(\Legendrian(\ngraph))$, and
\item the $(i,j)$-entry $b_{i,j}$ for $\qbasis(\quiver)=(b_{i,j})$ is the algebraic intersection number between $[\cycle_i]$ and $[\cycle_j]$
\[
b_{i,j} = i([\cycle_i], [\cycle_j]).
\]
\end{enumerate}
\end{definition}

In order to assign a cluster variable to each one-cycle.
Let us review  the microlocal monodromy functor from \cite{STZ2017}
\[
\mmon:\Sh_\Legendrian^\bullet \to\Loc^\bullet(\Legendrian).
\]
In our case, this functor sends microlocal rank-one sheaves $\cF_{\Legendrian(\ngraph)} \in \Sh_{\Legendrian(\ngraph)}^1(\disk^2\times \R)_0$, or equivalently, flags $\{\cF^\bullet(F_i)\}_{i \in I}\in\cM(\ngraph)$ to rank-one local systems $\mmon(\cF_{\Legendrian(\ngraph)})$ on the Legendrian surface $\Legendrian(\ngraph)$. 
Then (cluster) variables $\bfx$ for the triple $(\ngraph, \nbasis, \cF_{\Legendrian(\ngraph)})$ are defined by
\[
\bfx=\left(
\mmon(\cF_{\Legendrian(\ngraph)})([\cycle_1]),
\dots,
\mmon(\cF_{\Legendrian(\ngraph)})([\cycle_n])\right).
\]
Let us denote the above assignment by 
\[
\Psi(\ngraph, \nbasis, \cF_{\Legendrian(\ngraph)})=(\bfx(\Legendrian(\ngraph), \nbasis, \cF_{\Legendrian(\ngraph)}),\quiver(\Legendrian(\ngraph),\nbasis)).
\]
By the Legendrian isotopy invariance of $\Sh_{\Legendrian(\ngraph)}^1(\disk^2\times \R)_0$ in \cite{GKS2012}, and the functorial property of the microlocal monodromy functor $\mmon$ \cite{STZ2017}, the assignment $\Psi$ is well-defined up to isotopy of $\Legendrian(\ngraph)$. That is, if two triples $(\Legendrian(\ngraph),\nbasis, \cF_{\Legendrian(\ngraph)})$ and $(\Legendrian(\ngraph'),\nbasis', \cF_{\Legendrian(\ngraph')})$ are Legendrian isotopic, then they give us the same seed via $\Psi$.

Especially when an $N$-graph $\ngraph$ is good, see Definition~\ref{def:good N-graph}, $\cF_{\Legendrian(\ngraph)}$ is determined by the flags $\flags\in \Sh_{\legendrian}^1(\boundary\disk^2\times\R)$ at the boundary, where the Legendrian link $\legendrian$ is given by $\boundary\ngraph$.
So we have
\begin{theorem}\cite[\S7.2.1]{CZ2020}\label{thm:N-graph to seed}
Let $\ngraph\subset \disk^2$ be a good $N$-graph with a good tuple $\nbasis$ of cycles in $H_1(\Legendrian(\ngraph))$, and with flags $\flags$ on $\legendrian\subset J^1\sphere^1$ at the boundary. Then the assignment $\Psi$ to a seed in a cluster structure
\[
\Psi(\ngraph,\nbasis,\flags)= (\bfx(\Legendrian(\ngraph),\nbasis,\flags),\quiver(\Legendrian(\ngraph),\nbasis))
\]
is well-defined up to Legendrian isotopy.
\end{theorem}

As a corollary, the seed $\Psi(\ngraph, \nbasis, \flags)$ can be used to distinguish a pair of Legendrian surfaces and hence, by Lemma~\ref{lem:legendrian and lagrangian}, a pair of Lagrangian fillings.

\begin{corollary}\label{corollary:distinct seeds imples distinct fillings}
As in the above setup,
if two triples $(\ngraph, \nbasis, \flags)$, $(\ngraph', \nbasis', \flags)$ with the same boundary condition define different seeds, then two induced Lagrangian fillings $\pi\circ\iota(\Legendrian(\ngraph))$, $\pi\circ\iota(\Legendrian(\ngraph'))$ bounding $\iota(\legendrian)$ are not exact Lagrangian isotopic to each other.
\end{corollary}

The monodromy $\mmon(\cF_{\Legendrian(\ngraph)})$ along a loop $[\gamma]\in H_1(\Legendrian(\ngraph))$ can be obtained by restricting the constructible sheaf $\cF_{\Legendrian(\ngraph)}$ to a tubular neighborhood of $\gamma$. 
Let us investigate how the monodromy can be computed explicitly in terms of flags $\{\cF^\bullet(F_i)\}_{i \in I}$.

Let us consider an $\sfI$-cycle $[\cycle]$ represented by a loop $\cycle(e)$ for some monochromatic edge $e$ as in Figure~\ref{figure:I-cycle with flags}.
Let us denote four flags corresponding to each region by $F_1,F_2,F_3,F_4$, respectively.
Suppose that $e \subset \ngraph_i$, then by the construction of flag moduli space $\cM(\ngraph)$, a two-dimensional vector space $V:=\cF^{i+1}(F_*)/\cF^{i-1}(F_*)$ is independent of $*=1,2,3,4$. Moreover, $\cF^{i}(F_*)/\cF^{i-1}(F_*)$ defines a one-dimensional subspace $v_*\subset V$ for $*=1,2,3,4$, satisfying
\[
v_1\neq v_2 \neq v_3 \neq v_4 \neq v_1.
\]
Then $\mmon(\cF_{\Legendrian(\ngraph)})$ along the one-cycle $[\gamma(e)]$ is defined by the cross ratio \[
\mmon(\cF_{\Legendrian(\ngraph)})([\cycle])\coloneqq\langle v_1,v_2,v_3,v_4 \rangle=\frac{v_1 \wedge v_2}{v_2 \wedge v_3}\cdot\frac{v_3 \wedge v_4}{v_4\wedge v_1}.
\]

Suppose that local flags $\{F_j\}_{j\in J}$ near the upper $\sfY$-cycle $[\cycle_U]$ look like in Figure~\ref{figure:Y-cycle with flags I}. 
Let $\ngraph_i$ and $\ngraph_{i+1}$ be the $N$-subgraphs in red and blue, respectively.
Then the $3$-dimensional vector space $V=\cF^{i+2}(F_*)/\cF^{i-1}(F_*)$ is independent of $*\in J$. Now regard $a,b,c$ and $A,B,C$ are subspaces of $V$ of dimension one and two, respectively. Then the microlocal monodromy along the $\sfY$-cycle $[\cycle_U]$  becomes
\[
\mmon(\cF_{\Legendrian(\ngraph)})([\cycle_U])\coloneqq\frac{B(a)C(b)A(c)}{B(c)C(a)A(b)}.
\]
Here $B(a)$ can be seen as a paring between the vector $a$ and the covector $B$.

Now consider the lower $\sfY$-cycle $[\cycle_L]$ whose local flags given as in Figure~\ref{figure:Y-cycle with flags II}. We already have seen that the orientation convention of the loop in Figure~\ref{fig:I and Y cycle} for the upper and lower $\sfY$-cycle is different. Then microlocal monodromy along $[\cycle_L]$ follows the opposite orientation and becomes 
\[
\mmon(\cF_{\Legendrian(\ngraph)})([\cycle_L])\coloneqq\frac{C(a)B(c)A(b)}{C(b)B(a)A(c)}.
\]
Here, $B(a)$ is a pairing between the vector $B$ and covector $a$ which is the same as the above.

\begin{figure}[ht]
\begin{tikzcd}
\subfigure[$\sfI$-cycle with flags.\label{figure:I-cycle with flags}]{
\begin{tikzpicture}[baseline=.5ex]
\begin{scope}
\draw [dashed] (0,0) circle [radius=1.5];

\draw [green, line cap=round, line width=5, opacity=0.5] (-1/2,0) to (1/2,0);

\draw [blue, thick] ({-3*sqrt(3)/4},3/4)--(-1/2,0);
\draw [blue, thick] ({-3*sqrt(3)/4},-3/4)--(-1/2,0);
\draw [blue, thick] ({3*sqrt(3)/4},3/4)--(1/2,0);
\draw [blue, thick] ({3*sqrt(3)/4},-3/4)--(1/2,0);
\draw [blue, thick] (-1/2,0)--(1/2,0) node[above, midway] {$\cycle$};

\draw[thick,blue,fill=blue] (-1/2,0) circle (0.05);
\draw[thick,blue,fill=blue] (1/2,0) circle (0.05);

\node at (0,1) {$v_1$};
\node at (-1,0) {$v_2$};
\node at (0,-1) {$v_3$};
\node at (1,0) {$v_4$};

\end{scope}
\end{tikzpicture}}
&
\subfigure[Upper $\sfY$-cycle with flags.\label{figure:Y-cycle with flags I}]{
\begin{tikzpicture}[baseline=.5ex]

\begin{scope}[scale=0.9]
\draw [dashed] (0,0) circle [radius=1.5];

\draw [yellow, line cap=round, line width=5, opacity=0.5] (0,0) to ({cos(0-30)},{sin(0-30)});
\draw [yellow, line cap=round, line width=5, opacity=0.5] (0,0) to ({cos(120-30)},{sin(120-30)});
\draw [yellow, line cap=round, line width=5, opacity=0.5] (0,0) to ({cos(240-30)},{sin(240-30)});

\draw [blue, thick] ({1.5*cos(180-30)},{1.5*sin(180-30)})--(0,0)--({1.5*cos(60-30)},{1.5*sin(60-30)});
\draw [blue, thick] (0,0)--({1.5*cos(60+30)},{-1.5*sin(60+30)});

\draw [red, thick] ({1.5*cos(20-30)},{1.5*sin(20-30)})--({cos(0-30)},{sin(0-30)})--(0,0);
\draw [red, thick] ({1.5*cos(20+30)},{-1.5*sin(20+30)})--({cos(0-30)},{sin(0-30)});

\draw [red, thick] ({1.5*cos(100-30)},{1.5*sin(100-30)})--({cos(120-30)},{sin(120-30)})--(0,0);
\draw [red, thick] ({1.5*cos(140-30)},{1.5*sin(140-30)})--({cos(120-30)},{sin(120-30)});

\draw [red, thick] ({1.5*cos(220-30)},{1.5*sin(220-30)})--({cos(240-30)},{sin(240-30)})--(0,0);
\draw [red, thick] ({1.5*cos(260-30)},{1.5*sin(260-30)})--({cos(240-30)},{sin(240-30)});

\draw[thick,red,fill=red] ({cos(0-30)},{sin(0-30)}) circle (0.05);
\draw[thick,red,fill=red] ({cos(120-30)},{sin(120-30)}) circle (0.05);
\draw[thick,red,fill=red] ({cos(240-30)},{sin(240-30)}) circle (0.05);
\draw[thick,black,fill=white] (0,0) circle (0.05);
\end{scope}
\begin{scope}[yshift=-0.1cm]
\node at (0,1.7) {$(b,B)$};
\node[rotate=60] at ({1.7*cos(-30)},{1.7*sin(-30)}) {$(a,A)$};
\node[rotate=-60] at ({1.7*cos(-150)},{1.7*sin(-150)}) {$(c,C)$};

\node[rotate=100] at ({1.7*cos(10)},{1.7*sin(10)}) {$(a,ab)$};
\node[rotate=-100] at ({1.7*cos(170)},{1.7*sin(170)}) {$(c,bc)$};

\node[rotate=20] at ({1.7*cos(-70)},{1.7*sin(-70)}) {$(a,ac)$};
\node[rotate=-20] at ({1.7*cos(-110)},{1.7*sin(-110)}) {$(c,ac)$};

\node[rotate=-40] at ({1.7*cos(50)},{1.7*sin(50)}) {$(b,ab)$};
\node[rotate=40] at ({1.7*cos(130)},{1.7*sin(130)}) {$(b,bc)$};

\draw[red] (0,0) node[xshift=0.4cm] {$\cycle_U$};
\end{scope}
\end{tikzpicture}}
&
\subfigure[Lower $\sfY$-cycle with flags.\label{figure:Y-cycle with flags II}]{
\begin{tikzpicture}[baseline=.5ex]
\begin{scope}[scale=0.9]
\draw [dashed] (0,0) circle [radius=1.5];

\draw [yellow, line cap=round, line width=5, opacity=0.5] (0,0) to ({cos(0-30)},{sin(0-30)});
\draw [yellow, line cap=round, line width=5, opacity=0.5] (0,0) to ({cos(120-30)},{sin(120-30)});
\draw [yellow, line cap=round, line width=5, opacity=0.5] (0,0) to ({cos(240-30)},{sin(240-30)});

\draw [red, thick] ({1.5*cos(180-30)},{1.5*sin(180-30)})--(0,0)--({1.5*cos(60-30)},{1.5*sin(60-30)});
\draw [red, thick] (0,0)--({1.5*cos(60+30)},{-1.5*sin(60+30)});

\draw [blue, thick] ({1.5*cos(20-30)},{1.5*sin(20-30)})--({cos(0-30)},{sin(0-30)})--(0,0);
\draw [blue, thick] ({1.5*cos(20+30)},{-1.5*sin(20+30)})--({cos(0-30)},{sin(0-30)});

\draw [blue, thick] ({1.5*cos(100-30)},{1.5*sin(100-30)})--({cos(120-30)},{sin(120-30)})--(0,0);
\draw [blue, thick] ({1.5*cos(140-30)},{1.5*sin(140-30)})--({cos(120-30)},{sin(120-30)});

\draw [blue, thick] ({1.5*cos(220-30)},{1.5*sin(220-30)})--({cos(240-30)},{sin(240-30)})--(0,0);
\draw [blue, thick] ({1.5*cos(260-30)},{1.5*sin(260-30)})--({cos(240-30)},{sin(240-30)});

\draw[thick,blue,fill=blue] ({cos(0-30)},{sin(0-30)}) circle (0.05);
\draw[thick,blue,fill=blue] ({cos(120-30)},{sin(120-30)}) circle (0.05);
\draw[thick,blue,fill=blue] ({cos(240-30)},{sin(240-30)}) circle (0.05);
\draw[thick,black,fill=white] (0,0) circle (0.05);
\end{scope}
\begin{scope}[yshift=-0.1cm]
\node at (0,1.7) {$(b,B)$};
\node[rotate=60] at ({1.7*cos(-30)},{1.7*sin(-30)}) {$(a,A)$};
\node[rotate=-60] at ({1.7*cos(-150)},{1.7*sin(-150)}) {$(c,C)$};

\node[rotate=100] at ({1.7*cos(10)},{1.7*sin(10)}) {\small$(AB,A)$};
\node[rotate=-100] at ({1.7*cos(170)},{1.7*sin(170)}) {\small$(BC,C)$};

\node[rotate=20] at ({1.7*cos(-70)},{1.7*sin(-70)}) {\small$(AC,A)$};
\node[rotate=-20] at ({1.7*cos(-110)},{1.7*sin(-110)}) {\small$(AC,C)$};

\node[rotate=-40] at ({1.7*cos(50)},{1.7*sin(50)}) {\small$(AB,B)$};
\node[rotate=40] at ({1.7*cos(130)},{1.7*sin(130)}) {\small$(BC,B)$};

\draw[blue] (0,0) node[xshift=0.4cm] {$\cycle_L$};
\end{scope}
\end{tikzpicture}}
\end{tikzcd}
\caption{$\sfI$- and $\sfY$-cycles with flags.}
\label{fig:I and Y cycle with flags}
\end{figure}
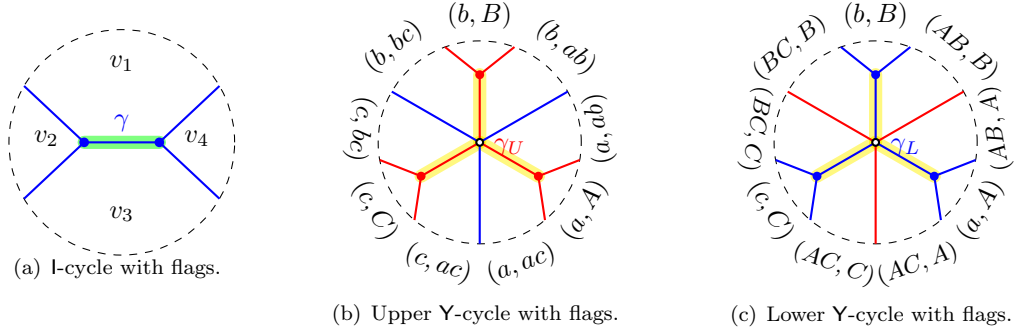

\subsection{Legendrian mutations in \texorpdfstring{$N$}{N}-graphs}

Let us define an operation called (\emph{Legendrian}) \emph{mutation} on $N$-graphs $\ngraph$ which corresponds to a geometric operation on the induced Legendrian surface $\Legendrian(\ngraph)$ that producing a smoothly isotopic but not necessarily Legendrian isotopic to $\Legendrian(\ngraph)$, see \cite[Definition~4.19]{CZ2020}.
Note that operation has an intimate relation with the wall-crossing phenomenon \cite{Aur2007}, Lagrangian surgery \cite{Pol1991}, and quiver (or cluster) mutations \cite{FZ1_2002}.

\begin{definition}\cite{CZ2020}\label{def:legendrian mutation}
Let $\ngraph$ be a (local) $N$-graph and $e\in \ngraph_i\subset \ngraph$ be an edge between two trivalent vertices corresponding to an $\sfI$-cycle $[\cycle]=[\cycle(e)]$. The mutation $\mutation_\cycle(\ngraph)$ of $\ngraph$ along $\cycle$ is obtained by applying the local change depicted in the left of Figure~\ref{fig:Legendrian mutation on N-graphs}.
\end{definition}

\begin{figure}[ht]
\begin{tikzcd}
\subfigure[A mutation along $\sfI$-cycle.\label{figure:I-mutation}]{
\begin{tikzpicture}[baseline=-.5ex,scale=1.2]
\begin{scope}
\draw [dashed] (0,0) circle [radius=1];
\draw [->,yshift=.5ex] (1.25,0) -- (1.75,0) node[midway, above] {$\mutation_\cycle$};
\draw [<-,yshift=-.5ex] (1.25,0) -- (1.75,0) node[midway, below] {$\mutation_{\cycle'}$};

\draw [green, line cap=round, line width=5, opacity=0.5] (-1/2,0) to (1/2,0);

\draw [blue, thick] ({-1*sqrt(3)/2},1*1/2)--(-1/2,0);
\draw [blue, thick] ({-1*sqrt(3)/2},-1*1/2)--(-1/2,0);
\draw [blue, thick] ({1*sqrt(3)/2},1*1/2)--(1/2,0);
\draw [blue, thick] ({1*sqrt(3)/2},-1*1/2)--(1/2,0);
\draw [blue, thick] (-1/2,0)-- node[midway,above] {$\cycle$}(1/2,0);

\draw[thick,blue,fill=blue] (-1/2,0) circle (0.05);
\draw[thick,blue,fill=blue] (1/2,0) circle (0.05);

\end{scope}

\begin{scope}[xshift=3cm]

\draw [dashed] (0,0) circle [radius=1];

\draw [green, line cap=round, line width=5, opacity=0.5] (0,-1/2) to (0,1/2);

\draw [blue, thick] (-1*1/2,{1*sqrt(3)/2}) to (0,1/2) to (0,-1/2)  node[midway,right] {$\cycle'$} to (-1*1/2,-{1*sqrt(3)/2});
\draw [blue, thick] (1*1/2,{1*sqrt(3)/2})--(0,1/2);
\draw [blue, thick] (1*1/2,-{1*sqrt(3)/2})--(0,-1/2);

\draw[thick,blue,fill=blue] (0,1/2) circle (0.05);
\draw[thick,blue,fill=blue] (0,-1/2) circle (0.05);
\end{scope}
\end{tikzpicture}
}
&
\subfigure[A mutations along $\sfY$-cycle.\label{figure:Y-mutation}]{
\begin{tikzpicture}[baseline=-.5ex,scale=1.2]
\begin{scope}

\draw [->,yshift=.5ex] (1.25,0) -- (1.75,0) node[midway, above] {$\mutation_\cycle$};
\draw [<-,yshift=-.5ex] (1.25,0) -- (1.75,0) node[midway, below] {$\mutation_{\cycle'}$};

\draw [dashed] (0,0) circle [radius=1];

\draw [yellow, line cap=round, line width=5, opacity=0.5] (0,1/2) to (0,0);
\draw [yellow, line cap=round, line width=5, opacity=0.5] ({sqrt(3)/4},-1/4) to (0,0);
\draw [yellow, line cap=round, line width=5, opacity=0.5] (-{sqrt(3)/4},-1/4) to (0,0);

\draw [blue, thick] (-1*1/2,{1*sqrt(3)/2}) to (0,1/2) to (0,0) node[right] {$\cycle$} to ({-sqrt(3)/4},-1/4) to (-1,0);
\draw [blue, thick] (1/2,{1*sqrt(3)/2}) to (0,1/2) to (0,0) to ({sqrt(3)/4},-1/4) to (1,0);
\draw [blue, thick] ({-sqrt(3)/4},-1/4) to ({-1/2},{-sqrt(3)/2});
\draw [blue, thick] ({sqrt(3)/4},-1/4) to ({1/2},{-sqrt(3)/2});

\draw [red, thick] (0,0) to ({sqrt(3)/2},1/2);
\draw [red, thick] (0,0) to ({-sqrt(3)/2},1/2);
\draw [red, thick] (0,0) to (0,-1);

\draw[thick,blue,fill=blue] (0,1/2) circle (0.05);
\draw[thick,blue,fill=blue] ({-sqrt(3)/4},-1/4) circle (0.05);
\draw[thick,blue,fill=blue] ({sqrt(3)/4},-1/4) circle (0.05);
\draw[thick,black,fill=white] (0,0) circle (0.05);
\end{scope}

\begin{scope}[xshift=3cm]

\draw [dashed] (0,0) circle [radius=1];

\draw [yellow, line cap=round, line width=5, opacity=0.5] (0,1/2) to (0,0);
\draw [yellow, line cap=round, line width=5, opacity=0.5] ({sqrt(3)/4},-1/4) to (0,0);
\draw [yellow, line cap=round, line width=5, opacity=0.5] (-{sqrt(3)/4},-1/4) to (0,0);

\draw [blue, thick] (-1/2,{sqrt(3)/2}) to ({-sqrt(3)/4},1/4) to (-1,0);
\draw [blue, thick] (1/2,{sqrt(3)/2}) to ({sqrt(3)/4},1/4) to (1,0);
\draw [blue, thick] ({-1/2},{-sqrt(3)/2}) to (0,-1/2) to ({1/2},{-sqrt(3)/2});
\draw [blue, thick] ({sqrt(3)/4},1/4) to (0,0);
\draw [blue, thick] ({-sqrt(3)/4},1/4) to (0,0);
\draw [blue, thick] (0,-1/2) to (0,0);

\draw [red, thick] ({-sqrt(3)/4},1/4) to ({-sqrt(3)/4},-1/4) to (0,-1/2) to ({sqrt(3)/4},-1/4) to ({sqrt(3)/4},1/4) to (0,1/2) to ({-sqrt(3)/4},1/4);
\draw [red, thick] ({sqrt(3)/4},1/4) to ({sqrt(3)/2},1/2);
\draw [red, thick] ({-sqrt(3)/4},1/4) to ({-sqrt(3)/2},1/2);
\draw [red, thick] (0,-1/2) to (0,-1);
\draw [red, thick] ({-sqrt(3)/4},-1/4) to (0,0);
\draw [red, thick] ({sqrt(3)/4},-1/4) to (0,0);
\draw [red, thick] (0,1/2) to (0,0) node[right] {$\cycle'$};

\draw[thick,red,fill=red] (0,1/2) circle (0.05);
\draw[thick,red,fill=red] ({-sqrt(3)/4},-1/4) circle (0.05);
\draw[thick,red,fill=red] ({sqrt(3)/4},-1/4) circle (0.05);
\draw[thick,black,fill=white] (0,0) circle (0.05);
\draw[thick,black,fill=white] ({-sqrt(3)/4},1/4) circle (0.05);
\draw[thick,black,fill=white] ({sqrt(3)/4},1/4) circle (0.05);
\draw[thick,black,fill=white] (0,-1/2) circle (0.05);
\end{scope}

\end{tikzpicture}
}
\end{tikzcd}
\caption{Legendrian mutations at $\sfI$- and $\sfY$-cycles.}
\label{fig:Legendrian mutation on N-graphs}
\end{figure}
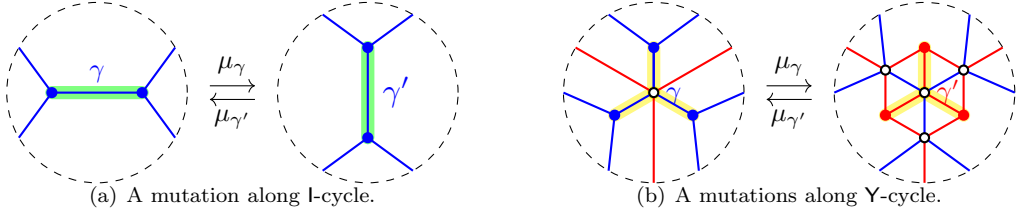

For the $\sfY$-cycle, the Legendrian mutation becomes as in the right of Figure~\ref{fig:Legendrian mutation on N-graphs}.
Note that the mutation at $\sfY$-cycle can be decomposed into a sequence of Move~\Move{I} and Move~\Move{II} together with a mutation at $\sfI$-cycle.

Let us remind our main purpose of finding exact embedded Lagrangian fillings for a Legendrian links.
The following lemma guarantees that Legendrian mutation preserves the embedding property of Lagrangian fillings.
\begin{proposition}\cite[Lemma~7.4]{CZ2020}
Let $\ngraph\subset \disk^2$ be a free $N$-graph. Then mutation $\mutation(\ngraph)$ at any $\sfI$- or $\sfY$-cycle is again free $N$-graph. 
\end{proposition}

\begin{proposition}
Let $\ngraph\subset \disk^2$ be a good $N$-graph.
Then mutation $\mutation_\gamma(\ngraph)$ at $\sfI$-cycle $\gamma$ is again good $N$-graph.
\end{proposition}

\begin{proof}
The proof is straightforward from the notion of the good $N$-graph in Definition~\ref{def:good N-graph} and of the Legendrian mutation depicted in Figure~\ref{figure:I-mutation}.
Note that the Legendrian mutation $\mutation_\gamma(\ngraph)$ at $\sfY$-cycle $\gamma$ is also good, since $\mutation_\gamma(\ngraph)$ is a composition of Moves \Move{I} and \Move{II}, and a mutation at $\sfI$-cycle.
\end{proof}

An important observation is the Legendrian mutation on $(\ngraph,\nbasis)$ induces a cluster mutation on the induced seed $(\bfx(\Legendrian(\ngraph),\nbasis,\flags),\quiver(\Legendrian(\ngraph),\nbasis))$.

\begin{proposition}[{\cite[\S7.2]{CZ2020}}]\label{proposition:equivariance of mutations}
Let $\ngraph\subset \disk^2$ be a good $N$-graph and $\nbasis$ be a good tuple of cycles in $H_1(\Legendrian(\ngraph))$.
Let $\mutation_{\cycle_i}(\ngraph,\nbasis)$ be a Legendrian mutation  of $(\ngraph,\nbasis)$ along a one-cycle $\cycle_i$  then the following holds:
for flags $\flags$ on $\legendrian$,
\[
\Psi(\mutation_{\cycle_i}(\ngraph,\nbasis),\flags)=\mutation_{i}(\Psi(\ngraph,\nbasis,\flags)).
\]
Here, $\mutation_{i}$ is the cluster $\mathcal{X}$-mutation
at the vertex $i$ \textup{(}cf. Remark~\ref{rmk_x_cluster_mutation}\textup{)}.
\end{proposition}

\section{Lagrangian fillings for of Legendrians type \texorpdfstring{$\dynADE$}{ADE}}

\subsection{Tripods}
Let $\legendrian\subset J^1\sphere^1$ be a Legendrian knot or link which bounds a Legendrian surface~$\Legendrian(\ngraph)$ in $J^1\disk^2$ for some free $N$-graph $\ngraph$.
We fix a good tuple $\nbasis$ of cycles in the sense of Definition~\ref{def:good cycle}, and fix flags $\flags$ on $\legendrian$.
Then by Theorem~\ref{thm:N-graph to seed}, we obtain a seed $\Psi(\ngraph,\nbasis,\flags)$ which is a pair of a set of cluster variables~$\bfx(\Legendrian(\ngraph),\nbasis,\flags)$ 
and a quiver $\quiver(\Legendrian(\ngraph),\nbasis)$. 

We say that the pair $(\ngraph,\nbasis)$ is \emph{of finite type} or \emph{of
infinite type} if so is the cluster algebra defined by
$\quiver(\Legendrian(\ngraph),\nbasis)$.
In particular, it is said to be of \emph{type~$\dynADE$} if the quiver
$\quiver(\Legendrian(\ngraph),\nbasis)$ is  of type
$\dynA_n, \dynD_n, \dynE_6, \dynE_7$ or $\dynE_8$ (see Definition~\ref{def_quiver_of_type_X}).
Braid words $\beta$ of Legendrians, $N$-graphs and good tuples of cycles $(\ngraph,\nbasis)$ of type $\dynADE$ are depicted in Table~\ref{table:ADE type}.

\begin{table}[ht]
\begin{tabular}{c|c|c}
\toprule
$\beta$ & $(\ngraph,\nbasis)$ & $\quiver$\\
\midrule
$\sigma_1^{n+3}$ &
\begin{tikzpicture}[baseline=-.5ex,scale=0.8]
\useasboundingbox(-3.5,-3.5)rectangle(3.5,3.5);
\draw[thick] (0,0) circle (3);
\draw[green, line width=5, opacity=0.5] (-1.5,0.5) -- (-0.5, -0.5) (1, 0) -- (1.5, -0.5);
\draw[yellow, line width=5, opacity=0.5] (-2.5,-0.5) -- (-1.5, 0.5) (-0.5, -0.5) -- (0, 0) (1.5, -0.5) -- (2.5, 0.5);
\draw[blue, thick, fill] (0:3) -- (2.5,0.5) circle (2pt) -- (45:3) (2.5,0.5) -- (1.5,-0.5) node[midway, above left] {$\gamma_n$} circle (2pt) -- (-45:3) (1.5,-0.5) -- (1,0) node[above] {$\gamma_{n-1}$} (0.5,0) node {$\cdots$} (0,0) node[above] {$\gamma_3$} -- (-0.5, -0.5) circle (2pt) -- (-90:3) (-0.5, -0.5) -- (-1.5, 0.5) node[midway, above right] {$\gamma_2$} circle (2pt) -- (135:3) (-1.5, 0.5) -- (-2.5, -0.5) node[midway, above left] {$\gamma_1$} circle (2pt) -- (-135:3);
\draw[blue, thick] (-2.5,-0.5) -- (-180:3);
\end{tikzpicture}
&
\begin{tikzpicture}[baseline=.5ex]
\useasboundingbox (-2.5,-2) rectangle (2.5,2);
\draw[fill, thick] (-2,0) node (A1) {} circle (1pt) (-1,0) node (A2) {} circle (1pt) (0,0) node (A3) {} circle (1pt) (1,0) node (A4) {} circle (1pt) (2,0) node (An) {} circle (1pt);
\draw[->] (A1) node[above] {$\scriptstyle 1$} -- (A2) node[above] {$\scriptstyle 2$};
\draw[->] (A3) node[above] {$\scriptstyle 3$} -- (A2);
\draw[->] (A3) -- (A4) node[above] {$\scriptstyle 4$};
\draw[dotted] (A4) -- (An) node[above] {$\scriptstyle n$};
\end{tikzpicture}\\
\hline
$\sigma_2\sigma_1^3\sigma_2\sigma_1^3\sigma_2\sigma_1^{n-1}$ &
\begin{tikzpicture}[baseline=-.5ex,scale=0.8]
\useasboundingbox(-3.5,-3.5)rectangle(3.5,3.5);
\draw[thick] (0,0) circle (3cm);
\draw[green, line cap=round, line width=5, opacity=0.5] (60:1) -- (50:2) (180:1) -- (170:2) (300:1) -- (290:1.5) (310:1.75) -- (290:2);
\draw[yellow, line cap=round, line width=5, opacity=0.5] (0,0) -- (60:1) (0,0) -- (180:1) (0,0) -- (300:1) (290:1.5) -- (310:1.75);
\draw[red, thick] (0,0) -- (0:3) (0,0) -- (120:3) (0,0) -- (240:3);
\draw[blue, thick, fill] (0,0) -- (60:1) node[midway, below right] {$\gamma_{n-2}$} circle (2pt) -- (90:3) (60:1) -- (50:2) node[midway, below right] {$\gamma_n$} circle (2pt) -- (30:3) (50:2) -- (60:3);
\draw[blue, thick, fill] (0,0) -- (180:1) circle (2pt) -- (210:3) (180:1) -- (170:2) node[midway, above right] {$\gamma_{n-1}$} circle (2pt) -- (150:3) (170:2) -- (180:3);
\draw[blue, thick, fill] (0,0) -- (300:1) circle (2pt) -- (340:3) (300:1) -- (290:1.5) node[midway, left] {$\gamma_{n-3}$} circle (2pt) -- (260:3) (290:1.5) -- (310:1.75) circle (2pt) -- (320:3) (310:1.75) -- (290:2) circle (2pt) -- (280:3);
\draw[blue, thick, dashed] (290:2) -- (300:3);
\draw[thick, fill=white] (0,0) circle (2pt);
\curlybrace[]{250}{350}{3.2};
\draw (300:3.5) node[rotate=30] {$n-1$};
\end{tikzpicture}
&
\begin{tikzpicture}[baseline=.5ex]
\useasboundingbox (-2.5,-2) rectangle (2.5,2);
\draw[fill, thick] (-2,0) node (D1) {} circle (1pt) (-1,0) node (D2) {} circle (1pt) (0,0) node (D3) {} circle (1pt) (1,0) node (D4) {} circle (1pt) (1,0) ++ (30:1) node (D5) {} circle (1pt) (1,0) ++ (-30:1) node (Dn) {} circle (1pt);
\draw[->] (D4) node[above] {$\scriptstyle n-2$} -- (D5) node[right] {$\scriptstyle n-1$};
\draw[->] (D4) -- (Dn) node[right] {$\scriptstyle n$};
\draw[->] (D4) -- (D3) node[above] {$\scriptstyle n-3$};
\draw[->] (D2) node[above] {$\scriptstyle n-4$} -- (D3);
\draw[dotted] (D1) node[above] {$\scriptstyle 1$} -- (D2);
\end{tikzpicture}\\
\hline
$\sigma_2\sigma_1^4\sigma_2\sigma_1^3\sigma_2\sigma_1^{n-2}$ &
\begin{tikzpicture}[baseline=-.5ex,scale=0.8]
\useasboundingbox(-3.5,-3.5)rectangle(3.5,3.5);
\draw[thick] (0,0) circle (3cm);
\draw[green, line cap=round, line width=5, opacity=0.5] (60:1) -- (50:2) (180:1) -- (170:1.5) (300:1) -- (290:1.5) (310:1.75) -- (290:2);
\draw[yellow, line cap=round, line width=5, opacity=0.5] (0,0) -- (60:1) (0,0) -- (180:1) (0,0) -- (300:1) (170:1.5) -- (190:2) (290:1.5) -- (310:1.75);
\draw[red, thick] (0,0) -- (0:3) (0,0) -- (120:3) (0,0) -- (240:3);
\draw[blue, thick, fill] (0,0) -- (60:1) node[midway, below right] {$\gamma_4$} circle (2pt) -- (90:3) (60:1) -- (50:2) node[midway, below right] {$\gamma_2$} circle (2pt) -- (30:3) (50:2) -- (60:3);
\draw[blue, thick, fill] (0,0) -- (180:1) circle (2pt) -- (216:3) (180:1) -- (170:1.5) node[midway, above right] {$\gamma_3$} circle (2pt) -- (144:3) (170:1.5) -- (190:2) node[midway, above left] {$\gamma_1$} circle (2pt) -- (192:3) (190:2) -- (168:3);
\draw[blue, thick, fill] (0,0) -- (300:1) circle (2pt) -- (340:3) (300:1) -- (290:1.5) node[midway, left] {$\gamma_5$} circle (2pt) -- (260:3) (290:1.5) -- (310:1.75) circle (2pt) -- (320:3) (310:1.75) -- (290:2) circle (2pt) -- (280:3);
\draw[blue, thick, dashed] (290:2) -- (300:3);
\draw[thick, fill=white] (0,0) circle (2pt);
\curlybrace[]{250}{350}{3.2};
\draw (300:3.5) node[rotate=30] {$n-2$};
\end{tikzpicture}
&
\begin{tikzpicture}[baseline=.5ex]
\useasboundingbox (-2.5,-2) rectangle (2.5,2);
\draw[fill, thick] (-2,0) node (E1) {} circle (1pt) (-1,0) node (E3) {} circle (1pt) (0,0) node (E4) {} circle (1pt) (0,1) node (E2) {} circle (1pt) (1,0) node (E5) {} circle (1pt) (2,0) node (En) {} circle (1pt);
\draw[->] (E4) node[above right] {$\scriptstyle 4$} -- (E3) node[above] {$\scriptstyle 3$};
\draw[->] (E4) -- (E2) node[right] {$\scriptstyle 2$};
\draw[->] (E4) -- (E5) node[above] {$\scriptstyle 5$};
\draw[dotted] (En) node[above] {$\scriptstyle n$} -- (E5);
\draw[->] (E1) node[above] {$\scriptstyle 1$} -- (E3);
\end{tikzpicture}\\
\bottomrule
\end{tabular}
\caption{$N$-graphs and their quivers of type~$\dynADE$}
\label{table:ADE type}
\end{table}

One can generalize these quivers of type~$\dynADE$  as follows:
\begin{definition}[Tripod quiver]
For $a,b,c\ge 1$, the \emph{tripod} $\quiver(a,b,c)$ of type $(a,b,c)$ is a bipartite quiver such that
\begin{enumerate}
\item the set of vertices is $[n]$ for $n=a+b+c-2$, 
\item the underlying graph is a boundary wedge sum of three quivers $\dynA_{a}, \dynA_{b}$, and $\dynA_{c}$, and 
\item the vertex where $\dynA_a, \dynA_b$, and $\dynA_c$ are glued together is called the \emph{central vertex}, labelled as $1$ and colored as $+$.
\end{enumerate}

We define an $N$-graph $\ngraph(a,b,c)$ on $\disk^2$  as the concatenation of three $N$-graphs $\ngraph(A_a), \ngraph(A_b)$ and $\ngraph(A_c)$ by making one $\sfY$-cycle $\cycle_1$ and define a good tuple $\nbasis(a,b,c)$ of cycles as the union of cycles in three $N$-graphs.
The $N$-graph obtained by switching colors from $\ngraph(a,b,c)$ and the induced set of chosen cycles will be denoted by $\bar\ngraph(a,b,c)$ and $\bar\nbasis(a,b,c)$, respectively. 

The pictorial definitions of $\quiver(a,b,c), \ngraph(a,b,c)$ and $\nbasis(a,b,c)$ are depicted in Figure~\ref{figure:tripod}.
\end{definition}

It is obvious that if $a,b$ or $c$ is one, then it is the same as $\dynA_n$ for $n=a+b+c-2$ up to relabelling.
Similarly, the cases that $a,b$ or $c$ is two include quivers of type $\dynD_n$ and $\dynE_n$.

\begin{figure}[ht]
\subfigure[The tripod $\quiver(a,b,c)$\label{figure:tripod quiver}]{
\begin{tikzpicture}[baseline=-.5ex,scale=0.7]
\draw[fill] (0,0) circle (1pt) node (O) {};
\draw[fill] (60:1) circle (1pt) node (A1) {};
\draw[fill] (60:2) circle (1pt) node (A2) {};
\draw[fill] (60:3) circle (1pt) node (A3) {};
\draw[fill] (60:4) circle (1pt) node (An) {};
\draw[fill] (180:1) circle (1pt) node (B1) {};
\draw[fill] (180:2) circle (1pt) node (B2) {};
\draw[fill] (180:3) circle (1pt) node (B3) {};
\draw[fill] (180:4) circle (1pt) node (Bn) {};
\draw[fill] (-60:1) circle (1pt) node (C1) {};
\draw[fill] (-60:2) circle (1pt) node (C2) {};
\draw[fill] (-60:3) circle (1pt) node (C3) {};
\draw[fill] (-60:4) circle (1pt) node (Cn) {};
\draw[->] (O) node[above left] {$\scriptstyle 1$}--(A1) node[above left] {$\scriptstyle 2$};
\draw[->] (A2) node[above left] {$\scriptstyle 3$}--(A1);
\draw[->] (A2)--(A3) node[above left] {$\scriptstyle 4$};
\draw[dotted] (A3) -- (An) node[above left] {$\scriptstyle a$};
\draw[->] (O)--(B1) node[above] {$\scriptstyle a+1$};
\draw[->] (B2) node[above] {$\scriptstyle a+2$}--(B1);
\draw[->] (B2)--(B3) node[above] {$\scriptstyle a+3$};
\draw[dotted] (B3)--(Bn) node[above] {$\scriptstyle a+b-1$};
\draw[->] (O)--(C1) node[above right] {$\scriptstyle a+b$};
\draw[->] (C2) node[above right] {$\scriptstyle a+b+1$}--(C1);
\draw[->] (C2)--(C3) node[above right] {$\scriptstyle a+b+2$};
\draw[dotted] (C3)--(Cn) node[above right] {$\scriptstyle a+b+c-2$};
\end{tikzpicture}
}

\subfigure[$(\ngraph(a,b,c), \nbasis(a,b,c))$]{
\begin{tikzpicture}[baseline=-.5ex,scale=0.6]
\draw[thick] (0,0) circle (3cm);
\draw[green, line cap=round, line width=5, opacity=0.5] (60:1) -- (50:1.5) (70:1.75) -- (50:2) (180:1) -- (170:1.5) (190:1.75) -- (170:2) (300:1) -- (290:1.5) (310:1.75) -- (290:2);
\draw[yellow, line cap=round, line width=5, opacity=0.5] (0,0) -- (60:1) (0,0) -- (180:1) (0,0) -- (300:1) (50:1.5) -- (70:1.75) (170:1.5) -- (190:1.75) (290:1.5) -- (310:1.75);
\draw[red, thick] (0,0) -- (0:3) (0,0) -- (120:3) (0,0) -- (240:3);
\draw[blue, thick, fill] (0,0) -- (60:1) circle (2pt) -- (100:3) (60:1) -- (50:1.5) circle (2pt) -- (20:3) (50:1.5) -- (70:1.75) circle (2pt) -- (80:3) (70:1.75) -- (50:2) circle (2pt) -- (40:3);
\draw[blue, thick, dashed] (50:2) -- (60:3);
\draw[blue, thick, fill] (0,0) -- (180:1) circle (2pt) -- (220:3) (180:1) -- (170:1.5) circle (2pt) -- (140:3) (170:1.5) -- (190:1.75) circle (2pt) -- (200:3) (190:1.75) -- (170:2) circle (2pt) -- (160:3);
\draw[blue, thick, dashed] (170:2) -- (180:3);
\draw[blue, thick, fill] (0,0) -- (300:1) circle (2pt) -- (340:3) (300:1) -- (290:1.5) circle (2pt) -- (260:3) (290:1.5) -- (310:1.75) circle (2pt) -- (320:3) (310:1.75) -- (290:2) circle (2pt) -- (280:3);
\draw[blue, thick, dashed] (290:2) -- (300:3);
\draw[thick, fill=white] (0,0) circle (2pt);
\curlybrace[]{10}{110}{3.2};
\draw (60:3.5) node[rotate=-30] {$a+1$};
\curlybrace[]{130}{230}{3.2};
\draw (180:3.5) node[rotate=90] {$b+1$};
\curlybrace[]{250}{350}{3.2};
\draw (300:3.5) node[rotate=30] {$c+1$};
\end{tikzpicture}
}
\subfigure[$(\bar\ngraph(a,b,c), \bar\nbasis(a,b,c))$]{
\begin{tikzpicture}[baseline=-.5ex,scale=0.6]
\draw[thick] (0,0) circle (3cm);
\draw[green, line cap=round, line width=5, opacity=0.5] (60:1) -- (50:1.5) (70:1.75) -- (50:2) (180:1) -- (170:1.5) (190:1.75) -- (170:2) (300:1) -- (290:1.5) (310:1.75) -- (290:2);
\draw[yellow, line cap=round, line width=5, opacity=0.5] (0,0) -- (60:1) (0,0) -- (180:1) (0,0) -- (300:1) (50:1.5) -- (70:1.75) (170:1.5) -- (190:1.75) (290:1.5) -- (310:1.75);
\draw[blue, thick] (0,0) -- (0:3) (0,0) -- (120:3) (0,0) -- (240:3);
\draw[red, thick, fill] (0,0) -- (60:1) circle (2pt) -- (100:3) (60:1) -- (50:1.5) circle (2pt) -- (20:3) (50:1.5) -- (70:1.75) circle (2pt) -- (80:3) (70:1.75) -- (50:2) circle (2pt) -- (40:3);
\draw[red, thick, dashed] (50:2) -- (60:3);
\draw[red, thick, fill] (0,0) -- (180:1) circle (2pt) -- (220:3) (180:1) -- (170:1.5) circle (2pt) -- (140:3) (170:1.5) -- (190:1.75) circle (2pt) -- (200:3) (190:1.75) -- (170:2) circle (2pt) -- (160:3);
\draw[red, thick, dashed] (170:2) -- (180:3);
\draw[red, thick, fill] (0,0) -- (300:1) circle (2pt) -- (340:3) (300:1) -- (290:1.5) circle (2pt) -- (260:3) (290:1.5) -- (310:1.75) circle (2pt) -- (320:3) (310:1.75) -- (290:2) circle (2pt) -- (280:3);
\draw[red, thick, dashed] (290:2) -- (300:3);
\draw[thick, fill=white] (0,0) circle (2pt);
\curlybrace[]{10}{110}{3.2};
\draw (60:3.5) node[rotate=-30] {$a+1$};
\curlybrace[]{130}{230}{3.2};
\draw (180:3.5) node[rotate=90] {$b+1$};
\curlybrace[]{250}{350}{3.2};
\draw (300:3.5) node[rotate=30] {$c+1$};
\end{tikzpicture}
}
\caption{The tripod $N$-graph $\ngraph(a,b,c)$ and the chosen set $\nbasis(a,b,c)$ of cycles}
\label{figure:tripod}
\end{figure}

Notice that the boundary of the Legendrian weave $\Legendrian(\ngraph(a,b,c))$ denoted by $\legendrian(a,b,c)$ is expressed as the braid word
\begin{align*}
\legendrian(a,b,c) &= \cl(\beta(a,b,c)),&
\beta(a,b,c)&=\sigma_2\sigma_1^{a+1}\sigma_2\sigma_1^{b+1}\sigma_2\sigma_1^{c+1}.
\end{align*}
Then this braid is equivalent to the following:
\begin{align*}
\beta(a,b,c)&=\sigma_2\sigma_1^{a+1}\sigma_2\sigma_1^{b+1}\sigma_2\sigma_1^{c+1}
=\sigma_2\sigma_1(\sigma_2\sigma_1)\sigma_2^a\sigma_1^{b-1}\sigma_2^c(\sigma_1\sigma_2)\sigma_1
=\Delta\sigma_1\sigma_2^a\sigma_1^{b-1}\sigma_2^c\Delta.
\end{align*}
Hence $\legendrian(a,b,c)\subset J^1\sphere^1$ corresponds to the rainbow closure of the braid $\beta_0(a,b,c)=\sigma_1\sigma_2^a\sigma_1^{b-1}\sigma_2^c$.
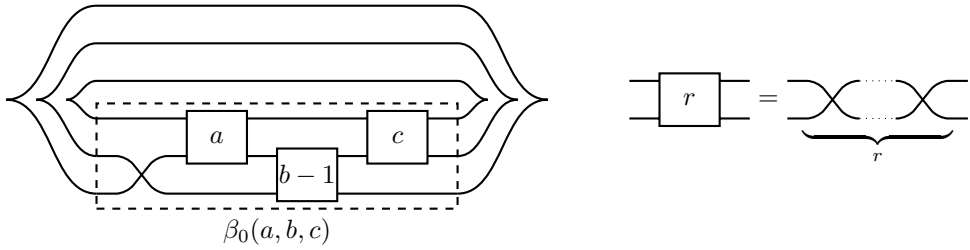
\begin{align*}
&\begin{tikzpicture}[yshift=0.75cm,baseline=-.5ex,xscale=0.8]
\draw[thick] (-3,-1)--(-1.5,-1) (-0.5,-1) -- (1.5, -1) (2.5,-1) -- (3,-1);
\draw[thick, rounded corners] (-3,-1.5) -- (-2.5,-1.5) -- (-2,-2) -- (0,-2) (-0.5,-1.5) -- (0, -1.5) (1, -1.5) -- (1.5,-1.5) (2.5, -1.5) -- (3,-1.5);
\draw[thick, rounded corners] (-3,-2) -- (-2.5,-2) -- (-2,-1.5) -- (-1.5,-1.5) (1,-2) -- (3,-2);
\draw[thick] (-1.5,-0.9) rectangle node {$a$} (-0.5, -1.6) (0,-2.1) rectangle node {$b-1$} (1,-1.4) (1.5,-1.6) rectangle node {$c$} (2.5,-0.9);
\draw[thick] (-3,-1) to[out=180,in=0] (-3.5,-0.75) to[out=0,in=180] (-3,-0.5) -- (3,-0.5) to[out=0,in=180] (3.5,-0.75) to[out=180,in=0] (3,-1);
\draw[thick] (-3,-1.5) to[out=180,in=0] (-4,-0.75) to[out=0,in=180] (-3,0) -- (3,0) to[out=0,in=180] (4,-0.75) to[out=180,in=0] (3,-1.5);
\draw[thick] (-3,-2) to[out=180,in=0] (-4.5,-0.75) to[out=0,in=180] (-3,0.5) -- (3,0.5) to[out=0,in=180] (4.5,-0.75) to[out=180,in=0] (3,-2);
\draw[thick, dashed] (-3,-2.2) rectangle (3,-0.8) (0,-2.2) node[below] {$\beta_0(a,b,c)$};
\end{tikzpicture}&
\begin{tikzpicture}[yshift=0.25cm,baseline=-.5ex,xscale=0.8]
\draw[thick] (-1.5,-0.5) -- (-1,-0.5) (-1.5,0) -- (-1,0) (0.5,-0.5) -- (0,-0.5) (0.5,0) -- (0,0) (-1,-0.6) rectangle node {$r$} (0,0.1);
\end{tikzpicture}&=
\begin{tikzpicture}[yshift=0.75cm,baseline=-.5ex,xscale=0.8]
\draw[thick, rounded corners] (-1.5,-0.5) -- (-1,-0.5) -- (-0.5, -1) -- (-0.3, -1) (-1.5,-1) -- (-1,-1) -- (-0.5, -0.5) -- (-0.3, -0.5) (1.5,-0.5) -- (1,-0.5) -- (0.5, -1) -- (0.3, -1) (1.5,-1) -- (1,-1) -- (0.5, -0.5) -- (0.3, -0.5);
\draw[dotted] (-0.3,-0.5) -- (0.3,-0.5) (-0.3,-1) -- node[below] {$\underbrace{\hphantom{\hspace{2cm}}}_r$} (0.3,-1);
\end{tikzpicture}
\end{align*}

\begin{remark}
One can easily check the quiver $\quiver^{\mathsf{brick}}(a,b,c)$ from the \emph{brick diagram} of $\beta_0(a,b,c)$ described in \cite{GSW2020b} looks as follows:
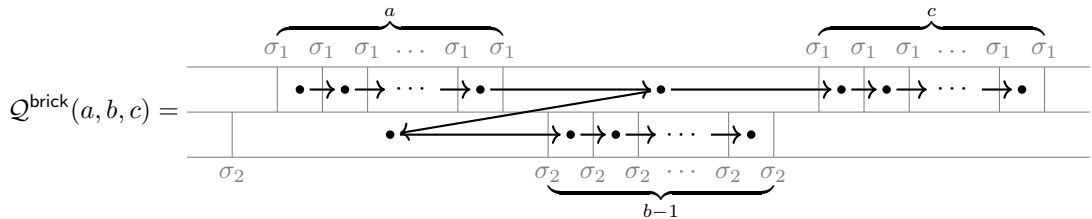
\[
\quiver^{\mathsf{brick}}(a,b,c)=\begin{tikzpicture}[baseline=-.5ex,scale=1.2]
\draw[gray] (-5,0.5) -- (5,0.5) (-5,0) -- (5,0) (-5,-0.5) -- (5,-0.5);
\draw[gray] (-4.5,0) -- (-4.5,-0.5) node[below] {$\sigma_2$} (-4,0) -- (-4, 0.5) node[above] {$\sigma_1$} (-3.5, 0) -- (-3.5, 0.5) node[above] {$\sigma_1$} (-3, 0) -- (-3, 0.5) node[above] {$\sigma_1$} (-2.5,0.5) node[above] {$\cdots$} (-2,0) -- (-2,0.5) node[above] {$\sigma_1$} (-1.5,0) -- (-1.5,0.5) node[above] {$\sigma_1$};
\draw[gray] (-1,0) -- (-1, -0.5) node[below] {$\sigma_2$} (-0.5,0) -- (-0.5, -0.5) node[below] {$\sigma_2$} (0, 0) -- (0, -0.5) node[below] {$\sigma_2$} (0.5,-0.5) node[below] {$\cdots$} (1,0) -- (1,-0.5) node[below] {$\sigma_2$} (1.5,0) -- (1.5, -0.5) node[below] {$\sigma_2$};
\draw[gray] (2,0) -- (2,0.5) node[above] {$\sigma_1$} (2.5,0) -- (2.5,0.5) node[above] {$\sigma_1$} (3,0) -- (3,0.5) node[above] {$\sigma_1$} (3.5,0.5) node[above] {$\cdots$} (4,0) -- (4,0.5) node[above] {$\sigma_1$} (4.5,0) -- (4.5,0.5) node[above] {$\sigma_1$};
\draw (-2.75,0.5) node[yshift=4ex] {$\overbrace{\hphantom{\hspace{3cm}}}^a$};
\draw (0.25,-0.5) node[yshift=-4ex] {$\underbrace{\hphantom{\hspace{3cm}}}_{b-1}$};
\draw (3.25,0.5) node[yshift=4ex] {$\overbrace{\hphantom{\hspace{3cm}}}^c$};
\draw[thick, fill] (-3.75, 0.25) circle (1pt) node (A1) {} (-3.25, 0.25) circle (1pt) node (A2) {} (-1.75, 0.25) circle (1pt) node (A3) {} (0.25, 0.25) circle (1pt) node (Aa) {} (-2.5,0.25) node (Adots) {$\cdots$} (2.25, 0.25) circle (1pt) node (B1) {} (2.75,0.25) circle (1pt) node (B2) {} (4.25, 0.25) circle (1pt) node (Bb) {} (3.5,0.25) node (Bdots) {$\cdots$};
\draw[thick, fill] (-2.75, -0.25) circle (1pt) node (C1) {} (-0.75, -0.25) circle (1pt) node (C2) {} (-0.25, -0.25) circle (1pt) node (C3) {} (1.25, -0.25) circle (1pt) node (Cc) {} (0.5, -0.25) node (Cdots) {$\cdots$};
\draw[thick,->] (A1) -- (A2);
\draw[thick,->] (A2) -- (Adots);
\draw[thick,->] (Adots) -- (A3);
\draw[thick,->] (A3) -- (Aa);
\draw[thick,->] (Aa) -- (B1);
\draw[thick,->] (B1) -- (B2);
\draw[thick,->] (B2) -- (Bdots);
\draw[thick,->] (Bdots) -- (Bb);
\draw[thick,->] (C1) -- (C2);
\draw[thick,->] (C2) -- (C3);
\draw[thick,->] (C3) -- (Cdots);
\draw[thick,->] (Cdots) -- (Cc);
\draw[thick,->] (Aa) -- (C1);
\end{tikzpicture}
\]
Then this quiver $\quiver^{\mathsf{brick}}(a,b,c)$ is obviously mutation equivalent to the bipartite quiver $\quiver(a,b,c)$.
\end{remark}

It is not hard to check that $\beta(1,b,c)$ is a stabilization of $\beta(\dynA_n)=\sigma_1^{n+3}$ for $n=b+c-1$ since
\begin{align*}
\beta(\dynA_n)&=\sigma_1^{n+3}=\sigma_1^{b+1}\sigma_1^{c+1}\stackrel{\Move{S}}{\longrightarrow}\sigma_2\sigma_1^2\sigma_2\sigma_1^{b+1}\sigma_2\sigma_1^{c+1}=\beta(1,b,c).
\end{align*}

\begin{lemma}\label{lemma:stabilized An}
The $N$-graph $\ngraph(1,b,c)$ is a stabilization of $\ngraph(\dynA_n)$ for $n=b+c-1$.
\end{lemma}
\begin{proof}
According to Remark~\ref{remark:stabilization} and Figure~\ref{fig:stab. on disk}, a stabilization of $\ngraph(\dynA_n)$ is given as the second picture in Figure~\ref{figure:proof of stabilization}.
Then by adding an annular $N$-graph corresponding to a sequence of $\Move{RIII}$, we obtain the third, which produces the fourth by applying the following generalized push-through move.
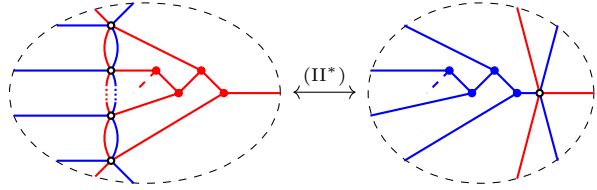
\[
\begin{tikzcd}
\begin{tikzpicture}[baseline=-.5ex,scale=0.6]
\draw[dashed] \boundellipse{0.75,0}{3}{2};
\draw[blue, thick](-1.2,1.5) -- (0,1.5) (-2.15,0.5) -- (0,0.5) (-2.15,-0.5) -- (0,-0.5) (-1.2,-1.5) -- (0,-1.5);
\draw[blue, thick] (0.5,2) --  (0,1.5) to[out=-60,in=60] (0,0.5) to[out=-60,in=100] (0.1,0.2) (0.1,-0.2) to[out=-100,in=60] (0,-0.5) to[out=-60,in=60] (0,-1.5) -- (0.5,-2);
\draw[blue, thick, densely dotted] (0.1, 0.2) -- (0.1, -0.2);
\draw[red, thick] (-0.35,1.85) --  (0,1.5) to[out=-120,in=120] (0,0.5) to[out=-120,in=80] (-0.1,0.2) (-0.1,-0.2) to[out=-80,in=120] (0,-0.5) to[out=-120,in=120] (0,-1.5) -- (-0.35,-1.85);
\draw[red, thick, densely dotted] (-0.1, 0.2) -- (-0.1, -0.2);
\draw[red, thick, fill] (3.75,0) -- (2.5,0) circle (2pt) -- (0,-1.5) (2.5,0) -- (2, 0.5) circle (2pt) -- (0,1.5) (2,0.5) -- (1.5,0) circle (2pt) -- (0,-0.5) (1.5,0) -- (1, 0.5) circle (2pt) -- (0,0.5);
\draw[red, thick, dashed] (1,0.5) -- (0.5,0);
\draw[fill=white, thick] (0,1.5) circle (2pt) (0,0.5) circle (2pt) (0,-0.5) circle (2pt) (0,-1.5) circle (2pt);
\end{tikzpicture}
\arrow[leftrightarrow, r,"\Move{II^*}"]&
\begin{tikzpicture}[baseline=-.5ex,scale=0.6]
\draw[dashed] \boundellipse{1.7,0}{2.5}{2};
\draw[blue, thick, fill] (3,0) -- (2.5,0) circle (2pt) -- (0,-1.5) (2.5,0) -- (2, 0.5) circle (2pt) -- (0,1.5) (2,0.5) -- (1.5,0) circle (2pt) -- (-0.75,-0.5) (1.5,0) -- (1, 0.5) circle (2pt) -- (-0.75,0.5);
\draw[blue, thick, dashed] (1,0.5) -- (0.5,0);
\draw[blue, thick] (3.4,1.5) -- (3,0) -- (3.4,-1.5);
\draw[red, thick] (2.5,1.9) -- (3,0) -- (2.5,-1.9) (3,0) -- (4.2,0);
\draw[fill=white, thick] (3,0) circle (2pt);
\end{tikzpicture}
\end{tikzcd}
\]

Now we add an annular $N$-graph consisting of $\Move{RIII}$'s as above to obtain the fifth $N$-graph in Figure~\ref{figure:proof of stabilization}, which is the same as $\ngraph(1,b,c)$ as desired up to Move~\Move{II} at the center.
\end{proof}

\begin{figure}[ht]
\[
\begin{tikzcd}
\ngraph(A_n)=\begin{tikzpicture}[baseline=-.5ex,scale=0.6]
\draw[thick] (0,0) circle (3cm);
\foreach \i in {120,240} {
\draw[blue, thick, fill] (0,0) -- (60+\i:1) circle (2pt) -- (100+\i:3) (60+\i:1) -- (50+\i:1.5) circle (2pt) -- (20+\i:3) (50+\i:1.5) -- (70+\i:1.75) circle (2pt) -- (80+\i:3) (70+\i:1.75) -- (50+\i:2) circle (2pt) -- (40+\i:3);
\draw[blue, thick, dashed] (50+\i:2) -- (60+\i:3);
}
\end{tikzpicture}
\arrow[r,"\Move{S}"]&
\begin{tikzpicture}[baseline=-.5ex,scale=0.6]
\draw[thick] (0,0) circle (3cm);
\foreach \i in {120,270} {
\draw[blue, thick, fill] (0,0) -- (60+\i:1) circle (2pt) -- (100+\i:3) (60+\i:1) -- (50+\i:1.5) circle (2pt) -- (20+\i:3) (50+\i:1.5) -- (70+\i:1.75) circle (2pt) -- (80+\i:3) (70+\i:1.75) -- (50+\i:2) circle (2pt) -- (40+\i:3);
\draw[blue, thick, dashed] (50+\i:2) -- (60+\i:3);
}
\draw[red, thick, fill] (240:3) -- (260:2) circle (2pt) -- (250:3) (260:2) -- (280:3);
\draw[blue, thick, rounded corners] (260:3) -- (265:2.5) -- (270:3);
\end{tikzpicture}\arrow[ld,"\text{Legendrian isotopy}"{sloped}]\\
\begin{tikzpicture}[baseline=-.5ex,scale=0.5]
\draw[thick](0,0) circle (5cm);
\draw[dashed] (0,0) circle (3cm) ;
\foreach \i in {120,270} {
\draw[blue, thick, fill] (60+\i:1) circle (2pt) (50+\i:1.5) circle (2pt) (70+\i:1.75) circle (2pt) (50+\i:2) circle (2pt);
\draw[blue, thick] (0,0) -- (60+\i:1) -- (100+\i:3) -- (100+\i:4) (60+\i:1) -- (50+\i:1.5) -- (20+\i:3) -- (20+\i:4) (50+\i:1.5) -- (70+\i:1.75) -- (80+\i:3) -- (80+\i:4) (70+\i:1.75) -- (50+\i:2) -- (40+\i:3) -- (40+\i:4);
\draw[blue, thick, dashed] (50+\i:2) -- (60+\i:3) -- (60+\i:4);
}
\draw[red, thick, fill] (260:2) circle (2pt);
\draw[red, thick] (240:5) -- (240:3) -- (260:2) (260:2) -- (250:3) -- (250:5) (260:2) -- (280:3) -- (290:4);
\draw[blue, thick, rounded corners] (270:5) -- (260:3) -- (265:2.5) -- (270:3) -- (290:4);
\draw[blue, thick] (140:4) -- (140:5) (160:4) -- (160:5) (200:4) -- (200:5) (220:4) -- (220:5);
\draw[blue, thick, dashed] (180:4) -- (180:5);
\draw[blue, thick] (290:4) to[out=360,in=240] (310:4) (330:4) to[out=40,in=280] (350:4) (350:4) to[out=420,in=-60] (10:4) (10:4) -- (30:5);
\draw[red, thick] (290:4) to[out=410,in=190] (310:4) (330:4) to[out=450,in=230] (350:4) (350:4) to[out=470,in=250] (370:4) (290:4) -- (290:5) (310:4) -- (310:5) (350:4) -- (350:5) (370:4) -- (370:5) (10:4) -- (50:5);
\draw[blue, thick, dashed] (310:4) to[out=380,in=260] (330:4);
\draw[red, thick, dashed] (310:4) to[out=430,in=210] (330:4) (330:4) -- (330:5);
\draw[thick, fill=white] (290:4) circle (2pt) (310:4) circle (2pt) (330:4) circle (2pt) (350:4) circle (2pt) (370:4) circle (2pt);
\end{tikzpicture}\arrow[r,"\Move{II^*}"]&
\begin{tikzpicture}[baseline=-.5ex,scale=0.6]
\draw[thick] (0,0) circle (3cm);
\foreach \i in {120} {
\draw[blue, thick, fill] (60+\i:1) circle (2pt) (50+\i:1.5) circle (2pt) (70+\i:1.75) circle (2pt) (50+\i:2) circle (2pt);
\draw[blue, thick] (0,0) -- (60+\i:1) -- (100+\i:3) (60+\i:1) -- (50+\i:1.5) -- (20+\i:3) (50+\i:1.5) -- (70+\i:1.75) -- (80+\i:3) (70+\i:1.75) -- (50+\i:2) -- (40+\i:3);
\draw[blue, thick, dashed] (50+\i:2) -- (60+\i:3);
}
\foreach \i in {250} {
\draw[red, thick, fill] (60+\i:1) circle (2pt) (50+\i:1.5) circle (2pt) (70+\i:1.75) circle (2pt) (50+\i:2) circle (2pt);
\draw[red, thick] (0,0) -- (60+\i:1) -- (100+\i:3) (60+\i:1) -- (50+\i:1.5) -- (20+\i:3) (50+\i:1.5) -- (70+\i:1.75) -- (80+\i:3) (70+\i:1.75) -- (50+\i:2) -- (40+\i:3);
\draw[red, thick, dashed] (50+\i:2) -- (60+\i:3);
}
\draw[red, thick, fill] (240:2) circle (2pt);
\draw[red, thick] (240:3) -- (240:2) (240:2) -- (0,0) -- (120:3) (240:2) -- (250:3);
\draw[blue, thick] (260:3) -- (0,0) -- (80:3);
\draw[thick, fill=white] (0,0) circle (2pt);
\end{tikzpicture}\arrow[ld,"\text{Legendrian isotopy}"{sloped}]\\
\begin{tikzpicture}[baseline=-.5ex,scale=0.5]
\draw[thick](0,0) circle (5cm);
\draw[dashed] (0,0) circle (3cm); 
\foreach \i in {120} {
\draw[blue, thick, fill] (60+\i:1) circle (2pt) (50+\i:1.5) circle (2pt) (70+\i:1.75) circle (2pt) (50+\i:2) circle (2pt);
\draw[blue, thick] (0,0) -- (60+\i:1) -- (100+\i:3) -- (100+\i:4) (60+\i:1) -- (50+\i:1.5) -- (20+\i:3) -- (20+\i:4) (50+\i:1.5) -- (70+\i:1.75) -- (80+\i:3) -- (80+\i:4) (70+\i:1.75) -- (50+\i:2) -- (40+\i:3) -- (40+\i:4);
\draw[blue, thick, dashed] (50+\i:2) -- (60+\i:3) -- (60+\i:4);
}
\foreach \i in {250} {
\draw[red, thick, fill] (60+\i:1) circle (2pt) (50+\i:1.5) circle (2pt) (70+\i:1.75) circle (2pt) (50+\i:2) circle (2pt);
\draw[red, thick] (0,0) -- (60+\i:1) -- (100+\i:3) -- (100+\i:4) (60+\i:1) -- (50+\i:1.5) -- (20+\i:3) -- (20+\i:4) (50+\i:1.5) -- (70+\i:1.75) -- (80+\i:3) -- (80+\i:4) (70+\i:1.75) -- (50+\i:2) -- (40+\i:3) -- (40+\i:4);
\draw[red, thick, dashed] (50+\i:2) -- (60+\i:3) -- (60+\i:4);
}
\draw[red, thick, fill] (240:2) circle (2pt);
\draw[red, thick] (240:5) -- (240:3) -- (240:2) (240:2) -- (0,0) -- (120:5) (240:2) -- (250:3) -- (270:4);
\draw[blue, thick] (270:4) -- (260:3) -- (0,0) -- (80:5);
\draw[thick, fill=white] (0,0) circle (2pt);
\draw[blue, thick] (140:4) -- (140:5) (160:4) -- (160:5) (200:4) -- (200:5) (220:4) -- (220:5);
\draw[blue, thick, dashed] (180:4) -- (180:5);
\begin{scope}[rotate=-20]
\draw[red, thick] (290:4) to[out=360,in=240] (310:4) (330:4) to[out=40,in=280] (350:4) (350:4) to[out=420,in=-60] (10:4) (10:4) -- (30:5);
\draw[blue, thick] (290:4) to[out=410,in=190] (310:4) (330:4) to[out=450,in=230] (350:4) (350:4) to[out=470,in=250] (370:4) (290:4) -- (290:5) (310:4) -- (310:5) (350:4) -- (350:5) (370:4) -- (370:5) (10:4) -- (50:5);
\draw[red, thick, dashed] (310:4) to[out=380,in=260] (330:4);
\draw[blue, thick, dashed] (310:4) to[out=430,in=210] (330:4) (330:4) -- (330:5);
\draw[thick, fill=white] (290:4) circle (2pt) (310:4) circle (2pt) (330:4) circle (2pt) (350:4) circle (2pt) (370:4) circle (2pt);
\end{scope}
\end{tikzpicture}\arrow[r,"\Move{II^*}"]&
\begin{tikzpicture}[baseline=-.5ex,scale=0.6]
\draw[thick] (0,0) circle (3cm);
\draw[blue, thick, fill] (180:0.5) -- (180:1) circle (2pt) -- (220:3) (180:1) -- (170:1.5) circle (2pt) -- (140:3) (170:1.5) -- (190:1.75) circle (2pt) -- (200:3) (190:1.75) -- (170:2) circle (2pt) -- (160:3);
\draw[blue, thick, dashed] (170:2) -- (180:3);
\draw[blue, thick, fill] (300:0.5) -- (300:1) circle (2pt) -- (340:3) (300:1) -- (290:1.5) circle (2pt) -- (260:3) (290:1.5) -- (310:1.75) circle (2pt) -- (320:3) (310:1.75) -- (290:2) circle (2pt) -- (280:3);
\draw[blue, thick, dashed] (290:2) -- (300:3);
\draw[blue, thick] (80:3) -- (180:0.5) -- (300:0.5) -- (40:3);
\draw[red, thick, fill] (240:3) -- (240:1) circle (2pt) -- (180:0.5) (240:1) -- (300:0.5);
\draw[red, thick] (120:3) -- (180:0.5) to[out=30,in=90] (300:0.5) -- (0:3);
\draw[fill=white, thick] (180:0.5) circle (2pt) (300:0.5) circle (2pt);
\end{tikzpicture}\arrow[ld, "\Move{II}"']\\
\begin{tikzpicture}[baseline=-.5ex,scale=0.6]
\draw[thick] (0,0) circle (3cm);
\draw[red, thick] (0,0) -- (0:3) (0,0) -- (120:3) (0,0) -- (240:3);
\draw[blue, thick, fill] (0,0) -- (60:1) circle (2pt) -- (40:3) (60:1) -- (80:3);
\draw[blue, thick, fill] (0,0) -- (180:1) circle (2pt) -- (220:3) (180:1) -- (170:1.5) circle (2pt) -- (140:3) (170:1.5) -- (190:1.75) circle (2pt) -- (200:3) (190:1.75) -- (170:2) circle (2pt) -- (160:3);
\draw[blue, thick, dashed] (170:2) -- (180:3);
\draw[blue, thick, fill] (0,0) -- (300:1) circle (2pt) -- (340:3) (300:1) -- (290:1.5) circle (2pt) -- (260:3) (290:1.5) -- (310:1.75) circle (2pt) -- (320:3) (310:1.75) -- (290:2) circle (2pt) -- (280:3);
\draw[blue, thick, dashed] (290:2) -- (300:3);
\draw[thick, fill=white] (0,0) circle (2pt);
\end{tikzpicture}
=\ngraph(1,b,c)
\end{tikzcd}
\]
\caption{The $N$-graph $\ngraph(1,b,c)$ is a stabilization of $\ngraph(\dynA_{n})$ for $n=b+c-1$}
\label{figure:proof of stabilization}
\end{figure}

\begin{definition}
Let $(\ngraph, \nbasis)$ and $(\ngraph', \nbasis')$ be pairs of $N$-graphs and good tuples of cycles.
We say that $(\ngraph, \nbasis)$ is \emph{Legendrian mutation equivalent} to $(\ngraph', \nbasis')$ if there exists a sequence $\mutation$ of Legendrian mutations which sends $(\ngraph, \nbasis)$ to $(\ngraph', \nbasis')$ up to equivalence. That is,
\[
[\mutation(\ngraph,\nbasis)] = [(\ngraph', \nbasis')].
\]

In particular, $(\ngraph, \nbasis)$ is said to be \emph{of type $\dynX$} or \emph{of type $(a,b,c)$} if it is Legendrian mutation equivalent to $(\ngraph(\dynX), \nbasis(\dynX))$ or $(\ngraph(a,b,c),\nbasis(a,b,c))$, respectively.
\end{definition}

\subsection{Coxeter mutation for tripods}

For a bipartite quiver $\quiver$, we have two sets of vertices $I_+$ and
$I_-$ so that all edges are oriented from $I_+$ to $I_-$.
By definition, for a tripod $\quiver(a,b,c)$,  the central vertex $1$ is lying in $I_+$. 
Let $\mutation_+$ and $\mutation_-$ be sequences of mutations defined by 
compositions of mutations corresponding to each and every vertex in $I_+$ and 
$I_-$, respectively.
A Coxeter mutation $\qcoxeter$ is the composition
\[
\qcoxeter=\mutation_- \mutation_+ = \prod_{i\in I_-} \mutation_i\cdot\prod_{i\in I_+} \mutation_i.
\]
\begin{remark}
For any sequence $\mutation$ of mutations, we will use the right-to-left convention. Namely, the rightmost mutation will be applied first on the quiver $\quiver$.
\end{remark}

Similarly, we define the Legendrian Coxeter mutation, which will be denoted by $\ncoxeter$, on a bipartite $N$-graph $\ngraph$ as follows:
\begin{definition}[Legendrian Coxeter mutation]
For a bipartite $N$-graph $\ngraph$ with decomposed sets of cycles $\nbasis=\nbasis_+\cup\nbasis_-$, we define the \emph{Legendrian Coxeter mutation} $\ncoxeter$ as the composition of Legendrian mutations
\[
\ncoxeter = \prod_{\cycle\in \nbasis_-}\mutation_\cycle \cdot \prod_{\cycle\in\nbasis_+}\mutation_\cycle.
\]
\end{definition}

\begin{lemma}\label{lemma:Legendriam Coxeter mutation of type An}
The effect of the Legendrian Coxeter mutation on $(\ngraph(\dynA_n),\nbasis(\dynA_n))$ is the clockwise $\frac{2\pi}{n+3}$-rotation.
\end{lemma}
\begin{proof}
We may assume that the Coxeter element $\ncoxeter$ can be represented by the sequence
\[
\ncoxeter=\mutation_-\mutation_+=(\mutation_{\cycle_2}\mutation_{\cycle_4}\mutation_{\cycle_6}\cdots)(\mutation_{\cycle_1}\mutation_{\cycle_3}\mutation_{\cycle_5}\dots).
\]
Then the action of $\ncoxeter$ on $\ngraph(\dynA_n)$ is as depicted in Figure~\ref{figure:Legendrian Coxeter mutation on An}, which is nothing but the clockwise $\frac{2\pi}{n+3}$-rotation of the original $N$-graph $(\ngraph(\dynA_n),\nbasis(\dynA_n))$ as claimed.
\end{proof}

\begin{figure}[ht]
\[
\begin{tikzcd}
\begin{tikzpicture}[baseline=-.5ex,scale=0.6]
\draw[thick] (0,0) circle (3);
\draw[green, line width=5, opacity=0.5] (-1.5,0.5) -- (-0.5, -0.5) (1, 0) -- (1.5, -0.5);
\draw[yellow, line width=5, opacity=0.5] (-2.5,-0.5) -- (-1.5, 0.5) (-0.5, -0.5) -- (0, 0) (1.5, -0.5) -- (2.5, 0.5);
\draw[blue, thick, fill] (0:3) -- (2.5,0.5) circle (2pt) -- (45:3) (2.5,0.5) -- (1.5,-0.5) circle (2pt) -- (-45:3) (1.5,-0.5) -- (1,0) (0.5,0) node {$\cdots$} (0,0) -- (-0.5, -0.5) circle (2pt) -- (-90:3) (-0.5, -0.5) -- (-1.5, 0.5) circle (2pt) -- (135:3) (-1.5, 0.5) -- (-2.5, -0.5) circle (2pt) -- (-135:3);
\draw[blue, thick] (-2.5,-0.5) -- (-180:3);
\draw (0,-3) node[below] {$(\ngraph(\dynA_n),\nbasis(\dynA_n))$};
\end{tikzpicture}
\arrow[r,|->,"\mutation_+"] &
\begin{tikzpicture}[baseline=-.5ex, scale=0.6]
\begin{scope}[yscale=-1]
\draw[thick] (0,0) circle (3);
\draw[green, line width=5, opacity=0.5] (-1.5,0.5) -- (-0.5, -0.5) (1, 0) -- (1.5, -0.5);
\draw[yellow, line width=5, opacity=0.5] (-2.5,-0.5) -- (-1.5, 0.5) (-0.5, -0.5) -- (0, 0) (1.5, -0.5) -- (2.5, 0.5);
\draw[blue, thick, fill] (0:3) -- (2.5,0.5) circle (2pt) -- (45:3) (2.5,0.5) -- (1.5,-0.5) circle (2pt) -- (-45:3) (1.5,-0.5) -- (1,0) (0.5,0) node {$\cdots$} (0,0) -- (-0.5, -0.5) circle (2pt) -- (-90:3) (-0.5, -0.5) -- (-1.5, 0.5) circle (2pt) -- (135:3) (-1.5, 0.5) -- (-2.5, -0.5) circle (2pt) -- (-135:3);
\draw[blue, thick] (-2.5,-0.5) -- (-180:3);
\end{scope}
\draw (0,-3) node[below] {$\mutation_+(\ngraph(\dynA_n),\nbasis(\dynA_n))$};
\end{tikzpicture}
\arrow[r,|->,"\mutation_-"] &
\begin{tikzpicture}[baseline=-.5ex, scale=0.6]
\begin{scope}[rotate=-45]
\draw[thick] (0,0) circle (3);
\draw[green, line width=5, opacity=0.5] (-1.5,0.5) -- (-0.5, -0.5) (1, 0) -- (1.5, -0.5);
\draw[yellow, line width=5, opacity=0.5] (-2.5,-0.5) -- (-1.5, 0.5) (-0.5, -0.5) -- (0, 0) (1.5, -0.5) -- (2.5, 0.5);
\draw[blue, thick, fill] (0:3) -- (2.5,0.5) circle (2pt) -- (45:3) (2.5,0.5) -- (1.5,-0.5) circle (2pt) -- (-45:3) (1.5,-0.5) -- (1,0) (0.5,0) node {\rotatebox{-45}{$\cdots$}} (0,0) -- (-0.5, -0.5) circle (2pt) -- (-90:3) (-0.5, -0.5) -- (-1.5, 0.5) circle (2pt) -- (135:3) (-1.5, 0.5) -- (-2.5, -0.5) circle (2pt) -- (-135:3);
\draw[blue, thick] (-2.5,-0.5) -- (-180:3);
\end{scope}
\draw (0,-3) node[below] {$\ncoxeter(\ngraph(\dynA_n),\nbasis(\dynA_n))$};
\end{tikzpicture}
\end{tikzcd}
\]
\caption{Legendrian Coxeter mutation $\ncoxeter$ on $(\ngraph(\dynA_n), \nbasis(\dynA_n))$}
\label{figure:Legendrian Coxeter mutation on An}
\end{figure}

\begin{remark}
The order of the Coxeter mutation is either $(n+3)/2$ if $n$ is odd or $n+3$ otherwise.
Since the Coxeter number $h=n+1$ for $\dynA_n$, this verifies Lemma~\ref{lemma:order of coxeter mutation}.
\end{remark}

Let $(\ngraph, \nbasis,\flags)$ be a triple of a good $N$-graph, a good tuple of cycles and flags $\flags$ on $\legendrian$.
Suppose that the quiver $\quiver(\ngraph,\nbasis)$ is bipartite and $\ncoxeter(\ngraph,\nbasis)$ is well-defined.
Then by Proposition~\ref{proposition:equivariance of mutations}, we have
\[
\Psi(\ncoxeter(\ngraph,\nbasis),\flags) = \qcoxeter(\Psi(\ngraph,\nbasis,\flags)).
\]

In particular, for quivers of type $\dynA_n$ or tripods we have the following corollary.
\begin{corollary}\label{corollary:Coxeter mutations}
For each $n\ge 1$ and $a,b,c\ge 1$, the Legendrian Coxeter mutation $\ncoxeter$ on $(\ngraph(\dynA_n),\nbasis(\dynA_n))$ or $(\ngraph(a,b,c),\nbasis(a,b,c))$ corresponds to the Coxeter mutation $\qcoxeter$ on $\quiver(\dynA_n)$ or $\quiver(a,b,c)$, respectively.
\end{corollary}

By the mutation convention mentioned above, for each tripod $\ngraph(a,b,c)$, we always take a mutation at the central $\sfY$-cycle $\cycle_1$ first.
After the Legendrian mutation on $(\ngraph(a,b,c),\nbasis(a,b,c))$ at $\cycle_1$, we have the $N$-graph on the left in Figure~\ref{figure:center mutation}.
Then there are three shaded regions that we can apply the generalized push-through moves so that we obtain the $N$-graph on the right in Figure~\ref{figure:center mutation}.
\begin{figure}[ht]
\subfigure[\label{figure:center mutation}After the mutation at the central vertex]{
\begin{tikzcd}[ampersand replacement=\&]
\begin{tikzpicture}[baseline=-.5ex,scale=0.8]
\begin{scope}
\fill[opacity=0.1](85:3) to[out=-90,in=150] (60:1.3) arc (60:-60:0.3) arc (120:240:0.7) to[out=-30,in=150] (-25:3) arc (-25:85:3);
\end{scope}
\begin{scope}[rotate=120]
\fill[opacity=0.1](85:3) to[out=-90,in=150] (60:1.3) arc (60:-60:0.3) arc (120:240:0.7) to[out=-30,in=150] (-25:3) arc (-25:85:3);
\end{scope}
\begin{scope}[rotate=240]
\fill[opacity=0.1](85:3) to[out=-90,in=150] (60:1.3) arc (60:-60:0.3) arc (120:240:0.7) to[out=-30,in=150] (-25:3) arc (-25:85:3);
\end{scope}
\draw[thick] (0,0) circle (3cm);
\draw[green, line cap=round, line width=5, opacity=0.5] (50:1.5) to[out=-60,in=60] (0:1) -- (-60:1) (70:1.75) -- (50:2) (170:1.5) to[out=60,in=180] (120:1) -- (60:1) (190:1.75) -- (170:2) (290:1.5) to[out=180,in=300] (240:1) -- (180:1) (310:1.75) -- (290:2);
\draw[yellow, line cap=round, line width=5, opacity=0.5] (0,0) -- (60:1) (0,0) -- (180:1) (0,0) -- (300:1) (50:1.5) -- (70:1.75) (170:1.5) -- (190:1.75) (290:1.5) -- (310:1.75);
\draw[blue, thick] (0,0) -- (0:1) (0,0) -- (120:1) (0,0) -- (240:1);
\draw[red, thick, fill] (0,0) -- (60:1) circle (2pt) (0,0) -- (180:1) circle (2pt) (0,0) -- (300:1) circle (2pt);
\draw[red, thick] (0:1) -- (0:3) (120:1) -- (120:3) (240:1) -- (240:3);
\draw[red, thick] (60:1) -- (120:1) -- (180:1) -- (240:1) -- (300:1) -- (0:1) -- cycle;
\draw[blue, thick] (100:3) to[out=-80,in=60] (120:1) (-20:3) to[out=-200,in=-60] (0:1) (220:3) to[out=40,in=180] (240:1);
\draw[blue, thick] (50:1.5) to[out=-60,in=60] (0:1) (170:1.5) to[out=60,in=180] (120:1) (290:1.5) to[out=180,in=300] (240:1);
\draw[blue, thick, fill] (50:1.5) circle (2pt) -- (20:3) (50:1.5) -- (70:1.75) circle (2pt) -- (80:3) (70:1.75) -- (50:2) circle (2pt) -- (40:3);
\draw[blue, thick, dashed] (50:2) -- (60:3);
\draw[blue, thick, fill] (170:1.5) circle (2pt) -- (140:3) (170:1.5) -- (190:1.75) circle (2pt) -- (200:3) (190:1.75) -- (170:2) circle (2pt) -- (160:3);
\draw[blue, thick, dashed] (170:2) -- (180:3);
\draw[blue, thick, fill] (290:1.5) circle (2pt) -- (260:3) (290:1.5) -- (310:1.75) circle (2pt) -- (320:3) (310:1.75) -- (290:2) circle (2pt) -- (280:3);
\draw[blue, thick, dashed] (290:2) -- (300:3);
\draw[thick,fill=white] (0:1) circle (2pt) (120:1) circle (2pt) (240:1) circle (2pt);
\draw[thick, fill=white] (0,0) circle (2pt);
\end{tikzpicture}\arrow[r,"\Move{II^*}"]\&
\begin{tikzpicture}[baseline=-.5ex,scale=0.6]
\draw[thick] (0,0) circle (5cm);
\draw[dashed]  (0,0) circle (3cm);
\fill[opacity=0.1, even odd rule] (0,0) circle (3) (0,0) circle (5);
\foreach \i in {1,2,3} {
\begin{scope}[rotate=\i*120]
\begin{scope}[shift=(60:0.5)]
\fill[opacity=0.1, rounded corners] (0,0) -- (0:2) arc (0:120:2) -- cycle;
\end{scope}
\draw[green, line cap=round, line width=5, opacity=0.5] (60:1) -- (70:1.5) (90:1.75) -- (70:2);
\draw[yellow, line cap=round, line width=5, opacity=0.5] (0,0) -- (60:1) (70:1.5) -- (90:1.75);
\draw[blue, thick, rounded corners] (0,0) -- (0:3.4) to[out=-75,in=80] (-40:4);
\draw[red, thick, fill] (0,0) -- (60:1) circle (2pt) (60:1) -- (70:1.5) circle (2pt) -- (90:1.75) circle (2pt) -- (70:2) circle (2pt);
\draw[red, thick, dashed, rounded corners] (70:2) -- (60:2.8) -- (60:3.3) to[out=0,in=220] (40:4) (40:4) to[out=120,in=-20] (60:4);
\draw[red, thick, rounded corners] (70:1.5) -- (40:2.8) -- (40:3.3) to[out=-20,in=200] (20:4) (70:2) -- (80:2.8) -- (80:3.3) to[out=20,in=240] (60:4) (90:1.75) -- (100:2.8) -- (100:3.3) to[out=40,in=260] (80:4);
\draw[red, thick, rounded corners] (60:1) -- (20:3) -- (20:3.5) to[out=-70,in=50] (-40:4) (20:4) to[out=-50,in=120] (0:4.5) -- (0:5);
\draw[red, thick] (20:4) to[out=100,in=-40] (40:4) (60:4) to[out=140,in=0] (80:4);
\draw[blue, thick] (20:5) -- (20:4) to[out=140,in=-80] (40:4) (60:5) -- (60:4) to[out=180,in=-40] (80:4) -- (80:5);
\draw[blue, thick, rounded corners] (20:4) to[out=-70,in=100] (-20:4.5) -- (-20:5);
\draw[blue, thick, dashed] (40:4) to[out=160,in=-60] (60:4) (40:4) -- (40:5);
\draw[fill=white, thick] (20:4) circle (2pt) (40:4) circle (2pt) (60:4) circle (2pt) (80:4) circle (2pt) (-40:4) circle (2pt);
\end{scope}
\draw[fill=white, thick] (0,0) circle (2pt);
}
\end{tikzpicture}
\end{tikzcd}}
\subfigure[\label{figure:coxeter mutation}After Legendrian Coxeter mutation]{$
\begin{tikzpicture}[baseline=-.5ex,scale=0.6]
\draw[thick] (0,0) circle (5cm);
\draw[dashed]  (0,0) circle (3cm);
\fill[opacity=0.1, even odd rule] (0,0) circle (3) (0,0) circle (5);
\foreach \i in {1,2,3} {
\begin{scope}[rotate=\i*120]
\draw[green, line cap=round, line width=5, opacity=0.5] (60:1) -- (50:1.5) (70:1.75) -- (50:2);
\draw[yellow, line cap=round, line width=5, opacity=0.5] (0,0) -- (60:1) (50:1.5) -- (70:1.75);
\draw[blue, thick, rounded corners] (0,0) -- (0:3.4) to[out=-75,in=80] (-40:4);
\draw[red, thick, fill] (0,0) -- (60:1) circle (2pt) (60:1) -- (50:1.5) circle (2pt) -- (70:1.75) circle (2pt) -- (50:2) circle (2pt);
\draw[red, thick, dashed, rounded corners] (50:2) -- (60:2.8) -- (60:3.3) to[out=0,in=220] (40:4) (40:4) to[out=120,in=-20] (60:4);
\draw[red, thick, rounded corners] (50:2) -- (40:2.8) -- (40:3.3) to[out=-20,in=200] (20:4) (70:1.75) -- (80:2.8) -- (80:3.3) to[out=20,in=240] (60:4) (60:1) -- (100:2.8) -- (100:3.3) to[out=40,in=260] (80:4);
\draw[red, thick, rounded corners] (50:1.5) -- (20:3) -- (20:3.5) to[out=-70,in=50] (-40:4) (20:4) to[out=-50,in=120] (0:4.5) -- (0:5);
\draw[red, thick] (20:4) to[out=100,in=-40] (40:4) (60:4) to[out=140,in=0] (80:4);
\draw[blue, thick] (20:5) -- (20:4) to[out=140,in=-80] (40:4) (60:5) -- (60:4) to[out=180,in=-40] (80:4) -- (80:5);
\draw[blue, thick, rounded corners] (20:4) to[out=-70,in=100] (-20:4.5) -- (-20:5);
\draw[blue, thick, dashed] (40:4) to[out=160,in=-60] (60:4) (40:4) -- (40:5);
\draw[fill=white, thick] (20:4) circle (2pt) (40:4) circle (2pt) (60:4) circle (2pt) (80:4) circle (2pt) (-40:4) circle (2pt);
\end{scope}
\draw[fill=white, thick] (0,0) circle (2pt);
}
\end{tikzpicture}
$}
\caption{Legendrian Coxeter mutation for $(\ngraph(a,b,c),\nbasis(a,b,c))$} 
\end{figure}
Notice that in each triangular shaded region, the $N$-subgraph looks like the $N$-graph of type $\dynA_{a-1}, \dynA_{b-1}$ or $\dynA_{c-1}$.
Moreover, the mutations corresponding to the rest sequence is just a composition 
of Coxeter mutations of type $\dynA_{a-1},\dynA_{b-1}$ and $\dynA_{c-1}$, 
which are essentially the same as the clock wise rotations.
Therefore, the result of the Coxeter mutation will be given as depicted in 
Figure~\ref{figure:coxeter mutation}.

Then one can observe that this is very similar to the original $N$-graph $\ngraph(a,b,c)$.
Indeed, the inside is identical to $\ngraph(a,b,c)$ but the colors are switched, which is $\bar \ngraph(a,b,c)$ by definition.
The complement of $\bar \ngraph(a,b,c)$ in $\qcoxeter(\ngraph(a,b,c),\nbasis(a,b,c))$ is an annular $N$-graph.

\begin{definition}[Coxeter padding]
For each triple $a,b,c$, the annular $N$-graph depicted in Figure~\ref{figure:coxeter padding} is denoted by $\coxeterpadding(a,b,c)$ and called the \emph{Coxeter padding} of type $(a,b,c)$.
We also denote the Coxeter padding with color switched by $\bar \coxeterpadding(a,b,c)$.
\end{definition}

\begin{figure}[ht]
\subfigure[$\coxeterpadding(a,b,c)$]{\makebox[0.48\textwidth]{
$
\begin{tikzpicture}[baseline=-.5ex,scale=0.6]
\draw[thick] (0,0) circle (5) (0,0) circle (3);
\foreach \i in {1,2,3} {
\begin{scope}[rotate=\i*120]
\draw[blue, thick, rounded corners] (0:3) -- (0:3.4) to[out=-75,in=80] (-40:4);
\draw[red, thick, dashed, rounded corners] (60:3) -- (60:3.3) to[out=0,in=220] (40:4) (40:4) to[out=120,in=-20] (60:4);
\draw[red, thick, rounded corners] (40:3) -- (40:3.3) to[out=-20,in=200] (20:4) (80:3) -- (80:3.3) to[out=20,in=240] (60:4) (100:3) -- (100:3.3) to[out=40,in=260] (80:4);
\draw[red, thick, rounded corners] (20:3) -- (20:3.5) to[out=-70,in=50] (-40:4) (20:4) to[out=-50,in=120] (0:4.5) -- (0:5);
\draw[red, thick] (20:4) to[out=100,in=-40] (40:4) (60:4) to[out=140,in=0] (80:4);
\draw[blue, thick] (20:5) -- (20:4) to[out=140,in=-80] (40:4) (60:5) -- (60:4) to[out=180,in=-40] (80:4) -- (80:5);
\draw[blue, thick, rounded corners] (20:4) to[out=-70,in=100] (-20:4.5) -- (-20:5);
\draw[blue, thick, dashed] (40:4) -- (40:5) (40:4) to[out=160,in=-60] (60:4);
\draw[fill=white, thick] (20:4) circle (2pt) (40:4) circle (2pt) (60:4) circle (2pt) (80:4) circle (2pt) (-40:4) circle (2pt);
\end{scope}
}
\end{tikzpicture}$
}}
\subfigure[$\bar\coxeterpadding(a,b,c)$]{\makebox[0.48\textwidth]{
$
\begin{tikzpicture}[baseline=-.5ex,scale=0.6]
\draw[thick] (0,0) circle (5) (0,0) circle (3);
\foreach \i in {1,2,3} {
\begin{scope}[rotate=\i*120]
\draw[red, thick, rounded corners] (0:3) -- (0:3.4) to[out=-75,in=80] (-40:4);
\draw[blue, thick, dashed, rounded corners] (60:3) -- (60:3.3) to[out=0,in=220] (40:4) (40:4) to[out=120,in=-20] (60:4);
\draw[blue, thick, rounded corners] (40:3) -- (40:3.3) to[out=-20,in=200] (20:4) (80:3) -- (80:3.3) to[out=20,in=240] (60:4) (100:3) -- (100:3.3) to[out=40,in=260] (80:4);
\draw[blue, thick, rounded corners] (20:3) -- (20:3.5) to[out=-70,in=50] (-40:4) (20:4) to[out=-50,in=120] (0:4.5) -- (0:5);
\draw[blue, thick] (20:4) to[out=100,in=-40] (40:4) (60:4) to[out=140,in=0] (80:4);
\draw[red, thick] (20:5) -- (20:4) to[out=140,in=-80] (40:4) (60:5) -- (60:4) to[out=180,in=-40] (80:4) -- (80:5);
\draw[red, thick, rounded corners] (20:4) to[out=-70,in=100] (-20:4.5) -- (-20:5);
\draw[red, thick, dashed] (40:4) -- (40:5) (40:4) to[out=160,in=-60] (60:4);
\draw[fill=white, thick] (20:4) circle (2pt) (40:4) circle (2pt) (60:4) circle (2pt) (80:4) circle (2pt) (-40:4) circle (2pt);
\end{scope}
}
\end{tikzpicture}$
}}
\caption{Coxeter paddings $\coxeterpadding(a,b,c)$ and $\bar\coxeterpadding(a,b,c)$}
\label{figure:coxeter padding}
\end{figure}

Notice that two Coxeter paddings $\coxeterpadding(a,b,c)$ and $\bar \coxeterpadding(a,b,c)$ can be glued without any ambiguity
and so we can also pile up Coxeter paddings $\coxeterpadding(a,b,c)$ and $\bar \coxeterpadding(a,b,c)$ alternatively as many times as we want.

We also define the concatenation of the Coxeter padding $\bar\coxeterpadding(a,b,c)$ on the pair $(\ngraph(a,b,c),\nbasis(a,b,c))$ as the pair $(\ngraph', \nbasis')$ such that
\begin{enumerate}
\item the $N$-graph $\ngraph'$ is obtained by gluing $\bar\coxeterpadding(a,b,c)$ on $\ngraph(a,b,c)$, and 
\item the tuple $\nbasis'$ of cycles is the set of $\sfI$- and $\sfY$-cycles identified with $\nbasis(a,b,c)$ in a canonical way.
\end{enumerate}

\begin{proposition}\label{proposition:effect of Legendrian Coxeter mutation}
The Legendrian Coxeter mutation on $(\ngraph(a,b,c), \nbasis(a,b,c))$ or $(\bar\ngraph(a,b,c), \bar\nbasis(a,b,c))$ is given as the concatenation
\begin{align*}
\ncoxeter(\ngraph(a,b,c), \nbasis(a,b,c)) &= \coxeterpadding(a,b,c)(\bar \ngraph(a,b,c), \bar\nbasis(a,b,c)),\\
\ncoxeter(\bar \ngraph(a,b,c), \nbasis(a,b,c)) &= \bar \coxeterpadding(a,b,c) (\ngraph(a,b,c), \nbasis(a,b,c)).
\end{align*}
\end{proposition}
\begin{proof}
This follows directly from the above observation.
\end{proof}

It is important that this proposition holds only when we take the Legendrian Coxeter mutation on the very standard $N$-graph with the tuple of cycles $(\ngraph(a,b,c), \nbasis(a,b,c))$.
Otherwise, the Legendrian Coxeter mutation will not be expressed as simple as above.

\begin{theorem}\label{theorem:infinite fillings}
For $a,b,c\ge 1$ with $\frac 1a+\frac1b+\frac1c\le 1$,
The Legendrian knot or link $\legendrian(a,b,c)$ in $J^1\sphere^1$ admits infinitely many distinct exact embedded Lagrangian fillings.
\end{theorem}
\begin{proof}
By Proposition~\ref{proposition:effect of Legendrian Coxeter mutation}, the effect of the Legendrian Coxeter mutation on $(\ngraph(a,b,c), \nbasis(a,b,c))$ is just to attach the Coxeter padding on $(\bar\ngraph(a,b,c),\bar\nbasis(a,b,c))$.
In particular, as mentioned earlier, for each $r\ge 0$, the iterated Legendrian Coxeter mutation
\[
\ncoxeter^r(\ngraph(a,b,c), \nbasis(a,b,c))
\]
is well-defined.
Each of these $N$-graphs define a Legendrian weave $\Legendrian(\ncoxeter^r(\ngraph(a,b,c), \nbasis(a,b,c)))$, whose Lagrangian projection is a Lagrangian filling 
\[
L_r(a,b,c)\coloneqq(\pi\circ\iota)(\Legendrian(\ncoxeter^r(\ngraph(a,b,c), \nbasis(a,b,c)))
\]
as desired. Therefore it suffices to prove that Lagrangians $L_r(a,b,c)$ for $r\ge 0$ are pairwise distinct up to exact Lagrangian isotopy, when $\frac1a+\frac1b+\frac1c\le 1$.

Now suppose that $\frac1a+\frac1b+\frac1c\le1$, or equivalently, $\quiver(a,b,c)$ is of infinite type. 
Then the order of the Coxeter mutation is infinite by Lemma~\ref{lemma:order of coxeter mutation} and so is the order of the Legendrian Coxeter mutation by Corollary~\ref{corollary:Coxeter mutations}.
In particular, for fixed flags $\flags$ on $\legendrian$, the set 
\[
\left\{\Psi(\ncoxeter^r(\ngraph(a,b,c), \nbasis(a,b,c)),\flags)\mid r\ge 0\right\}
\]
is the set of infinitely many pairwise distinct seeds in the cluster pattern for $\quiver(a,b,c)$.
Hence by Corollary~\ref{corollary:distinct seeds imples distinct fillings}, we have pairwise distinct Lagrangian fillings $L_r(a,b,c)$.
\end{proof}

\subsection{\texorpdfstring{$N$}{N}-graphs of type \texorpdfstring{$\dynADE$}{ADE}}

In this section, we will prove one of the main theorem.
\begin{theorem}\label{theorem:ADE type}
Let $\legendrian$ be a Legendrian knot or link which is either $\legendrian(\dynA_n)$ or $\legendrian(a,b,c)$ of type $\dynADE$.
Then it admits exact embedded Lagrangian fillings as many as seeds in its seed pattern of the same type.
\end{theorem}

Indeed, this theorem follows from the generalized questions.
\begin{question}
For given $N$-graph $\ngraph$ with a chosen set $\nbasis$ of cycles, can we take a Legendrian mutation as many times as we want? Or equivalently, after applying a mutation $\mutation_k$ on $(\ngraph, \nbasis)$, is the tuple $\mutation_k(\nbasis)$ still good in $\mutation_k(\ngraph)$?
\end{question}

This question has been raised previously in \cite[Remark~7.13]{CZ2020}.
One of the main reason making the question nontrivial is that 
the potential difference of geometric and algebraic intersections between two cycles.
More concretely, two cycles $\cycle_1$ and $\cycle_2$ as shown in Figure~\ref{fig:geometric and algebraic intersection}, can never be isotoped off to each other but their signed intersections following the rule in Figure~\ref{fig:I-cycle with orientation and intersections} vanishes. 
Hence in the corresponding quiver to the first local $N$-graph, there are no arrows between the corresponding vertices $1$ and $2$.
However, after a sequence of Move $\Move{II}$, we can deform $\cycle_2$ into $\cycle(e)$ for an edge $e$ as depicted in the third picture of Figure~\ref{fig:geometric and algebraic intersection}.
The mutation $\mutation_{\cycle(e)}$ transforms $\cycle_1$ to $\cycle_1'$, which is \emph{not} good and so it is not clear how to define a mutation $\mutation_{\cycle_1'}$.

\begin{figure}[ht]
\begin{tikzcd}
\begin{tikzpicture}
\begin{scope}

\draw[dashed] \boundellipse{0,0}{1.5}{1};

\draw [green, line cap=round, line width=5, opacity=0.5] (-1,0) to (-0.5,0);
\draw [green, line cap=round, line width=5, opacity=0.5] (-0.5,0) to (0,0.5);
\draw [green, line cap=round, line width=5, opacity=0.5] (-0.5,0) to (0,-0.5);

\draw [yellow, line cap=round, line width=5, opacity=0.5] (1,0) to (0.5,0);
\draw [yellow, line cap=round, line width=5, opacity=0.5] (0.5,0) to (0,0.5);
\draw [yellow, line cap=round, line width=5, opacity=0.5] (0.5,0) to (0,-0.5);

\draw[blue, thick] (-1.15,0.65) to (-0.5,0) to (0.5,0) to (1.15,0.65);
\draw[blue, thick] (-1.15,-0.65) to (-0.5,0) to (0.5,0) to (1.15,-0.65);
\draw[red, thick] (0,1) to (0,0.5) to (0.5,0) to (1,0) to (1.4,0.4);
\draw[red, thick] (0,1) to (0,0.5) to (-0.5,0) to (-1,0) to (-1.4,0.4);
\draw[red, thick] (0,-1) to (0,-0.5) to (0.5,0) to (1,0) to (1.4,-0.4);
\draw[red, thick] (0,-1) to (0,-0.5) to (-0.5,0) to (-1,0) to (-1.4,-0.4);
\draw[thick,red,fill=red] (0,1/2) circle (0.05);
\draw[thick,red,fill=red] (0,-1/2) circle (0.05);
\draw[thick,red,fill=red] (1,0) circle (0.05);
\draw[thick,red,fill=red] (-1,0) circle (0.05);
\draw[thick,black,fill=white] (0.5,0) circle (0.05);
\draw[thick,black,fill=white] (-0.5,0) circle (0.05);

\node at (-0.5,0) [above]{\small$\gamma_1$};
\node at (0.5,0) [above]{\small$\gamma_2$};
\end{scope}
\end{tikzpicture}
\arrow[r, mapsto, "\Move{II}"]&
\begin{tikzpicture}

\draw [green, line cap=round, line width=5, opacity=0.5] 
(-1,0) to (-0.5,0)
(-0.5,0) to[out=90, in=180] (0,0.5)
(-0.5,0) to[out=-90, in=180] (0,-0.5);

\draw [yellow, line cap=round, line width=5, opacity=0.5] 
(0,0.5) -- (0.5,0.5) to[out=0,in=0] (0.5,-0.5) -- (0,-0.5);

\draw[dashed] \boundellipse{0,0}{1.75}{1.25};
\draw[thick, red] 
(-1.55,0.55) -- (-1,0) -- (-0.5,0) to[out=90,in=180] (0,0.5) -- (0,1.25)
(0,0.5) -- (0.5,0.5) -- (1,1)
(0.5,0.5) -- (0.5,-0.5) -- (1,-1)
(0.5,-0.5) -- (0,-0.5) -- (0,-1.25)
(0,-0.5) to[out=180,in=-90] (-0.5,0)
(-1,0) -- (-1.55,-0.55);
\draw[thick, red, fill] (-1,0) circle (1.5pt)
(0,0.5) circle (1.5pt)
(0,-0.5) circle (1.5pt);
\draw[thick, blue] (-1.32,0.82) -- (-0.5,0) -- (0,0) -- (0.5,0.5) -- (0.5,1.18)
(0.5,0.5) to[out=0,in=0] (0.5,-0.5) -- (0.5,-1.18)
(0.5,-0.5) -- (0,0)
(-0.5,0) -- (-1.32,-0.82);
\draw[thick, blue, fill] (0,0) circle (1.5pt);
\draw[thick,black,fill=white] (0.5,0.5) circle (1.5pt)
(0.5,-0.5) circle (1.5pt)
(-0.5,0) circle (1.5pt);
\end{tikzpicture}
\arrow[dl, mapsto, "\Move{II}\circ\Move{II}"]&\\
\begin{tikzpicture}

\draw [green, line cap=round, line width=5, opacity=0.5] 
(-1,0) to (-0.5,0)
(-0.5,0) to[out=90, in=180] (0.5,0.5) -- (1,0.5)
(-0.5,0) to[out=-90, in=180] (0.5,-0.5) -- (1,-0.5);

\draw [yellow, line cap=round, line width=5, opacity=0.5] 
(1,0.5) -- (1,-0.5);

\draw[dashed] \boundellipse{0.25,0}{2}{1.5};
\draw[thick,red] (-1.6,0.6) -- (-1,0)-- (-0.5,0) to[out=90,in=180] (0.5,0.5) -- (0.5,1) -- (0,1.5)
(0.9,1.4) -- (0.5,1) -- (0.5,-1) -- (0.9,-1.4)
(0,-1.5) -- (0.5,-1)
(0.5,-0.5) to[out=180,in=-90] (-0.5,0)
(-1,0) -- (-1.6,-0.6);
\draw[thick,blue] (-1.38,0.88) -- (-0.5,0) -- (0,0) -- (0.5,0.5) to[out=180,in=180] (0.5,1) -- (0.5,1.5)
(0.5,1) to[out=0,in=90] (1,0.5) -- (1,-0.5) to[out=-90,in=0] (0.5,-1) -- (0.5,-1.5)
(0.5,-1) to[out=180,in=180] (0.5,-0.5) to (0,0)
(-0.5,0) -- (-1.38,-0.88)
(0.5,0.5) -- (1,0.5)
(0.5,-0.5) -- (1,-0.5);

\draw[thick, red, fill] (-1,0) circle (1.5pt);

\draw[thick, blue, fill] (0,0) circle (1.5pt)
(1,0.5) circle (1.5pt) (1,-0.5) circle (1.5pt);
\draw[thick,black,fill=white] (0.5,0.5) circle (1.5pt)
(0.5,-0.5) circle (1.5pt)
(-0.5,0) circle (1.5pt)
(0.5,1) circle (1.5pt)
(0.5,-1) circle (1.5pt);
\node at (1,0) [right]{$\cycle_2$};
\end{tikzpicture}
\arrow[r, mapsto, "\mutation_{\cycle_2}"]&
\begin{tikzpicture}
\draw[dashed] \boundellipse{0.25,0}{2}{1.5};
\draw [green, line cap=round, line width=5, opacity=0.5]
(-1,0) -- (-0.5,0) to[out=45,in=180] (0.5,0.5) to[out=0,in=90] (1,0) to[out=-90,in=0] (0.5,-0.5) to[out=180,in=-45](-0.5,0);
\draw [yellow, line cap=round, line width=5, opacity=0.5] (1,0) -- (1.5,0) ;
\draw[thick,red] 
(-1.6,0.6) -- (-1,0) -- (-0.5,0) to[out=45, in=180] (0.5,0.5) to[out=45,in=-45] (0.5,1) -- (0,1.5)
(-1.6,-0.6) -- (-1,0)
(-0.5,0) to[out=-45, in=180] (0.5,-0.5) to[out=-45,in=45] (0.5,-1) -- (0,-1.5)
(0.5,-1) -- (0.9, -1.4)
(0.5,0.5) -- (0.5,-0.5)
(0.5,1) -- (0.9,1.4);
\draw[thick,blue]
(-1.38,0.88) -- (-0.5,0) -- (0,0) -- (0.5,0.5) to[out=135, in=-135] (0.5,1) -- (0.5,1.5)
(0,0) -- (0.5,-0.5) to[out=-135,in=135] (0.5,-1) -- (0.5,-1.5)
(-1.38,-0.88) -- (-0.5,0)
(0.5,0.5) to[out=0,in=90] (1,0) to[out=-90,in=0] (0.5,-0.5)
(1,0) -- (1.5,0)
(0.5,1) to[out=0,in=90] (1.5,0) to[out=-90,in=0] (0.5,-1);
\draw[thick,blue, fill] 
(0,0)  circle (1.5pt) 
(1,0) circle (1.5pt)
(1.5,0) circle (1.5pt);
\draw[thick, red, fill] 
(-1,0)  circle (1.5pt); 
\draw[thick,black, fill=white]
(-0.5,0) circle (1.5pt)
(0.5,0.5) circle (1.5pt)
(0.5,1) circle (1.5pt)
(0.5,-0.5) circle (1.5pt)
(0.5,-1) circle (1.5pt); 

\node at (-0.5,0) [above]{\small$\gamma_1'$};
\node at (1.25,0) [above]{\small$\gamma_2'$};
\end{tikzpicture}
\end{tikzcd}
\caption{Non-disjoint cycles $\gamma_1$ and $\gamma_2$ with signed intersection number zero, and a mutation $\mutation_e$.}
\label{fig:geometric and algebraic intersection}
\end{figure}
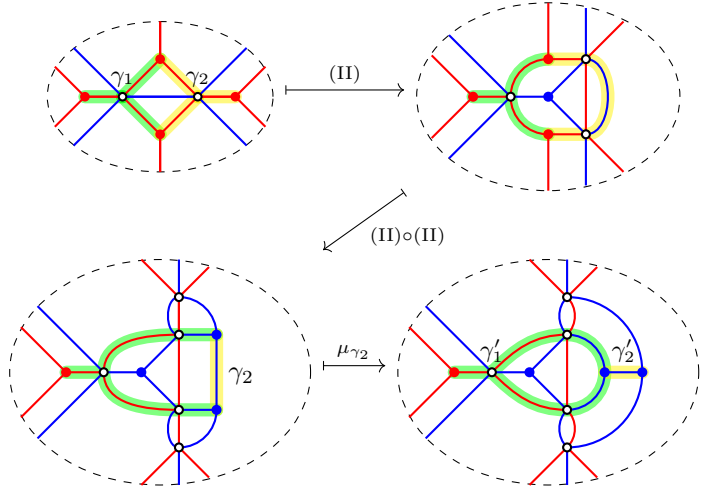

Instead of attacking this question directly, we will prove the following:

\begin{proposition}\label{proposition:realizability}
Let $\legendrian$ be as above and $\flags$ be flags on $\legendrian$.
Suppose that $\seed$ is a seed in the seed pattern of the same type with the initial seed
\[
\seed_{t_0} = \begin{cases}
\Psi(\ngraph(\dynA_n),\nbasis(\dynA_n),\flags) & \legendrian=\legendrian(\dynA_n);\\
\Psi(\ngraph(a,b,c),\nbasis(a,b,c),\flags) & \legendrian=\legendrian(a,b,c).
\end{cases}
\]
Then $\legendrian$ admits either an $N$-graph $(\ngraph, \nbasis)$ on $\disk^2$ such that 
$\ngraph$ is either a $2$-graph if $\legendrian=\legendrian(\dynA_n)$ or a $3$-graph if $\legendrian=\legendrian(a,b,c)$, and 
\[
\seed=\Psi(\ngraph, \nbasis, \flags).
\]
\end{proposition}

Under the aid of this proposition, one can prove Theorem~\ref{theorem:ADE type}.
\begin{proof}[Proof of Theorem~\ref{theorem:ADE type}]
Let $\legendrian$ be given as above.
Then by Proposition~\ref{proposition:realizability}, we have pairs of $N$-graphs and good tuples of cycles which have a one-to-one correspondence $\Psi$ with seeds in the seed pattern of $\quiver(a,b,c)$.
Hence any pair of the Lagrangian fillings coming from these $N$-graphs is never exact Lagrangian isotopic by Corollary~\ref{corollary:distinct seeds imples distinct fillings}.
This completes the proof.
\end{proof}

We will use the following observations: let $P(\Roots)$ be the generalized associahedron for the root system $\Roots$ of type $\dynADE$ (cf. Theorem~\ref{thm_FZ_finite_type} and \S\ref{sec_comb_of_exchange_graphs}).
\begin{enumerate}
\item There is one-to-one correspondence between the sets of vertices and seeds.
\item There is one-to-one correspondence between the set of facets, faces of codimension 1, and the set of almost positive roots $\Roots_{\ge -1}$.
\item For the initial seed $\initialseed$, we may assume that the facets of codimension one including $\initialseed$ correspond to negative simple roots. Namely, there are exactly $n$-facets 
\[
\facet = \{F_{-\alpha_i}\mid \alpha_i\in \SRoots\}.
\]
\item The orbits of $\facet$ under the action of the Legendrian Coxeter mutation $\qcoxeter$ exhaust all facets.
\end{enumerate}

\begin{proof}[Proof of Proposition~\ref{proposition:realizability}]
For a Legendrian link $\legendrian$ of type $\dynADE$, we fix flags $\flags$ on $\legendrian$.
Let us define the initial $2$-graph or $3$-graph with the chosen tuple of cycles $(\ngraph_{t_0}, \nbasis_{t_0})$ as 
\[
(\ngraph_{t_0}, \nbasis_{t_0})=\begin{cases}
(\ngraph(\dynA_n), \nbasis(\dynA_n)) & \legendrian=\legendrian(\dynA_n);\\
(\ngraph(a,b,c), \nbasis(a,b,c)) & \legendrian=\legendrian(a,b,c),
\end{cases}
\]
which defines the initial seed $\initialseed$ via $\Psi$
\begin{align*}
\seed_{t_0} = \Psi(\ngraph_{t_0}, \nbasis_{t_0},\flags) = (\bfx(\Legendrian(\ngraph_{t_0}), \nbasis_{t_0}, \flags), \quiver(\Legendrian(\ngraph_{t_0}), \nbasis_{t_0})).
\end{align*}

Suppose that $\seed$ is a seed in the cluster pattern. 
Then we need to to prove that there exists an $N$-graph $(\ngraph,\nbasis)$ such that $\Psi(\ngraph,\nbasis,\flags) = \seed$.

By Proposition~\ref{prop_FZ_finite_type_Coxeter_element} and Lemma~\ref{lemma:normal form}, there exist an integer $r$ and a sequence $\mutation'$ of mutations such that
\[
\seed=\mutation'(\qcoxeter^r(\initialseed)),
\]
where $\mutation'$ joins $\qcoxeter^r(\initialseed)$ and $\seed$ inside a facet.

If $\legendrian=\legendrian(\dynA_n)$, then $\ncoxeter$ acts on $(\ngraph_{t_0}, \nbasis_{t_0})$ as the $\left(\frac{2\pi}{n+2}\right)$-rotation, which obviously commutes with Legendrian mutation $\mutation'$.
Hence it suffices to show the well-definedness of $\mutation'(\ngraph_{t_0},\nbasis_{t_0})$.

Otherwise, as seen earlier, the action of the Legendrian Coxeter mutation $\ncoxeter^r$ on $(\ngraph_{t_0},\nbasis_{t_0})$ is obtained by the concatenation of sequences of $\coxeterpadding=\coxeterpadding(a,b,c)$ and $\bar \coxeterpadding=\bar\coxeterpadding(a,b,c)$ to either $(\ngraph_{t_0},\nbasis_{t_0})$ or $(\bar\ngraph_{t_0}, \bar\nbasis_{t_0})$.
\[
\ncoxeter^r(\ngraph_{t_0},\nbasis_{t_0}) = \begin{cases}
\coxeterpadding\bar\coxeterpadding\cdots \bar\coxeterpadding (\ngraph_{t_0},\nbasis_{t_0})& r\text{ is even},\\
\coxeterpadding\bar\coxeterpadding\cdots \coxeterpadding (\bar\ngraph_{t_0},\bar\nbasis_{t_0})& r\text{ is odd}.
\end{cases}
\]

Let us regard the sequence $\mutation'$ of mutations as the sequence of Legendrian mutations.
Since the concatenation of $\coxeterpadding$ or $\bar\coxeterpadding$ do not touch any chosen cycle in $\nbasis_{t_0}$, two operations---the concatenation of $\coxeterpadding$ or $\bar\coxeterpadding$, and the mutation $\mutation'$--- commute.
Therefore
\[
\mutation'(\ncoxeter^r(\ngraph_{t_0},\nbasis_{t_0}))=\begin{cases}
\coxeterpadding\bar\coxeterpadding\cdots \bar\coxeterpadding (\mutation'(\ngraph_{t_0},\nbasis_{t_0}))& r\text{ is even},\\
\coxeterpadding\bar\coxeterpadding\cdots \coxeterpadding (\mutation'(\bar\ngraph_{t_0},\bar\nbasis_{t_0}))& r\text{ is odd}, 
\end{cases}
\]
and the proposition follows if $\mutation'(\ngraph_{t_0},\nbasis_{t_0})$ and $\mutation'(\bar\ngraph_{t_0},\bar\nbasis_{t_0})$ are well-defined.
Since $\ngraph_{t_0}$ and $\bar\ngraph_{t_0}$ are essentially the same, it suffices to show the well-definedness of $\mutation'(\ngraph_{t_0},\nbasis_{t_0})$ as before.

Now we will prove the well-definedness of $\mutation'(\ngraph_{t_0}, \nbasis_{t_0})$ in both cases by using induction on $n$.
Suppose that $\mutation'$ is a sequence of mutations in a facet $F_{\beta}$ for some $\beta\in\Roots_{\ge-1}$, and $\qcoxeter^r(F_{-\alpha_i})=F_{\beta}$ for some $\alpha_i\in \SRoots$.
Then the facet $F_{\beta}$ is combinatorially equivalent to the lower dimensional
generalized associahedron
\[
F_{\beta}\cong P(\Roots([n] \setminus \{i\}))
\cong P(\Roots_1)\times\cdots \times P(\Roots_m).
\]
Here, $\Roots([n] \setminus \{i\})$ is not necessarily irreducible 
and we denote by $\Roots_1,\dots,\Roots_\ell, \ell\le 3$ the root systems 
satisfying that $\Roots([n] \setminus \{i\}) = \Roots_1 \times \cdots 
\times \Roots_\ell$. 
Moreover, in terms of quivers, if we denote the connected components of $\quiver\setminus\{i\}$ by $\quiver^{(1)},\dots,\quiver^{(\ell)}$, then we may say that $\Phi_j$ and $\quiver^{(j)}$ are of the same type.
Therefore the sequence $\mutation'$ of mutations can be decomposed into $\mutation^{(1)},\dots, \mutation^{(\ell)}$ on $\quiver^{(1)},\dots, \quiver^{(\ell)}$, respectively. 

Similarly, in $N$-graph $\ngraph_{t_0}$, the $i$-th cycle $[\cycle_i]$ separates $(\ngraph_{t_0}, \nbasis_{t_0})$ into at most three parts $\{(\ngraph_1,\nbasis_1), \dots, (\ngraph_\ell,\nbasis_\ell)\}$, as seen in Figure~\ref{figure:N graph separation}.
This means that 
\[
\mutation^{(j)}(\ngraph_j,\nbasis_j)\text{ is well-defined for all }1\le j\le\ell\Longrightarrow\mutation'(\ngraph_{t_0},\nbasis_{t_0})\text{ is well-defined}.
\]

Indeed, if $\legendrian=\legendrian(\dynA_n)$, then we have the following two cases:
\begin{enumerate}
\item if $\cycle_i$ corresponds to a bivalent vertex, then for some $1\le r,s$ with $r+s+1=n$, we have two $2$-subgraphs
\[
\{(\ngraph(\dynA_r),\nbasis(\dynA_r)), (\ngraph(\dynA_s),\nbasis(\dynA_s))\};
\]
\item if $\cycle_i$ corresponds to a leaf, then we have the $2$-subgraph
\[
\{(\ngraph(\dynA_{n-1}), \nbasis(\dynA_{n-1})\}.
\]
\end{enumerate}

Otherwise, if $\legendrian=\legendrian(a,b,c)$, then we have the following three cases:
\begin{enumerate}
\item if $\cycle_i$ corresponds to the central vertex, then we have three $3$-subgraphs
\[
\{(\ngraph_{(3)}(\dynA_{a-1}), \nbasis_{(3)}(\dynA_{a-1})), (\ngraph_{(3)}(\dynA_{b-1}), \nbasis_{(3)}(\dynA_{b-1})),(\ngraph_{(3)}(\dynA_{c-1}), \nbasis_{(3)}(\dynA_{c-1}))\},
\]
\item  if $\cycle_i$ corresponds to a bivalent vertex, then for some $1\le r,s$ with $r+s+1=a$, up to permuting indices $a,b,c$, we have two $3$-subgraphs
\[
\{(\ngraph_{(3)}(\dynA_s),\nbasis_{(3)}(\dynA_s)),(\ngraph(r,b,c),\nbasis(r,b,c))\},
\]
\item otherwise,  if $\cycle_i$ corresponds to a leaf, then up to permuting indices $a,b,c$, we have the $3$-subgraph
\[
\{(\ngraph(a-1,b,c), \nbasis(a-1,b,c))\}.
\]
\end{enumerate}
Some of separations are depicted in Figure~\ref{figure:N graph separation}.
Here, $\ngraph_{(3)}(\dynA_s)$ is the $3$-graph which looks like the $2$-graph $\ngraph(\dynA_s)$.
Indeed, there are no edges in red and so the well-definedness of each mutation on $\ngraph_{(3)}(\dynA_s)$ is the same as $\ngraph(\dynA_s)$.
Therefore we can safely replace $\ngraph_{(3)}(\dynA_s)$ with $\ngraph(\dynA_s)$ for the proof.

However, for each $1\le j\le \ell$, the $N$-subgraph $\ngraph_j$ is either $\ngraph(a',b',c')$ with $n'=a'+b'+c'-2$ or $\ngraph(\dynA_{n'})$, where $n'<n$.
Therefore the proposition follows from the induction on $n$ once we establish the initial step, which is when $n=1$, that is, either
\[
(\ngraph(1,1,1),\nbasis(1,1,1))\quad\text{ or }\quad (\ngraph(\dynA_1),\nbasis(\dynA_1)).
\]
Since there are no obstructions for mutations on these $N$-graphs, we are done for the initial condition for the induction.
\end{proof}

\begin{figure}[ht]
\begin{align*}
\begin{tikzpicture}[baseline=-.5ex,scale=0.7]
\draw[thick] (0,0) circle (3cm);
\draw[green, line cap=round, line width=5, opacity=0.5] (60:1) -- (50:1.5) (70:1.75) -- (50:2) (180:1) -- (170:1.5) (190:1.75) -- (170:2) (300:1) -- (290:1.5) (310:1.75) -- (290:2);
\draw[yellow, line cap=round, line width=5, opacity=0.5] (50:1.5) -- (70:1.75) (170:1.5) -- (190:1.75) (290:1.5) -- (310:1.75);
\begin{scope}
\clip(0,0) circle (3);
\draw[line width=7, dashed] (0,0) -- (60:1) (0,0) -- (180:1) (0,0) -- (300:1);
\draw[yellow!50, line cap=round, line width=5] (0,0) -- (60:1) (0,0) -- (180:1) (0,0) -- (300:1);
\foreach \i in {0,120,240} {
\begin{scope}[rotate=\i]
\draw[red, thick] (0,0) -- (0:3);
\draw[blue, thick, fill] (0,0) -- (60:1) circle (2pt) -- (100:3) (60:1) -- (50:1.5) circle (2pt) -- (20:3) (50:1.5) -- (70:1.75) circle (2pt) -- (80:3) (70:1.75) -- (50:2) circle (2pt) -- (40:3);
\draw[blue, thick, dashed] (50:2) -- (60:3);
\end{scope}
}
\end{scope}
\draw[thick, fill=white] (0,0) circle (2pt);
\end{tikzpicture}&\to
\begin{tikzpicture}[baseline=-.5ex,scale=0.7]
\foreach \i in {0,120,240} {
\begin{scope}[rotate=\i,shift=(60:0.5)]
\draw[green, line cap=round, line width=5, opacity=0.5] (60:1) -- (50:1.5) (70:1.75) -- (50:2) ;
\draw[yellow, line cap=round, line width=5, opacity=0.5] (50:1.5) -- (70:1.75);
\draw[thick](0:3) -- (0,0) -- (120:3) arc (120:0:3);
\draw[blue, thick, fill] (0,0) -- (60:1) circle (2pt) -- (100:3) (60:1) -- (50:1.5) circle (2pt) -- (20:3) (50:1.5) -- (70:1.75) circle (2pt) -- (80:3) (70:1.75) -- (50:2) circle (2pt) -- (40:3);
\draw[blue, thick, dashed] (50:2) -- (60:3);
\end{scope}
}
\draw (60:4) node[rotate=-30] {$(\ngraph_{(3)}(\dynA_{a-1}),\nbasis_{(3)}(\dynA_{a-1}))$};
\draw (180:4) node[rotate=90] {$(\ngraph_{(3)}(\dynA_{b-1}),\nbasis_{(3)}(\dynA_{b-1}))$};
\draw (300:4) node[rotate=30] {$(\ngraph_{(3)}(\dynA_{c-1}),\nbasis_{(3)}(\dynA_{c-1}))$};
\end{tikzpicture}\\
\begin{tikzpicture}[baseline=-.5ex,scale=0.7]
\draw[thick] (0,0) circle (3cm);
\draw[green, line cap=round, line width=5, opacity=0.5] (60:1) -- (50:1.5) (70:1.75) -- (50:2) (180:1) -- (170:1.5) (190:1.75) -- (170:2) (300:1) -- (290:1.5) (310:1.75) -- (290:2);
\draw[yellow, line cap=round, line width=5, opacity=0.5] (0,0) -- (60:1) (0,0) -- (180:1) (0,0) -- (300:1) (170:1.5) -- (190:1.75) (290:1.5) -- (310:1.75);
\draw[line width=7, dashed] (70:1.75) -- (50:1.5);
\draw[yellow!50, line cap=round, line width=5] (70:1.75) -- (50:1.5);
\draw[red, thick] (0,0) -- (0:3) (0,0) -- (120:3) (0,0) -- (240:3);
\draw[blue, thick, fill] (0,0) -- (60:1) circle (2pt) -- (100:3) (60:1) -- (50:1.5) circle (2pt) -- (20:3) (50:1.5) -- (70:1.75) circle (2pt) -- (80:3) (70:1.75) -- (50:2) circle (2pt) -- (40:3);
\draw[blue, thick, dashed] (50:2) -- (60:3);
\draw[blue, thick, fill] (0,0) -- (180:1) circle (2pt) -- (220:3) (180:1) -- (170:1.5) circle (2pt) -- (140:3) (170:1.5) -- (190:1.75) circle (2pt) -- (200:3) (190:1.75) -- (170:2) circle (2pt) -- (160:3);
\draw[blue, thick, dashed] (170:2) -- (180:3);
\draw[blue, thick, fill] (0,0) -- (300:1) circle (2pt) -- (340:3) (300:1) -- (290:1.5) circle (2pt) -- (260:3) (290:1.5) -- (310:1.75) circle (2pt) -- (320:3) (310:1.75) -- (290:2) circle (2pt) -- (280:3);
\draw[blue, thick, dashed] (290:2) -- (300:3);
\draw[thick, fill=white] (0,0) circle (2pt);
\end{tikzpicture}&\to
\begin{tikzpicture}[baseline=-.5ex,scale=0.7]
\draw[thick] (0,0) circle (3cm);
\draw[green, line cap=round, line width=5, opacity=0.5] (60:1) -- (50:1.5) (180:1) -- (170:1.5) (190:1.75) -- (170:2) (300:1) -- (290:1.5) (310:1.75) -- (290:2);
\draw[yellow, line cap=round, line width=5, opacity=0.5] (0,0) -- (60:1) (0,0) -- (180:1) (0,0) -- (300:1) (170:1.5) -- (190:1.75) (290:1.5) -- (310:1.75);
\draw[red, thick] (0,0) -- (0:3) (0,0) -- (120:3) (0,0) -- (240:3);
\draw[blue, thick, fill] (0,0) -- (60:1) circle (2pt) -- (90:3) (60:1) -- (50:1.5) circle (2pt) -- (30:3) (50:1.5) -- (60:3);
\draw[blue, thick, fill] (0,0) -- (180:1) circle (2pt) -- (220:3) (180:1) -- (170:1.5) circle (2pt) -- (140:3) (170:1.5) -- (190:1.75) circle (2pt) -- (200:3) (190:1.75) -- (170:2) circle (2pt) -- (160:3);
\draw[blue, thick, dashed] (170:2) -- (180:3);
\draw[blue, thick, fill] (0,0) -- (300:1) circle (2pt) -- (340:3) (300:1) -- (290:1.5) circle (2pt) -- (260:3) (290:1.5) -- (310:1.75) circle (2pt) -- (320:3) (310:1.75) -- (290:2) circle (2pt) -- (280:3);
\draw[blue, thick, dashed] (290:2) -- (300:3);
\draw[thick, fill=white] (0,0) circle (2pt);
\draw (0,-3.5) node {$(\ngraph(r,b,c),\nbasis(r,b,c))$};
\begin{scope}[xshift=5cm,yshift=-2cm]
\draw[green, line cap=round, line width=5, opacity=0.5] (70:1.5) -- (50:2);
\draw[thick](0:3) -- (0,0) -- (120:3) arc (120:0:3);
\draw[blue, thick, fill] (0,0) -- (70:1.5) circle (2pt) -- (90:3) (70:1.5) -- (50:2) circle (2pt) -- (30:3);
\draw[blue, thick, dashed] (50:2) -- (60:3);
\draw (0,-0.5) node {$(\ngraph_{(3)}(\dynA_s),\nbasis_{(3)}(\dynA_s))$};
\end{scope}
\end{tikzpicture}
\end{align*}
\caption{Separations of $(\ngraph(a,b,c),\nbasis(a,b,c))$ at $\cycle_i$}
\label{figure:N graph separation}
\end{figure}

\begin{remark}
In the above proof, it is not claimed that two mutations $\mutation'$ and $\ncoxeter$ commute.
Indeed, if we first mutate $(\ngraph_{t_0},\nbasis_{t_0})$ via $\mutation'$, then the result may not look like either $(\ngraph_{t_0},\nbasis_{t_0})$ or $(\bar\ngraph_{t_0},\bar\nbasis_{t_0})$ and hence $\ncoxeter$ will not work as expected.
\end{remark}

\section{Lagrangian fillings admitting cluster structures of type \texorpdfstring{$\dynBCFG$}{BCFG}}

In this section, we will construct cluster structures of type $\dynBCFG$ on certain $N$-graphs by using the folding of $N$-graphs.
Throughout this section, let us assume that a triple~$(\dynX,\dynY, G)$ is one of 
\begin{align*}
&(\dynA_{2n-1}, \dynB_n,\Z/2\Z),&
&(\dynD_{n+1}, \dynC_n,\Z/2\Z),&
&(\dynE_{6}, \dynF_4,\Z/2\Z),&
&(\dynD_{4}, \dynG_2,\Z/3\Z)
\end{align*}
and that the group $G$ is generated by $\tau$.

\begin{remark}
For $\quiver(\dynD_{n+1})=\quiver(n-1,2,2), \quiver(\dynE_6)=\quiver(2,3,3)$ and $\quiver(\dynD_4)=\quiver(2,2,2)$, we will use the labelling convention of tripods depicted in Figure~\ref{figure:tripod quiver}, which is different from the usual convention of Dynkin diagrams shown in Table~\ref{Dynkin}.
\end{remark}

\subsection{Rotations and \texorpdfstring{$N$}{N}-graphs of type \texorpdfstring{$\dynA_{2n-1}$}{A(2n-1)} and \texorpdfstring{$\dynD_{4}$}{D(4)}}
Let $(\ngraph, \nbasis)$ be a pair of a $2$- or $3$-graph and a good tuple of cycles of type $\dynX=\dynA_{2n-1}$ or $\dynD_4$, respectively.
We define a new pair $(\tau(\ngraph),\tau(\nbasis))$ such that
$\tau(\ngraph)$ and $\tau(\nbasis)$ are obtained by the $(2\pi/N)$-rotation on $(\ngraph, \nbasis)$.

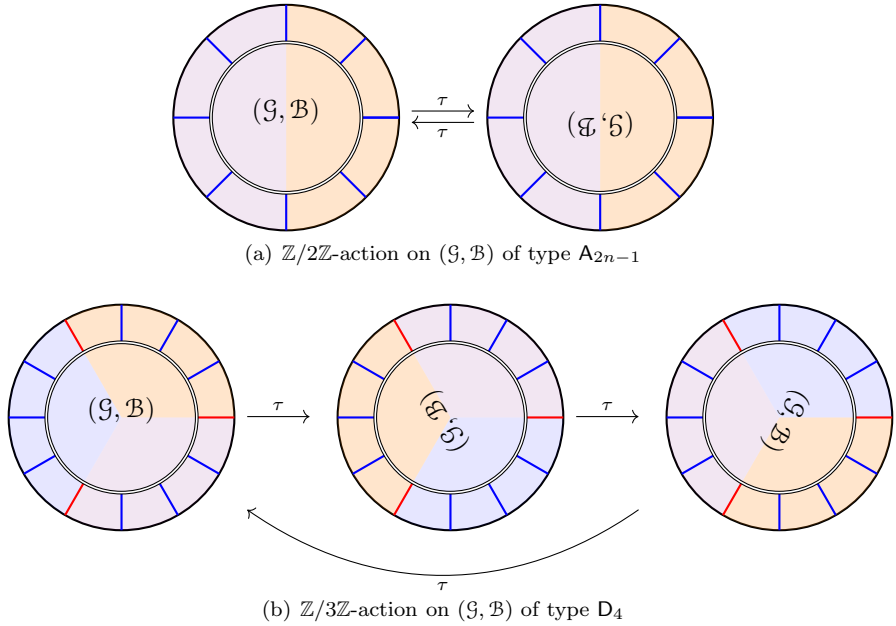
\begin{figure}[ht]
\subfigure[$\Z/2\Z$-action on $(\ngraph,\nbasis)$ of type $\dynA_{2n-1}$]{
\begin{tikzcd}[ampersand replacement=\&]
\begin{tikzpicture}[baseline=-.5ex, scale=0.5]
\draw[thick] (0,0) circle (3);
\fill[orange, opacity=0.2] (-90:3) arc(-90:90:3) -- cycle;
\fill[violet, opacity=0.1] (90:3) arc(90:270:3) -- cycle;
\foreach \i in {0,...,8} {
\draw[blue,thick] ({\i*45}:3) -- ({\i*45}:2);
}
\draw[double] (0,0) node {$(\ngraph,\nbasis)$} circle (2);
\end{tikzpicture}
\arrow[r,"\tau",yshift=.5ex]\&
\begin{tikzpicture}[baseline=-.5ex, scale=0.5]
\draw[thick] (0,0) circle (3);
\fill[orange, opacity=0.2] (90:3) arc(90:-90:3) -- cycle;
\fill[violet, opacity=0.1] (-90:3) arc(-90:-270:3) -- cycle;
\foreach \i in {0,...,8} {
\draw[blue,thick] ({\i*45}:3) -- ({\i*45}:2);
}
\draw[double] (0,0) node[rotate=180] {$(\ngraph,\nbasis)$} circle (2);
\end{tikzpicture}
\arrow[l, "\tau", yshift=-.5ex]
\end{tikzcd}
}

\subfigure[$\Z/3\Z$-action on $(\ngraph,\nbasis)$ of type $\dynD_4$]{
\begin{tikzcd}[ampersand replacement=\&]
\begin{tikzpicture}[baseline=-.5ex, scale=0.5]
\draw[thick] (0,0) circle (3);
\fill[orange, opacity=0.2] (0,0) -- (0:3) arc(0:120:3) -- cycle;
\fill[violet, opacity=0.1] (0,0) -- (-120:3) arc(-120:0:3) -- cycle;
\fill[blue, opacity=0.1] (0,0) -- (120:3) arc (120:240:3) -- cycle;
\foreach \i in {0, 120, 240} {
\begin{scope}[rotate=\i]
\draw[blue, thick] (30:3) -- (30:2) (60:3) -- (60:2) (90:3) -- (90:2);
\draw[red, thick] (0:3) -- (0:2);
\end{scope}
}
\draw[double] (0,0) node {$(\ngraph,\nbasis)$} circle (2);
\end{tikzpicture}
\arrow[r,"\tau"]\&
\begin{tikzpicture}[baseline=-.5ex, scale=0.5]
\draw[thick] (0,0) circle (3);
\fill[orange, opacity=0.2] (0,0) -- (120:3) arc(120:240:3) -- cycle;
\fill[violet, opacity=0.1] (0,0) -- (0:3) arc(0:120:3) -- cycle;
\fill[blue, opacity=0.1] (0,0) -- (-120:3) arc (-120:0:3) -- cycle;
\foreach \i in {0, 120, 240} {
	\begin{scope}[rotate=\i]
	\draw[blue, thick] (30:3) -- (30:2) (60:3) -- (60:2) (90:3) -- (90:2);
	\draw[red, thick] (0:3) -- (0:2);
	\end{scope}
}
\draw[double] (0,0) node[rotate=120] {$(\ngraph,\nbasis)$} circle (2);
\end{tikzpicture}
\arrow[r,"\tau"]\&
\begin{tikzpicture}[baseline=-.5ex, scale=0.5]
\draw[thick] (0,0) circle (3);
\fill[orange, opacity=0.2] (0,0) -- (-120:3) arc(-120:0:3) -- cycle;
\fill[violet, opacity=0.1] (0,0) -- (120:3) arc(120:240:3) -- cycle;
\fill[blue, opacity=0.1] (0,0) -- (0:3) arc (0:120:3) -- cycle;
\foreach \i in {0, 120, 240} {
\begin{scope}[rotate=\i]
\draw[blue, thick] (30:3) -- (30:2) (60:3) -- (60:2) (90:3) -- (90:2);
\draw[red, thick] (0:3) -- (0:2);
\end{scope}
}
\draw[double] (0,0) node[rotate=-120] {$(\ngraph,\nbasis)$} circle (2);
\end{tikzpicture}
\arrow[ll,"\tau", bend left=35]
\end{tikzcd}
}
\caption{Rotation actions on $N$-graphs of type $\dynA_{2n-1}$ and $\dynD_4$}
\label{figure:rotation action}
\end{figure}

We say that $(\ngraph, \nbasis)$ is \emph{$G$-admissible} if 
\begin{enumerate}
\item the $N$-graph $\ngraph$ has the $(2\pi/N)$-rotation symmetry so that $\tau(\ngraph)=\ngraph$,
\item the tuples of cycles $\nbasis$ and $\tau(\nbasis)$ are identical up to relabelling as follows: if $\dynX=\dynA_{2n-1}$,
\begin{align*}
\cycle_i&\leftrightarrow \cycle_{2n-i},
\end{align*}
and if $\dynX=\dynD_4$,
\begin{align*}
\cycle_2&\to \cycle_3,&
\cycle_3&\to \cycle_4,&
\cycle_4&\to \cycle_2.
\end{align*}
\end{enumerate}
In particular, $\tau$ preserves $\cycle_n$ if $\dynX=\dynA_{2n-1}$ and $\cycle_1$ if $\dynX=\dynD_4$.
Figure~\ref{figure:G-admissible N graphs for AD} shows examples and non-examples of $G$-admissible $N$-graphs.

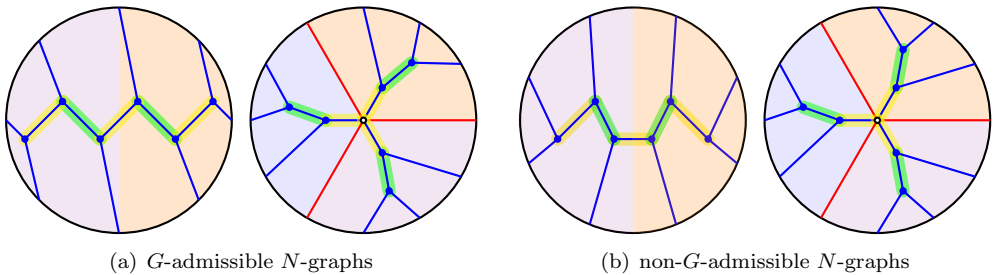
\begin{figure}[ht]
\subfigure[$G$-admissible $N$-graphs]{\makebox[0.45\textwidth]{
\begin{tikzpicture}[baseline=-.5ex,scale=0.5]
\fill[orange, opacity=0.2] (-90:3) arc(-90:90:3) -- cycle;
\fill[violet, opacity=0.1] (90:3) arc(90:270:3) -- cycle;
\draw[thick] (0,0) circle (3);
\draw[green,line cap=round, line width=5, opacity=0.5] (-1.5,0.5) -- (-0.5, -0.5) 
(0.5, 0.5) -- (1.5, -0.5);
\draw[yellow,line cap=round, line width=5, opacity=0.5] (-2.5,-0.5) -- (-1.5, 0.5) (-0.5, -0.5) -- (0.5, 0.5) (1.5, -0.5) -- (2.5, 0.5);
\draw[blue, thick, fill] (0:3) -- (2.5,0.5) circle (2pt) -- (45:3) (2.5,0.5) -- (1.5,-0.5) circle (2pt) -- (-45:3) (1.5,-0.5) -- (0.5,0.5) circle (2pt) -- (90:3) (0.5,0.5) -- (-0.5, -0.5) circle (2pt) -- (-90:3) (-0.5, -0.5) -- (-1.5, 0.5) circle (2pt) -- (135:3) (-1.5, 0.5) -- (-2.5, -0.5) circle (2pt) -- (-135:3);
\draw[blue, thick] (-2.5,-0.5) -- (-180:3);
\begin{scope}[xshift=6.5cm]
\draw[orange, opacity=0.2, fill] (0:3) arc(0:120:3) (120:3) -- (0,0) -- (0:3);
\draw[violet, opacity=0.1, fill] (0:3) -- (0,0) -- (-120:3) arc(-120:0:3) (0:3);
\draw[blue, opacity=0.1, fill] (120:3) arc(120:240:3) (240:3) -- (0,0) -- (120:3);
\draw[thick] (0,0) circle (3);
\foreach \i in {0,120,240} {
\begin{scope}[rotate=\i]
\draw[green, line cap=round, line width=5, opacity=0.5] (60:1) -- (50:2);
\draw[yellow, line cap=round, line width=5, opacity=0.5] (0,0)--(60:1);
\draw[red,thick](0:3) -- (0,0);
\draw[blue, thick, fill] (0,0) -- (60:1) circle (2pt) -- (90:3) (60:1) -- (50:2) circle (2pt) -- (30:3) (50:2) -- (60:3);
\end{scope}
}
\draw[thick,fill=white] (0,0) circle (2pt);
\end{scope}
\end{tikzpicture}
}}
\subfigure[non-$G$-admissible $N$-graphs]{\makebox[0.45\textwidth]{
\begin{tikzpicture}[baseline=-.5ex,scale=0.5]
\draw[thick] (0,0) circle (3);
\draw[yellow, line cap=round, line width=5, opacity=0.5] (-0.5, -0.5) -- (0.5,-0.5) (-2,-0.5) -- (-1, 0.5) (2, -0.5) -- (1, 0.5);
\draw[green, line cap=round, line width=5, opacity=0.5] (-0.5, -0.5) -- (-1,0.5) (0.5, -0.5) -- (1, 0.5);
\draw[blue, thick,fill] (22.5:3) -- (2, -0.5) circle(2pt) -- (-22.5:3) (2,-0.5) -- (1,0.5) circle (2pt) -- (67.5:3) (1,0.5) -- (0.5,-0.5) circle(2pt) -- (-67.5:3) (0.5,-0.5) -- (0,-0.5);
\begin{scope}[xscale=-1]
\draw[blue, thick,fill] (22.5:3) -- (2, -0.5) circle(2pt) -- (-22.5:3) (2,-0.5) -- (1,0.5) circle (2pt) -- (67.5:3) (1,0.5) -- (0.5,-0.5) circle(2pt) -- (-67.5:3) (0.5,-0.5) -- (0,-0.5);
\end{scope}
\fill[orange, opacity=0.2] (-90:3) arc(-90:90:3) -- cycle;
\fill[violet, opacity=0.1] (90:3) arc(90:270:3) -- cycle;

\begin{scope}[xshift=6.5cm]
\draw[orange, opacity=0.2, fill] (0:3) arc(0:120:3) (120:3) -- (0,0) -- (0:3);
\draw[violet, opacity=0.1, fill] (0:3) -- (0,0) -- (-120:3) arc(-120:0:3) (0:3);
\draw[blue, opacity=0.1, fill] (120:3) arc(120:240:3) (240:3) -- (0,0) -- (120:3);
\draw[thick] (0,0) circle (3);
\begin{scope}[yscale=-1, rotate=-120]
\draw[green, line cap=round, line width=5, opacity=0.5] (60:1) -- (50:2);
\draw[yellow, line cap=round, line width=5, opacity=0.5] (0,0)--(60:1);
\draw[blue, thick, fill] (0,0) -- (60:1) circle (2pt) -- (90:3) (60:1) -- (50:2) circle (2pt) -- (30:3) (50:2) -- (60:3);
\end{scope}
\draw[red,thick](0:3) -- (0,0);
\foreach \i in {120,240} {
\begin{scope}[rotate=\i]
\draw[green, line cap=round, line width=5, opacity=0.5] (60:1) -- (50:2);
\draw[yellow, line cap=round, line width=5, opacity=0.5] (0,0)--(60:1);
\draw[red,thick](0:3) -- (0,0);
\draw[blue, thick, fill] (0,0) -- (60:1) circle (2pt) -- (90:3) (60:1) -- (50:2) circle (2pt) -- (30:3) (50:2) -- (60:3);
\end{scope}
}
\draw[thick,fill=white] (0,0) circle (2pt);
\end{scope}
\end{tikzpicture}
}}
\caption{$G$-admissible or non-$G$-admissible $N$-graphs for Legendrians of type $\dynA_{2n-1}$ and $\dynD_4$}
\label{figure:G-admissible N graphs for AD}
\end{figure}

\begin{remark}\label{remark:coxeter padding symmetry 1}
Notice that the Coxeter padding $\coxeterpadding$ for $\ngraph(\dynA_{2n-1})$ is empty and $\coxeterpadding(2,2,2)$ has obviously the $\Z/3\Z$-rotational symmetry.
For each $G$-admissible $(\ngraph, \nbasis)$ of type $\dynA_{2n-1}$ or $\dynD_4$,
so is the following
\[
\coxeterpadding\cdots\bar\coxeterpadding\coxeterpadding(\ngraph,\nbasis)\quad\text{ or }\quad
\bar\coxeterpadding\cdots\bar\coxeterpadding\coxeterpadding(\ngraph,\nbasis).
\]
\end{remark}

\begin{lemma}\label{lemma:no weird cycles in An}
Let $\quiver$ be a quiver of type $\dynA_{2n-1}$.
Suppose that $\quiver$ is invariant under the action
\begin{align*}
\tau(i)&=2n-i
\end{align*}
for all $i\in[2n-1]$.
Then there is no oriented cycle of the form
\[
j\to i\to\tau(j)\to\tau(i)\to j
\]
for any $i,j\neq n$.
\end{lemma}
\begin{proof}
It is well known that any minimal cycle in $\quiver$ is of length 3.
Therefore, if such an oriented cycle exists, then there must be an edge $i-\tau(i)$ or $j-\tau(j)$ in $\quiver$. Hence $b_{i,\tau(i)}\neq0$ or $b_{j,\tau(j)}\neq0$ for $\qbasis=(b_{k,\ell})=\qbasis(\quiver)$.

This is impossible because $\quiver$ is $\Z/2\Z$-admissible and so
\[
b_{i,\tau(i)} = b_{\tau(i),\tau(\tau(i))} = b_{\tau(i), i} = -b_{i,\tau(i)}\quad\Longrightarrow\quad
b_{i,\tau(i)}=0.
\]
Therefore we are done.
\end{proof}

\begin{proposition}\label{proposition:admissibility for An}
Let $(\ngraph, \nbasis)$ be of type $\dynA_{2n-1}$ as above. If $(\ngraph, \nbasis)$ is $\Z/2\Z$-admissible, then so is the quiver $\quiver(\Legendrian(\ngraph), \nbasis)$.
\end{proposition}
\begin{proof}
For $\quiver = \quiver(\Legendrian(\ngraph),\nbasis)$, since the generator $\tau\in \Z/2\Z$ acts on $\nbasis$ as $\tau(\cycle_i)=\cycle_{2n-i}$, we have the $\Z/2\Z$-action on the set $[2n-1]$ of vertices of $\quiver$ as $\tau(i) = 2n-i$ and therefore for each $i\in[2n-1]$,
\[
i\sim 2n-i.
\]
We will check the conditions (a), (b), (c), and (d) for admissibility according to Definition~\ref{definition:admissible quiver}.

\noindent (a) Since all vertices in $\quiver$ are mutable, the condition (a) is obviously satisfied.

\noindent (b) Let $\qbasis = (b_{i,j}) = \qbasis(\quiver)$. Then for each $i, j\in[2n-1]$, the entry $b_{i,j}$ is given by the algebraic intersection number $(\cycle_i, \cycle_j)$, which is the same as $(\cycle_{2n-i},\cycle_{2n-j})$ since $\ngraph$ has the $\pi$-rotation symmetry.
Hence 
\[
b_{i,j} = b_{\tau(i),\tau(j)}.
\]

\noindent (c) On the other hand, for each $i\in[2n-i]$, we have 
\[
b_{i,\tau(i)}=(\cycle_i, \cycle_{\tau(i)})= (\cycle_{\tau(i)}, \cycle_{\tau(\tau(i))}) = (\cycle_{\tau(i)}, \cycle_i) = -b_{i,\tau(i)},
\]
which implies that 
\[
b_{i,\tau(i)}=0.
\]

\noindent (d) Finally, we need to prove that for each $i, j$,
\[
b_{i,j}b_{\tau(i),j}\ge 0.
\]

If $j=n$, then since $\tau(n)=n$, we have
\[
b_{i,n}b_{\tau(i),n}=b_{i,n}b_{\tau(i),\tau(n)} = b_{i,n}b_{i,n}\ge 0.
\]
Similarly, if $i=n$, then 
\[
b_{n, j}b_{\tau(n),j} = b_{n,j}b_{n,j}\ge 0.
\]

Suppose that for some $i, j\neq n$,
\[
b_{i,j}b_{\tau(i),j}<0.
\]
By changing the roles of $i$ and $\tau(i)$ if necessary, we may assume that $b_{i,j}<0<b_{\tau(i),j}$.
Then we also have
\[
b_{\tau(i),\tau(j)}<0<b_{i,\tau(j)},
\]
which implies that there is an oriented cycle in $\quiver$
\[
j\to i \to \tau(j) \to \tau(i) \to j.
\]
However, this contradicts to Lemma~\ref{lemma:no weird cycles in An} and therefore $\quiver$ satisfies all conditions in Definition~\ref{definition:admissible quiver}.
\end{proof}

Similarly, we have the following proposition as well.
\begin{proposition}\label{proposition:admissibility for D4}
Let $(\ngraph, \nbasis)$ be of type $\dynD_{4}$. If $(\ngraph, \nbasis)$ is $\Z/3\Z$-admissible, then so is the quiver $\quiver(\ngraph, \nbasis)$.
\end{proposition}
\begin{proof}
Let $\quiver=\quiver(\ngraph,\nbasis)$. Then by definition of the $\Z/3\Z$-action, we have
\[
2\sim3\sim 4.
\]

\noindent (a) and (b) This is obvious as before.

\noindent (c) Let $\qbasis=(b_{i,j})=\qbasis(\quiver)$. Suppose that $b_{2,3}\neq0$. Then by (b),
\[
b_{2,3}=b_{3,4}=b_{4,2}\neq0
\]
and so $\quiver$ has a directed cycle either
\[
2\to3\to4\to 2\quad\text{or}\quad 2\to4\to3\to 2.
\]
Then according to $b_{1,2}$, the underlying graph of the quiver $\quiver$ is either the complete graph $K_4$ or a disconnected graph, but both are impossible. Therefore
\[
b_{2,3}=b_{3,4}=b_{4,2}=0.
\]

\noindent (d) The only entries we need to check are $b_{1,j}$'s, which are all equal by (b). Therefore
\[
b_{1,j}b_{1,j'}\ge 0.\qedhere
\]
\end{proof}

\subsection{Partial rotations and \texorpdfstring{$N$}{N}-graphs of type \texorpdfstring{$\dynD_{n+1}$}{D(n+1)} and \texorpdfstring{$\dynE_{6}$}{E(6)}}

Recall from Table~\ref{table:ADE type} that 
the Legendrians $\legendrian(\dynD_{n+1})=\legendrian(n-1,2,2)$ and $\legendrian(\dynE_6)=\legendrian(2,3,3)$ whose braid representatives are
\begin{align*}
\beta(\dynD_{n+1})&=\sigma_2\sigma_1^n\sigma_2\sigma_1^3\sigma_2\sigma_1^3,&
\beta(\dynE_6)&=\sigma_2\sigma_1^3\sigma_2\sigma_1^4\sigma_2\sigma_1^4.
\end{align*}

Let us identify $\disk^2$ with the unit disk in $\C$ and define the ray $R_\theta$ in $\disk^2$ as
\[
R_\theta=\{(r,\theta)\in\disk\subset\C\mid 0\le r\le 1\}.
\]

In each case, we assume that three $\sigma_2$'s are on the end points of three rays $R_0, R_{2\pi/3}$ and $R_{4\pi/3}$, which are points $\{1, e^{2\pi/3}, e^{4\pi/3}\}\subset S^1\subset\C$.
We also assume that two same blocks of $\sigma_1^*$ in each $\legendrian$ are contained in the angle $[2\pi/3, 4\pi/3]$ and $[4\pi/3, 2\pi]$.

Let $(\ngraph, \nbasis)$ be a $3$-graph of type $\dynX=\dynD_{n+1}$ or $\dynE_6$.
We consider the intersection $(\ngraph,\nbasis)\cap R_\theta$ between $(\ngraph, \nbasis)$ with the ray $R_\theta$, which consists of colored points or intervals possibly together with labels $\cycle_j$.

\begin{definition}[Ray symmetry]
We say that the pair $(\ngraph, \nbasis)$ is \emph{ray-symmetric} if the intersections $(\ngraph,\nbasis)\cap R_\theta$ for $\theta=0, 2\pi/3, 4\pi/3$ are the same up to rotation and avoid all trivalent vertices and hexagonal points except at the origin.
\begin{equation}
(R_0, (\ngraph,\nbasis)\cap R_0) \cong (R_{2\pi/3}, (\ngraph,\nbasis)\cap R_{2\pi/3})\cong(R_{4\pi/3}, (\ngraph,\nbasis)\cap R_{4\pi/3}).
\end{equation}
\end{definition}

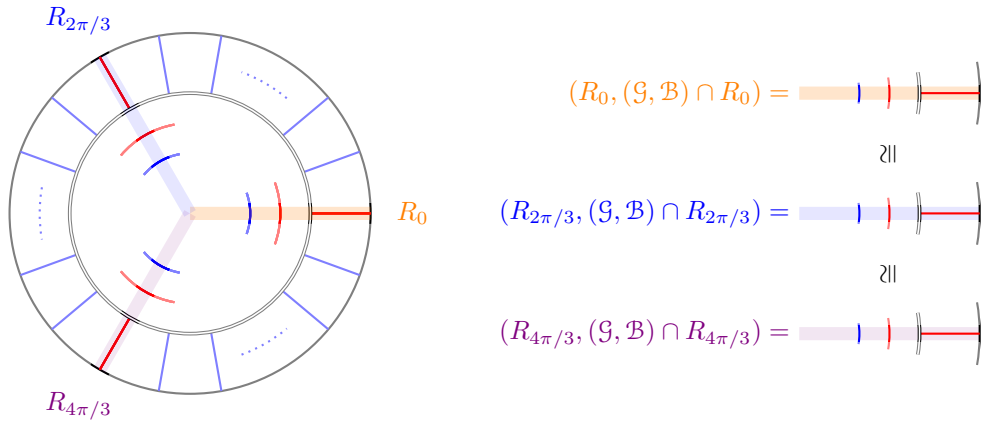
\begin{figure}[ht]
\[
\begin{tikzpicture}[baseline=-.5ex,scale=0.8]
\draw[thick] (0,0) circle (3);
\foreach \i in {0,120, 240} {
\begin{scope}[rotate=\i]
\foreach \j in {1,2,4,5} {
\draw[blue, thick] (\j*20:2) -- (\j*20:3);
\draw[red, thick] (0:2) -- (0:3);
\draw[red, thick] (-20:1.5) arc (-20:20:1.5);
\draw[blue, thick] (-20:1) arc (-20:20:1);
}
\draw[blue, thick, dotted] (50:2.5) arc (50:70:2.5);
\end{scope}
}
\draw[double] (0,0) circle (2);
\foreach \i in {0,120, 240} {
\begin{scope}[shift=({\i+60}:0.2)]
\fill[white,opacity=0.5] (0,0) -- (\i:3) arc (\i:{\i+120}:3) -- cycle;
\end{scope}
}
\draw[orange, opacity=0.2, line width=5] (0,0) -- (0:3);
\draw[blue, opacity=0.1, line width=5] (0,0) -- (120:3);
\draw[violet, opacity=0.1, line width=5] (0,0) -- (240:3);
\draw[orange] (0:3.7) node {$R_0$};
\draw[blue] (120:3.7) node {$R_{2\pi/3}$};
\draw[violet] (240:3.7) node {$R_{4\pi/3}$};
\end{tikzpicture}\qquad
\begin{tikzpicture}[baseline=-.5ex,scale=0.8]
\begin{scope}[yshift=2cm]
\draw[orange, opacity=0.2, line width=5] (0,0) -- (3,0);
\draw[orange] (0,0) node[left] {$(R_0, (\ngraph,\nbasis)\cap R_0)=$};
\draw[thick] (-10:3) arc (-10:10:3);
\draw[red, thick] (0:2) -- (0:3);
\draw[double] (-10:2) arc (-10:10:2);
\draw[red, thick] (-10:1.5) arc (-10:10:1.5);
\draw[blue, thick] (-10:1) arc (-10:10:1);
\fill[white, opacity=0.5] (0,0.15) rectangle (3,0.6);
\fill[white, opacity=0.5] (0,-0.15) rectangle (3,-0.6);
\end{scope}
\begin{scope}
\draw[blue, opacity=0.1, line width=5] (0,0) -- (3,0);
\draw[blue] (0,0) node[left] {$(R_{2\pi/3}, (\ngraph,\nbasis)\cap R_{2\pi/3})=$};
\draw[thick] (-10:3) arc (-10:10:3);
\draw[red, thick] (0:2) -- (0:3);
\draw[double] (-10:2) arc (-10:10:2);
\draw[red, thick] (-10:1.5) arc (-10:10:1.5);
\draw[blue, thick] (-10:1) arc (-10:10:1);
\fill[white, opacity=0.5] (0,0.15) rectangle (3,0.6);
\fill[white, opacity=0.5] (0,-0.15) rectangle (3,-0.6);
\end{scope}
\begin{scope}[yshift=-2cm]
\draw[violet, opacity=0.1, line width=5] (0,0) -- (3,0);
\draw[violet] (0,0)  node[left] {$(R_{4\pi/3}, (\ngraph,\nbasis)\cap R_{4\pi/3})=$};
\draw[thick] (-10:3) arc (-10:10:3);
\draw[red, thick] (0:2) -- (0:3);
\draw[double] (-10:2) arc (-10:10:2);
\draw[red, thick] (-10:1.5) arc (-10:10:1.5);
\draw[blue, thick] (-10:1) arc (-10:10:1);
\fill[white, opacity=0.5] (0,0.15) rectangle (3,0.6);
\fill[white, opacity=0.5] (0,-0.15) rectangle (3,-0.6);
\end{scope}
\draw (1.5,1) node[rotate=90] {$\cong$};
\draw (1.5,-1) node[rotate=90] {$\cong$};
\end{tikzpicture}
\]
\caption{Ray-symmetricity}
\label{figure:ray symmetricity}
\end{figure}

Then we define a $\Z/2\Z$-action on a ray-symmetric $(\ngraph,\nbasis)$ as follows:
\begin{enumerate}
\item cut $\disk^2$ into three sectors $\disk^2_1, \disk^2_2$ and $\disk^2_3$ along the rays $R_\theta$ for $\theta=0,2\pi/3$ and $4\pi/3$ so that $(\ngraph,\nbasis)$ gives us three $3$-subgraphs
\begin{align*}
&\{(\ngraph_1,\nbasis_1),(\ngraph_2,\nbasis_2),(\ngraph_3,\nbasis_3)\},&
\ngraph_i&=\ngraph\cap \disk^2_i,
\end{align*}
\item change two subgraphs contained in sectors whose angles are in between $[2\pi/3,4\pi/3]$ and $[4\pi/3, 2\pi]$ by rotating certain angles.
\item The result will be denoted by $(\tau(\ngraph), \tau(\nbasis))$.
\end{enumerate}

Notice that each subgraph $\ngraph_i$ may not satisfy the condition of $N$-graphs but the final result will be an well-defined $3$-graph since $\ngraph$ is ray-symmetric.
However, if $\ngraph$ is not ray-symmetric, then the $\Z/2\Z$-action is never well-defined.
We call this action the \emph{partial rotation} and see Figure~\ref{figure:partial rotation} for the pictorial definition.

\begin{figure}[ht]
\[
\begin{tikzcd}[ampersand replacement=\&, column sep=5pc, row sep=3pc]
\begin{tikzpicture}[baseline=-.5ex, scale=0.5]
\draw[thick] (0,0) circle (3);
\fill[orange, opacity=0.2] (0,0) -- (0:3) arc(0:120:3) -- cycle;
\fill[violet, opacity=0.1] (0,0) -- (-120:3) arc(-120:0:3) -- cycle;
\fill[blue, opacity=0.1] (0,0) -- (120:3) arc (120:240:3) -- cycle;
\foreach \j in {1,2,4,5} {
\draw[blue, thick] (\j*20:2) -- (\j*20:3);
\draw[red, thick] (0:2) -- (0:3);
}
\draw[blue, thick, dotted] (50:2.5) arc (50:70:2.5);
\foreach \i in {120, 240} {
\begin{scope}[rotate=\i]
\foreach \j in {1,3} {
\draw[blue, thick] (\j*30:2) -- (\j*30:3);
\draw[red, thick] (0:2) -- (0:3);
}
\draw[blue, thick, dotted] (50:2.5) arc (50:70:2.5);
\end{scope}
}
\draw[double] (0,0) circle (2);
\draw (60:1) node[rotate=-30] {$(\ngraph_1,\nbasis_1)$};
\begin{scope}[rotate=120]
\draw (60:1) node[rotate=90] {$(\ngraph_2,\nbasis_2)$};
\end{scope}
\begin{scope}[rotate=240]
\draw (60:1) node[rotate=-150] {$(\ngraph_3,\nbasis_3)$};
\end{scope}
\end{tikzpicture}
\arrow[r,"\tau", yshift=.5ex]\arrow[d, "\text{cut}"', xshift=-.5ex]\&
\begin{tikzpicture}[baseline=-.5ex, scale=0.5]
\draw[thick] (0,0) circle (3);
\fill[orange, opacity=0.2] (0,0) -- (0:3) arc(0:120:3) -- cycle;
\fill[blue, opacity=0.1] (0,0) -- (-120:3) arc(-120:0:3) -- cycle;
\fill[violet, opacity=0.1] (0,0) -- (120:3) arc (120:240:3) -- cycle;
\foreach \j in {1,2,4,5} {
\draw[blue, thick] (\j*20:2) -- (\j*20:3);
\draw[red, thick] (0:2) -- (0:3);
}
\draw[blue, thick, dotted] (50:2.5) arc (50:70:2.5);
\foreach \i in {120, 240} {
\begin{scope}[rotate=\i]
\foreach \j in {1,3} {
\draw[blue, thick] (\j*30:2) -- (\j*30:3);
\draw[red, thick] (0:2) -- (0:3);
}
\draw[blue, thick, dotted] (50:2.5) arc (50:70:2.5);
\end{scope}
}
\draw[double] (0,0) circle (2);
\draw (60:1) node[rotate=-30] {$(\ngraph_1,\nbasis_1)$};
\begin{scope}[rotate=120]
\draw (60:1) node[rotate=90] {$(\ngraph_3,\nbasis_3)$};
\end{scope}
\begin{scope}[rotate=240]
\draw (60:1) node[rotate=-150] {$(\ngraph_2,\nbasis_2)$};
\end{scope}
\end{tikzpicture}\arrow[l, "\tau", yshift=-.5ex]\arrow[d, "\text{cut}", xshift=.5ex]\\
\begin{tikzpicture}[baseline=-.5ex, scale=0.5]
\begin{scope}[shift=(60:1)]
\fill[orange, opacity=0.2] (0,0) -- (0:3) arc(0:120:3) -- cycle;
\draw[thick] (0:3) arc (0:120:3);
\foreach \j in {1,2,4,5} {
\draw[blue, thick] (\j*20:2) -- (\j*20:3);
\draw[red, thick] (0:2) -- (0:3);
\draw[red, thick] (120:2) -- (120:3);
}
\draw[double] (0:2) arc (0:120:2);
\draw[blue, thick, dotted] (50:2.5) arc (50:70:2.5);
\draw (60:1) node[rotate=-30] {$(\ngraph_1,\nbasis_1)$};
\end{scope}
\begin{scope}[rotate=120]
\begin{scope}[shift=(60:1)]
\fill[blue, opacity=0.1] (0,0) -- (0:3) arc (0:120:3) -- cycle;
\draw[thick] (0:3) arc (0:120:3);
\foreach \j in {1,3} {
\draw[blue, thick] (\j*30:2) -- (\j*30:3);
\draw[red, thick] (0:2) -- (0:3);
\draw[red, thick] (120:2) -- (120:3);
}
\draw[blue, thick, dotted] (50:2.5) arc (50:70:2.5);
\draw[double] (0:2) arc (0:120:2);
\draw (60:1) node[rotate=90] {$(\ngraph_2,\nbasis_2)$};
\end{scope}
\end{scope}
\begin{scope}[rotate=240]
\begin{scope}[shift=(60:1)]
\fill[violet, opacity=0.1] (0,0) -- (0:3) arc (0:120:3) -- cycle;
\draw[thick] (0:3) arc (0:120:3);
\foreach \j in {1,3} {
\draw[blue, thick] (\j*30:2) -- (\j*30:3);
\draw[red, thick] (0:2) -- (0:3);
\draw[red, thick] (120:2) -- (120:3);
}
\draw[blue, thick, dotted] (50:2.5) arc (50:70:2.5);
\draw[double] (0:2) arc (0:120:2);
\draw (60:1) node[rotate=-150] {$(\ngraph_3,\nbasis_3)$};
\end{scope}
\end{scope}
\end{tikzpicture}\arrow[r, "\text{partial rot.}", yshift=.5ex]\arrow[u, "\text{glue}"', xshift=.5ex]
\&
\begin{tikzpicture}[baseline=-.5ex, scale=0.5]
\begin{scope}[shift=(60:1)]
\fill[orange, opacity=0.2] (0,0) -- (0:3) arc(0:120:3) -- cycle;
\draw[thick] (0:3) arc (0:120:3);
\foreach \j in {1,2,4,5} {
\draw[blue, thick] (\j*20:2) -- (\j*20:3);
\draw[red, thick] (0:2) -- (0:3);
\draw[red, thick] (120:2) -- (120:3);
}
\draw[double] (0:2) arc (0:120:2);
\draw[blue, thick, dotted] (50:2.5) arc (50:70:2.5);
\draw (60:1) node[rotate=-30] {$(\ngraph_1,\nbasis_1)$};
\end{scope}
\begin{scope}[rotate=240]
\begin{scope}[shift=(60:1)]
\fill[blue, opacity=0.1] (0,0) -- (0:3) arc (0:120:3) -- cycle;
\draw[thick] (0:3) arc (0:120:3);
\foreach \j in {1,3} {
\draw[blue, thick] (\j*30:2) -- (\j*30:3);
\draw[red, thick] (0:2) -- (0:3);
\draw[red, thick] (120:2) -- (120:3);
}
\draw[blue, thick, dotted] (50:2.5) arc (50:70:2.5);
\draw[double] (0:2) arc (0:120:2);
\draw (60:1) node[rotate=-150] {$(\ngraph_2,\nbasis_2)$};
\end{scope}
\end{scope}
\begin{scope}[rotate=120]
\begin{scope}[shift=(60:1)]
\fill[violet, opacity=0.1] (0,0) -- (0:3) arc (0:120:3) -- cycle;
\draw[thick] (0:3) arc (0:120:3);
\foreach \j in {1,3} {
\draw[blue, thick] (\j*30:2) -- (\j*30:3);
\draw[red, thick] (0:2) -- (0:3);
\draw[red, thick] (120:2) -- (120:3);
}
\draw[blue, thick, dotted] (50:2.5) arc (50:70:2.5);
\draw[double] (0:2) arc (0:120:2);
\draw (60:1) node[rotate=90] {$(\ngraph_3,\nbasis_3)$};
\end{scope}
\end{scope}
\end{tikzpicture}
\arrow[l, "\text{partial rot.}", yshift=-.5ex]\arrow[u, "\text{glue}", xshift=-.5ex]
\end{tikzcd}
\]
\caption{Partial rotation}
\label{figure:partial rotation}
\end{figure}
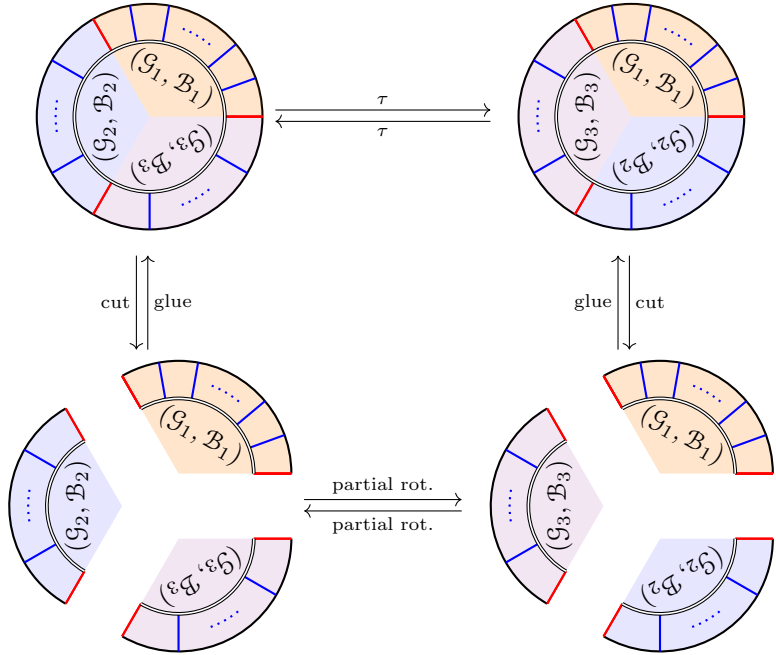

We say that $(\ngraph, \nbasis)$ is $\Z/2\Z$-admissible if it is invariant under the partial rotation up to relabeling of cycles as follows:
\begin{enumerate}
\item if $\dynX=\dynD_{n+1}$, then
\[
\cycle_n\leftrightarrow \cycle_{n+1},
\]
\item if $\dynX=\dynE_6$, then
\begin{align*}
\cycle_3&\leftrightarrow \cycle_5,&
\cycle_4&\leftrightarrow \cycle_6.
\end{align*}
\end{enumerate}

\begin{remark}\label{remark:coxeter padding symmetry 2}
Similar to Remark~\ref{remark:coxeter padding symmetry 1}, the Coxeter padding $\coxeterpadding=\coxeterpadding(n-1,2,2)$ for $\ngraph(\dynD_{n+1})$ or $\coxeterpadding(2,3,3)$ for $\ngraph(\dynE_6)$ has also the $\Z/2Z$-symmetry under the partial rotation.
Therefore for any $\Z/2\Z$-admissible $(\ngraph, \nbasis)$ under the partial rotation of type $\dynD_{n+1}$ or $\dynE_6$, so is the following
\[
\coxeterpadding\cdots\bar\coxeterpadding\coxeterpadding(\ngraph,\nbasis)\quad\text{ or }\quad
\bar\coxeterpadding\cdots\bar\coxeterpadding\coxeterpadding(\ngraph,\nbasis).
\]
\end{remark}

\begin{lemma}\label{lemma:no weird cycles in E6}
Let $\quiver$ be a quiver of type $\dynE_6$.
Suppose that $\quiver$ is invariant under the action
\begin{align*}
\tau(3)&=5,&
\tau(4)&=6.
\end{align*}
Then there is no oriented cycle, which is either
\[
3\to4\to5\to6\to3\quad\text{or}\quad
3\to6\to5\to4\to3.
\]
\end{lemma}

\begin{proof}
We will use the essentially same argument as the proof of Lemma~\ref{lemma:no weird cycles in An}.

If $\quiver=\mutation(\quiver(\dynE_6))$, then by Proposition~\ref{prop_FZ_finite_type_Coxeter_element} and Lemma~\ref{lemma:normal form}, there exist an integer $r$ and a sequence $\mutation'$ of mutations such that
\[
\quiver=\mutation'(\qcoxeter^r(\quiver(\dynE_6))) = \mutation'(\quiver(\dynE_6)).
\]
Moreover, $\mutation'$ misses at least one mutation $\mutation_i$.

\begin{enumerate}
\item If $i=1$, then $\quiver(\dynE_6)\setminus\{1\}$ consists of three quivers 
\begin{align*}
\{2\}&\subset\quiver(\dynA_1),&
\{3,4\}&\subset\quiver(\dynA_2),&
\{5,6\}&\subset\quiver(\dynA_2).
\end{align*}
In particular, there are no direct edges between two sets of vertices $\{3,4\}$ and $\{5,6\}$. 

\item If $i=2$, then $\mutation'$ can be regarded as a sequence of mutations of type $\dynA_5$. By Lemma~\ref{lemma:no weird cycles in An}, we are done.
\item If $i=3$, then we may assume that $\mutation'$ also misses $\mutation_5$ due to the symmetry. Hence we have separated quivers
\begin{align*}
\{1,2\}&\subset\quiver(\dynA_2),&
\{4\}&\subset\quiver(\dynA_1),&
\{6\}&\subset\quiver(\dynA_1).
\end{align*}
Then after the mutation, the vertices $4$ and $6$ can be joined only with $3$ and $5$, respectively, and so we never have an edge between $3$ and $6$ or between $4$ to $5$.
\item If $i=4$, then as above, we may assume that $\mutation'$ misses $\mutation_6$ as well and we may regard $\mutation'$ as a sequence of mutations on $\quiver(\dynD_4)$
\[
\{1,2,3,5\}\subset\quiver(\dynD_4).
\]
Moreover, $\mutation'$ consists of mutations corresponding to $\Z/2\Z$-orbits, which are $\mutation_1, \mutation_2$ and $\mutation_{3,5}=(\mutation_3\mutation_5)$.
Here the group $\Z/2\Z$ folds $\dynD_4$ onto $\dynC_3$, and up to Coxeter mutations, there are only three facets
\begin{align*}
F_{-\alpha_1}^{\dynC_3}&\cong P(\Roots(\dynC_2)),&
F_{-\alpha_2}^{\dynC_3}&\cong P(\Roots(\dynA_1))\times P(\Roots(\dynA_1)),&
F_{-\alpha_3}^{\dynC_3}&\cong P(\Roots(\dynA_2)),
\end{align*}

Hence all possible quivers are obtained by one of the following ways:
\begin{align*}
&\underbrace{\mutation_1\mutation_{3,5}\mutation_1\cdots}_{k}(\quiver(\dynD_4)),& 0\le k&\le 11\\
&\underbrace{\mutation_2\mutation_{3,5}\mutation_2\cdots}_{k}(\quiver(\dynD_4)),& 0\le k&\le 3\\
&\underbrace{\mutation_1\mutation_2\mutation_1\cdots}_{k}(\quiver(\dynD_4)),& 0\le k&\le 9
\end{align*}
Note that two quivers
\[
(\mutation_1\mutation_{3,5}\mutation_1\mutation_{3,5}\mutation_1\mutation_{3,5})(\quiver(\dynD_4))
\quad\text{ and }\quad
(\mutation_1\mutation_2\mutation_1\mutation_2\mutation_1)(\quiver(\dynD_4))
\]
are obtained from $\quiver(\dynD_4))$ by permuting vertices $3\leftrightarrow 5$ and $1\leftrightarrow 2$, respectively.
Finally, one can directly check that we have no such cycles by the exhaustive search in this full list.\qedhere
\end{enumerate}
\end{proof}

\begin{proposition}\label{proposition:admissibility for Dn and E6}
Let $(\ngraph, \nbasis)$ be of type $\dynX=\dynD_{n+1}$ or $\dynE_6$.
If $(\ngraph, \nbasis)$ is $\Z/2\Z$-admissible, then so is the quiver $\quiver(\ngraph,\nbasis)$.
\end{proposition}
\begin{proof}
\noindent (a) and (b): This is obvious as before.

\noindent (c) Let $\qbasis=(b_{i,j})=\qbasis(\quiver)$. Then by (b),
\[
b_{i,\tau(i)}=b_{\tau(i), \tau(\tau(i))} = b_{\tau(i), i} = -b_{i, \tau(i)}\quad
\Longrightarrow\quad
b_{i,\tau(i)}=0.
\]

\noindent (d) If $\dynX=\dynD_{n+1}$, then we only need to show
\[
b_{i,n}b_{i,n+1}\ge 0
\]
for $i<n$. This is obvious since 
\[
b_{i,n+1} = b_{\tau(i), \tau(n+1)} = b_{i,n}.
\]

If $\dynX=\dynE_6$, then all we need to show are inequalities
\begin{align*}
b_{i,j}b_{i,j+2}&\ge 0,&
b_{3,4}b_{3,6}&\ge 0
\end{align*}
for $i=1,2$ and $j=3,4$.

The first inequality is obvious since
\begin{align*}
b_{i,j+2}&=b_{\tau(i),\tau(j+2)} = b_{i,j}.
\end{align*}

Suppose that $b_{3,4}b_{3,6}<0$. Then since $b_{3,4}=b_{5,6}$ and $b_{3,6}=b_{5,4}$, the $\quiver$ has a loop either
\[
3\to 4\to 5\to 6 \to 3\quad\text{or}\quad
3\to 6\to 5\to 4 \to 3.\qedhere
\]
\end{proof}

\subsubsection{$\quiver(\dynA_{2n-1})$ as a tripod $\quiver(1,n,n)$}
As observed in Lemma~\ref{lemma:stabilized An}, one can think $\quiver(1,n,n)$ and $\ngraph(1,n,n)$ for $\dynA_{2n-1}$ instead of $\quiver(\dynA_{2n-1})$ and $\ngraph(\dynA_{2n-1})$.

The major difference is now we have to use $3$-graphs and partial rotations instead of $2$-graphs and $\pi$-rotations.
Then it can be easily checked that the above two notions are identical and so the $\Z/2\Z$-admissibility for $3$-graphs of type $\dynA_{2n-1}$ is also well-defined.
Moreover, the $3$-graph analogue under the partial rotation of Proposition~\ref{proposition:admissibility for An} will be true.

\subsection{Global foldability of \texorpdfstring{$N$}{N}-graphs}
Let $(\ngraph, \nbasis)$ be of type $\dynX$. We say that $(\ngraph, \nbasis)$ is \emph{globally foldable} with respect to $G$ if $(\ngraph, \nbasis)$ is $G$-admissible and for any sequence of mutable $G$-orbits $I_1,\dots, I_\ell$, there exists a $G$-admissible $(\ngraph', \nbasis')$ such that
\[
\quiver(\Legendrian(\ngraph'), \nbasis') = (\mutation_{I_\ell}\cdots\mutation_{I_1})(\quiver(\Legendrian(\ngraph),\nbasis)).
\]

\begin{theorem}\label{theorem:global foldability of N graphs}
The $N$-graph with a good tuple of cycles $(\ngraph(\dynX), \nbasis(\dynX))$ is globally foldable with respect to $G$.
\end{theorem}
\begin{proof}
Let us define the initial quiver $\quiver_{t_0}$ by 
\[
\quiver_{t_0} = \quiver(\Legendrian(\ngraph(\dynX)), \nbasis(\dynX)).
\]

For a sequence of mutable $G$-orbits $I_1,\dots, I_\ell$, we have an integer $r$ and $\mutation'^{\dynY}$ by Proposition~\ref{prop_FZ_finite_type_Coxeter_element} and Lemma~\ref{lemma:normal form} such that in the cluster pattern of type $\dynY$, two sequences of mutations
\[
\left((\mutation_{I_\ell}^{\dynX}\cdots\mutation_{I_1}^{\dynX})(\quiver_{t_0})\right)^{G} = \mutation'^{\dynY}\left((\qcoxeter^{\dynY})^r\left(\quiver_{t_0}^{G}\right)\right)
\]
will produce the same seed.

On the other hand, as seen in Remark~\ref{remark:folding and Coxeter mutation}, the Coxeter mutation $\qcoxeter^{\dynX}$ will correspond to the Coxeter mutation $\qcoxeter^{\dynY}$ via the folding.
Moreover, $\mutation'^{\dynY}$ comes from a sequence $\mutation'^{\dynX}$ of mutations via the folding such that $\mutation'^{\dynX}$ is the composition of mutations at $G$-orbits and happens inside some facet $F_\beta\subset P(\Roots(\dynX))$.
Hence we have
\begin{align*}
\mutation'^{\dynY}\left((\qcoxeter^{\dynY})^r\left(\quiver_{t_0}^{G}\right)\right)&=
\mutation'^{\dynY}\left(( (\qcoxeter^{\dynX})^r (\quiver_{t_0}))^{G}\right)=\left(\mutation'^{\dynX}((\qcoxeter^{\dynX})^r(\quiver_{t_0}))\right)^{G}.
\end{align*}

By Proposition~\ref{proposition:realizability}, there exists a pair $(\ngraph', \nbasis')$ satisfying that
\begin{align*}
\quiver(\Legendrian(\ngraph'), \nbasis')&=\quiver(\mutation'^{\dynX}((\qcoxeter^{\dynX})^r(\quiver_{t_0})))=(\mutation_{I_\ell}^{\dynX}\cdots\mutation_{I_1}^{\dynX})(\quiver_{t_0}).
\end{align*}

Finally, we need to show that $(\ngraph',\nbasis')$ can be assumed to be $G$-admissible.
As in the proof of Proposition~\ref{proposition:realizability}, $(\ngraph', \nbasis')$ is obtained by taking a $\mutation'^{\dynX}$ on either $(\ngraph(\dynX), \nbasis(\dynX))$ or $(\bar\ngraph(\dynX), \nbasis(\dynX))$ and attaching Coxeter paddings.
As observed in Remarks~\ref{remark:coxeter padding symmetry 1} and \ref{remark:coxeter padding symmetry 2}, the Coxeter paddings themselves are already $G$-admissible and the process attaching them preserve the $G$-admissibility in each case.
Therefore we only need to show the $G$-admissibility of $\mutation'^{\dynX}(\ngraph(\dynX), \nbasis(\dynX))$.

We will use the essentially same strategy as the proof of Proposition~\ref{proposition:realizability}.
Since $\mutation'^{\dynX}$ misses some $\mutation_{\cycle_i}$, it misses all $\mutation_{\cycle_{i'}}$ for $i\sim i'$.
Then one can split $(\ngraph(\dynX), \nbasis(\dynX))$ in a $G$-admissible way. That is, the set of $N$-subgraphs
\[
\{(\ngraph_1, \nbasis_1),\dots,(\ngraph_\ell, \nbasis_\ell)\}
\]
is closed under the $G$-action. In this case, $G$ may permute $N$-subgraphs as well.
Now we split $\mutation'^{\dynX}$ into $\{\mutation'_1,\dots, \mutation'_\ell\}$ such that each $\mutation'_i$ is a sequence of mutations of $\quiver(\ngraph_i,\nbasis_i)$.
Then $\mutation'^{\dynX}(\ngraph(\dynX),\nbasis(\dynX))$ is $G$-admissible if so is $\mutation'_i(\ngraph_i, \nbasis_i)$ for each $1\le i\le \ell$.

Since each $(\ngraph_i, \nbasis_i)$ is strictly simpler than $(\ngraph(\dynX), \nbasis(\dynX))$ in terms of the number of vertices and is again of type $\dynADE$, the rest of the proof follows from induction and we omit the detail.
\end{proof}

As a direct consequence, we will prove the following theorem:
\begin{theorem}\label{theorem:BCFG type}
The following holds:
\begin{enumerate}
\item The Legendrian link $\lambda(\dynA_{2n-1})$ has $\binom{2n}{n}$ $\Z/2\Z$-admissible $N$-graphs which admits the cluster pattern of type $\dynB_n$.
\item The Legendrian link $\lambda(\dynD_{n+1})$ has $\binom{2n}{n}$ $\Z/2\Z$-admissible $N$-graphs which admits the cluster pattern of type $\dynC_n$.
\item The Legendrian link $\lambda(\dynE_{6})$ has $105$ $\Z/2\Z$-admissible $N$-graphs which admits the cluster pattern of type $\dynF_4$.
\item The Legendrian link $\lambda(\dynD_{4})$ has $8$ $\Z/3\Z$-admissible $N$-graphs which admits the cluster pattern of type $\dynG_2$.
\end{enumerate}
\end{theorem}
\begin{proof}
By Theorem~\ref{theorem:global foldability of N graphs}, it is already known that for each $\dynX$, the quiver of type $\dynX$ is globally foldable with respect to $G$.
By Propositions~\ref{proposition:admissibility for An}, \ref{proposition:admissibility for D4} and \ref{proposition:admissibility for Dn and E6}, the quiver $\quiver(\ngraph(\dynX),\nbasis(\dynX))$ is also globally foldable with respect to $G$.

Let $\seed_{t_0}=\Psi(\ngraph(\dynX),\nbasis(\dynX),\flags)=(\bfx(\Legendrian(\ngraph(\dynX)),\nbasis(\dynX),\flags), \quiver(\Legendrian(\ngraph(\dynX)),\nbasis(\dynX)))$ be the initial seed.
Without loss of generality, we may denote cluster variables in $\bfx$ by
\[
\bfx=(x_1,\dots,x_{\operatorname{rk}(\dynX)})
\]
and we define a field homomorphism $\psi:\field=\C(x_1,\dots,x_{\operatorname{rk}(\dynX)})\to\field^G=\C(x_{I_1},\dots, x_{I_{\operatorname{rk}(\dynY)}})$ by
\[
\psi(x_i) = x_I
\]
for any $i$ in a $G$-orbit $I$.
Then by construction, the initial seed $\seed_{t_0}$ is $(G,\psi)$-admissible. See \S\ref{sec:folding}.

Finally, by Proposition~\ref{proposition:folded cluster pattern}, folded seeds form a seed pattern of type $\dynY$ as desired, and we are done.
\end{proof}

\bibliographystyle{plain}
\bibliography{references}

\begin{thebibliography}{10}

\bibitem{Arn1990}
V.~I. Arnold.
\newblock {\em Singularities of caustics and wave fronts}, volume~62 of {\em
  Mathematics and its Applications (Soviet Series)}.
\newblock Kluwer Academic Publishers Group, Dordrecht, 1990.

\bibitem{Aur2007}
Denis Auroux.
\newblock Mirror symmetry and {$T$}-duality in the complement of an
  anticanonical divisor.
\newblock {\em J. G\"{o}kova Geom. Topol. GGT}, 1:51--91, 2007.

\bibitem{BFFH2018}
L.~Bossinger, X.~Fang, G.~Fourier, M.~Hering, and M.~Lanini.
\newblock Toric degenerations of {${\rm Gr}(2, n)$} and {${\rm Gr}(3, 6)$} via
  plabic graphs.
\newblock {\em Ann. Comb.}, 22(3):491--512, 2018.

\bibitem{Bourbaki02}
Nicolas Bourbaki.
\newblock {\em Lie groups and {L}ie algebras. {C}hapters 4--6}.
\newblock Elements of Mathematics (Berlin). Springer-Verlag, Berlin, 2002.
\newblock Translated from the 1968 French original by Andrew Pressley.

\bibitem{CalderoKeller06}
Philippe Caldero and Bernhard Keller.
\newblock From triangulated categories to cluster algebras. {II}.
\newblock {\em Ann. Sci. \'{E}cole Norm. Sup. (4)}, 39(6):983--1009, 2006.

\bibitem{Cas2020}
Roger Casals.
\newblock Lagrangian skeleta and plane curve singularities.
\newblock arXiv:2009.06737, 2020.

\bibitem{CG2020}
Roger Casals and Honghao Gao.
\newblock Infinitely many lagrangian fillings.
\newblock arXiv:2001.01334, 2020.

\bibitem{CZ2020}
Roger Casals and Eric Zaslow.
\newblock Legendrian weaves.
\newblock arxiv:2007.04943, 2020.

\bibitem{CFZ02_polytopal}
Fr\'{e}d\'{e}ric Chapoton, Sergey Fomin, and Andrei Zelevinsky.
\newblock Polytopal realizations of generalized associahedra.
\newblock volume~45, pages 537--566. 2002.
\newblock Dedicated to Robert V. Moody.

\bibitem{Che2002}
Yuri Chekanov.
\newblock Differential algebra of {L}egendrian links.
\newblock {\em Invent. Math.}, 150(3):441--483, 2002.

\bibitem{EHK2016}
Tobias Ekholm, Ko~Honda, and Tam\'{a}s K\'{a}lm\'{a}n.
\newblock Legendrian knots and exact {L}agrangian cobordisms.
\newblock {\em J. Eur. Math. Soc. (JEMS)}, 18(11):2627--2689, 2016.

\bibitem{EGH2000}
Y.~Eliashberg, A.~Givental, and H.~Hofer.
\newblock Introduction to symplectic field theory.
\newblock Number Special Volume, Part II, pages 560--673. 2000.
\newblock GAFA 2000 (Tel Aviv, 1999).

\bibitem{FWZ_chapter45}
Sergey Fomin, Lauren Williams, and Andrei Zelevinsky.
\newblock Introduction to cluster algebras. chapters 4-5.
\newblock {\em arXiv preprint arXiv:1707.07190}, 2017.

\bibitem{FZ1_2002}
Sergey Fomin and Andrei Zelevinsky.
\newblock Cluster algebras. {I}. {F}oundations.
\newblock {\em J. Amer. Math. Soc.}, 15(2):497--529, 2002.

\bibitem{FZ2_2003}
Sergey Fomin and Andrei Zelevinsky.
\newblock Cluster algebras. {II}. {F}inite type classification.
\newblock {\em Invent. Math.}, 154(1):63--121, 2003.

\bibitem{FZ_Ysystem03}
Sergey Fomin and Andrei Zelevinsky.
\newblock {$Y$}-systems and generalized associahedra.
\newblock {\em Ann. of Math. (2)}, 158(3):977--1018, 2003.

\bibitem{FZ4_2007}
Sergey Fomin and Andrei Zelevinsky.
\newblock Cluster algebras. {IV}. {C}oefficients.
\newblock {\em Compos. Math.}, 143(1):112--164, 2007.

\bibitem{GSW2020a}
Honghao Gao, Linhui Shen, and Daping Weng.
\newblock Augmentations, fillings, and clusters.
\newblock arXiv:2008.10793, 2020.

\bibitem{GSW2020b}
Honghao Gao, Linhui Shen, and Daping Weng.
\newblock Positive braid links with infinitely many fillings.
\newblock arXiv:2009.00499, 2020.

\bibitem{Gei2008}
Hansj{\"o}rg Geiges.
\newblock {\em An introduction to contact topology}, volume 109.
\newblock Cambridge University Press, 2008.

\bibitem{GKS2012}
St\'{e}phane Guillermou, Masaki Kashiwara, and Pierre Schapira.
\newblock Sheaf quantization of {H}amiltonian isotopies and applications to
  nondisplaceability problems.
\newblock {\em Duke Math. J.}, 161(2):201--245, 2012.

\bibitem{Humphreys}
James~E. Humphreys.
\newblock {\em Introduction to {L}ie algebras and representation theory},
  volume~9 of {\em Graduate Texts in Mathematics}.
\newblock Springer-Verlag, New York-Berlin, 1978.
\newblock Second printing, revised.

\bibitem{Kal2006}
Tam\'{a}s K\'{a}lm\'{a}n.
\newblock Braid-positive {L}egendrian links.
\newblock {\em Int. Math. Res. Not.}, pages Art ID 14874, 29, 2006.

\bibitem{NRSSZ2015}
Lenhard Ng, Dan Rutherford, Vivek Shende, Steven Sivek, and Eric Zaslow.
\newblock Augmentations are sheaves.
\newblock arxiv:1502.04939, 2015.

\bibitem{Pan2017}
Yu~Pan.
\newblock Exact {L}agrangian fillings of {L}egendrian {$(2,n)$} torus links.
\newblock {\em Pacific J. Math.}, 289(2):417--441, 2017.

\bibitem{Pol1991}
L.~Polterovich.
\newblock The surgery of {L}agrange submanifolds.
\newblock {\em Geom. Funct. Anal.}, 1(2):198--210, 1991.

\bibitem{SW2019}
Linhui Shen and Daping Weng.
\newblock Cluster structures on double bott-samelson cells.
\newblock arXiv:1904.07992, 2019.

\bibitem{STWZ2019}
Vivek Shende, David Treumann, Harold Williams, and Eric Zaslow.
\newblock Cluster varieties from {L}egendrian knots.
\newblock {\em Duke Math. J.}, 168(15):2801--2871, 2019.

\bibitem{STZ2017}
Vivek Shende, David Treumann, and Eric Zaslow.
\newblock Legendrian knots and constructible sheaves.
\newblock {\em Invent. Math.}, 207(3):1031--1133, 2017.

\bibitem{TZ2018}
David Treumann and Eric Zaslow.
\newblock Cubic planar graphs and {L}egendrian surface theory.
\newblock {\em Adv. Theor. Math. Phys.}, 22(5):1289--1345, 2018.

\end{thebibliography}

\end{document}